\documentclass[oneside,final,11pt]{amsart}

\RequirePackage{amsthm,amsmath,amsfonts,amssymb}
\RequirePackage[colorlinks,citecolor=blue,urlcolor=blue]{hyperref}
\RequirePackage{graphicx}

\usepackage[marginparwidth=2.4cm, marginparsep=1mm]{geometry}
\usepackage{graphicx}
\usepackage{soul}
\usepackage{xcolor}
\usepackage{amssymb}
\usepackage{hyperref}
\usepackage{enumitem}
\usepackage{bm}
\usepackage{marginnote}
\usepackage{stmaryrd}
\usepackage{bbm}
\usepackage{caption}

\usepackage{tikz}
\usetikzlibrary{matrix,graphs,arrows,positioning, shapes, calc,decorations.markings,decorations.pathmorphing,shapes.symbols,hobby}

\usepackage{mathtools}
\usepackage{mathrsfs}
\usepackage[style=alphabetic,sorting=nty,natbib=true,maxnames=99,isbn=false,doi=false,url=false,firstinits=true,hyperref=auto,arxiv=abs,backend=bibtex]{biblatex}
\AtEveryBibitem{%
 \clearlist{language}}
\DeclareFieldFormat[article,inbook,incollection,inproceedings,patent,thesis,unpublished]{title}{#1\isdot}
\renewbibmacro{in:}{%
 \ifentrytype{article}{}{%
   \printtext{\bibstring{in}\intitlepunct}}} 
\defbibheading{apa}[\refname]{\section*{#1}}
\DeclareFieldFormat{sentencecase}{\MakeSentenceCase{#1}} \renewbibmacro*{title}{%
 \ifthenelse{\iffieldundef{title}\AND\iffieldundef{subtitle}}{}
 {\ifthenelse{\ifentrytype{article}\OR\ifentrytype{inbook}%
     \OR\ifentrytype{incollection}\OR\ifentrytype{inproceedings}%
     \OR\ifentrytype{inreference}} {\printtext[title]{%
       \printfield[sentencecase]{title}%
       \setunit{\subtitlepunct}%
       \printfield[sentencecase]{subtitle}}}%
   {\printtext[title]{%
       \printfield[titlecase]{title}%
       \setunit{\subtitlepunct}%
       \printfield[titlecase]{subtitle}}}%
   \newunit}%
 \printfield{titleaddon}}

\definecolor{darkblue}{rgb}{0.13,0.13,0.39}
\hypersetup{colorlinks=true,urlcolor=darkblue,citecolor=darkblue,linkcolor=darkblue, pdftitle=TwinPeaks, pdfauthor={K. Matetski}}

\mathtoolsset{showmanualtags,showonlyrefs}

\newtheorem{thm}{Theorem}[section]
\newtheorem{lem}[thm]{Lemma}
\newtheorem{lemma}[thm]{Lemma}
\newtheorem{assump}[thm]{Assumption}

\newtheorem{prop}[thm]{Proposition}
\newtheorem{proposition}[thm]{Proposition}
\newtheorem{cor}[thm]{Corollary}

\newtheorem{conjec}[thm]{Conjecture}

\theoremstyle{definition}
\newtheorem{rem}[thm]{Remark}
 \newtheorem*{rem*}{Remark}

\newtheorem{defn}[thm]{Definition} 
 \newcounter{assum}

\newcommand{\thmtitle}[1]{{\bf(#1)}}

\newcommand{\Corner}{\mathrm{Corner}}

\newcommand{\fav}{\mathsf{Fav}}

\renewcommand{\L}{\CP}
\newcommand{\WB}{W\!B}
\newcommand{\Wf}{W\!f}

\newcommand\dist{\stackrel{\uptext{\tiny{dist}}}{=}}

\newcommand\hit{\uptext{Hit}}
\newcommand\nohit{\uptext{No\! hit}}

\newcommand{\op}{\mathrm{op}}
\newcommand{\fh}{\mathfrak{h}}
\newcommand{\h}{\fh^{{\scriptscriptstyle(}n{\scriptscriptstyle)}}}
\newcommand{\hz}[1]{\fh^{{\scriptscriptstyle(}n{\scriptscriptstyle)},#1}}

\newcommand{\fg}{\mathfrak{g}}
\newcommand{\ff}{\mathfrak{f}}

\newcommand{\fd}{\mathfrak{d}}
\newcommand{\Airy}{Airy$_2$~}

\newcommand{\TTP}{\mathsf{TP}}
\newcommand{\TP}[1]{\mathsf{TP}_{\!\!#1}}
\newcommand{\TPno}{\mathsf{TP}}
\newcommand{\TPS}[1]{\mathsf{S}_{#1}}
\newcommand{\diam}{\mathrm{diam}}
\renewcommand{\d}{\mathrm{d}}

\renewcommand{\SS}{\mathbb{S}}
\newcommand{\Dom}{\mathsf{Dom}}
\newcommand{\mylim}[1]{\underset{#1}{\wt{\mathrm{lim}}}}
\newcommand{\intint}[1]{\llbracket1,#1\rrbracket}
\makeatletter
\newcommand*\bigcdot{\mathpalette\bigcdot@{.5}}
\newcommand*\bigcdot@[2]{\mathbin{\vcenter{\hbox{\scalebox{#2}{$\m@th#1\bullet$}}}}}
\makeatother

\newcommand{\bJ}{\mathsf{J}}

\newcommand{\I}{{\rm i}}

\newcommand{\R}{\mathbb{R}}
\newcommand{\nn}{\mathbb{N}}
\newcommand{\N}{\nn}
\newcommand{\zz}{\mathbb{Z}}
\newcommand{\aip}{\mathcal{A}}
\newcommand{\CT}{\mathcal{T}}
\newcommand{\CC}{\mathcal{C}}

\newcommand{\dens}{\CF}

\newcommand{\CH}{\mathcal{H}}
\newcommand{\CB}{\mathcal{B}}

\newcommand{\CL}{\mathcal{L}}

\newcommand{\CS}{\mathcal{S}}
\renewcommand{\S}{\CS}

\newcommand{\CX}{\mathcal{X}}

\newcommand{\1}[1]{\mathbf{1}_{#1}}

\newcommand{\eps}{\varepsilon}

\newcommand{\wt}{\widetilde}

\newcommand{\ts}{\hspace{0.1em}}
\newcommand{\tts}{\hspace{0.05em}}
\newcommand{\tsm}{\hspace{-0.1em}}

 \newcommand{\hlim}{\mathfrak{h}}
\renewcommand{\epsilon}{\varepsilon}
\newcommand{\F}{\mathscr{F}}
\DeclareMathOperator*{\argmax}{arg\,max}
\newcommand{\mc}{\mathcal}
\newcommand{\Q}{\mathbb Q}
\newcommand{\msf}{\mathsf}
\newcommand{\E}{\mathbb E}
\newcommand{\one}{\mathbbm{1}}

\newcommand{\as}{a.s.\ }
\newcommand{\smallconst}{c'}

\newcommand{\red}[1]{{\color{red}#1}}

\makeatletter
\newcommand\RedeclareMathOperator{%
  \@ifstar{\def\rmo@s{m}\rmo@redeclare}{\def\rmo@s{o}\rmo@redeclare}%
}
\newcommand\rmo@redeclare[2]{%
  \begingroup \escapechar\m@ne\xdef\@gtempa{{\string#1}}\endgroup
  \expandafter\@ifundefined\@gtempa
     {\@latex@error{\noexpand#1undefined}\@ehc}%
     \relax
  \expandafter\rmo@declmathop\rmo@s{#1}{#2}}
\newcommand\rmo@declmathop[3]{%
  \DeclareRobustCommand{#2}{\qopname\newmcodes@#1{#3}}%
}
\@onlypreamble\RedeclareMathOperator
\makeatother
\newcommand{\uptext}[1]{\text{\upshape{#1}}}
\DeclareMathOperator{\epi}{\uptext{epi}}

\DeclareMathOperator{\hypo}{\uptext{hypo}}

\DeclareMathOperator{\UC}{\uptext{UC}}
\DeclareMathOperator{\LC}{\uptext{LC}}
\DeclareMathOperator{\sgn}{sgn}
\DeclareMathOperator{\Ai}{\uptext{Ai}}
\DeclareMathOperator{\tr}{\mathop{\uptext{tr}}}
\RedeclareMathOperator{\det}{\mathop{\uptext{det}}}
\RedeclareMathOperator{\ker}{\mathop{\uptext{ker}}}
\RedeclareMathOperator{\exp}{\mathop{\uptext{exp}}}
\RedeclareMathOperator{\log}{\mathop{\uptext{log}}}
\RedeclareMathOperator*{\lim}{\mathop{\uptext{lim}}}
\RedeclareMathOperator*{\sup}{\mathop{\uptext{sup}}}
\RedeclareMathOperator*{\limsup}{\mathop{\uptext{lim\hspace{1pt}sup}}}
\RedeclareMathOperator*{\max}{\mathop{\uptext{max}}}
\RedeclareMathOperator*{\inf}{\mathop{\uptext{inf}}}
\RedeclareMathOperator*{\min}{\mathop{\uptext{min}}}

\newcommand{\Max}{\mathrm{Max}}

\newcommand{\Argmax}{\mathrm{Arg max}}

\newcommand{\fT}{\mathbf{S}}
\newcommand{\ft}{\mathbf{t}}

\newcommand{\fx}{x}
\newcommand{\bx}{\mathbf{x}}
\newcommand{\by}{\mathbf{y}}
\newcommand{\fM}{M}

\newcommand{\fD}{\mathbf{D}}
\newcommand{\fy}{y}

\newcommand{\fB}{\mathbf{B}}
\newcommand{\fP}{\mathbf{P}}
\newcommand{\fQ}{\mathbf{Q}}
\newcommand{\fA}{\mathbf{A}}
\newcommand{\fK}{\mathbf{K}}
\newcommand{\fR}{\mathbf{R}}
\newcommand{\fG}{\mathbf{G}}
\newcommand{\fH}{\mathbf{H}}

\newcommand{\CI}{\mathcal{I}}
\newcommand{\CJ}{\mathcal{J}}
\newcommand{\CA}{\mathcal{A}}
\newcommand{\CP}{\mathcal{P}}
\newcommand{\CF}{\mathcal{F}}
\renewcommand{\P}{\mathbb P}

\let\emptyset\varnothing

\def\dash---{\kern.16667em---\penalty\exhyphenpenalty\hskip.16667em\relax}

\numberwithin{equation}{section}

\title[Exceptional times when the KPZ fixed point violates Johansson's conjecture]{Exceptional times when the KPZ fixed point violates\\ Johansson's conjecture on maximizer uniqueness}
\date{\today}

\author[I. Corwin]{Ivan Corwin}
\address[I. Corwin]{Columbia University, Department of Mathematics, 2990 Broadway, Office 603, New York, NY 10027, USA}
\email{ic2354@columbia.edu}

\author[A. Hammond]{Alan Hammond}
\address[A. Hammond]{Departments of Mathematics and Statistics, University of California at Berkeley, 899 Evans Hall, Berkeley, CA, 94720-3840, USA}
\email{alanmh@stat.berkeley.edu}

\author[M. Hegde]{Milind Hegde}
\address[M. Hegde]{Department of Mathematics, University of California at Berkeley, 1039 Evans Hall, Berkeley, CA, 94720-3840, USA}
\email{milind.hegde@berkeley.edu}

\author[K. Matetski]{Konstantin Matetski}
\address[K. Matetski]{Columbia University, Department of Mathematics, 2990 Broadway, Office 517, New York, NY 10027, USA}
\email{matetski@math.columbia.edu}

\begin{document}

\begin{abstract}
In 2002, Johansson conjectured that the maximum of the \Airy process minus the parabola $x^2$ is almost surely achieved at a unique location \cite[Conjecture 1.5]{JohanssonPoly}. This result was proved a decade later by Corwin and Hammond \cite[Theorem 4.3]{BrownianGibbs}; Moreno Flores, Quastel and Remenik \cite{MaxOfAiry2}; and Pimentel \cite{pimentel_2014}. Up to scaling, the \Airy process minus the parabola $x^2$ arises as the fixed time spatial marginal of the KPZ fixed point when started from  {\it narrow wedge} initial data. We extend this maximizer uniqueness result to the fixed time spatial marginal of the KPZ fixed point when begun from any element of a very broad class of initial data.
%
%

None of these results rules out the possibility that at {\it random} times, the KPZ fixed point spatial marginal violates maximizer uniqueness. To understand this possibility, we study the probability that the KPZ fixed point has, at a given time, two or more locations where its value is close to the maximum, obtaining quantitative upper and lower bounds in terms of the degree of closeness for a very broad class of initial data. We also compute a quantity akin to the joint density of the locations of two maximizers and the maximum value. 
As a consequence, the set of times of maximizer non-uniqueness almost surely has Hausdorff dimension at most two-thirds. 

Our analysis relies on the exact formula for the distribution function of the KPZ fixed point obtained by Matetski, Quastel and Remenik in \cite{fixedpt}, the variational formula for the KPZ fixed point involving the Airy sheet constructed by Dauvergne, Ortmann and Vir\'ag in \cite{Landscape}, and the Brownian Gibbs property for the \Airy process minus the parabola $x^2$ demonstrated by Corwin and Hammond in \cite{BrownianGibbs}. 

\end{abstract}

\maketitle

\tableofcontents


\section{Introduction and main result}

The Kardar-Parisi-Zhang (KPZ) fixed point is a Markovian temporal evolution on spatial functions which is believed to be the universal limit of a wide class of random models of interface growth. Begun from narrow wedge initial data, the KPZ fixed point at given time has an almost surely unique maximizer. This is Johansson's conjecture~\cite{JohanssonPoly}, proved in \cite{BrownianGibbs,MaxOfAiry2,pimentel_2014}. An event that has zero probability at any given time may however occur on a set of exceptional times of vanishing measure.

In this work, we initiate the study of such exceptional times of non-uniqueness for the KPZ fixed point. We consider a general class of initial data for the KPZ fixed point and investigate the times at which the spatial process has two or more maximizers. Our main result, stated in Section~\ref{sec_main_result}, asserts that, for any fixed time, this non-uniqueness has zero probability (so that Johansson's conjecture is in fact valid for general initial data), while the set of such random times has Hausdorff dimension at most two-thirds almost surely.

In an earlier version of this article our main result was that the Hausdorff dimension equals $\frac{2}{3}$ on the event of the set of exceptional times being non-empty, which we claimed has positive probability. Unfortunately a flaw in our argument made these earlier claims out of reach, so we have revised our manuscript to prove an upper bound only. In the meantime these claims and several extensions which we conjectured (see Section~\ref{sec:conjectures}) have been proven by Dauvergne in \cite{dauvergne2022non}.

In recent years, research into the KPZ universality class has been invigorated by several fundamental new techniques including the Brownian Gibbs resampling property \cite{BrownianGibbs}, the KPZ fixed point transition probability formulas \cite{fixedpt}, and the variational representation of the directed landscape \cite{Landscape}. As we explain in Section~\ref{sec:heuristic}, our work is the first instance in which all three of these approaches have been brought together. In particular, in order to establish the two-thirds upper bound on the Hausdorff dimension, we compute and bound certain iterated derivatives of the Fredholm determinant formulas giving the exact transition probability formulas for the KPZ fixed point. To establish lower bounds on the probability of a given time being one of the existence of two or more near-maximizer (which would be one of the estimates needed to show that the set of exceptional times is non-empty with positive probability or establish a dimension lower bound, though we do not to show these),
 we make use of a pre-limiting version of the parabolic Airy sheet (a marginal of the directed landscape) for Brownian last passage percolation which can be encoded in terms of a variational problem for $\beta=2$ Dyson Brownian motion. The $\beta=2$ Dyson Brownian motion enjoys the Brownian Gibbs property which allows us to closely compare its law to that of independent Brownian motions. Filtering this through the variational problem, we are able to provide lower bounds on the probability that there are two or more approximate maximizers.

We close the introduction in Section~\ref{sec:conjectures} with some conjectures for other exceptional sets.
The concerned problems include showing that the exceptional set of times of maximizer non-uniqueness is almost surely dense in $(0,\infty)$; establishing the Hausdorff dimension of times with three (or a given greater number of) maximizers; and probing exceptional sets involving non-uniqueness of geodesics in the directed landscape.



\subsection{Introducing the KPZ fixed point through last passage percolation}

The KPZ universality class is a large collection of one-dimensional randomly forced systems, whose members may loosely be characterized by three features: slope-dependent lateral growth; smoothing behavior; and local random perturbations. Examples are last passage percolation; stochastic interface growth on a one-dimensional substrate; directed polymers in random environment; driven lattice gas models; and reaction-diffusion models in two-dimensional random media: see for instance \cite{corwinReview,quastelCDM,quastelSpohn,hht,Takeuchi}. These systems are expected to evince universal behavior---in particular, scaled universal limiting structures.
The KPZ fixed point is believed to be the universal limit of all KPZ class systems. A natural aspect of its fractal geometry will be the object of our attention.


The precise definition of the KPZ fixed point is somewhat involved and will be given in Section \ref{sec:KPZ_fp}. For the pedagogical purpose of this introduction, we begin with a very concrete model, geometric \emph{last passage percolation} (LPP) on $\zz^2$, which is essentially known to converge to the KPZ fixed point. By working with this model, we hope to illustrate some of the important concepts related to the KPZ fixed point and to motivate our main results, which come in Section \ref{sec_main_result}. Since our discussion on LPP is purely for illustrative purposes and our main results are for the KPZ fixed point, we will not precisely state or qualify all results mentioned.

\medskip
Consider a family of independent geometric random variables $\{\omega_{i,j}\}_{i,j \in \zz}$ with parameter $q \in (0,1)$, i.e., $\P(\omega_{i,j} = k) = q(1-q)^k$ for integer $k \geq 0$. For two points $\by, \bx \in \zz^2$ such that $\by \leq \bx$ entry-wise, we define the set $\Pi_{\by \to \bx}$ of all upright paths $\pi = (\pi_0, \pi_1, \ldots)$ in $\zz^2$ from $\by$ to $\bx$; i.e., paths starting at $\pi_0 = \by$, ending at $\bx$ and such that $\pi_i - \pi_{i-1} \in \{(1,0), (0,1)\}$. Then the \emph{point-to-point last passage time} is
\begin{equation}\label{eq:LPP_time}
G_{\by \to \bx} := \max_{\pi \in \Pi_{\by \to \bx}} \sum_{i = 0}^{|\bx - \by|} \omega_{\pi_i},
\end{equation}
where $|\bx - \by| := |x_1 - y_1| + |x_2 - y_2|$. One can readily check that the last passage times satisfy the recurrence relation
\begin{equation}
G_{\by \to \bx} = \max \bigl\{G_{\by \to \bx - (1,0)}, G_{\by \to \bx - (0,1)}\bigr\} + \omega_{\bx},
\end{equation}
for all points $\bx$ lying to the right and above $\by$. These last passage times are random variables, since they depend on the environment $\{\omega_{i,j}\}_{i,j \in \zz}$. A path $\pi \in \Pi_{\by \to \bx}$ can be considered to be a discrete directed path going from $\by$ to $\bx$ whose energy is given by the sum in \eqref{eq:LPP_time}. Then the quantity $G_{\by \to \bx}$ is the maximal energy of such paths, and the random path achieving the maximum can be regarded as a zero-temperature random directed polymer (with given endpoints).

\begin{figure}[h]
\centering
\scalebox{0.8}{\includegraphics{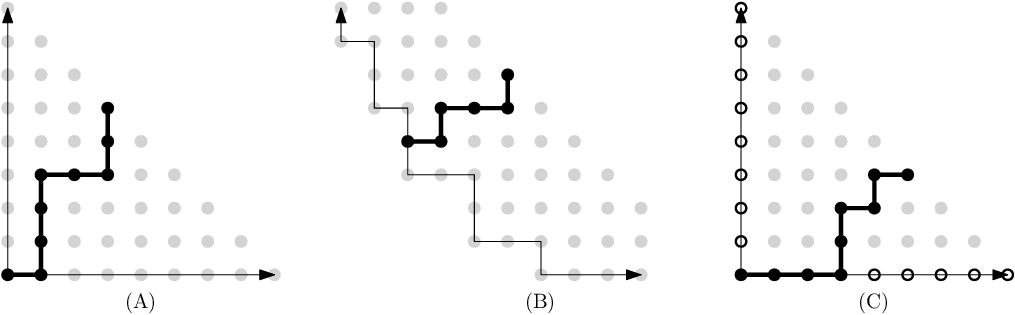}}
 \captionsetup{width=.9\linewidth}
\caption{(A): LPP with a specified starting and ending point. The grey bullets are associated with random weights $\omega_{i,j}$ and the thick black line is the upright path $\pi$ which maximizes the sum of weights along it. The quantity $H_N(x)$, defined by \eqref{eq:H}, measures the centred and scaled last passage time at various points along the anti-diagonal line through $(N,N)$. (B): In this version of LPP, defined in \eqref{eq:LPP general starting points}, we allow for the starting points to be anywhere along the thin black line on the lower left of the figure. (C): In this version, the weights $\omega$ on the boundary of the quadrant, depicted by black circles, may have a different law than those in the bulk, and may, in fact, be inhomogeneous.}
\label{Fig:LPPs}
\end{figure}

\subsubsection{The \Airy process and Johansson's uniqueness conjecture}

Let us fix the starting point of polymers $\by = (0,0)$ as depicted in Figure \ref{Fig:LPPs}(A). Then, for an integer $N \geq 1$, we define the random function $x \mapsto H_N(x)$ by linearly interpolating the values obtained by the identity
\begin{equation}\label{eq:H}
G_{(0,0) \to (N + x, N - x)} = c_1 N + c_2 N^{\frac{1}{3}} H_N \bigl( c_3 N^{-\frac{2}{3}} x\bigr),
\end{equation}
where $c_1 = 2\sqrt q / (1 - \sqrt q)$, $c_2 = q^{1/6} (1 + \sqrt q)^{1/3} / (1-q)$ and $c_3 = (1 - \sqrt q) c_2 / (1 + \sqrt q)$. The function $H_N(x)$ weakly converges as $N \to \infty$ to $\CA(x) - x^2$ in the topology of uniform convergence on compact sets \cite{JohanssonPoly}, where $\CA$ is an \Airy process, first defined in \cite{prahoferSpohn}. The process $\CA(x)$ is stationary and has marginal distributions equal to the Tracy-Widom GUE distribution \cite{tracyWidom, tracyWidom2}; we refer to \cite{cqr} and \cite{quastelRem-review} for a review of the \Airy process.

The just introduced polymers run point-to-point, and we may consider polymers whose starting point is fixed at $\by=(0,0)$, but whose ending point is free to lie anywhere along a given line. More precisely, we consider the maximum of the left-hand side of \eqref{eq:H} over all $x=-N, \ldots, N$. The path which attains this maximum energy is a random directed polymer with \emph{unconstrained} ending point. The random endpoint is the  maximizer of the function $H_N$, and it is natural to enquire into its law.
A first query in this vein is whether the random endpoint is uniquely defined in the large $N$ limit; this effectively amounts to determining whether the process  $\CA(x) - x^2$ has a unique maximizer.

Johansson conjectured \cite[Conjecture 1.5]{JohanssonPoly} that, indeed, the maximizer of $\CA(x) - x^2$ is almost surely unique. Predicated on the validity of this conjecture, \cite[Theorem 1.6]{JohanssonPoly} demonstrated that the random variables
\begin{equation}
\CX_N := \inf \left\{x \in \R : \sup_{z \leq x} H_N(z) = \sup_{z \in \R} H_N(z)\right\}
\end{equation}
converge weakly as $N \to \infty$ to the unique maximizer $\CX$ of $x \mapsto \CA(x) - x^2$.

\smallskip
There are now three proofs of {\em Johansson's conjecture}, which we will briefly review.

\smallskip
Corwin and Hammond in \cite{BrownianGibbs} introduced the method of {\em Brownian Gibbs} resampling---a technique important in this paper, employed in Section~\ref{sec:TP_probability}---and showed thereby that, on any compact interval, the \Airy process is absolutely continuous with respect to a Brownian motion of rate two.  The almost sure uniqueness of the maximizer of Brownian motion on a compact interval then led to Johansson's conjecture.
Moreno Flores, Quastel and Remenik \cite{MaxOfAiry2} computed the joint density function of the maximizer $\CX$ and the maximal value $\CA(\CX) - \CX^2$ using the continuum statistic formula for the \Airy process proved in \cite{cqr}. Since this density integrates to one, they were able to conclude that the maximizer is unique almost surely.
In \cite{pimentel_2014}, Pimentel characterized the uniqueness of the maximizer of a stochastic process by perturbing the process by the addition of a small affine shift. This method yields Johansson's conjecture as a special case.

Each technique that yields Johansson's conjecture has its strengths and drawbacks. For example, the exact formula for the density obtained in \cite{MaxOfAiry2} yields rather strong upper bounds on probabilities. These, in turn, allowed Quastel and Remenik \cite{MR3300961} to obtain an upper estimate on the tail distribution of directed polymers. However, the formula was not useful for proving a lower bound and the authors had to use probabilistic methods instead. Stronger bounds were obtained in \cite{Schehr} (see also \cite{BLS}) via Riemann-Hilbert techniques. In our work too, we require information not easily gleaned by any single technique, and we will have recourse to several approaches to furnish the proofs of our principal results. We will elaborate  in Section~\ref{sec:heuristic}.

\subsubsection{General initial data for the KPZ fixed point}

More general initial locations of polymers in the LPP model may be allowed. Indeed, for a collection $\vec\by$ of starting points $\by_i \in \zz^2$, $i \in \nn$, and for an endpoint $\bx \in \zz^2$, we define the last passage time
\begin{equation}\label{eq:LPP general starting points}
G_{\vec \by \to \bx} := \max_{i \in \nn} G_{\by_i \to \bx}.
\end{equation}
This is now the maximum energy of a polymer running from any of the points $\by_i$ to $\bx$. Figure~\ref{Fig:LPPs}(B) depicts this version of LPP when the set of points $\vec \by$ forms a downright path (the thin black line on the lower left of the figure).
The scaling limit of $H_N(x)$, still defined by \eqref{eq:H}, is dictated by the form of this  downright path of starting points. For instance, if the set of points $\vec\by$ forms a zigzag path (down, then right, and repeat), then $H_N(x)$ converges to the Airy$_1$ process \cite{bfp} rather than to the \Airy process. In general, the mapping from the limiting behavior of the set $\vec\by$ to the limiting behavior of $H_N(x)$ as $N \to \infty$ is facilitated by the KPZ fixed point. The limit of $H_N(x)$ matches the time-one spatial marginal $\fh_1(x)$ of the KPZ fixed point with initial data related to the limit of the $\vec\by$. The process $\CA(x) - x^2$  arises precisely when the initial data for the KPZ fixed point is given by the so-called narrow wedge. Geometric last passage percolation has a simple correspondence to the height function evolution of a discrete-time version of the corner growth model (or totally asymmetric simple exclusion process) and in that interpretation, the general initial data corresponds to the initial height function: see for instance \cite{corwin2016}.

Another route to access more general initial data for the KPZ fixed point arises by returning to the fixed initial location $\by = (0,0)$  version of LPP but now associating to each point on the boundary of the first quadrant an $\omega$ which is distributed differently (in a way that may depend on location) compared to the independent and identically distributed geometric $\omega$ inside the quadrant. Figure \ref{Fig:LPPs}(C) depicts this version of LPP (see also the caption). Depending on how the boundary $\omega$ are distributed, the $H_N(x)$ defined by \eqref{eq:H} will converge to different limits. As before, these correspond to the time-one spatial marginal of the KPZ fixed point with initial data related to the limit of the boundary $\omega$'s.

The first result of this paper (Theorem~\ref{thm:johansson} ahead) is that, for a quite large class of initial data with decay at infinity, the KPZ fixed point almost surely has a unique maximizer. This result should also follow from recent results of \cite{sarkar2020brownian} using a proof technique analogous to that of \cite{BrownianGibbs} via absolute continuity with respect to Brownian motion. Our proof of Theorem~\ref{thm:johansson} is in the vein of \cite{MaxOfAiry2}, using exact determinantal formulas from \cite{fixedpt}.

\subsubsection{Temporal evolution of the KPZ fixed point}

\begin{figure}[h]
\centering
\includegraphics[width=\linewidth]{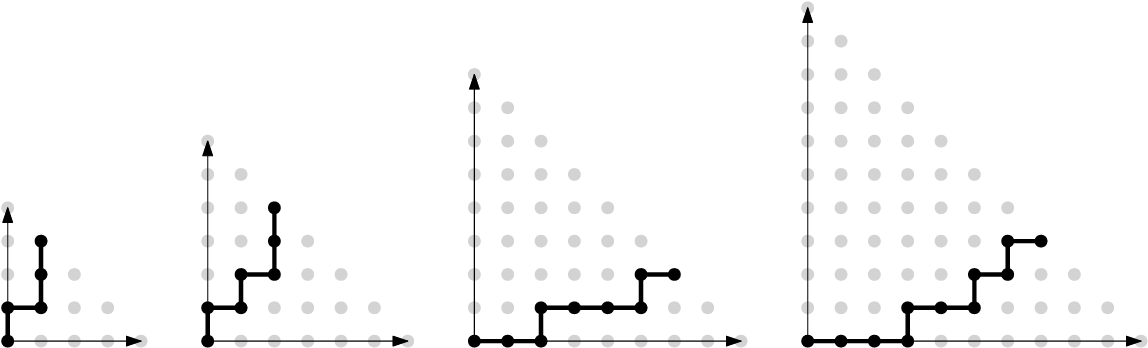}
 \captionsetup{width=.9\linewidth}
\caption{LPP evolution with respect to time. Time is measured in the diagonal direction and space is in the anti-diagonal direction. The thick black line depicts the maximizing path over all possible endpoints at a given time. Between the first and second pictures (and between the third and fourth), this endpoint is rather stable. However, between the second and third pictures, there is a large jump in the endpoint. In the large $N$ limit, these jumps occur in the exceptional set $\CT_{T}$ of times at which the maximizer is not unique.}
\label{Fig:LPPsovertime}
\end{figure}

We may add a time parameter $t$ into the LPP model (as depicted in Figure \ref{Fig:LPPsovertime}) by replacing $N$ by~$\lfloor tN \rfloor$. In this way, we extend $H_N(x)$ to $H_N(t,x)$ where the later is defined by the relation
\begin{equation}\label{eq:H2}
G_{\vec \by \to ( \lfloor tN \rfloor + x, \lfloor tN \rfloor - x)} = c_1 t N + c_2 N^{\frac{1}{3}} H_N \bigl(t, c_3 N^{-\frac{2}{3}} x\bigr).
\end{equation}
The limit of $H_N(t,x)$ as $N \to \infty$ is given by the \emph{KPZ fixed point} $\fh_t(x)$ started from initial data given in terms of the limit of the starting points $\vec\by$. The KPZ fixed point was first constructed in \cite{fixedpt} as a limit of the totally asymmetric simple exclusion process (TASEP), which is one of the simplest models in the universality class (see also \cite{NQR20} for an analogous derivation for reflected Brownian motions). TASEP is essentially equivalent to $H_N(t,x)$. More precisely, there is a coupling of the exponential random variable version of our LPP model and the totally asymmetric simple exclusion process \cite{BaikBook} such that $H_N$ describes the height function for the latter. A proof that geometric LPP has the same limit as the exponential version of the model is given in \cite{TASEPandOthers}. We do not reproduce a precise statement of this limit result  since we will only be concerned with the limit process itself.


In their construction of the KPZ fixed point, \cite{fixedpt}  also proved that the map $t \mapsto \fh_t$ is a Markov process on a suitable space of upper semi-continuous functions. Various other properties of this Markov process are discussed in \cite[Section~4]{fixedpt}, such as exact determinantal formulas for transition probabilities. A variational formulation of the KPZ fixed point which provides a coupling of all initial data on the same probability space was proved in \cite{NQR20} based on the {\em directed landscape} constructed in \cite{Landscape}. For physical background on KPZ universality and an overview of some of the subject's main developments, see \cite{Takeuchi, quastelCDM, corwinReview, borodinPetrov, quastelSpohn}.

There are several models for which convergence to the KPZ fixed point was proved only for special initial data, and just two (TASEP and Brownian last passage percolation) for which it is proven for general initial data. However, the KPZ fixed point is conjectured to be the universal scaling limit of the height functions for all models in the KPZ universality class: the height functions $h(t,x)$ are conjectured to converge, after recentring and under the KPZ 1:2:3 scaling, to $\fh_t(x)$; indeed, it is expected that, up to model-dependent values for $c_1,c_2$ and $c_3$,  $h_{\eps}(t,x):=c_2\eps^{1/2} h(\eps^{-3/2} t, c_3\eps^{-1} x) - c_1 \eps^{-3/2} t$ converges to $\fh_t(x)$ with initial data corresponding to the $t=0$ limit of $h_{\eps}(0,x)$.

For special initial data, the one-dimensional distributions of the KPZ fixed point coincide with the Tracy-Widom GOE and GUE distributions, which originally appeared in random matrix theory \cite{tracyWidom, tracyWidom2, RM}. The GOE arises when $\fh_0(x)\equiv 0$, the so-called ``flat'' initial data. The GUE arises when $\fh_0(x)=\fd_u(x)$, the so-called ``narrow wedge''. For at $u \in \R$, $\fd_u(x)$ is defined as
\begin{equation}\label{eq:NW}
\fd_u(x) :=
\begin{cases}
0, &\text{if}~ x = u,\\
-\infty, &\text{otherwise}.
\end{cases}
\end{equation}
Despite looking somewhat odd, this is the most fundamental initial data for the KPZ fixed point.

As mentioned earlier, the following result can be found in \cite[Eq.~4.12]{fixedpt}: for any fixed $t > 0$, the spatial marginal of the KPZ fixed point starting from the narrow wedge \eqref{eq:NW} is given, as a process in $x$, by
\begin{equation}\label{eq:KPZ_NW}
\fh_{t}(x) = \aip (t, x - u) - \frac{(x-u)^2}{t},
\end{equation}
where $\aip (t, x) := t^{1/3}\aip (t^{-2/3} x)$ and $\aip(x)$ is the \Airy process. In particular, under narrow-wedge initial data, $\fh_1(x)$ has the same law as $\aip(x) - x^2$. The uniqueness of the random variable~$\CX$, defined above, implies that, for any given time $t \geq 0$, the KPZ fixed point $\fh_{t}$ started from a narrow wedge almost surely attains its maximum at a unique point.

\subsection{The main results of the paper on Hausdorff dimension}\label{sec_main_result}

\begin{figure}[h]
\centering
\scalebox{0.6}{\includegraphics{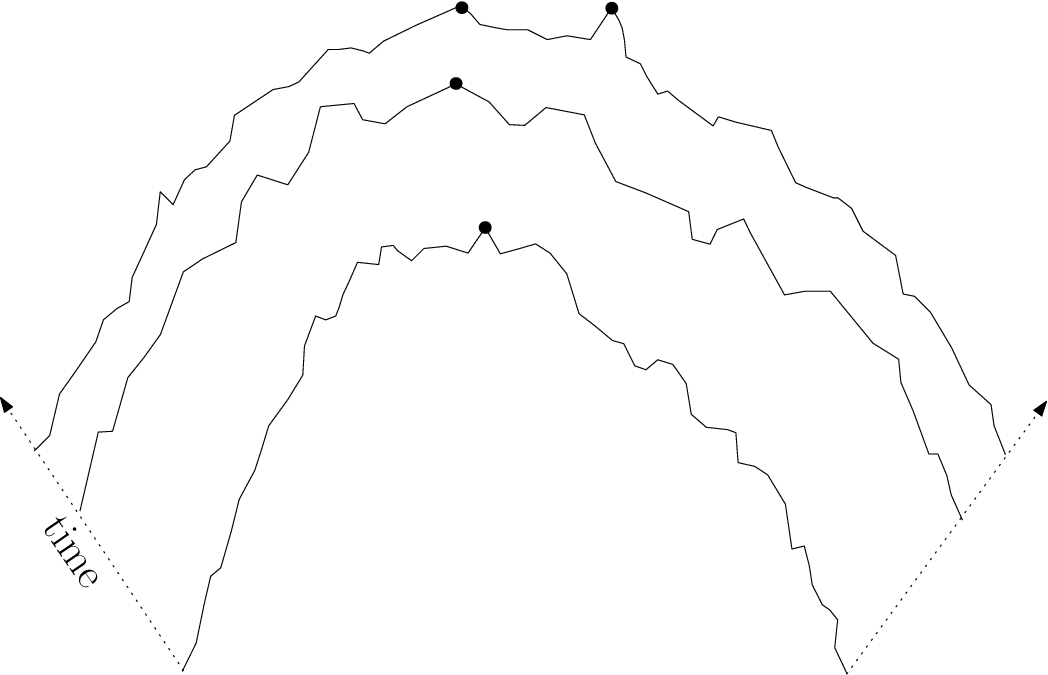}}
 \captionsetup{width=.9\linewidth}
\caption{The KPZ fixed point $\fh_t(x)$ spatial process depicted at three times; for clarity, we have shown the curves separated, but there need not be any ordering of the function values as time proceeds. The first two of the three times are typical ones at which the maximizer is unique (see the black bullet at its location), while the third is exceptional. There are then two maximizers, one close to those earlier and the other at a dramatic remove.}
\label{Fig:fixedpointevolution}
\end{figure}

As we have mentioned, one of our results will be that $\fh_t(x)$ attains its maximum at a unique location almost surely for each fixed $t> 0$ under general initial data. However, this does not preclude the existence of random \emph{exceptional} times at which $\fh_t$ has several maximizers.
A classical example of a non-empty set of exceptional times is the zero set  $\mathsf{Zero} = \{ t \geq 0 : B(t) = 0\}$ of standard Brownian motion $B$. For every $T > 0$, the Hausdorff dimension of $\mathsf{Zero} \cap [0,T]$ is almost surely $\frac{1}{2}$---see \cite[Theorem~4.24]{morters2010brownian}.
The random geometry of an exceptional time is trivial in this case---we merely have $B(t) = 0$ at such times. In our case, and others, there is however a rich geometry associated to exceptional times.

Indeed,
the study of random dynamics that leaves invariant an equilibrium measure, and the question of the existence and Hausdorff dimension of an exceptional set of times at which an almost sure property of the equilibrium law is violated, are celebrated topics in modern probability theory.
Discrete Fourier analytic tools have been employed to investigate~\cite{BKS1999} the sensitivity to noise of critical percolation. This led to the proof~\cite{SS2010} that, under a natural dynamic which left the critical percolation distribution invariant, there exist exceptional times where there exists an infinite cluster; further, the Hausdorff dimension of these exceptional times has been identified~\cite{GPS2010} as $31/36$.

In our case, we will consider exceptional times when the $\fh_t$ fails to have a unique maximizer. They can be viewed as moments of instability in the life of the random polymer, at which its endpoint loses its uniqueness. At such moments, the polymer's future is uncertain, with the passage of even an instant of further time being liable to propel the polymer endpoint to a macroscopic distance from its present location (see Figures \ref{Fig:LPPsovertime} and \ref{Fig:fixedpointevolution}). (This does not contradict existing results on the continuity properties ofsuchpolymers, eg. from \cite{Landscape}, as those results are typically stated on the event that the polymer is unique.)

Let $T > 0$ and $A\geq 0$. The exceptional set of random times in $[0,T]$ that are such moments of instability, when the polymer endpoint may subsequently jump by a distance greater than~$A$, is
\begin{equation}\label{eq:TwinPeaksSet}
	\CT_{T, A}(\fh) := \bigl\{t \in [0,T] : \exists\, x_1, x_2 \in \Argmax(\fh_t),\; |x_1-x_2|> A\bigr\},
\end{equation}
where the set $\Argmax(\fh_t) \subset \R$ contains all points $x$ at which $\fh_t(x)$ is the maximum value of $\fh_t$. When $A=0$, we adopt the shorthand $\CT_{T}$ for $\CT_{T,0}$; in this case, the condition on $x_1, x_2$ may be equivalently written as $x_1\neq x_2$.

Of course, the set $\CT_T$ is random and depends on the initial data $\fh_0$ of the KPZ fixed point. For example, if $\fh_0$ does not have sufficiently fast decay at infinity, then $\fh_t$ can be unbounded, in which case $\CT_{T}(\fh)$ is empty.
We thus need to impose some assumptions on the initial data $\fh_0$ to ensure that our question is well-defined. For instance, we need that $\fh_t$ is always almost surely bounded so that maximizers exist. We impose a slightly stronger condition, that $\fh_0$ decays in a square-root manner at infinity and is identically $-\infty$ for all sufficiently negative arguments.

\begin{assump}\thmtitle{Decay at infinity}\label{a:initial_state_pd} 
The initial data $\fh_0 : \R \to [-\infty, \infty)$ of the KPZ fixed point is a non-random, upper semicontinuous function that is not identically equal to $-\infty$, and for which (a)~there exist $\alpha\in\R$ and $\gamma>0$ such that the bound $\fh_0(y) \leq \alpha - \gamma |y|^{1/2}$ holds for all $y \in \R$ and (b)~there exists $\lambda \in \R$ such that $\fh_0(y) = -\infty$ for $y\leq -\lambda$.
\end{assump}


\noindent Assumption~\ref{a:initial_state_pd} will refer to both parts (a) and (b), and we will occasionally make use of only part~(a).

We will show in Lemma~\ref{lem:KPZ_is_bounded} that Assumption~\ref{a:initial_state_pd}(a) guarantees that, at any time $t > 0$, the KPZ fixed point $\fh_t$ is
almost surely bounded and decays at infinity. This yields the existence of maximizers and leads to the natural question of their almost sure uniqueness at fixed times, a form of Johansson's conjecture for general initial data; here is our result affirming this general form of the conjecture.

\begin{thm}\label{thm:johansson}
Let the initial data of the KPZ fixed point $\fh_0$ satisfy Assumption~\ref{a:initial_state_pd}(a). Then $\fh_t$ has a unique maximizer almost surely for every fixed $t>0$.
\end{thm}

Theorem~\ref{thm:johansson} should be provable by using the recent result of \cite{sarkar2020brownian} which shows the absolute continuity of $\fh_t$ with respect to Brownian motion on compact intervals. Such a proof would closely follow the original proof of Johansson's conjecture in the narrow-wedge case presented in \cite{BrownianGibbs}; this requires a statement such as Lemma~\ref{lem:KPZ_is_bounded} on the decay of $\fh_t$ at infinity. Our proof of Theorem~\ref{thm:johansson}, however, is independent of \cite{sarkar2020brownian} and relies on the formulas for the KPZ fixed point proved in \cite{fixedpt}. In this way, it is more in line with the approach to Johansson's conjecture in the narrow-wedge case presented in \cite{MaxOfAiry2}. However, in contrast to \cite{MaxOfAiry2}, where a density of one maximizer was computed, from which uniqueness was concluded, we compute the probability that two maximizers appear at a given time and show that it equals zero.

Theorem~\ref{thm:johansson} implies that the Lebesgue measure of $\CT_{T,A}(\fh)$
is zero almost surely. It is thus natural to examine the fractal dimension of the exceptional set of violating times.  We refer to Appendix~\ref{sec:Hausdorff} for the definition of
the Hausdorff dimension $\dim(X)$ for a set $X \subset \R$, because it is this notion of dimension that we will use to measure the exceptional set when it is non-empty. Our first result concerns the Hausdorff dimension of the set~$\CT_{T,A}(\fh)$ and proves an upper bound.

\begin{thm}\label{thm:main}
Let the initial data of the KPZ fixed point $\fh_0$ satisfy Assumption~\ref{a:initial_state_pd}.
Then, for any $T > 0$ and $A\geq 0$,
\begin{equation}\label{eq:dimTP}
\P_{\fh_0}\Bigl(\dim(\CT_{T,A}(\fh)) \leq \tfrac{2}{3}\Bigr) \, = \, 1.
\end{equation}
\end{thm}


In fact, we obtain the upper bound on the Hausdorff dimension under only Assumption~\ref{a:initial_state_pd}(a). In our approach, we require the full Assumption~\ref{a:initial_state_pd} to show a lower bound on the probability of the existence of multiple near-maximizers at a given time, which we will shortly state as Theorem~\ref{thm:TP_probability}, because of a restriction in a ``melon'' transformation in Brownian last passage percolation that we rely on; we will say a little more about this transformation in Section~\ref{sec:heuristic}, and about the restriction that it imposes in Section~\ref{sec:TP_probability}. We do not attempt the extension to more general initial data in the present article. We explain the difficulties in obtaining a matching lower bound on the Hausdorff dimension in Section~\ref{sec:why no lower bound}.


Finally, in the case of narrow wedge initial data, we observe that it satisfies Assumption~\ref{a:initial_state_pd} and so we obtain the following corollary immediately. 

\begin{cor}\label{cor:nw}
For narrow wedge initial data $\fh_0=\fd_0$ and for any $T>0$, 
\begin{equation}\label{eq:dimTPNW}
\P_{\fd_0}\Bigl(\dim(\CT_{T}(\fh)) \leq \tfrac{2}{3} \Bigr) \, = \, 1.
\end{equation}
\end{cor}

\subsection{Quantitative upper and lower bounds on near-twin peaks}

Theorem~\ref{thm:main} is the consequence of an upper bound on the probability that $\fh_t$, for given $t>0$, satisfies an almost twin peaks condition. To our knowledge these are the first results giving explicit probability bounds for the KPZ fixed point under general initial data, and may prove to be of independent interest. Indeed, (weaker) upper bounds on similar twin peaks probabilities in the narrow-wedge case were obtained in \cite{calvert2019brownian,LPPtools} and used in \cite{LPPdynamics} in the related context of Brownian LPP to understand a phase transition under a particular dynamic.

Let us define what we mean by an almost twin peaks more precisely. For $\varepsilon>0$ and fixed $\beta\in(0,1)$, define the $\varepsilon$-twin peaks set $\TP{A, L}^{\eps}$ by
\begin{equation}\label{eq:TP}
\TP{A, L}^{\eps} = \left\{f: \R\to\R\cup{-\infty} : |\Max(f)| \leq (\beta L)^{1/2}, \TPS{A, L}^{\eps}(f) \neq\emptyset\right\},
\end{equation}
where
$$\TPS{A, L}^{\eps}(f) := \Bigl\{(x_1, x_2) \in [-\beta L, \beta L]^2 : x_2 - x_1 > A,\; \Max(f) - f(x_i) \leq \eps ~\text{for}~ i=1,2\Bigr\}.$$
In words, $\TP{A,L}^\eps$ is the event of having two locations, in an interval of length of order $L$, at distance at least $A$ apart, where the value of $\h_t$ is within $\varepsilon$ of the global maximum; additionally, the maximum is in absolute value at most of order $L^{1/2}$.

The upper bound on the Hausdorff dimension can be proved by understanding the probability of exhibiting an $\varepsilon$-twin peaks at time $t$, which is our next main result. Here, we use the cut-off $\fh_t^L$ of the KPZ fixed point, which is essentially $\fh_t|_{[-L,L]}$; it is defined in \eqref{eq:cut-off}.

\begin{thm}\label{thm:densities2-intro}
Let the initial data $\fh_0 \in \UC$ of the KPZ fixed point satisfy $\fh_0(x) \leq \alpha$ for some $\alpha \in \R$ and for any $x \in \R$. 
Then, for any $0 < L \leq \bar L$, $A > 0$ and $0 < T_0 < T$, 
\begin{equation}\label{eq:q_def}
\P_{\fh_0} \Bigl(\fh^L_t \in \TP{A, \bar L}^{\eps}\Bigr) \leq C \eps t^{-\frac{1}{3}},
\end{equation}
for $t \in [T_0, T]$ and $\eps \in (0,1)$, and for a constant $C \geq 0$ depending on $T_0$, $\bar L$, $A$, $\beta$ (which comes into the definition \eqref{eq:TP}) and $\alpha$ (i.e., the dependence on $\fh_0$ is merely via the constant $\alpha$).
\end{thm}

Observe that this theorem immediately yields Theorem~\ref{thm:johansson} by taking $\varepsilon\to0$ and then $L,\bar L, A\to\infty$.


As we will see in the proof outline section ahead, the Hausdorff dimension upper bound of $\frac{2}{3}$ that we prove is a consequence of the $\frac{1}{3}$ exponent of $t$ and $1$ exponent of $\varepsilon$ in the previous result. In the next result we show that the $\varepsilon$-dependence is sharp by proving a matching lower bound.

\begin{thm}\label{thm:TP_probability}
    Let the initial state $\fh_0$ of the KPZ fixed point satisfy Assumption~\ref{a:initial_state_pd}, and suppose that $A>0$ and $t>0$. There exist $L_0$ and $\smallconst > 0$ (both depending on $A$, $t$, and $\fh_0$) such that, for all $L\geq L_0$ and $\eps \in (0,1)$,
    %
    \begin{equation}
        \P_{\fh_0} \Bigl(\fh_t \in \TP{A, L}^{\eps}\Bigr) \geq \smallconst \eps.
    \end{equation}
    Further, $\smallconst$ and $L_0$ can be taken to depend on $t>0$ in a continuous manner.
\end{thm}


\subsection{A density of two maximizers}

Our proof of Theorem~\ref{thm:densities2-intro} relies on existence of the joint density of two maximizers and the maximal value of the KPZ fixed point $\fh_t$. More precisely, for $t, L > 0$ and $M \in \R$, define the function
\begin{align}
\MoveEqLeft[11]
\dens_{t, M, L}^{(\eta_1; \eta_2)}(x_1, x_2) := \P_{\fh_0}\Big(\text{For each }i\in \{1,2\}, \,\,\fh_{t}(y_i) > M - \eps_i \text{ for}\\[-2pt]
&\text{ some } y_i \in [x_i, x_i + \delta_i),
 \text{ and } \fh_{t}(z) \leq M \text{ for } z \in [-L, L]\Big),\label{eq:density}
\end{align}
where $\eta_i = (\eps_i, \delta_i)$ with $\eps_i, \delta_i > 0$. %
Here, we have suppressed dependence on the initial data $\fh_0$. The function \smash{$\dens_{t, M, L}^{(\eta_1; \eta_2)}(x_1, x_2)$} equals the probability that the KPZ fixed point visits $\eps_1$- and $\eps_2$-neighbourhoods of its maximum in $\delta_1$- and $\delta_2$-neighbourhoods around the respective points $x_1$ and $x_2$, with the maximum being close to $M$. The next result implies that properly normalized functions $\smash{\dens^{(\eta_1; \eta_2)}_{t, M, L}(x_1, x_2)}$ converge to a non-trivial limit as $\eta_1, \eta_2 \searrow 0$. To compute the limit of this function, we fix a constant $\alpha_\star > \frac{1}{2}$; define the domain
\begin{equation}
	\Dom := \bigl\{(\eps, \delta) \in \R^2\, :\, \delta > 0,\, 0 < \eps \leq \delta^{\alpha_\star}\bigr\} ; \label{eq:Dom}
\end{equation}
and denote the limit in this domain
\begin{equation}\label{eq:mylim}
	\mylim{\eta \searrow 0} f(\eta) := \lim_{\eta \in \Dom,\, \eta \searrow 0} f(\eta),
\end{equation}
for a suitable function $f$. It should be noted that, when we take the limit $\eta \searrow 0$, we are implicitly claiming that the double limit in $\epsilon$ and $\delta$ in $\Dom$ exists and does not depend on how that limit is performed.

\begin{thm}\label{thm:densities}
Assume that the initial data of the KPZ fixed point satisfies $\fh_0(x) \leq \alpha$ for some $\alpha\in \mathbb{R}$ and for all $x \in \R$. For every fixed $t > 0$ the  following limit exists: 
\begin{equation}\label{eq:F-limit}
	\dens_{t, M}(x_1, x_2) := 
    \lim_{L \to \infty}\ \mylim{\eta_1 \searrow 0}\ \mylim{\eta_2 \searrow 0}\ \frac{1}{\eps_1 \delta_1 \eps_2 \delta_2} \dens^{(\eta_1; \eta_2)}_{t, M, L}(x_1, x_2),
\end{equation}
pointwise for $x_1, x_2 \in (-L, L)$ and $M \in \R$, such that $x_2 - x_1 > 0$.
Further, $\dens_{t,M}$ satisfies a scale invariance property:
\begin{equation}
\dens_{t, M}(x_1, x_2) = t^{-2} \dens_{1, t^{-1/3} M} \bigl(t^{-\frac{2}{3}} x_1, t^{-\frac{2}{3}} x_2\bigr),
\end{equation}
where the right-hand function is defined for the initial data $\fh_0^{(t)}(x) := t^{-1/3}\fh_0(t^{2/3} x)$. Moreover, this function can be written explicitly as
\begin{equation}
\begin{split}
\dens_{1, M}(x_1, x_2) &= \det\Bigl(I- \fK^{\hypo(\fh_0)}_{1/2} \Bigl(\fR - \fA^{(0; \star)} - \fA^{(\star; 0)} + \fA^{(\star; \star)}\Bigr) \Bigr) \\
&\quad - \det\Bigl(I- \fK^{\hypo(\fh_0)}_{1/2} \Bigl(\fR - \fA^{(0; \star)}\Bigr) \Bigr)\\
&\quad - \det \Bigl(I- \fK^{\hypo(\fh_0)}_{1/2} \Bigl(\fR - \fA^{(\star; 0)}\Bigr) \Bigr) + \det \Bigl(I-\fK^{\hypo(\fh_0)}_{1/2}\fR\Bigr),
\end{split}
\end{equation}
where the kernels in the Fredholm determinants are defined in \eqref{eq:Kd} and Section~\ref{sec:L-to-infinity}.
\end{thm}

(To be precise, the function $\dens_{1,M}$ should be called a super-density, because its integral over the arguments $x_1$, $x_2$ and $M$ exceeds $1$; in fact, we expect this integral to be infinite.)


\subsection{Heuristics and structure}
\label{sec:heuristic}

We will indicate three main facts; how they will combine into a proof of the upper bound of the Hausdorff dimension of the exceptional set and the non-emptiness of that set; and, in outline, how we will prove each of them.
In so doing, we will indicate the structure of the paper, which is also illustrated in Figure~\ref{fig:outline}.


In a crude but instructive simplification, a set has dimension $\alpha$ if, for any $\eps > 0$, it may be fattened to become a union of order $\eps^{-\alpha}$ intervals of length $\eps$.
The fattened set will be a proxy set $\CT_T^\eps(\fh) = \{t\leq T: \fh_t \in \TPno^\eps\}$; namely, the set
of times $t$ at which the KPZ fixed point $\fh_t$ lies in the {\em twin peaks} set $\TPno^\eps$ (similar to $\TP{A,L}^\eps$ introduced earlier) 
of upper semicontinuous functions $f$ such that $\Max(f) - f(x_i) \leq \eps$ for some points $x_1, x_2 \in \R$ such that $|x_1-x_2|\geq 1$. We refer to these points $x_1$ and $x_2$ as \emph{near maximizers}. The reader may recall that in our earlier statements of bounds on the probability of $\fh_t$ lying in $\TP{A,L}^\varepsilon$, we consider near maximizers whose distance $|x_1-x_2|$ is at least a given small value $A$; and for which $|x_1|,|x_2|\leq L$ for large $L$. In this present heuristic discussion, we ignore these variables. Here are the three main facts that we demonstrate:
%
\begin{enumerate}[label=(\roman*)]
  \item\label{heur:1} The  KPZ fixed point $\fh_t(x)$ is H\"older-$\tfrac{1}{3}^{{\scriptscriptstyle-}}$ in $t$, uniformly over compact sets in $x$.

  \item\label{heur:2} The probability that $\fh_t \in \TPno^\eps$ is at most $C\eps t^{-1/3}$ (Theorem~\ref{thm:densities2-intro}).

  \item\label{heur:3} The same probability is at least $C(t) \eps$ (Theorem~\ref{thm:TP_probability}).
\end{enumerate}

(Strictly speaking, only the first two points are needed for an upper bound on the dimension. The third point aids in the heuristic explanation of why the dimension should be \emph{equal} to $\frac{2}{3}$, and would be needed to prove a matching dimension lower bound; see Section~\ref{sec:why no lower bound} ahead for why are unable to do this.)

Given these, we may see why the Hausdorff dimension of $\CT_T(\fh)$ will be two-thirds when this set is non-empty.
Indeed, Item \ref{heur:1} implies that the maximum of $\fh_t$ is H\"older \smash{$\frac{1}{3}^{\scriptscriptstyle-}$} in time, which suggests that intervals contained in $\CT_T^\eps(\fh)$ should typically be at least of size of order $\eps^3$. Item \ref{heur:2} gives us a complementary upper bound on the typical length of intervals in $\CT_T^\eps(\fh)$. In particular, it shows that, once the set $\CT_T^\eps(\fh)$ is entered, in the passage of time of order $\eps^3$, it will be exited  (and hence the interval will have ended).
Finally, Item \ref{heur:3} along with Item \ref{heur:2} implies that the total size of~$\CT_T^\eps(\fh)$ is of order $\eps$. Thus, $\CT_T^\eps(\fh)$ is a set of size $\epsilon$
which consists of intervals of length~$\eps^3$; the number of intervals must be of order $\eps^{-2}$, and such an order of intervals of length $\epsilon$ will be needed to cover~$\CT_T(\fh)$.
The infimum $\CH^\alpha_\eps(\CT_T(\fh))$ of $\alpha$-values   of coverings of maximum diameter~$\eps$ recalled in~\eqref{eq:pre-measure} should be of order $\eps^{-2 + 3\alpha}$, a quantity whose behavior changes as $\eps\to0$ depending on whether $\alpha$ is greater or less than $2/3$. Thus is the Hausdorff dimension of two-thirds predicted. Of course, we only prove an upper bound of two-thirds; see ahead in Section~\ref{sec:why no lower bound} for a discussion of where our methods face difficulties in proving a lower bound.


Our task is thus to prove the three facts.  The respective proofs appear in Sections~\ref{sec:maximum}, \ref{sec:TwinPeaks}, and \ref{sec:TP_probability}. The first two pieces are fit together to prove Theorem~\ref{thm:main} in the final Section~\ref{sec:main_proof}. Next we explain in outline how the three tasks will be accomplished.

\smallskip
Lemma~\ref{lem:Holder-time} is the precise rendering of Item  \ref{heur:1}. That $\fh_t(x)$ is H\"older-$\frac{1}{3}^{\scriptscriptstyle-}$ in $t$ for fixed $x$ is known from \cite{fixedpt}, but our proof of the extra local uniformity in space makes use of a variational formula, recalled in Section~\ref{sec:KPZ_fp}, for the KPZ fixed point in terms of the directed landscape recently constructed in \cite{Landscape}. The proof also requires that $\fh_t(x)$ is H\"older-$\frac{1}{2}^{\scriptscriptstyle-}$ in $x$, locally uniformly in $t$, which we prove in Lemma~\ref{lem:Holder-space}. Finally, the spatially uniform temporal continuity's implication that $\Max(\fh_t)$ is H\"older-$\frac{1}{3}^{\scriptscriptstyle-}$ in time is recorded in Lemma~\ref{lem:max_Holder}.

\smallskip
The previously stated Theorem~\ref{thm:densities2-intro} is the rigorous manifestation of Item \ref{heur:2}, and is derived by way of a density whose existence is asserted in Theorem~\ref{thm:densities}. The much more involved proof of these results uses exact formulas,
also recalled in Section~\ref{sec:KPZ_fp}, for the KPZ fixed point in terms of Fredholm determinants that have been obtained in \cite{fixedpt}. More precisely, for a fixed $t > 0$ and for any $\delta, \eps > 0$, we will consider in \eqref{eq:density} the probability that two near maximizers of $\fh_t$ are located in $\delta$-neighbourhoods of fixed points $x_1 \neq x_2$, while the maximum of $\fh_t$ is in an $\eps$-neighbourhood of a fixed value $M$.
Dividing this probability by $\delta^2 \eps^2$ and taking the limits $\eps, \delta \to 0$, we will obtain in effect  the density, conditionally on maximizer non-uniqueness, of the maximum value and two maximizers of $\fh_t$; this is the limit and the density whose existence is the content of Theorem~\ref{thm:densities}. In Proposition~\ref{prop:densities}, scaling properties and bounds on this limit are asserted.
(To call this just mentioned limit a density is really an improper formulation because the concerned law is $\sigma$-finite, with infinite mass on maximizer pairs at distance close to zero.) The derivation of this result involves  double differentiation of a Fredholm determinant, arising from the distribution function of the KPZ fixed point. The justification of this differentiation and the computation of the derivatives are quite technical, involving asymptotic analysis of path integral formulations of the Fredholm determinant kernels, defined by means of  a Brownian scattering transform.
These calculations rely on properties of trace class operators and Fredholm determinants that are summarized in Appendix~\ref{sec:Fredholm}; and bounds on kernels involving Airy functions that are reviewed in Appendix~\ref{sec:kernels_bounds}.

 Theorem~\ref{thm:densities2-intro} is then obtained by integration of the just indicated density function over $x_1$, $x_2$ and $M$. 
Combining Item \ref{heur:2} with the temporal H\"older regularity of $\fh_t$ and of its maximum, given in Item~\ref{heur:1}, yields the upper bound on the Hausdorff dimension of $\CT_T(\fh)$.

\smallskip
Theorem~\ref{thm:TP_probability} is the precise rendering of Item  \ref{heur:3}. Its proof, in contrast to that of Item \ref{heur:2}, relies on a Gibbsian resampling argument, akin in spirit to  the Brownian Gibbs property used in such works as \cite{BrownianGibbs,dauvergne2018basic,hammond2017brownian,calvert2019brownian}; we discuss this property in greater detail in Section~\ref{sec:TP_probability}.
The argument works in the prelimiting model of Brownian last passage percolation: the random variable $\smash{\h_t}$, namely the prelimiting version of $\fh_t$, will be written in terms of a last passage value through a random environment obtained by a ``melon'' transformation of the original Brownian environment (the term ``melon'' is used in \cite{Landscape}, though one could instead call this a ``continuous RSK correspondence'' or a ``Pitman transform''). This relies on a remarkable identity of last passage values through the original and melon-transformed environments proved in \cite{Landscape}. (The identity was later found to be implied by the earlier work \cite{noumi2002tropical} in the more general setting of positive temperature polymer models; see \cite{corwin2020invariance} for more details and three proofs of the positive temperature identity, and \cite{dauvergne2020hidden} for a fourth). It is through this melon transform that we are able to access general initial data.

The melon transformed environment coincides with Dyson Brownian motion with $\beta=2$. This ensemble of random curves is Gibbsian,  enjoying a simple Brownian resampling invariance. By harnessing this invariance, we may exploit the melon-transformed environment, resampling $\smash{\h_t}$ in order to generate an event of probability at least $C_2\epsilon$ on which $\smash{\h_t \in \TPno^\eps}$. Taking $n\to\infty$ is the final step on the road to
Theorem~\ref{thm:TP_probability}.

The lower bound in Item \ref{heur:3} is suggested by (but does not follow from) the spatial Brownian absolute continuity of $\fh_t$ established in~\cite{sarkar2020brownian}. Indeed, a Brownian bridge has probability proportional to $\eps$ of having a unit-order separated location that comes within $\eps$ of the maximal value of the bridge. A much more quantitative comparison result than absolute continuity is needed, however, in order to transfer small probability events between the Brownian measure and the spatial process for $\fh_t$. Our approach does not seek such a comparison, but rather attacks the problem directly.

Although Gibbsian resampling is not a new technique in studying models in the KPZ universality class, our use of it departs from existing methods in two quite substantial ways. In previous works obtaining quantitative probability bounds on KPZ objects via Gibbs resampling, the Gibbsian line ensemble which is being resampled is the object of interest. 
In the present work, we are interested in last passage percolation in the environment defined by the Gibbsian line ensemble. If we restrict attention to narrow wedge initial data, we recover the earlier considered cases where the line ensemble is the object of interest (i.e., the last passage problem trivializes). (Of course, a number of other works, such as \cite{Landscape}, have made use of resampling arguments with the line ensemble to obtain important qualitative results---indeed, we crucially rely on the important identity from \cite{Landscape} mentioned above.) The second departure is that our resampling is performed on a random interval defined relative to the location of the unique global maximizer. In this sense, the resampling may be called non-Markovian. To illustrate the complication that this introduces, consider a Brownian bridge. Its law on any given interval is invariant under replacing the path by an independent Brownian bridge affinely shifted to fit in. If we look at its law on an interval starting at the maximizer of the path, such a resampling invariance no longer holds. To handle this in the Brownian bridge case, as well as in our more general case, we rely on a path decomposition of Markov processes at certain special times given in~\cite{millar1978path}.


 \begin{center}
  \footnotesize
  \begin{figure}
  \begin{tikzpicture}[scale=0.9, auto,
    block_main/.style = {rectangle, draw=black, thick, fill=white, text width=13em, text centered, minimum height=4em, font=\small},
  block_bound/.style = {rectangle, draw=black, fill=white, thick, text width=10em, text centered, minimum height=4em},
  block_item1/.style = {rectangle, draw=black, fill=gray!30, thick, text width=11em, text centered, minimum height=4em},
  block_item2/.style = {rectangle, draw=black, fill=gray!30, thick, text width=11em, text centered, minimum height=4em},
  block_item3/.style = {rectangle, draw=black, fill=gray!30, thick, text width=14em, text centered, minimum height=4em},
      block_KPZ/.style = {rectangle, draw=black, fill=white, thick, text width=12em, text centered, minimum height=4em},
        line/.style ={draw, thick, -latex', shorten >=0pt}]
    \node [block_main] (main) at (-4.9,0) {The Hausdorff dimension of exceptional times is at most $\frac{2}{3}$ (Theorem~\ref{thm:main})};
    \node [block_KPZ] (KPZ_bdd) at (0,-2.5) {The KPZ fixed point is \as bounded (Lemma~\ref{lem:KPZ_is_bounded})};
    \node [block_item1] (item1) at (-4.8,-2.5) {The max of the KPZ fixed point is \\ $\tfrac{1}{3}^{{\scriptscriptstyle-}}$-H\"{o}lder (Lemma~\ref{lem:max_Holder})};
    \node [block_item2] (item2) at (0,0) {The probability that $\fh_t \in \TPno^\eps$ is at most $C\eps t^{-1/3}$ (Theorem~\ref{thm:densities2-intro})};
    \node [block_item3] (item3) at (5,0) {The probability that $\fh_t \in \TPno^\eps$ is at least $C(t) \eps$ (Theorem~\ref{thm:TP_probability})};
    \node [block_KPZ] (KPZ) at (3,-5) {The KPZ fixed point formula and its properties (Propositions~\ref{prop:KPZfp} \& \ref{prop:sym})};
    \node [block_KPZ] (density) at (5,-2.5) {Existence and properties of the density of two maximizers (Proposition~\ref{prop:densities})};
    \node [block_KPZ] (landscape) at (-2.5,-5) {Properties of the directed landscape (Proposition~\ref{prop:Landscape_bound})};

    \begin{scope}[every path/.style=line]
      \path (KPZ_bdd)   -- (main);
      \path (KPZ)   -- (density);
      \path (density) -- (item2);
      \path (item2) -- (main);
      \path (item1) -- (main);
      \path (item2) -- (main);
      \path (landscape) -- (item1);
      \path (landscape) -- (KPZ_bdd);
    \end{scope}
  \end{tikzpicture}
  \captionof{figure}{Flowchart of the paper. The grey blocks correspond to the three items described in Section~\ref{sec:heuristic}.}
  \label{fig:outline}
  \end{figure}
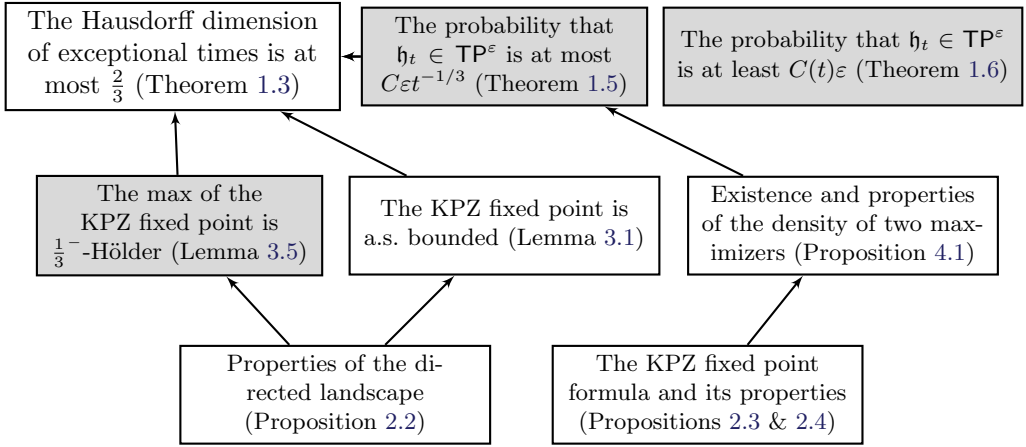
\end{center}

\subsection{The difficulty in proving a matching lower bound}\label{sec:why no lower bound}
Typically in proving bounds on Hausdorff dimension, the lower bound is the harder one to prove, and it is the same here. We try to give an indication now of why our methods face difficulties in proving a lower bound.

The basic problem is in understanding the behavior of the KPZ fixed point at very small times in a uniform way. In particular, to obtain a lower bound on the Hausdorff dimension, one needs to have some control on the probability of \emph{two} times lying in $\CT_{T,A}$, where the two times may be arbitrarily close. By the Markov property of the fixed point, this is equivalent to understanding the probability that $t\in\CT_{T,A}$ for arbitrarily small $t$ from an initial condition taken from a fairly wide class, which we approach by studying the probability that $\fh_t\in\TP{A}^\eps$ for small $\varepsilon$.

Heuristically, if the initial condition is very flat, there are many locations from which a potential maximizer can arise, which would increase the probability that $\fh_t\in\TP{A}^{\eps}$ for small $t$ (when the parabolic decay effect of the fixed point is weaker). And indeed, Theorem~\ref{thm:densities2-intro} provides a bound only when $t$ is bounded away from zero---a technically precise version of the problem.

Our approach to obtaining the upper bound in Theorem~\ref{thm:densities2-intro} has been by way of integrating the density from Theorem~\ref{thm:densities}; the density is a joint density of the location of the maximizers and the value of the maximum.
On a technical level, the problem we face is that the density blows up when the maximum value argument is taken to $-\infty$; by KPZ scaling, as $t\to 0$, the maximum value being larger than a negative constant $-M$ at time $t$ is the same as it being greater than $-Mt^{-1/3}$ at time 1, and so the maximum value argument must be allowed to go to $-\infty$ if we want to allow small $t$. The blowup of the density in this regime is somewhat to be expected by our heuristic from the previous paragraph: when the maximum is very negative, the shape of the profile should be essentially flat on a large interval. And it is this density blowup which leads to the imposition that $t$ be away from zero in Theorem~\ref{thm:densities2-intro}. For the same reason, our bounds do not imply that the set of exceptional times is non-empty with positive probability.

A similar problem was encountered in \cite{dauvergne2022non}, who adopted a thinning procedure and an understanding of the envelope of growth (a law of iterated logarithm) for the KPZ fixed point at short times, which sufficed to handle the problem in the context of the argument framework developed there. We discuss the approach of \cite{dauvergne2022non} a bit more in Section~\ref{sec:Dau22}.

\subsection{Conjectures for some related exceptional sets}
\label{sec:conjectures}

Theorem~\ref{thm:main} offers an assertion concerning fractal geometry for exceptional sets embedded in the KPZ fixed point. Several assertions may naturally be conjectured concerning related fractals embedded either in the KPZ fixed point or in the space-time Airy sheet. Here we state, and explain the reasoning for, some such conjectures.
Our first conjecture concerns more basic structural properties of the set $\CT_T$ than its dimension.

\begin{conjec}\label{conj:density}
For any $\fh_0$ satisfying Assumption~\ref{a:initial_state_pd}, $\CT_\infty(\fh)$ is almost surely dense in $[0, \infty)$.
\end{conjec}
Let us assume, for the moment, that we can show that for general initial data and any $T>0$,  $\CT_T\neq \emptyset$ almost surely. Then by the Markov property we can apply this result to the fixed point at any given time $t>0$ by considering the process started with initial data given by $\fh_t$. This would imply that around any rational time, there is a time arbitrarily close when non-uniqueness holds---thus the almost sure density result. 

In dynamical percolation at criticality on the faces of the honeycomb lattice~\cite{SS2010,GPS2010}, the probability that a given face lies in an infinite open component at some time on a given unit interval is known to be bounded away from zero and one. However, the set of times at which there exists an infinite open component is presumably almost surely dense---and a zero-one law seems to promise such a result, though no such proof is recorded to our knowledge. Given our conjecture, we see an analogy between exceptional times for the KPZ fixed point at which two maximizers exist, separated to unit order, with times in dynamic percolation at which the infinite open cluster visits a given region; and between those exceptional times at which the pair of maximizers exist, without any demand of separation, and moments in dynamical percolation at which the infinite open cluster exists, without any condition being imposed on its geometry.


\medskip
Our principal result, Theorem~\ref{thm:main},  concerns the Hausdorff dimension of exceptional times in the KPZ fixed point at which there exist  two maximizers, and a natural question concerns the Hausdorff dimensions of times at which there are a given number of  maximizers exceeding two. We formulate a conjecture on the values of these dimensions.
\begin{conjec}\label{conj:multiple maximizers}
The Hausdorff dimension of the set of times with three maximizers is one-third; the set of times with four such has zero Hausdorff dimension. For five maximizers, there almost surely exist no such exceptional times.
\end{conjec}

Admitting this conjecture, the set of times at which there exist four maximizers could be non-empty, but we believe this set to be empty almost surely.

Duncan Dauvergne has recently proven Conjectures~\ref{conj:density} and \ref{conj:multiple maximizers} in \cite{dauvergne2022non}. In the case of four maximizers, he also shows that the set is either almost surely empty or almost surely dense.

We argue for Conjecture~\ref{conj:multiple maximizers} in the case of three maximizers,  developing the heuristic presented for the case of two in the preceding section.
 Let $\CT_T(3)$ denote the set of times with three maximizers, which, by a countability argument, we may suppose to be at mutual distance at least one; and
 let $\CT_T^\eps(3)$ denote the set of times $t$ at which there exist three points at mutual distance at least one, for each of which, $\fh_t$ is within $\eps$ of its maximum. Local Brownianity of $\fh_t$  suggests that $\P(t\in \CT_T^\eps(3)) = \Theta(\varepsilon^2)$. Because $\fh_t$ is $\tfrac{1}{3}^{{\scriptscriptstyle-}}$-H\"older in time, $\CT_T^\eps(3)$ may plausibly be decomposed into intervals of length of order $\varepsilon^3$. Setting $\delta = \varepsilon^3$, the number of intervals of length $\delta$ needed to cover $\CT_T^\varepsilon(3)$, and hence $\CT_T(3)$, should be of order $\varepsilon^2/\varepsilon^3 = \varepsilon^{-1} = \delta^{-1/3}$, whence our prediction.

\medskip
Our results and the above conjectures concern the  KPZ fixed point, in which the initial time is fixed at zero. The fractal geometry of exceptional sets may also be explored in the directed landscape~$\mc{L}$ constructed in \cite{Landscape}. This landscape ascribes to any pair of points $(y,s),(x,t) \in \R^2$ with $s < t$
a random real-valued weight $\mc{L}(y,s;x,t)$ that is a scaled expression for the energy in Brownian LPP of the geodesic whose route in scaled coordinates starts at $(y,s)$ and ends at $(x,t)$. A fractal object in the landscape is the difference profile $\R \to \R: z \mapsto \mc{L}(1,0;z,1) - \mc{L}(-1,0;z,1)$. This is a non-decreasing random function whose set of points of increase almost surely has Hausdorff dimension one-half~\cite{basu2019fractal}. For generic point pairs $(y,s),(x,t) \in \R^2$ with $s < t$, there is a unique path---the geodesic---that interpolates the pair's elements whose scaled energy attains the value  $\mc{L}(y,s;x,t)$. Exceptionally, two such paths exist that are disjoint except at their shared endpoints. The set $(x,y) \in \R^2$ for which $(y,0)$ and $(x,1)$ are such a pair of points has Hausdorff dimension one-half almost surely~\cite{BatesGangulyHammond}.

To formulate conjectures that develop the theme of the present article but concern the coupling offered by the system $\mc{L}$, let $X$ denote the set of triples $(y,s,t)$ where $y \in \R$ and $s,t \in \R$ with $s < t$.
Each element indexes a scaled energy profile $\R \to \R: x \mapsto \mc{L}(y,s;x,t)$ which may, exceptionally, have several maximizers.
Indeed, for  $k \in \N$ with $k \geq 2$, we may set $E_k$ equal to the subset of elements of $X$ for which $|s-t|\geq 1$ and for which there are at least $k$ maximizers $x_1,\ldots, x_k$ for which $|x_i-x_j|\geq 1$ for all $1\leq i<j\leq k$.

The space $X$ should be equipped with a metric if Hausdorff dimension questions are to be well-posed.
A natural notion of distance on $X$ stipulates that the distance
$(y_1,s_1,t_1)$ and $(y_2,s_2,t_2)$
equals
$\vert t_1 - t_2 \vert + \vert s_1 - s_2 \vert + \vert y_1 - y_2 \vert^{3/2}$ (this is not a metric as it does not satisfy the triangle inequality).
The power of $3/2$ in the final term accounts for the shape of a KPZ space-time box: if such a box has height $\eps^3$, it will have width $\eps^2$. Despite the mentioned notion of distance not being a metric, this specification of the radius of boxes can be used to define a scaling-adapted notion of Hausdorff measure and dimension.
Such a scaling-adapted specification of a parabolic form of Hausdorff dimension was used by Cafferelli, Kohn and Nirenberg \cite{caffarelli1982partial} in their treatment of partial regularity of weak solutions of the Navier-Stokes' equations.

\begin{conjec}
When $X$ is equipped with the just specified metric, the Hausdorff dimension of~$E_k$ is almost surely equal to $\max \{ 3-k/3, 0\}$.
Moreover, $E_{10}$ is empty almost surely.
\end{conjec}
To make a case for the conjecture, let an $\eps$-box in $X$ take the form of a rectangle whose length is $\eps^3$ in the $t$ and $s$ coordinates, and $\eps^2$ in the $x$ coordinate.
Let $E_k(\eps)$ denote a fattening of the exceptional set $E_k$, wherein a given triple qualifies for membership if there exist $k$ points at pairwise distance at least one at which  the indexed profile achieves a shortfall from its maximum of at most~$\eps$.
By profile Brownianity on unit-order scales, we expect that the intersection of $E_k(\eps)$ with a unit-order ball in $X$ has Lebesgue measure~$\eps^{k-1}$.
But the measure of an $\eps$-box is $\eps^{3 + 3 + 2} = \eps^{8}$.
Thus, the number of boxes needed to cover $E_k(\eps)$ should be of order $\eps^{k-9}$.
The diameter $\delta$ of an $\eps$-box equals $\eps^3$ in the metric we specified above.
The number of boxes in the covering is thus $\delta^{k/3 - 3}$.
This inference points to the sought conclusion that
the Hausdorff dimension of $E_k$ equals $3 - k/3$; provided, of course, that this quantity is not negative.

\subsection{A brief discussion of \cite{dauvergne2022non}}\label{sec:Dau22}
As mentioned, \cite{dauvergne2022non} establishes Conjectures~\ref{conj:density} and \ref{conj:multiple maximizers} as well as a stronger form of Theorem~\ref{thm:main}. Here we discuss how Dauvergne's approach differs from ours.

The basic tool used in \cite{dauvergne2022non} is that almost sure statements about the set of times of twin peaks (or of a greater number of maximizers) for any given initial condition can be transferred to any other initial condition. This relies on $\fh_t$ being absolutely continuous to Brownian motion on compact intervals for any fixed $t>0$ and a very broad class of initial data \cite{sarkar2020brownian}.

This observation makes the subsequent analysis much more streamlined, as one can pick specific and different initial conditions to prove upper and lower bounds on the Hausdorff dimension (though one must restrict to the case where the separation between peaks can be arbitrarily small, i.e., $A=0$; for $A>0$ the twin peaks set will be empty with positive probability, which rules out a transfer strategy). The choice of initial conditions is made such that they are each suited to their respective tasks. In \cite{dauvergne2022non}, a Bessel process initial condition is used for proving an upper bound, and a collection of $k$ unit-order separated narrow wedges is chosen for the lower bound on the dimension of $k$ distinct maximizers.

In contrast, our approach tries to deal with the Hausdorff dimension separately for each initial condition. This requires us to prove estimates, such as Theorems~\ref{thm:densities2-intro} and \ref{thm:TP_probability}, which hold for a broad class of initial data, making the required analysis much more delicate. As we indicated above, some of the difficulties that arose in proving a lower bound seemed to come from the possibility of very flat initial conditions, a possibility which does not need to be directly handled if one can restrict to two specific initial conditions as in \cite{dauvergne2022non}.

\subsection{Notation}
Results from other papers that are restated herein will be called ``Propositions''. For two real numbers, $a$ and $b$, $a\wedge b:= \min(a,b)$ and $a\vee b:=\max(a,b)$. For a function $f : \R \to \R \cup \{\pm \infty\}$ we define $\Max(f) := \sup_{x \in \R} f(x)$, which may equal  $\infty$ even if $f$ is real-valued.  The elements of $\Argmax(f) \subset \R$ are $x \in \R$ such that $f(x) = \Max(f)$. The set $\Argmax(f)$ can be empty if $f$ is unbounded. 
 We will say (and already have said!) that a function $f$ is $\beta^{{\scriptscriptstyle-}}$-H\"older if, for all $\alpha<\beta$, it is $\alpha$-H\"older.

\subsection*{Acknowledgments}
The authors wish to acknowledge the support of the Fernholz Foundation which funded Minerva lectures by A. Hammond at Columbia University in spring 2019. This project started during that visit. We wish to thank Christophe Garban for helpful discussions about dynamic percolation; and Jim Pitman for pointing us to the paper \cite{millar1978path}. We also wish to thank Shirshendu Ganguly and Promit Ghosal for preliminary discussions on this subject.
Finally, we thank the referees for providing a careful reading and identifying some mistakes, in the course of correcting which we discovered the flaw which led to the present form of our results.

I.~Corwin was partially supported by the NSF grants DMS-1811143 and DMS-1664650, and a Packard Fellowship for Science and Engineering. A.~Hammond
 was partially supported by NSF grant DMS-$1512908$ and a Miller Professorship from the Miller Institute for Basic Research in Science.
M.~Hegde was partially supported by the Richman Fellowship of the U.C. Berkeley mathematics department and NSF grant DMS-1855550. K.~Matetski was partially supported by NSF grant DMS-1953859.

\section{The KPZ fixed point}
\label{sec:KPZ_fp}

The KPZ fixed point is a conjectural universal limit of a variety of space-time growth processes as well as directed polymers and interacting particle systems \cite{cqrFixedPt, corwinReview} (see also the lecture notes \cite{1710.02635}).
It was only recently constructed in \cite{fixedpt}, as a limit of TASEP; in \cite{NQR20}, convergence of one-sided reflected Brownian motions to the KPZ fixed point was proved, and, in \cite{TASEPandOthers}, solvability and scaling limits of a general class of determinantal processes are studied. Through a combination of works \cite{fixedpt},  \cite{Landscape} and \cite{NQR20}, a characterization of the KPZ fixed point in terms of a variational problem has been obtained. The fixed point's universality, termed {\it strong KPZ universality}, remains a wide open problem; see, however, recent progress in \cite{sarkar2020fixedpoint,virag2020fixedpoint}. The fixed point is a highly non-trivial object, with work needed to define it. This section is devoted to recalling key results and properties of the KPZ fixed point that we will use.

There are two approaches to the fixed point, each offering a different advantage, perspective and piece of the construction. The earlier approach, due to \cite{fixedpt}, is based on exact formulas; the second, due to \cite{Landscape}, studies a probabilistic object. Although these two works study limits of slightly different versions of TASEP, the approaches are united in \cite{NQR20}. 

The method of \cite{fixedpt} provides a Fredholm determinant formula for the transition probability distribution for the KPZ fixed point.
This formula appears later as \eqref{eq:fpform}; it employs notation and operators defined in Section \ref{sec_defs}. It is proved in \cite{fixedpt} that, for every choice of initial data (within a suitable class), there exists a unique Feller process with transition probabilities given by this key formula. This is the content of the upcoming Proposition~\ref{prop:KPZfp} in which the notation $\fh_{t}(x; \fh_0)$ denotes the KPZ fixed point $\fh_{t}(x)$ started with initial data $\fh_0$ at time $t = 0$. We also will record, as Proposition \ref{prop:sym}, several key properties of this process. The approach in \cite{fixedpt} does not provide a coupling of $\fh_{t}(x; \fh_0)$ over all initial data (i.e., a stochastic flow).

The coupling of all initial data for the KPZ fixed point is achieved in \cite{Landscape} by the construction of the directed landscape, an object that, when parabolic curvature is removed, is the space-time Airy sheet whose existence was conjectured in \cite{cqrFixedPt}.
This four-parameter field can be viewed as a Green's function; the fixed point is then recovered from it by convolution against the desired initial data. That is, for any specific initial data~$\fh_0$, the resulting space-time process is equal in law to $\fh_{t}(x; \fh_0)$: see Definition~\ref{def:KPZfp} and Proposition~\ref{prop:KPZfp} below. This variational representation for the KPZ fixed point is described in Section \ref{sec_var_form} along with some key properties of the directed landscape. Since the original posting of this manuscript, there has been progress \cite{dauvergne2021scaling} in establishing convergence to the directed landscape in a number of models apart from the original one used in \cite{Landscape}.

Since Theorem \ref{thm:main} is only concerned with the KPZ fixed point for a single (though arbitrary) initial data, we will mainly rely on the Fredholm determinant transition probability formula of \cite{fixedpt} in our analysis.  All of this analysis appears in the proof of our key technical result, Proposition~\ref{prop:densities}. The directed landscape is invoked in proving two lemmas in Section \ref{sec_max_kpz} which pertain to the existence and temporal H\"older continuity of the maximum of the KPZ fixed point, as well as in proving Theorem~\ref{thm:TP_probability}.

\subsection{A variational formula}
\label{sec_var_form}

We start by defining the KPZ fixed point through a variational formula. To do this, we will use the parabolic Airy sheet $\CS : \R^2 \to \R$, which was introduced in \cite{cqrFixedPt} and constructed in \cite{Landscape} as a scaling limit of Brownian last passage percolation. The projection $\CS(0, \bigcdot)$ of the random continuous function~$\CS$  has the distribution of the \Airy process minus the parabola~$x^2$. The parabolic Airy sheet of scale $s > 0$ is defined by
\begin{equation}
\CS_s(x,y) := s \CS \bigl(s^{-2} x, s^{-2} y\bigr),
\end{equation}
for $x,y \in \R$. Some properties of the parabolic Airy sheet can be found in Definition~1.2 and Lemma~9.1 of \cite{Landscape}.

The \emph{directed landscape} can be defined using the parabolic Airy sheet. Denote $\R^4_{\uparrow} := \{(x, s; y, t) \in \R^4 : s < t\}$, where $s$ and $t$ will be viewed as temporal, and $x$ and $y$ as spatial, variables. The directed landscape is a random continuous function $\CL: \R^4_{\uparrow} \to \R$, which satisfies
\begin{equation}
\CL(x, r; y, t) = \max_{z \in \R} \bigl\{\CL(x, r; z, s) + \CL(z, s; y, t)\bigr\},
\end{equation}
for $(x, r; y, t) \in \R^4_{\uparrow}$ and $s \in (r,t)$. It is characterized by the property that for any $k\in \N$, $t_1,\ldots, t_k>0$ and $s_1,\ldots, s_k>0$, provided that $(t_i,t_i+s_i^3)$ are pairwise disjoint intervals for $1\leq i\leq k$, the collection of processes
 $\left\{\CL(\bigcdot, t_i; \bigcdot, t_i + s_i^3)\right\}_{i=1}^{k}$ are independent parabolic Airy sheets of scale $s_i$, $1\leq i\leq k$.

\begin{defn}\label{def:KPZfp}
The KPZ fixed point $\fh_t$ for $t > 0$, starting from $\fh_0$ at time $t = 0$, is defined by the variational formula
\begin{equation}\label{eq:variational}
\fh_t(x; \fh_0) := \sup_{y \in \R} \bigl\{ \fh_0(y) + \CL(y, 0; x, t) \bigr\},
\end{equation}
provided that the supremum is finite.
\end{defn}

In order that the supremum in \eqref{eq:variational} be finite almost surely, we need to restrict the growth of $\fh_0(y)$ as $y \to \pm \infty$. The assumption made in  \cite{fixedpt} is of at most linear growth of $\fh_0$ for large $\vert y \vert$.

If we take the initial data $\fh_0$ to be the narrow wedge \eqref{eq:NW}, then, from \eqref{eq:variational}, we recover \eqref{eq:KPZ_NW}. We denote by $\P_{\fh_0}$ the probability distribution of the KPZ fixed point starting at $\fh_0$; $\E_{\fh_0}$ is the associated expectation. When the initial data is clear, we prefer to write $\fh_t(x)$ instead of $\fh_t(x; \fh_0)$.

\cite[Corollary 10.7]{Landscape} yields the following bound on the directed landscape:
\begin{prop}\label{prop:Landscape_bound}
For any $\delta > 0$,
\begin{align}\label{eq:Landscape_bound1}
\Bigl|\CL(y, s; x, t) + \tfrac{(y - x)^2}{t-s}\Bigr| \leq \CC_\delta \bigl(1 + \Vert y, s, x, t \Vert^{\frac{1}{5}}\bigr) (t-s)^{\frac{1}{3} - \delta},
\end{align}
for all $(y, s; x, t) \in \R^4_{\uparrow}$. Here, $\Vert \bigcdot \Vert$ denotes the $\ell_1$-norm on $\R^4$;  $\CC_\delta$ is a random variable, which depends on $\delta$ but is independent of $(y, s, x, t)$ and satisfies $\P (\CC_\delta > m) \leq a \exp(-c m^{3/2})$ for any $m \geq 0$ and for some positive $\delta$-dependent constants $a$ and $c$. In particular, $\CC_\delta$ is almost surely finite.
\end{prop}
The bound \eqref{eq:Landscape_bound1} is not optimal, but it will suffice for our purpose. We have applied the estimate $\log (1 + |x|) \leq \beta |x|^{\alpha}$, valid for any $\alpha > 0$ and some $\beta = \beta(\alpha) > 0$, to \cite[Corollary~10.7]{Landscape}.
\subsection{A formula for the distribution function}
\label{sec_defs}

Here we provide a formula for the distribution function of the KPZ fixed point \eqref{eq:variational}.
We need some definitions  from \cite[Sections~3.3 and 4.1]{fixedpt} wherein are found proofs---for example, of existence of limits---that the definitions are well-posed.

\subsubsection{Functions and graphs}

For a function $f : \R \to [-\infty, \infty]$ we define its \emph{hypograph} $\hypo(f) := \{(x,y) \in \R^2: y \le f(x)\}$ and \emph{epigraph} $\epi(f) := \{(x,y) \in \R^2: y \ge f(x)\}$. We recall that a function $f$ is upper (or lower) semicontinuous if and only if $\hypo(f)$ (or $\epi(f)$) is closed.
We declare that an upper semicontinuous function $f : \R \to [-\infty, \infty)$ belongs to  the set $\UC$  if $f \not\equiv -\infty$ and if, for some real $\alpha$ and $\gamma$, $f(x) \leq \alpha + \gamma |x|$ for all $x\in \R$. We equip $\UC$ with the topology of local $\UC$ convergence, a description of which can be found in \cite[Section~3.1]{fixedpt}. We denote by  $\LC := \{f : -f \in \UC\}$  the counterpart class of lower semicontinuous functions.

\subsubsection{Integral operators involving Airy functions}

The classical Airy function $\Ai$ can be defined via the contour integral $\Ai(z) := \frac1{2\pi\I} \int_{\langle} e^{w^3 / 3-z w} \d w$, where ``${\langle}~$'' is the positively oriented contour  running from $e^{-\I\pi/3}\infty$ to $e^{\I\pi/3}\infty$ through $0$ (see \cite{abrSteg, AiryBook, NIST:DLMF} for properties of $\Ai$). Using this function, we define the family of operators
\begin{equation}\label{eq:groupS}
\fT_{t,x} := \exp \bigl\{ x \partial^2 + \tfrac{t}3\tts\partial^3 \bigr\}, \qquad x, t\in\R^2\setminus \{x<0, t= 0\},
\end{equation}
which act on the domain $\mathscr{C}_0^\infty(\R)$ of smooth, compactly supported functions. This operator (and many others considered below) can be written explicitly as an integral operator with respect to a kernel. We will use the same notation for the kernel and the operator.
The operator $\fT_{t,x}$  acts on functions $f \in \mathscr{C}_0^\infty(\R)$ as $(\fT_{t,x} f)(z) = \int_\R \fT_{t,x}(z,y) f(y)\, \d y$ with the kernel $\fT_{t,x}(z,y) = \fT_{t,x}(z-y)$, which for $t > 0$ is given by
\begin{equation}\label{eq:fTdef}
\fT_{t,x}(z) := \frac1{2\pi\I} \int_{\langle}\ts e^{\frac{t}3 w^3+x  w^2+z w} \d w\, =\, t^{-1/3} e^{\frac{2 x^3}{3t^2}-\frac{zx}{t}}\tsm\Ai \bigl(-t^{-\frac{1}{3}} z+t^{-\frac{4}{3}}x^2\bigr).
\end{equation}
For $t<0$, the kernel is defined by setting $\fT_{t,x}(z) := \fT_{-t,x}(-z)$, which yields $\fT_{t,x}=(\fT_{-t,x})^*$, where the latter is the adjoint operator. Using properties of the Airy function, one shows that
\begin{equation}\label{eq:S_properties}
\fT_{s,x}\fT_{t,y}=\fT_{s+t,x+y}, \qquad\qquad (\fT_{t,x})^*\fT_{t,-x} = I,
\end{equation}
as long as all parameters avoid the set $\{x<0, t= 0\}$. Here, $I$ is the identity operator. 
In Appendix~\ref{sec:kernels_bounds}, we prove several bounds on the kernel \eqref{eq:fTdef}.

\subsubsection{The Brownian scattering transform}

Let $\fB$ be a Brownian motion with rate two, so that $\E [\fB(\ell)^2] = 2 \ell$. Let $p_\ell(u) := \P_{\fB(0)=0} (\fB(\ell)=u)$ be its transition kernel, given by
\begin{equation}\label{eq:heat_kernel}
p_\ell(u) = \frac{1}{\sqrt{4 \pi \ell}} e^{- \frac{u^2}{4 \ell}}, \qquad p_0(u) = \delta_u,
\end{equation}
for $\ell > 0$ and $u \in \R$. For any $\ell_1 < \ell_2$, and for a function $\ff \in \UC$, we define the ``No hit'' operator by
 \begin{equation}\label{eq:no_hit}
\fP_{[\ell_1,\ell_2]}^{\nohit(\ff)}(u_1,u_2) := \P_{\fB(\ell_1)=u_1, \fB(\ell_2) = u_2}\!\Big(\fB(y)>\ff(y)\text{ for } y \in [\ell_1,\ell_2] \Big) p_{\ell_2 - \ell_1}(u_2 - u_1),
\end{equation}
where the denoted event is that of a Brownian bridge, from $(\ell_1, u_1)$ to $(\ell_2, u_2)$, staying above the function $\ff$. We further define the ``Hit'' operator
\begin{equation}\label{eq:hit_def}
\fP_{[\ell_1,\ell_2]}^{\hit(\ff)} := I- \fP_{[\ell_1,\ell_2]}^{\nohit(\ff)},
\end{equation}
where, as before, $I$ is the identity operator.

For $f \in L^2(\R)$, $t > 0$, and  for  $\kappa > 0$ fixed,  we define the multiplicative operator
\begin{equation}\label{eq:Gamma_L}
\Gamma_t f(u) := e^{G_t(u)}f(u), \quad \text{with}\quad G_t(u) := t^{-\frac{1}{2}} \kappa \sgn(u) |u|^{\frac{3}{2}}.
\end{equation}
It is stated in the proof of Theorem~4.1 in \cite{fixedpt} that, for every $t > 0$, there exists a value $\kappa_\star$ such that, for $0 < \kappa < \kappa_\star$, the \emph{Brownian scattering transform} may be defined as a map from any function $\ff \in \UC$ to a $t$-dependent operator on $L^2(\R)$,
\begin{equation}\label{eq:Kd}
\fK^{\hypo(\ff)}_{t} := \lim_{\small\substack{\ell_1\to -\infty\\\ell_2\to \infty}} \Gamma_t (\fT_{t,\ell_1})^*\, \fP_{[\ell_1,\ell_2]}^{\hit(\ff)}\fT_{t,-\ell_2} \Gamma_t,
\end{equation}
where $\fT_{t,x}$ is specified in \eqref{eq:groupS}, and where the limit is in the trace norm in the space $L^2(\R)$. Moreover, if $\ff$ decays at $\pm \infty$, then the bound on the trace norm of $\fK^{\hypo(\ff)}_{t}$ depends on $\ff$ merely through the maximal value of $\ff$. Indeed, the initial data of the KPZ fixed point considered in \cite{fixedpt} has at most linear growth at infinity; i.e., $\fh_0(x) \leq \bar \alpha + \bar \gamma |x|$ for all $x \in \R$. Obviously, if $\fh_0$ decays at infinity, it is in this class of initial data with $\bar \gamma = 0$. Then the last formula in \cite[Appendix~A.1]{fixedpt} gives a bound on $\fK^{\hypo(\fh_0)}_{1/2}$, which depends on the initial data only via~$\bar \alpha$. 

In \cite[Section~4.1]{fixedpt}, the Brownian scattering transform $\fK^{\hypo(\ff)}_{t}$ was in fact defined as the limit of the kernels \eqref{eq:Kd} without using $\Gamma_t$, but by restricting the space to $L^2([a,\infty))$ for any fixed $a$. Such restriction of the space appears naturally from the determinantal formula for the multipoint distribution function of the KPZ fixed point \cite[Eq.~4.7]{fixedpt}, and one of the roles of $a$ is to avoid divergence of the $\ell$-dependent kernel in \eqref{eq:Kd} at $-\infty$. Since we will use the continuum statistics formula in Proposition~\ref{prop:KPZfp} below, divergence of the kernel is controlled by the operator $\Gamma_t$ in \eqref{eq:Kd}.

The object $\fK^{\hypo(\ff)}_{t}$ is an integral operator, whose existence was proved in \cite[Theorem~4.1]{fixedpt}. For any $\fg \in \LC$, we similarly define the integral operator on $L^2(\R)$
\begin{equation}\label{eq:Kd_epi}
\fK^{\epi(\fg)}_{-t} := \varrho\fK^{\hypo(-\fg)}_t\varrho,
\end{equation}
where $\varrho$ is the reflection operator $\varrho f(u) := f(-u)$.

\subsubsection{The KPZ fixed point formula}
\label{sec:KPZfp}

Using these definitions, we can provide a formula for the distribution function of the KPZ fixed point  \eqref{eq:variational}, which can be found in \cite[Section~4.2]{fixedpt}. Some basic properties of the Fredholm determinant which appears on the right-hand side of \eqref{eq:fpform} may be found in Appendix~\ref{sec:Fredholm}. 

\begin{prop}\label{prop:KPZfp}
For any $\fh_0\in\UC$, the KPZ fixed point $t \mapsto \fh_t(\,\bigcdot\,; \fh_0)$ is a Feller process on $\UC$, whose distribution function for any $t > 0$ and $\fg \in \LC$ is given by the formula
\begin{equation}\label{eq:fpform}
 \P_{\fh_0} \bigl( \fh_t(x) \le \fg(x),\, \forall\, x\in\R\bigr) = \det \Bigl(I - \fK^{\hypo(\fh_0)}_{t/2}\fK^{\epi(\fg)}_{-t/2} \Bigr)_{L^2(\R)}.
\end{equation}
\end{prop}

We note that the trace norm of the operator $\fK^{\hypo(\fh_0)}_{t/2}\fK^{\epi(\fg)}_{-t/2}$ is bounded, because, in view of \eqref{eq:norms_products} and \eqref{eq:norms}, it is a composition of two trace class operators; in particular, the Fredholm determinant in \eqref{eq:fpform} is well-defined.

Properties of the KPZ fixed point can be found in \cite[Section~4.3]{fixedpt}. Here, we list only those which will be used in this article.

\begin{prop}\label{prop:sym}
Let $\fh_t(x; \fh_0)$ denote the KPZ fixed point with initial data $\fh_0\in \UC$. Then it has the following properties, where all of the identities in distribution are for random functions in the spatial variable on $x\in\R$.
\begin{enumerate}[label=\normalfont{(\roman*)},ref=\roman*]
\item \label{123a}  \emph{(1:2:3 scaling invariance)} For any $\alpha>0$, one has $\alpha\tts\fh_{\alpha^{-3}t}(\alpha^{-2}x; \fh^{\alpha}_0)\dist \fh_t(x;\fh_0)$ as a process in both $x$ and $t$, where $\fh^\alpha_0(x) := \alpha^{-1}\fh_0(\alpha^{2}x)$.
\item \label{sis} \emph{(Stationarity in space)} For any $t > 0$ and $u \in \R$, $\fh_t(x +u; \fh_0(\,\bigcdot\, -u))\dist \fh_t(x; \fh_0)$, where we write $\fh_0(\,\bigcdot\, -u)$ for the function $y \mapsto \fh_0(y -u)$.
\item \label{ai} \emph{(Affine invariance)} Let $\ff(x) = \fh_0(x) + a + c x$ for some constants $a, c$. Then one has $\fh_t(x; \ff)\dist \fh_t(x+\frac12ct; \fh_0) + a + cx + \frac14c^2t$ as a process in $x$ and $t$.
\item  \label{str} \emph{(Skew time reversibility)} For any functions $\ff,\fg\in\UC$,
\begin{equation}
\P\big(\fh_t(x; \fg)\le -\ff(x), \,\forall\,x \in \R\big) =\P\big(\fh_t(x;\ff)\le -\fg(x), \,\forall\,x \in \R\big).
\end{equation}
\item \label{pom} \emph{(Preservation of max)} For any $\ff_1,\ff_2\in\UC$ and $t > 0$, the KPZ fixed points $\fh_t(\bigcdot\,; \ff_1)$ and $\fh_t(\bigcdot\,; \ff_2)$ can be coupled so that $\fh_t(x; \ff_1)\vee \fh_t(x; \ff_2) \dist \fh_t(x;\ff_1\vee \ff_2)$ for all $x \in \R$.
\item \label{mon} \emph{(Monotonicity)} For any $t > 0$ and $\ff_1,\ff_2\in\UC$ such that $\ff_1 \geq \ff_2$, the KPZ fixed points $\fh_t(\bigcdot\,; \ff_1)$ and $\fh_t(\bigcdot\,; \ff_2)$ can be coupled so that $\fh_t(x;\ff_1) \geq \fh_t(x; \ff_2)$ for all $x \in \R$ \as
\item \label{space} \emph{($\tfrac{1}{2}^{{\scriptscriptstyle-}}-$H\"older regularity in space)} For any fixed $t>0$ and $\beta\in (0,\tfrac12)$, the function $x \mapsto \fh_t(x; \fh_0)$ is almost surely locally $\beta$-H\"{o}lder continuous.
\item \label{time} \emph{($\tfrac{1}{3}^{{\scriptscriptstyle-}}$-H\"older regularity in time)} For any fixed $x \in \R$ and $\beta\in (0,\tfrac13)$, the function $t \mapsto \fh_t(x; \fh_0)$ is almost surely locally $\beta$-H\"{o}lder continuous.
\end{enumerate}
\end{prop}

\noindent The monotonicity (\ref{mon}) follows from the preservation of maximum (\ref{pom}). Proofs of spatial and temporal regularities of the KPZ fixed point can be found in~\cite[Theorem~4.13 and Proposition~4.24]{fixedpt}.

\begin{lem}
For any $\fh_0\in \UC$, the KPZ fixed point $t \mapsto \fh_t$ enjoys the strong Markov property.
\end{lem}

\begin{proof}
The strong Markov property follows from time-continuity of the KPZ fixed point (see Proposition~\ref{prop:sym}(\ref{time})), the Feller property (see Proposition~\ref{prop:KPZfp}), and \cite[Theorem~2]{Dynkin}.
\end{proof}

%

\section{Maximum of the KPZ fixed point}\label{sec_max_kpz}
\label{sec:maximum}

In this section, we prove that, if the initial data has parabolic decay at infinity, then the KPZ fixed point at every time is almost surely bounded above and decays at infinity. Moreover, we obtain the temporal H\"{o}lder regularity of the maximum value, provided that it lies in a bounded interval. Along the way, we will prove H\"older continuity of $\fh_t(x)$ in each variable, locally uniformly in the other variable.

We start by proving boundedness of the KPZ fixed point.

\begin{lem}\label{lem:KPZ_is_bounded}
For any initial data $\fh_0$ satisfying Assumption~\ref{a:initial_state_pd}(a) and for any $T > 0$, there exists a random variable $\CB \in \R$, which is almost surely finite, such that the KPZ fixed point $\fh_{t}$ satisfies
\begin{equation}\label{eq:KPZ_is_bounded}
  \fh_t(x) \leq \CB - \frac{1}{8}\min\left(\sqrt{2}\gamma, \frac{1}{t}\right) |x|^{1/2},
\end{equation}
for all $x \in \R$ and $t \in [0, T]$, where the constant $\gamma > 0$ is from Assumption~\ref{a:initial_state_pd}. In particular, for every fixed $t \geq 0$, the KPZ fixed point $\fh_t(x)$ is almost surely bounded  above, and $\lim_{x \to \pm \infty} \fh_t(x) = -\infty$ almost surely.
\end{lem}

\begin{proof}
We first observe that, by the continuity of $(t,x)\mapsto \fh_t(x)$, it is enough to prove \eqref{eq:KPZ_is_bounded} under the assumption that $|x|\geq 1$.

We will use the variational formula \eqref{eq:variational} and properties of the directed landscape. Let us fix some $\delta>0$ and denote $F_\delta(y, s, x, t) := \CC_\delta (1 + \Vert y, s, x, t \Vert^{1 / 5})$, where $\Vert \bigcdot \Vert$ is the $\ell_1$-norm on $\R^4$ and $\CC_\delta$ is the $\delta$-dependent random variable from \eqref{eq:Landscape_bound1}, which is positive and almost surely finite. Then, for any $(y, s; x, t) \in \R^4_{\uparrow}$, \eqref{eq:Landscape_bound1} yields
\begin{equation}\label{eq:Landscape_bound2}
- F_\delta(y, s, x, t) (t-s)^{\frac{1}{3} - \delta} - \tfrac{(x-y)^2}{t-s} \leq \CL(y, s; x, t) \leq F_\delta(y, s, x, t) (t-s)^{\frac{1}{3} - \delta} - \tfrac{(x-y)^2}{t-s}.
\end{equation}
Since we are going to use $F_\delta$ and $\CC_\delta$ for a fixed value $\delta$, we prefer to make our notation slightly lighter and omit the subscripts $\delta$ in the rest of this proof. From \eqref{eq:variational}, \eqref{eq:Landscape_bound2} and Assumption~\ref{a:initial_state_pd}, we obtain
\begin{equation}\label{eq:KPZ_is_bounded1}
\fh_t(x) \leq \sup_{y \in \R} \Bigl\{ \alpha - \gamma |y|^{1/2} + F(y, 0, x, t) t^{\frac{1}{3} - \delta} - \tfrac{(x-y)^2}{t} \Bigr\}.
\end{equation}
We consider $t \in [0, T]$ for fixed $T > 0$.

We want to upper bound $-\gamma |y|^{1/2}/2 - (x-y)^2/t$ by $-c_{\gamma,t} |x|^{1/2}$ for some $c_{\gamma,t}>0$ for all $x, y \in \R$. First, by the triangle inequality we have $(x-y)^2/t \geq (|x|-|y|)^2/t$. Next, we write 
$$\frac{\gamma}{2}|y|^{1/2}+\frac{(|x|-|y|)^2}{t} = \frac{\gamma}{2}|y|^{1/2}+|x|^2\frac{(1-|y|/|x|)^2}{t}.$$
If $|y|/|x| \leq 1/2$, we ignore the first term and obtain a lower bound of $|x|^{1/2}/(4t)$, using $|x|^2\geq |x|^{1/2}$since $|x|\geq 1$ by assumption. If $|y|/|x|\geq 1/2$, we ignore the second term and obtain a lower bound of $\gamma |x|^{1/2}/2\sqrt{2}$. Thus we see
$$-\frac{\gamma}{2} |y|^{1/2} - \frac{(x-y)^2}{t} \leq -c_{\gamma,t}|x|^{1/2}$$
with $c_{\gamma,t} = \min(1/(4t), \gamma/2\sqrt{2})$. Substituting this in \eqref{eq:KPZ_is_bounded1}, we obtain
\begin{align*}
\MoveEqLeft[6]
\sup_{y\in\R}\left\{\alpha - \gamma |y|^{1/2} + F(y, 0, x, t) t^{\frac{1}{3} - \delta} - \frac{(x-y)^2}{t}\right\}\\
&\leq \sup_{y\in\R}\left\{\alpha - \frac{\gamma}{2}|y|^{1/2}+F(y, 0, x, t)t^{\frac{1}{3} - \delta} - c_{\gamma,t}|x|^{1/2}\right\}\\
&\leq \sup_{x, y\in\R}\left\{\alpha - \frac{\gamma}{2}|y|^{1/2}+F(y, 0, x, t)t^{\frac{1}{3} - \delta} - \frac{1}{2}c(\gamma,t)|x|^{1/2}\right\} - \frac{1}{2}c_{\gamma,t}|x|^{1/2}.
\end{align*}
Now, since $F(y, 0, x, t) \leq \CC(1+|x|^{1/5}+|y|^{1/5} + |t|^{1/5})$, it is easy to see that the first term in the previous displayed inequality is an almost surely finite random constant that depends on only $\gamma$ and $T$; we label it $\CB$. Boundedness and decay at infinity of the KPZ fixed point follow readily from \eqref{eq:KPZ_is_bounded}.
\end{proof}

\subsection{Locally uniform spatial \& temporal H\"{o}lder regularity}

To prove the H\"older regularity of the maximum of $\fh_t$ (Lemma~\ref{lem:max_Holder}), we will need that the temporal H\"{o}lder regularity Proposition~\ref{prop:sym}(\ref{time}) holds locally uniformly in $x$. This is Lemma~\ref{lem:Holder-time} ahead. Its proof will first require spatial H\"older regularity of $\fh_t$ to hold locally uniformly in $t$, which is the next statement.

\begin{lem}\label{lem:Holder-space}
For any initial state $\fh_0$ satisfying Assumption~\ref{a:initial_state_pd}(a), $K>0$, $0 < T_0 < T$ and  $\delta\in (0,\tfrac12]$, one has almost surely
\begin{equation}
  \sup_{s \in [T_0, T]} \sup_{\substack{x\neq y,\\|x|,|y|\leq K}} \frac{|\fh_s(x) - \fh_s(y)|}{|x - y|^{1/2-\delta}} < \infty.
\end{equation}
\end{lem}

We will need some modulus of continuity bounds for $\mc L$ which is a special case of \cite[Proposition~10.5]{Landscape}. 

\begin{lemma}\label{l.landscape mod of con}
Fix $L>0$ and $\varepsilon>0$. There exists a random constant $C$ which depends on $L$ and $\varepsilon$ and which is almost surely finite such that, for $x,y,z\in[-L,L]$ and $0<s<L$, $s+\varepsilon<t<L$,
\begin{align*}
\MoveEqLeft[12]
\left|\mc L(z,0;x,s)+\frac{(x-z)^2}{s} - \mc L(z,0;y,t)-\frac{(y-z)^2}{t}\right|\\
&\leq C\left(\tau^{1/3}\log^{2/3}\tau^{-1}+ \xi^{1/2}\log^{1/2}\xi^{-1}\right),
\end{align*}
where $\tau=t-s$ and $\xi=|x-y|$.
\end{lemma}

\begin{proof}[Proof of Lemma~\ref{lem:Holder-space}]
We make use of Lemma~\ref{l.landscape mod of con} which implies that, for $\delta \in (0,1/2]$,  $0< T_0 \leq T$, and  $L>0$, there exists a random constant $\mc C_L$ (depending also on $\delta$, $T_0$, and $T$) such that, for all $s\in[T_0,T]$ and  $x$, $y$, and $z$ with $|x|, |y|, |z| \leq L$,
\begin{equation}\label{e.landscape holder}
|\mc L(z,0; y, s) - \mc L(z, 0; x, s)| \leq \mc C_L|x-y|^{\frac{1}{2}-\delta}.
\end{equation}
Now we observe that, for any $x,y\in\R$,
\begin{align*}
\fh_s(y) - \fh_s(x) &= \sup_{z\in\R} \left(\fh_0(z) + \mc L(z,0;y,s)\right) - \sup_{z\in\R}\left(\fh_0(z) + \mc L(z,0;x,s)\right)\\
&\leq \mc L(z_s^*(y), 0; y,s) - \mc L(z_s^*(y), 0; x,s),
\end{align*}
where $z_s^*(y)$ is the element of minimum absolute value in the set of maximizers of the first supremum.

Define $\bar z = \sup_{s\in[T_0, T]}\sup_{|y|\leq K} |z_s^*(y)|$. We {\em claim} that $\bar z$ is almost surely finite. Before proving the claim, we show how it completes the proof of Lemma~\ref{lem:Holder-space}.

Since $\delta$, $T_0$, and $T$ are fixed, let, for $L\in\N$, $\mc C_L$ be the corresponding constant in \eqref{e.landscape holder}. Define $\mc C$ by setting it to be $\mc C_L$ on the event that $\bar z\in [L, L+1]$, for every $L\geq \lceil K \rceil$, and to be $\mc C_{\lceil K\rceil}$ on the event that $\bar z \in [0, \lceil K\rceil]$. Then we have that $\fh_s(y) -\fh_s(x) \leq \mc C|x-y|^{1/2-\delta}$ for all $|x|,|y| \leq K$ and $s\in [T_0,T]$; by symmetry, the same upper bound holds for $|\fh_s(y) - \fh_s(x)|$. Since $\mc C$ is clearly almost surely finite, Lemma~\ref{lem:Holder-space} follows conditional upon the claim.

We now prove the claim that $\bar z <\infty$ almost surely. Let $z_0\in\R$ be such that $\fh_0(z_0) > -\infty$. It is enough to prove that $\fh_0(z_0) + \inf_{s\in[T_0,T]}\inf_{|y|\leq K} \mc L(z_0,0;y,s) =: -R > -\infty$, i.e., $R<\infty$, and that there exists a random constant $M$ depending on only $T_0$, $T$, and $K$ such that $\sup_{|z| \geq M} (\fh_0(z)+\mc L(z,0;y,s)) < -R$ for all $|y|\leq K$.

That $R<\infty$ follows immediately from the continuity of $\mc L$ and the definition of $z_0$. To find the random $M$ satisfying the conditions, we observe that, from Assumption~\ref{a:initial_state_pd} and \eqref{eq:Landscape_bound2} with $\delta=1/3$,
\begin{align*}
\fh_0(z)+\mc L(z,0;y,s) &\leq \alpha-\gamma |z|^{1/2} + F_{1/3}(z,0,y,s) - \frac{(z-y)^2}{s}\\
&\leq \alpha-\gamma |z|^{1/2} + F_{1/3}(z,0,K,T) - \frac{(z-y)^2}{T_0}.
\end{align*}
It is easy to see that we can find $M$ large enough depending on only $\alpha$, $\gamma$, $K$, $T$ and $T_0$ such that $|z| > M$ implies that the previous expression is smaller than $-R$  for all $y$ with $|y|\leq K$. This completes the proof of the claim, and thus of  Lemma~\ref{lem:Holder-space}.
\end{proof}

\begin{lem}\label{lem:Holder-time}
For any initial state $\fh_0$ satisfying Assumption~\ref{a:initial_state_pd}(a), $L > 0$,  $0 < T_0 < T$ and  $\delta\in (0,\tfrac13]$, one has almost surely
\begin{equation}
	\sup_{|x| \leq L}\sup_{\substack{s \neq t\\ s,t \in [T_0, T]}} \frac{|\fh_t(x) - \fh_s(x)|}{|t - s|^{1/3-\delta}} < \infty.
\end{equation}
\end{lem}

\begin{proof}
For $s < t$,  the variational formula \eqref{eq:variational} yields
\begin{equation}\label{e.t,s difference variational formula}
	\fh_t(x) - \fh_s(x) = \sup_{y \in \R} \bigl\{ \fh_s(y) - \fh_s(x) + \CL(y, s; x, t) \bigr\}.
\end{equation}
We wish to bound $\fh_s(y) -\fh_s(x)$ by $\mc C|y-x|^{1/2-\delta}$ for a random constant $\mc C$ uniform in $s\in[T_0,T]$ using Lemma~\ref{lem:Holder-space}. However, this requires that $y$ be restricted to a compact interval, and so we start by showing that we may localize the supremum in the previous display to a compact interval. More precisely, letting $y_{s,t}^*(x)$ be the maximizer of the supremum in \eqref{e.t,s difference variational formula} of minimum absolute value, we claim that, almost surely,
\begin{equation}\label{e.localize sup}
\sup_{|x|\leq L}\sup_{s\neq t\in[T_0, T]} |y_{s,t}^*(x)| < \infty.
\end{equation}
First let $R' = -\inf_{|x|\leq L}\inf_{s\in[T_0,T]} \fh_s(x)$. By the variational formula \eqref{eq:variational} and the continuity of $\CL$, we have that $R'<\infty$ almost surely. Thus we see that
$$\sup_{y \in \R} \bigl\{ \fh_s(y) + \CL(y, s; x, t) \bigr\} \geq \fh_s(x) + \CL(x,s; x,t) \geq -R' + \inf_{|x|\leq L}\inf_{s\neq t\in[T_0,T]} \CL(x,s; x,t) =: -R.$$
Bounds  \eqref{eq:Landscape_bound2} on $\CL$ imply that $R<\infty$ almost surely; note that, by definition, $R$ is uniform over $|x|\leq L$ and $s,t\in[T_0,T]$. Thus, to prove \eqref{e.localize sup}, it is enough to show that there exists a random $M$ large enough such that, for all $s\neq t\in[T_0,T]$ and $|x|\leq L$,
\begin{equation}\label{e.defM}
\sup_{|y|\geq M} \bigl\{\fh_s(y) + \CL(y,s; x,t) \bigr\} < -R.
\end{equation}
As in the proof of Lemma~\ref{lem:Holder-space}, this follows by using that $\CL(y,s; x,t) \leq F_{1/3}(y,s,x,t) - \frac{(x-y)^2}{t-s}$ and a uniform decay estimate on $\fh_s(y)$. This decay estimate, provided by Lemma~\ref{lem:KPZ_is_bounded}, asserts that $\fh_s(y) \leq \CB - \frac{1}{8}\min(t^{-1}, \sqrt{2}\gamma) |y|^{1/2} \leq \CB - \frac{1}{8}\min(T^{-1}, \sqrt{2}\gamma) |y|^{1/2}$ for all $s,t\in[T_0, T]$ and $y\in \R$, where $\CB$ is uniform over $s\in[T_0,T]$.

Thus, using the random $M$ which satisfies \eqref{e.defM} above, we see that \eqref{e.t,s difference variational formula} may be written as
\begin{equation}\label{e.difference variational formula 2}
\fh_t(x) - \fh_s(x) = \sup_{|y| \leq M} \bigl\{ \fh_s(y) - \fh_s(x) + \CL(y, s; x, t) \bigr\}.
\end{equation}
Now, from Lemma~\ref{lem:Holder-space}, for given $\delta_1>0$, there exists a random constant $\mc C_1$ such that, for all $x,y \in[-M,M]$ and $s\in[T_0,T]$,
\begin{equation}\label{e.uniform space holder}
|\fh_s(y) - \fh_s(x)| \leq \mc C_1|y-x|^{\frac{1}{2}-\delta_1};
\end{equation}
though $M$ is random, we know such a random and a.s. finite constant $\mc C_1$ exists by applying Lemma~\ref{lem:Holder-space} on each of the events $M\in[m,m+1]$ for $m\in\N$ and getting a different random constant on eachof these disjoint events.

Applying this and bounds on $\mc L$ from \eqref{eq:Landscape_bound2} in \eqref{e.difference variational formula 2} yields, for $|x|\leq L$ and $s,t \in[T_0,T]$,
\begin{align*}
\fh_t(x)-\fh_s(x) &\leq \sup_{|y|\leq M} \left\{\mc C_1|y-x|^{\frac{1}{2}-\delta_1} + F_{\delta}(y,s,x,t)|t-s|^{\frac{1}{3}-\delta} - \frac{(x-y)^2}{t-s}\right\} \\
&\leq \sup_{|y|\leq M} \left\{\mc C_1|y-x|^{\frac{1}{2}-\delta_1} - \frac{(x-y)^2}{t-s}\right\} + F_{\delta}(M,T,L,T)|t-s|^{\frac{1}{3}-\delta}.
\end{align*}
Simple calculus yields that the supremum is at most $\mc C_2 |t-s|^{1/3-\delta_2}$ for a random constant $\mc C_2$, where $\delta_2 = 8\delta_1/(9+6\delta_1)$. We set $\delta_1$ so that $\delta_2$ equals the given positive value $\delta$ in Lemma~\ref{lem:Holder-time}, thus obtaining  $\fh_t(x)-\fh_s(x) \leq \mc C_3 |t-s|^{1/3-\delta}$ for a random constant $\mc C_3$ and for all $|x|\leq L$ and $s,t\in[T_0,T]$.

For the lower bound, again \eqref{e.uniform space holder} and \eqref{eq:Landscape_bound2} imply
\begin{align*}
\fh_t(x)-\fh_s(x) &\geq \sup_{|y|\leq M} \left\{-\mc C_1|y-x|^{\frac{1}{2}-\delta} - F_{\delta}(y,s,x,t)|t-s|^{\frac{1}{3}-\delta} - \frac{(x-y)^2}{t-s}\right\},\\
&\geq -\mc C_4 |t-s|^{\frac{1}{3}-\delta},
\end{align*}
by noting that $F_{\delta}(y,s,x,t) \leq F_\delta(M,T,L,T)$ and considering the value of the remaining expression at $y = x$, which satisfies $|y| \leq M$ by assuming $M>L$ if necessary. This completes the proof of Lemma~\ref{lem:Holder-time}.
\end{proof}

Our next aim is to obtain temporal H\"{o}lder regularity of the maximum of the following cutoff of the KPZ fixed point:
\begin{equation}\label{eq:cut-off}
\fh_t^L(x) :=
\begin{cases}
\fh_t(x), &\text{if}~ x \in [-L, L],\\
-\infty, &\text{otherwise},
\end{cases}
\end{equation}
defined for $L > 0$.

\begin{lem}\label{lem:max_Holder}
Let the initial data of the KPZ fixed point satisfy Assumption~\ref{a:initial_state_pd}(a). For  $\delta \in (0, \frac{1}{3})$, $T >T_0> 0$ and $L > 0$, one has
\begin{equation}\label{eq:max_Holder}
|\Max(\fh^L_{t}) - \Max(\fh_{s}^L)|\leq \CC |t-s|^{\frac{1}{3} - \delta},
\end{equation}
for all $s \neq t \in [T_0,T]$, where the random variable $\CC$ is almost surely finite.
 \end{lem}

\begin{proof}
This is a quick consequence of Lemma~\ref{lem:Holder-time} on locally spatially uniform temporal H\"older regularity. Letting $x_t^*\in [-L,L]$ be the maximizer of $\fh_t^L$, we see that
$$\Max(\fh^L_{t}) - \Max(\fh_{s}^L) \leq \fh_t(x^*_t) - \fh_s(x^*_t) \leq \mc C |t-s|^{\frac{1}{3}-\delta} \, ,$$
where the almost surely finite constant $\mc C$ depends only on $\delta$, $L$, $T_0$, and $T$ in view of Lemma~\ref{lem:Holder-time}. By symmetry, this completes the proof.
\end{proof}


 \section{The upper bound on the twin peaks probability}
 \label{sec:TwinPeaks}

Here we prove Theorem~\ref{thm:densities2-intro}, a rigorous rendering of Item~\ref{heur:2} in Section~\ref{sec:heuristic}, and a result that immediately yields Theorem~\ref{thm:johansson}. We derive Theorem~\ref{thm:densities2-intro} from a bound on the density for the presence of two maximizers of the KPZ fixed point $\fh_t$ at two given spatial points. The latter is computed in Proposition~\ref{prop:densities} by careful differentiation of the formula \eqref{eq:fpform}. We note that Theorem~\ref{thm:densities} follows from Proposition~\ref{prop:densities}. We refer to Figure~\ref{fig:outline_upper} for the structure of this section.

\begin{center}
  \footnotesize
  \begin{figure}[h]
  \begin{tikzpicture}[scale=0.9, auto,
  	block_main/.style = {rectangle, draw=black, thick, fill=white, text width=17em, text centered, minimum height=4em, font=\small},
	block_density1/.style = {rectangle, draw=black, fill=white, thick, text width=11em, text centered, minimum height=4em},
	block_density2/.style = {rectangle, draw=black, fill=white, thick, text width=10em, text centered, minimum height=4em},
	block_density3/.style = {rectangle, draw=black, fill=white, thick, text width=11em, text centered, minimum height=4em},
	block_rewrite/.style = {rectangle, draw=black, fill=white, thick, text width=17em, text centered, minimum height=4em},
	block_kernels/.style = {rectangle, draw=black, fill=white, thick, text width=18em, text centered, minimum height=4em},
        line/.style ={draw, thick, -latex', shorten >=0pt}]
      		\node [block_main] (upper_bound) at (-4,0) {The upper bound on the twin peaks probability (Theorem~\ref{thm:densities2-intro})};
		\node [block_main] (density) at (3,0) {Existence and properties of the density of two maximizers and the maximal value (Proposition~\ref{prop:densities})};
		\node [block_density1] (scaling) at (-4,-2.7) {Scaling property of the density, following from scaling of the KPZ fixed point (Section~\ref{ss:scaling_of_P})};
		\node [block_density2] (existence) at (1.5,-2.7) {Existence of the density, relying on double differentiation of the KPZ fixed point formula (Section~\ref{ss:first_deriv})};
		\node [block_density3] (bound) at (6.5,-2.7) {A bound on the density, obtained from the formula computed in Section~\ref{ss:first_deriv} (Section~\ref{ss:bound_on_P})};
		\node [block_rewrite] (rewrite) at (-2,-5.5) {Rewriting the prelimiting density in terms of distribution functions of the KPZ fixed point (Section~\ref{sec:KPZ_application})};
		\node [block_kernels] (kernels) at (5,-5.5) {Convergence of kernels in the trace class norm, which allows differentiation of Fredholm determinants (Lemma~\ref{lem:deriv1} and auxiliary results provided in Lemmas~\ref{lem:Q_limits}--\ref{lem:A_formulas} and Section~\ref{sec:Auxiliary})};
		
    \begin{scope}[every path/.style=line]
		\path (density) -- (upper_bound);
		\path (existence) -- (density);
		\path (scaling) -- (density);
		\path (bound) -- (density);
		\path (existence) -- (bound);
		\path (rewrite) -- (existence);
		\path (rewrite) -- (scaling);
		\path (kernels) -- (existence);
    \end{scope}
  \end{tikzpicture}
  \captionof{figure}{Structure of Section \ref{sec:TwinPeaks}.}
  \label{fig:outline_upper}
  \end{figure}
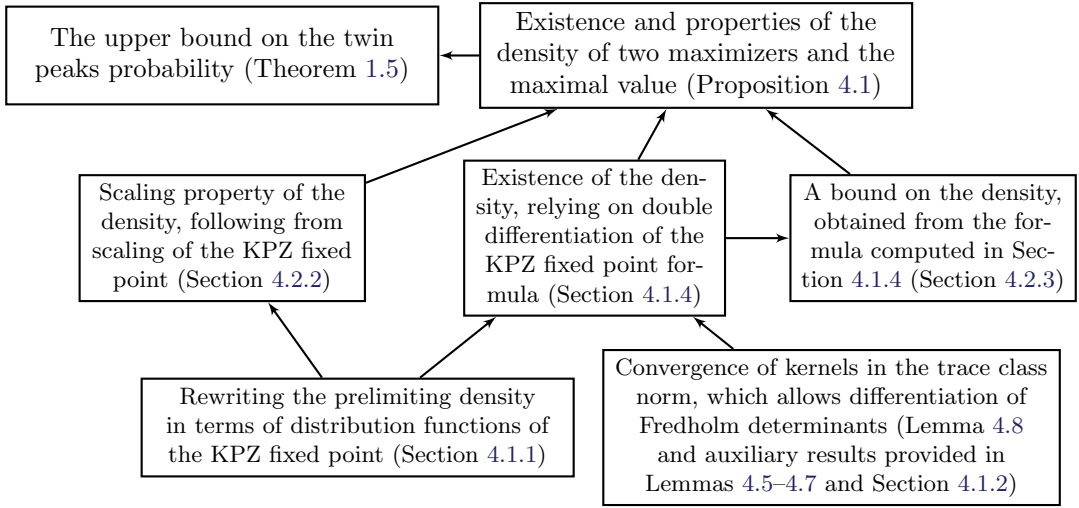
\end{center}

In order to use the formula \eqref{eq:fpform}, it is convenient to restrict the spatial variable $x$ of the KPZ fixed point $\fh_t(x)$ to a finite interval. Doing so permits the  use of the kernel on the right-hand side of \eqref{eq:Kd} before we take the limits. This kernel has the advantage of being given by an explicit formula. For fixed  $\beta \in (0,1)$, and $L > 0$, we denote
\begin{equation}\label{eq:setJ}
\bJ_{L} := [-\beta L, \beta L].
\end{equation}
We will estimate the probability that the KPZ fixed point has two near maximizers that differ by at least some fixed $A > 0$. To this end, for a function $f : \R \to \R \cup \{-\infty\}$ that is bounded from above, and for three values $A, L, \eps > 0$, we specify the set
 \begin{equation}\label{eq:set_S}
 \TPS{A, L}^{\eps}(f) := \bigl\{(x_1, x_2) \in \bJ_{L}^2 : x_2 - x_1 \geq A,\; \Max(f) - f(x_i) \leq \eps ~\text{for}~ i=1,2\bigr\}.
 \end{equation}
 Then we define the set
 \begin{equation}\label{eq:setTP}
 	\TP{A, L}^{\eps} := \left\{ f : \R \to \R \cup \{-\infty\} : \Max(f) \in \bJ_{L^{1/2}}, \; \TPS{A, L}^{\eps}(f) \neq \emptyset\right\}.
 \end{equation}
In particular, if $L \to \infty$ and $\eps = 0$ then $\TP{A} := \TP{A, \infty}^{0}$ contains functions $f$ which attain their maxima at two points at distance at least $A$. In the case that $L$ tends to infinity, we denote the set $\TP{A}^\eps := \TP{A, \infty}^{\eps}$, which coincides  with the twin peaks set $\TPno^\eps$ after the limit $A \to 0$ is taken.

The restriction that the maximum of $f \in \TP{A, L}^{\eps}$ lies in a bounded interval $\bJ_{L^{1/2}}$ is rather technical and simplifies computations in the proof of Proposition~\ref{prop:densities}.
We do not simply take $\beta = 1$ and $\bJ_{L} = [- L,  L]$, because then we would have to deal with divergences of the type $\smash{(L \pm \fx)^{-1 / 2}}$ in some upcoming bounds such as \eqref{eqs:D_bounds}. Having $\beta < 1$ slightly simplifies our computations; in particular, it allows us to bound $\smash{(L \pm \fx)^{-1 / 2}}$ by the constant $L^{-1 / 2} (1 - \beta)^{-1 / 2}$.
That the spatial variable lies in an interval of length proportional to $L$ while the maximum value lies  in an interval of length proportional to $L^{1/2}$ corresponds to the 1:2:3 scaling invariance of the KPZ fixed point provided in Proposition~\ref{prop:sym}(\ref{123a}).

Next, we derive an upper bound on the probability that the KPZ fixed point maximum lies close to a given value and is nearly adopted at locations close to two given points. For $\eps_i, \delta_i > 0$, set $\eta_i = (\eps_i, \delta_i)$. For $t, L > 0$ and $M \in \R$, define the function $\dens_{t, M, L}^{(\eta_1; \eta_2)} : \bJ_L^2 \to [0,1]$ according to \eqref{eq:density}. 
%
%
To compute the limit we are going to use the notion \eqref{eq:mylim}. The choice of the domain \eqref{eq:Dom} is dictated by our proof of Lemma~\ref{lem:Q_limits}. 

\begin{prop}\label{prop:densities}
For $\eta_i = (\eps_i, \delta_i)$ with positive components, define $|\eta_i| = \eps_i \delta_i > 0$. Assume that the initial data of the KPZ fixed point satisfies $\fh_0(x) \leq \alpha$ for some $\alpha\in \mathbb{R}$ and for all $x \in \R$. For every fixed $t > 0$, $L > 0$ and $A > 0$ the  following  limits exist:
\begin{subequations}\label{eqs:densities}
\begin{align}
	\dens^{(\eta_1; \star)}_{t, M, L}(x_1, x_2) &:= \mylim{\eta_2 \searrow 0} \frac{1}{|\eta_2|} \dens^{(\eta_1; \eta_2)}_{t, M, L}(x_1, x_2),\label{e:densities1}\\
	\dens_{t, M, L}(x_1, x_2) &:= \mylim{\eta_1 \searrow 0} \frac{1}{|\eta_1|} \dens^{(\eta_1; \star)}_{t, M, L}(x_1, x_2),\label{e:densities}
\end{align}
\end{subequations}
uniformly over $x_1, x_2 \in \bJ_{L}$, such that $x_2 - x_1 \geq A$, and locally uniformly over $M \in \R$; the limit \eqref{e:densities1} holds uniformly over $\eta_1 \in (0, \frac{A}{2}]^2$, i.e., for $\eps_1$ and $\delta_1$ in $(0, \frac{A}{2}]$.

Furthermore, the following limit holds pointwise 
\begin{equation}\label{eq:P_limit}
\dens_{1, M}(x_1, x_2) := \lim_{L \to \infty} \dens_{1, M, L}(x_1, x_2).
\end{equation}

The function $\dens_{t, M, L}(x_1, x_2)$ is continuous in $M \in \R$ and $x_1, x_2 \in \bJ_{L}$ such that $x_1 < x_2$, and satisfies the scaling identity
\begin{equation}\label{e:densities_scale}
\dens_{t, M, L}(x_1, x_2) = t^{-2} \dens_{1, t^{-1/3} M, t^{-2/3} L} \bigl(t^{-\frac{2}{3}} x_1, t^{-\frac{2}{3}} x_2\bigr),
\end{equation}
where the right-hand function is defined by \eqref{e:densities} with the initial data $\fh_0^{(t)}(x) := t^{-1/3}\fh_0(t^{2/3} x)$.

Finally, for any $\bar L > 0$, there exists a constant $C \geq 0$, depending on $L$, $\bar L$ and $\alpha$, such that 
\begin{equation}\label{eq:P_bound}
\dens_{1, M, L}(x_1, x_2) \leq C |x_2 - x_1|^{-\frac{3}{2}},
\end{equation}
for any  $x_1, x_2 \in \bJ_L$ such that $x_1 < x_2$, and any $|M| \leq \bar L$.
\end{prop}

To prove this result, we write the function $\dens_{t, M, L}^{(\eta_1; \eta_2)}(x_1, x_2)$ from \eqref{eq:density} in terms of distribution functions of the form \eqref{eq:fpform}. We then take the limits in \eqref{eqs:densities}, a task, undertaken in  Section~\ref{sec:KPZ_application}, that amounts to a double differentiation of a Fredholm determinant. In order to use Lemma~\ref{lem:det},  we prove in Section~\ref{sec:trace} that the involved kernels are trace class, and we compute  in Section~\ref{ss:first_deriv} these derivatives of the Fredholm determinant. Finally, we prove properties of the limit \eqref{e:densities} in Section~\ref{sec:properties_of_P}.

\subsection{Proof of Proposition~\ref{prop:densities}}

\subsubsection{Rewriting the probability}
\label{sec:KPZ_application}

The scaling property, Proposition~\ref{prop:sym}(\ref{123a}), of the KPZ fixed point  implies that, for the initial data $\fh_0$, the probability \eqref{eq:density} can be written as
\begin{align}
\MoveEqLeft[11]
\dens_{t, M, L}^{(\eta_1; \eta_2)}(x_1, x_2) = \P_{\fh^{(t)}_0}\Big(\text{For each }i\in \{1,2\},\,\, \fh_{1}(t^{-\frac{2}{3}} y_i) >\,  t^{-\frac{1}{3}}(M - \eps_i) \text{ for some } \\
&y_i \in [x_i, x_i + \delta_i), \text{and } \fh_{1}(t^{-\frac{2}{3}} z) \leq t^{-\frac{1}{3}} M \text{ for } z \in [-L, L]\Big), \label{eq:density2}
\end{align}
where $\fh_0^{(t)}(x) := t^{-1/3}\fh_0(t^{2/3} x)$. Note that if $\fh_0(x) \leq \alpha$ for all $x \in \R$, then $\fh_0^{(t)}(x) \leq t^{-1/3} \alpha$ for all $x \in \R$. Hence, it is sufficient to prove the limits \eqref{eqs:densities} for $t = 1$. Indeed, we will prove that the limits
\begin{subequations}\label{eqs:limitsF}
\begin{align}
\dens^{(\eta_1; \star)}_{1, M, L}(x_1, x_2) &= \mylim{\eta_2 \searrow 0} \frac{1}{|\eta_2|} \dens^{(\eta_1; \eta_2)}_{1, M, L}(x_1, x_2),\label{eq:limitF1}\\
\dens_{1, M, L}(x_1, x_2) &= \mylim{\eta_1 \searrow 0} \frac{1}{|\eta_1|} \dens^{(\eta_1; \star)}_{1, M, L}(x_1, x_2), \label{eq:limitF2}
\end{align}
\end{subequations}
exist uniformly as in \eqref{eqs:densities}, where the limit \eqref{eq:limitF1} holds uniformly over $\eta_1 \in (0, \frac{A}{2}]^2$. To make our notation lighter, we prefer to fix the values of $L$, $\fM$, $\fx_1$ and $\fx_2$ throughout this proof and not to indicate the dependence of various kernels on them.

We can write the function $\dens^{(\eta_1; \eta_2)}_{1, M, L}$ using the KPZ fixed point formula \eqref{eq:fpform}. For this, we need to write the event in \eqref{eq:density} in terms of the distribution function as on the left-hand side of \eqref{eq:fpform}. For $x_1, x_2 \in \bJ_L$ and for $\eta_i = (\eps_i, \delta_i) \geq 0$---a notation indicating that the components are non-negative---such that $\delta_1, \delta_2 < x_2 - x_1$, we define the function
\begin{equation}\label{eq:g_def}
 \fg_{\eta_1, \eta_2}(y) := \fM - \eps_1 \cdot \1{y \in [\fx_1, \fx_1 + \delta_1)} - \eps_2 \cdot \1{y \in [\fx_2, \fx_2 + \delta_2)}.
\end{equation}
Then our assumption on $\delta_i$ implies $[\fx_1, \fx_1 + \delta_1) \cap [\fx_2, \fx_2 + \delta_2) = \emptyset$. Using the inclusion-exclusion formula, the probability \eqref{eq:density} with $t = 1$ can be written as
\begin{equation}
\begin{split}
\dens_{1, \fM, L}^{(\eta_1; \eta_2)}(\fx_1, \fx_2) 
&=\P_{\fh_0}\Bigl( \fh_1(y) \leq \fg_{0, 0}(y) \text{ for } y \in [-L, L] \Bigr)\\
&\quad- \P_{\fh_0}\Bigl( \fh_1(y) \leq \fg_{\eta_1, 0}(y) \text{ for } y \in [-L, L] \Bigr)\\
&\quad- \P_{\fh_0}\Bigl( \fh_1(y) \leq \fg_{0, \eta_2}(y) \text{ for } y \in [-L, L] \Bigr)\\
&\quad+ \P_{\fh_0}\Bigl( \fh_1(y) \leq \fg_{\eta_1, \eta_2}(y) \text{ for } y \in [-L, L] \Bigr),
\end{split}\label{eq:density_approx}
\end{equation}
where we write $\eta_i = 0$ if $\eps_i = \delta_i = 0$.
To evaluate these probabilities, we use the KPZ fixed point formula \eqref{eq:fpform}. An advantage in restricting the spatial variable to the interval $[-L, L]$ is that we can use the prelimiting version of formula \eqref{eq:Kd} for the kernel in the KPZ fixed point formula \eqref{eq:fpform}. All Fredholm determinants are computed over $L^2(\R)$ and our notation omits reference to this.

Before using formula \eqref{eq:fpform}, we need to make several definitions. For $\ff \in \UC$ and $\ell_1 < \ell_2$, define the restriction of $\ff$ to $[\ell_1,\ell_2]$ by
\begin{equation}
\ff^{[\ell_1,\ell_2]}(x) :=
\begin{cases}
\ff(x), &\text{for}~x \in [\ell_1,\ell_2],\\
-\infty, &\text{otherwise}.
\end{cases}
\end{equation}
Then we define the operator on the right-hand side of \eqref{eq:Kd} before taking the limits
\begin{equation}\label{eq:Kd_prelimit}
\fK^{\hypo(\ff)}_{t, [\ell_1,\ell_2]} := \Gamma_t (\fT_{t,\ell_1})^*\, \fP_{[\ell_1,\ell_2]}^{\hit(\ff^{[\ell_1,\ell_2]})}\fT_{t,-\ell_2} \Gamma_t.
\end{equation}
The counterpart $\epi$-operator is defined by \eqref{eq:Kd_epi}. We use the shorthand
 \begin{equation}\label{eq:fR}
 \fR_L^{(\eta_1; \eta_2)} := \fK^{\epi(\fg_{\eta_1, \eta_2})}_{1/2, [-L,L]},
 \end{equation}
 where $\ff \in \UC$ and $\fg \in \LC$. 
 Using the KPZ fixed point formula \eqref{eq:fpform}, we can then write the probabilities in \eqref{eq:density_approx} in the form
\begin{align}
\dens_{1, \fM, L}^{(\eta_1; \eta_2)}(\fx_1, \fx_2) &= \det \Bigl(I-\fK^{\hypo(\fh_0)}_{1/2} \fR_L^{(0; 0)}\Bigr) - \det\Bigl(I-\fK^{\hypo(\fh_0)}_{1/2} \fR_L^{(\eta_1; 0)}\Bigr)\\
&\qquad\qquad - \det \Bigl(I-\fK^{\hypo(\fh_0)}_{1/2} \fR_L^{(0; \eta_2)}\Bigr) + \det \Bigl(I-\fK^{\hypo(\fh_0)}_{1/2} \fR_L^{(\eta_1; \eta_2)}\Bigr),\label{eq:P_det}
\end{align}
where as before we write $\eta_i = 0$ if $\eps_i = \delta_i = 0$. Dependence on $\fx_1$, $\fx_2$ and $\fM$ of the right-hand side is through the function $\fg_{\eta_1, \eta_2}$ in the kernel \eqref{eq:fR}. 

One can see that computation of the iterated limits \eqref{e:densities} boils down to double differentiation of a Fredholm determinant. More precisely, for a function $f : \R^2 \to \R$, we denote the limit
\begin{equation}\label{eq:deriv_eta}
D_\eta f(\eta) := \mylim{\bar \eta \searrow 0} \,\frac{f(\eta + \bar \eta) - f(\eta)}{|\bar \eta|},
\end{equation}
provided it exists, where $\bar \eta \in \Dom$ is such that $|\bar \eta| > 0$ (recall the definition of the domain \eqref{eq:Dom} and the limit \eqref{eq:mylim}). Then, dividing \eqref{eq:P_det} by $|\eta_1| |\eta_2|$ and taking the limits as in \eqref{eqs:limitsF}, we get
\begin{equation}
\begin{split}
\dens^{(\eta_1; \star)}_{1, M, L}(x_1, x_2) &= D_{\eta_2} \det \Bigl(I-\fK^{\hypo(\fh_0)}_{1/2} \fR_L^{(\eta_1; \eta_2)}\Bigr) \Bigr|_{\eta_2 = 0}\\
&\qquad- D_{\eta_2} \det \Bigl(I-\fK^{\hypo(\fh_0)}_{1/2} \fR_L^{(0,  \eta_2)}\Bigr) \Bigr|_{\eta_2 = 0},\\[4pt]
\dens_{1, \fM, L}(\fx_1, \fx_2) &= D_{\eta_1} D_{\eta_2} \det \Bigl(I-\fK^{\hypo(\fh_0)}_{1/2} \fR_L^{(\eta_1; \eta_2)}\Bigr) \Bigr|_{\eta_1= \eta_2 = 0},
\end{split}\label{eqs:det_derivs}
\end{equation}
where the limits are uniform over $x_1$, $x_2$ and $M$ as in \eqref{e:densities}.

We compute the derivatives \eqref{eqs:det_derivs} in Section~\ref{ss:first_deriv}, making use of Lemma~\ref{lem:det} for differentiating Fredholm determinants. The required limits of the kernels $\fR_L^{(\eta_1; \eta_2)}$ in trace norm are computed in Lemma~\ref{lem:deriv1}, which relies on properties of auxiliary kernels provided in Lemmas~\ref{lem:Q_limits}--\ref{lem:A_formulas}.

 \subsubsection{Auxiliary results}
 \label{sec:Auxiliary}

The following lemmas provide some asymptotics, which we will use in the forthcoming section.

\begin{lemma}\label{lem:g_converges}
For some $a > 0$ and every $0 \leq \delta \leq a$, let the functions $g_{\delta} : [0, \infty) \to \R$ be equibounded with $g_\delta(0) = 0$. Moreover, assume that $g_\delta$ is continuously differentiable on $[0, \infty)$ and that the sequence $g'_\delta$ converges locally uniformly on $[0, \infty)$ as $\delta \searrow 0$ to a continuous function $g'_0 : [0, \infty) \to \R$. Then for every $\alpha > 0$ the following limit holds locally uniformly in $v \in [0, \infty)$:
\begin{equation}\label{eq:g_converges}
	\lim_{\delta \searrow 0} \frac{1}{\delta^\alpha} g_\delta (\delta^\alpha v) = g'_0(0) v.
\end{equation}
Moreover, there is $\delta_\star > 0$ and a constant $C \geq 0$, depending only on $\alpha$ and $\delta_\star$, such that
\begin{equation}\label{eq:g_is_bounded}
	\sup_{0 \leq \delta \leq \delta_\star} \sup_{0 < v \leq \delta^\alpha} \frac{|g_\delta(v)|}{v} \leq C.
\end{equation}
\end{lemma}

\begin{proof}
Using properties of the functions $g_\delta$, for every $v \geq 0$ we have
\begin{equation}
\frac{1}{\delta^\alpha} g_\delta (\delta^\alpha v) = \frac{1}{\delta^\alpha} \bigl(g_\delta (\delta^\alpha v) - g_\delta(0)\bigr) = \frac{1}{\delta^\alpha} \int_0^{\delta^\alpha v} g'_\delta (r) \d r.
\end{equation}
This yields, for any $R > 0$,
\begin{align}
\sup_{0 \leq v \leq R}\left|\frac{1}{\delta^\alpha} g_\delta (\delta^\alpha v) - g_0' (0) v \right|&= \sup_{0 \leq v \leq R} \left|\frac{1}{\delta^\alpha} \int_0^{\delta^\alpha v} \bigl(g'_\delta (r) - g'_0 (0) \bigr) \d r \right|\\
&\leq R \sup_{0 \leq r \leq \delta^\alpha R} \bigl| g'_\delta (r) - g'_0 (0) \bigr|,
\end{align}
which vanishes as $\delta \searrow 0$, because $g'_\delta$ converge locally uniformly to $g_0'$.

The bound \eqref{eq:g_is_bounded} follows from the limit \eqref{eq:g_converges}.
\end{proof}

\begin{lemma}\label{lem:heat2}
For some $a > 0$ and every $0 < \delta \leq a$, let $g_{\delta} : [0, \infty) \to \R$ be a bounded continuous function, such that $g_\delta$ converge locally uniformly as $\delta \searrow 0$ to a continuous function $g_0 : [0, \infty) \to \R$. Then, for any $\alpha > 0$,
\begin{equation}\label{eq:heat2}
\lim_{\delta \searrow 0} \sup_{0 \leq \eps \leq \delta^\alpha} \int_{0}^\infty \d v\; p_{\delta}(v) g_\delta(v + \eps) = \frac{1}{2} g_0(0),
\end{equation}
where $p_\ell$ is the heat kernel \eqref{eq:heat_kernel}.
\end{lemma}

\begin{proof}
Denoting $p(u) := p_1(u)$, using the scaling property of the heat kernel $p_{\delta}(\sqrt{\delta} u) = p(u) / \sqrt \delta$, and changing the integration variable $v = \sqrt \delta w$, the integral in \eqref{eq:heat2} can be written as
\begin{align}\label{eq:change_of_variables}
	\int_{0}^\infty \d w\; p(w) g_\delta(\sqrt \delta w + \eps).
\end{align}
Since $\int_{0}^\infty \d w\; p(w) = \frac{1}{2}$, we can write
\begin{equation}
	\int_{0}^\infty \d v\; p_{\delta}(v) g_\delta(v + \eps) - \frac{1}{2} g_0(0) = \int_{0}^\infty \d w\; p(w) \bigl(g_\delta(\sqrt \delta w + \eps) - g_0(0)\bigr).
\end{equation}
The absolute value of this expression is bounded by
\begin{equation}
	\int_{0}^\infty \d w\; p(w) \sup_{0 \leq \eps \leq \delta^\alpha} \bigl|g_\delta(\sqrt \delta w + \eps) - g_0(0)\bigr|.
\end{equation}
Then boundedness of the functions $g_\delta$ and fast decay of $p$ at infinity allow us to apply the dominated convergence theorem and to conclude that the limit $\delta \searrow 0$ of this expression equals
\begin{equation}
	\int_{0}^\infty \d w\; p(w) \lim_{\delta \searrow 0} \sup_{0 \leq \eps \leq \delta^\alpha} \bigl|g_\delta(\sqrt \delta w + \eps) - g_0(0)\bigr| = 0.
\end{equation}
From this the limit \eqref{eq:heat2} follows.
\end{proof}

For $\ell_1 < \ell_2$ and $a \in \R$, we define the density function
\begin{equation}\label{e:Theta_def}
\Theta^{(a)}_{[\ell_1, \ell_2]}(u_1, u_2) := \P_{\fB(\ell_1)=u_1, \fB(\ell_2) = u_2}\Bigl(\fB(y) > a \text{ for } y \in [\ell_1,\ell_2]\Bigr) p_{\ell_2 - \ell_1}(u_2 - u_1),
\end{equation}
where $\fB$ is a Brownian motion with variance two, whose transition kernel $p_\ell$ is defined in \eqref{eq:heat_kernel}. In other words, $\Theta^{(a)}_{[\ell_1, \ell_2]}(u_1, u_2)$ may be interpreted as the probability that Brownian motion  moves from $(\ell_1, u_1)$ to $(\ell_2, u_2)$ while staying strictly above $a$. By the reflection principle, we may compute \eqref{e:Theta_def} explicitly:
\begin{equation}\label{e:Theta_formula}
\Theta^{(a)}_{[\ell_1, \ell_2]}(u_1, u_2) := \Bigl(p_{\ell_2 - \ell_1}(u_2 - u_1) - p_{\ell_2 - \ell_1}(u_1 + u_2 - 2a)\Bigr) \cdot \1{u_1, u_2 > a}.
\end{equation}

For fixed $u_1, u_2 > -M$ and for $\delta > 0$, let us define the functions
\begin{equation}\label{eq:functions_f_g}
	f_\delta(v) = \fP_{[-L,\fx_2 + \delta]}^{\nohit\,(-\fg_{\eta_1, 0})} (u_1, v), \qquad g_{\delta}(v) = \Theta^{(-\fM)}_{[x_2 + \delta, L]} (v, u_2).
\end{equation}
It is important that the function $f_\delta$ depends on $\eta_1$, which we however prefer not to indicate in our notation. We can prove that these functions have the properties stated in the preceding lemmas.

\begin{lemma}\label{lem:f_g}
The function $f_\delta$ is differentiable on $[-M, \infty)$ and both $f_\delta$ and $f'_\delta$ are continuous in $v \in [-M, \infty)$ and equibounded in $v \in [-M, \infty)$ and $u_1 \in \R$, for $\delta \in [0,1]$. Moreover, $f_\delta(v) = 0$ for $v \leq 0$.

Furthermore, if $x_2 - x_1 \geq A$, then for $0 \leq \delta < L - x_2$ the functions $g_{\delta} (\,\bigcdot\, - M)$ have the properties listed in Lemmas~\ref{lem:g_converges} and \ref{lem:heat2}, uniformly in $u_2 \in \R$.
\end{lemma}

\begin{proof}
The property $f_\delta(v) = 0$ for $v \leq 0$ follows readily from the definition \eqref{eq:functions_f_g}. Using the kernel \eqref{e:Theta_def}, the Markov property allows to write
\begin{equation}\label{eq:f_explicitly}
	f_\delta(v) = \Bigl(\fP_{[-L,\fx_1 + \delta_1]}^{\nohit\,(-\fg_{\eta_1, 0})} \, *\, \Theta^{(-\fM)}_{[\fx_1 + \delta_1, x_2 + \delta]}\Bigr) (u_1, v),
\end{equation}
where we write ``$*$'' for the convolution of two kernels. Since we have $\fx_1 + \delta_1 < x_2$ and the function $\fP_{[-L,\fx_1 + \delta_1]}^{\nohit\,(-\fg_{\eta_1, 0})}$ is integrable, we conclude that both $f_\delta$ and $f'_\delta$ are continuous in $v \in [-M, \infty)$ and equibounded in $v \in [-M, \infty)$ and $u_1 \in \R$.

Formula \eqref{e:Theta_formula} yields $g_{\delta}(v) = \bigl(p_{L - \fx_2 - \delta}(u_2 - v) - p_{L - \fx_2 - \delta}(v + u_2 + 2 \fM)\bigr) \cdot \1{v, u_2 > -\fM}$, which are continuous and equibounded on $\R$, and converge uniformly as $\delta \searrow 0$ to
\begin{equation}
	g(v) = \Bigl(p_{L - \fx_2}(u_2 - v) - p_{L - \fx_2}(v + u_2 + 2 \fM)\Bigr) \cdot \1{v, u_2 > -\fM}.
\end{equation}
Moreover, $g_{\delta}$ is continuously differentiable on $[-M, \infty)$, whose derivatives converge uniformly on this interval to $g'(v) = - p'_{L - \fx_2}(u_2 - v) - p'_{L - \fx_2}(v + u_2 + 2 \fM)$, where $p'_t(u) = \partial_u p_t(u)$. All these properties hold uniformly in $u_2 \in \R$.
\end{proof}

 \subsubsection{Trace class bounds on kernels}
\label{sec:trace}

The aim of this section is to compute limits of the kernel $\fR_L^{(\eta_1; \eta_2)}$, which are provided in Lemma~\ref{lem:deriv1}. For this, we need to compute limits of some auxiliary kernels, which we do in the forthcoming lemmas.

Let us fix $L > 0$, $\bar L > 0$ and $0 < A < \beta L$, where the constant $\beta$ is from the definition of $\bJ_L$ in \eqref{eq:setJ}; and let us define the sets
 \begin{equation}\label{eq:setsS}
 S = \{(x_1, x_2) \in \bJ^2_L : x_2 - x_1 \geq A \}, \qquad \bar S = \{ M \in \R : |M| \leq \bar L\}.
 \end{equation}
Now we prove that the differentiation of the Fredholm determinants, which appears in the derivation of \eqref{eqs:det_derivs}, can be performed uniformly over $(x_1, x_2) \in S$ and $M \in \bar S$. For this, we use Lemma~\ref{lem:det} and show that the respective operators are differentiable in trace class.

Recall the kernel \eqref{eq:no_hit}. For $L > 0$ and $\eta_i = (\eps_i, \delta_i) \geq 0$,
we denote
\begin{equation}\label{e:Q_def}
\fQ_L^{(\eta_1; \eta_2)}(u_1, u_2) := \fP_{[-L,L]}^{\nohit\,(-\fg_{\eta_1, \eta_2})}(u_1,u_2),
\end{equation}
which depends on $\fM$, $\fx_1$ and $\fx_2$, because the function $\fg_{\eta_1, \eta_2}$ in \eqref{eq:g_def} does.
In the next lemma, we compute this function's derivatives with respect to $\eta_1$ and $\eta_2$.

\begin{lem}\label{lem:Q_limits}
The following limits exist
\begin{subequations}\label{eqs:Q_limits}
\begin{align}
\fQ_L^{(\eta_1; \star)}(u_1, u_2) &:= \mylim{\eta_2 \searrow 0} \frac{1}{|\eta_2|} \Bigl(\fQ_L^{(\eta_1; \eta_2)} - \fQ_L^{(\eta_1; 0)}\Bigr)(u_1, u_2),\label{eq:Q_eps1_0}\\
\fQ_L^{(\star; 0)}(u_1, u_2) &:= \mylim{\eta_1 \searrow 0} \frac{1}{|\eta_1|} \Bigl(\fQ_L^{(\eta_1; 0)} - \fQ_L^{(0; 0)}\Bigr)(u_1, u_2),\label{eq:Q_0_eps2} \\
\fQ_L^{(\star; \star)}(u_1, u_2) &:= \mylim{\eta_1 \searrow 0} \frac{1}{|\eta_1|} \Bigl(\fQ_L^{(\eta_1; \star)} - \fQ_L^{(0; \star)}\Bigr)(u_1, u_2),\label{eq:Q_0_0}
\end{align}
\end{subequations}
uniformly over $u_1, u_2 \in \R$, $(x_1, x_2) \in S$ and $M \in \bar S$; and the limit \eqref{eq:Q_eps1_0} holds uniformly over $\eta_1 \in (0, \frac{A}{2}]^2$. Moreover, these kernels are rank-one and are given by the formulas
\begin{subequations}
\label{eqs:Q_formulas}
\begin{align}
\hspace{0.5cm}\fQ_L^{(\eta_1; \star)}(u_1, u_2) &= -\partial_{2} \fP_{[-L,\fx_2]}^{\nohit\,(-\fg_{\eta_1, 0})} (u_1, -\fM)\, \cdot\,\partial_{1} \Theta^{(-\fM)}_{[\fx_2, L]} (-\fM, u_2),\label{eq:Q_eps1_0_formula}\\
\fQ_L^{(\star; 0)}(u_1, u_2) &= -\partial_{2} \Theta^{(-\fM)}_{[-L, \fx_1]} (u_1, -\fM)\, \cdot\, \partial_{1} \Theta^{(-\fM)}_{[\fx_1, L]} (-\fM, u_2),\label{eq:Q_0_eps2_formula} \\
\fQ_L^{(\star; \star)}(u_1, u_2) &= -\frac{(\fx_2 - \fx_1)^{-\frac{3}{2}}}{\sqrt{4\pi}}\, \cdot\, \partial_{2} \Theta^{(-\fM)}_{[-L, \fx_1]} (u_1, -\fM)\, \cdot\, \partial_{1} \Theta^{(-\fM)}_{[\fx_2, L]} (-\fM, u_2),\label{eq:Q_0_0_formula}
\end{align}
\end{subequations}
where $\partial_1, \partial_2$ are the derivatives with respect to the first and second arguments respectively.
\end{lem}

\begin{proof}
We recall our notation $\eta_i = (\eps_i, \delta_i) \geq 0$, and throughout the proof we consider values $\eps_i, \delta_i \in (0,\frac{A}{2}]$. This, in particular, implies that the two intervals in the definition \eqref{eq:g_def} do not intersect. We start by proving existence of the limit \eqref{eq:Q_eps1_0} and the formula \eqref{eq:Q_eps1_0_formula}.

Definitions \eqref{e:Q_def}, \eqref{eq:no_hit} and the definition in \eqref{eq:g_def} of the function $ \fg_{\eta_1, \eta_2}$ yield
\begin{equation}
\begin{split}
 \MoveEqLeft[7]
 \frac{\Bigl(\fQ_L^{(\eta_1; \eta_2)} - \fQ_L^{(\eta_1; 0)}\Bigr)(u_1, u_2)}{p_{2 L}(u_2 - u_1)}\\
 &=\P_{\fB(-L)=u_1, \fB(L)=u_2}\Bigl(\fB(y)>-\fg_{\eta_1, \eta_2}(y) \text{ for } y \in [-L,L]\Bigr) \\
 &\qquad- \P_{\fB(-L)=u_1, \fB(L)=u_2}\Bigl(\fB(y)>-\fg_{\eta_1, 0}(y) \text{ for } y \in [-L,L]\Bigr)\\
 &= - \P_{\fB(-L)=u_1, \fB(L)=u_2}\Bigl(\fB(y)>-\fg_{\eta_1, 0}(y) \text{ for } y \in [-L,L],\, \\
 & \hspace{2cm} \fB(z) \in (-\fM, -\fM + \eps_2] \text{ for some } z \in [\fx_2, \fx_2 + \delta_2)\Bigr),
 \end{split} \label{eq:Q_formula}
\end{equation}
where $\fB$ is a Brownian bridge with variance two. It will be convenient to view the Brownian bridge $\fB$ going from the point $u_2$ at time $L$ to the point $u_1$ at time $-L$. Let us then define the hitting time $\tau^{(a)} := \sup \{z \leq x_2 + \delta_2: \fB(z) = a\}$. Then formula \eqref{eq:Q_formula} can be written as
\begin{equation}
\begin{split}
\MoveEqLeft[6]
\frac{\Bigl(\fQ_L^{(\eta_1; \eta_2)} - \fQ_L^{(\eta_1; 0)}\Bigr)(u_1, u_2)}{p_{2 L}(u_2 - u_1)}\\
&= - \P_{\fB(-L)=u_1, \fB(L)=u_2}\Bigl(\fB(y)>-\fg_{\eta_1, 0}(y) \text{ for } y \in [-L,L],\, \\
& \hspace{6cm} \tau^{(-\fM + \eps_2)} \in [\fx_2, \fx_2 + \delta_2)\Bigr).
\end{split}\label{eq:Q_formula1}
\end{equation}
We note that if the value of the Brownian motion $ \fB(x_2 + \delta_2)$ is in $(-M, -M + \eps_2)$, then hitting of the box $[x_2, x_2 + \delta_2) \times (-M, -M + \eps_2]$ happens at time $x_2 + \delta_2$ (this is due to continuity of Brownian motion). Since we excluded the point $x_2 + \delta_2$ in \eqref{eq:Q_formula1}, hitting can happen only at the level $-M + \eps_2$ on the time interval $[x_2, x_2 + \delta_2)$. Moreover, if the value of $\fB(x_2 + \delta_2)$ is in $[-M + \eps_2, \infty)$, then hitting of the box happens almost surely at $\tau^{(-\fM + \eps_2)} \in [\fx_2, \fx_2 + \delta_2)$.

\begin{figure}[h]
\centering
\includegraphics[width=\textwidth]{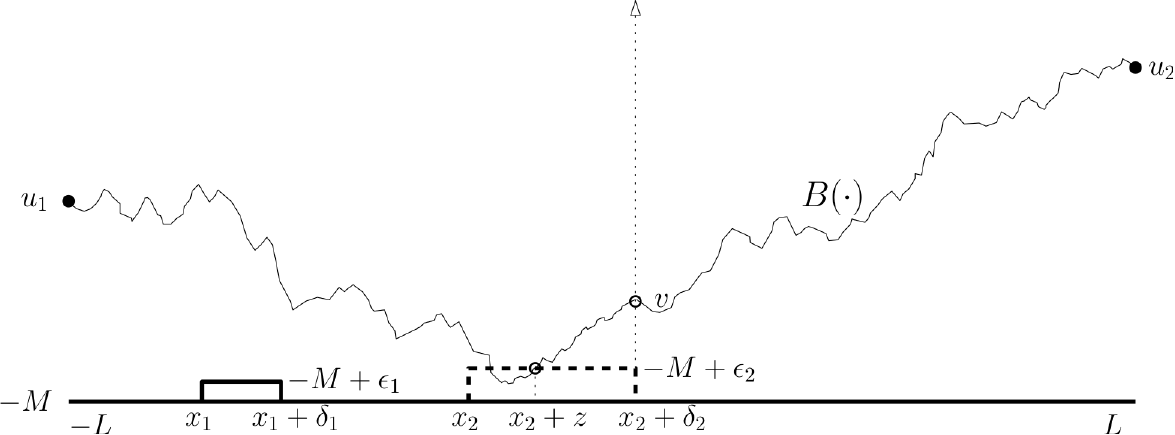}
 \captionsetup{width=.9\linewidth}
\caption{An illustration of \eqref{eq:Qform}. The Brownian motion ${\fB}$ starts at height $u_2$ at time $L$ and runs backwards in time until height $v$ at time $x_2+\delta_2$, conditioned to avoid the solid line at height $-M$ all the while. From height $v$, the path runs until a time $z$ with sub-probability measure $F^{(-M + \eps_2)}_{[z, \delta_2]} (v)dz$ on the interval $z\in [x_2,x_2+\delta)$. This point is the time when the Brownian motion hits height $-M+\eps_2$, provided it occurs in the time interval $[x_2,x_2+\delta_2)$. From there, the Brownian motion continues to height $u_1$ at time $-L$, conditioned to avoid the solid line at height $-M$ and the bump of height $-M+\eps_1$ on the time interval $[x_1,x_1+\delta_1)$.}
\label{Fig:Qform}
\end{figure}

Define the first hitting density function (existence follows from the Brownian reflection principle)
\begin{equation}\label{eq:density_stopping_time}
	F^{(a)}_{[z, \delta_2]} (v) := \frac{\d}{\d z}\P_{\fB(x_2 + \delta_2) = v} \bigl(\tau^{(a)} \leq x_2 + z\bigr),
\end{equation}
for $z < \delta_2$ and $v \geq a$. Let us also use the functions $f_\delta$ and $g_{\delta}$, defined in \eqref{eq:functions_f_g}, and the kernel \eqref{e:Theta_def}. Then the strong Markov property of Brownian motion yields the decomposition formula
\begin{equation}\label{eq:Qform}
\Bigl(\fQ_L^{(\eta_1; \eta_2)} - \fQ_L^{(\eta_1; 0)}\Bigr)(u_1, u_2) = - \int_{0}^{\delta_2} \d z \int_{-M + \eps_2}^{\infty} \d v\; f_z(-M + \eps_2) \,\cdot\, F^{(-M + \eps_2)}_{[z, \delta_2]} (v) \,\cdot\, g_{\delta_2}(v),
\end{equation}
where we used the property $\fB \bigl(\tau^{(-M + \eps_2)}\bigr) = -M + \eps_2$ and the fact that the ``No hit'' probability \eqref{eq:no_hit} is the same for the Brownian motion $\fB$ with the reverse time direction. See Figure~\ref{Fig:Qform} for an illustration.

From Eq.~6.3 in Section~2.6 of \cite{karatzas1998brownian} for $z < \delta_2$ and $v > a$ we have the formula
\begin{equation}\label{eq:density_stopping_time2}
	F^{(a)}_{[z, \delta_2]} (v) = \frac{v - a}{\delta_2 - z} p_{\delta_2 - z} (v - a) = - 2 p'_{\delta_2 - z} (v - a),
\end{equation}
where $p'_{\ell} (u) = \partial_u p_{\ell} (u)$. For $z < \delta_2$, this function is continuously extended to $v = a$ by $0$. Combining this formula with \eqref{eq:Qform} and substituting it into \eqref{eq:Q_formula1} yields
\begin{equation}
	\begin{split}
	\MoveEqLeft[6]
	\Bigl(\fQ_L^{(\eta_1; \eta_2)} - \fQ_L^{(\eta_1; 0)}\Bigr)(u_1, u_2)\\
	&= 2 \int_{0}^{\delta_2} \d z \int_{-M + \eps_2}^{\infty} \d v\; f_z(-M + \eps_2) \,\cdot\, p'_{\delta_2 - z}(v + M - \eps_2) \,\cdot\, g_{\delta_2}(v).
	\end{split}
\end{equation}
Since $f_z(-M) = 0$, we can write the preceding expression as
\begin{equation}
\begin{split}
\MoveEqLeft[6]
	\Bigl(\fQ_L^{(\eta_1; \eta_2)} - \fQ_L^{(\eta_1; 0)}\Bigr)(u_1, u_2)\\
	&= 	2 \eps_2 \int_{0}^{\delta_2} \d z \int_{-M + \eps_2}^{\infty} \d v\; \nabla_{\!\eps_2} f_z(-M) \,\cdot\, p'_{\delta_2 - z}(v + M - \eps_2) \,\cdot\, g_{\delta_2}(v).
\end{split}
\end{equation}
where $\nabla_{\!\eps_2} f_z(v) = \frac{1}{\eps_2} (f_z(v + \eps_2) - f_z(v))$ is a difference quotient. Applying integration by parts to the integral with respect to $v$, we get
\begin{equation*}
\begin{split}
\MoveEqLeft[7]
	\Bigl(\fQ_L^{(\eta_1; \eta_2)} - \fQ_L^{(\eta_1; 0)}\Bigr)(u_1, u_2)\\
	&= 	- 2 \eps_2 \int_{0}^{\delta_2} \d z\; \nabla_{\!\eps_2} f_z(-M) \,\cdot\, p_{\delta_2 - z}(0) \,\cdot\, g_{\delta_2}(-M + \eps_2)\\
	&\quad- 2 \eps_2 \int_{0}^{\delta_2} \d z \int_{-M + \eps_2}^{\infty} \d v\; \nabla_{\!\eps_2} f_z(-M) \,\cdot\, p_{\delta_2 - z}(v + M - \eps_2) \,\cdot\, g'_{\delta_2}(v),
\end{split}
\end{equation*}
where we used fast decay of the heat kernel at infinity. Hence, we have
\begin{equation}
\begin{split}
\MoveEqLeft[5]
	\frac{1}{|\eta_2|} \Bigl(\fQ_L^{(\eta_1; \eta_2)} - \fQ_L^{(\eta_1; 0)}\Bigr)(u_1, u_2)\\
	&= - \frac{2}{\delta_2} \int_{0}^{\delta_2} \d z\; \nabla_{\!\eps_2} f_z(-M) \,\cdot\, p_{\delta_2 - z}(0) \,\cdot\, g_{\delta_2}(-M + \eps_2)\\
	&\quad- \frac{2}{\delta_2} \int_{0}^{\delta_2} \d z \int_{-M + \eps_2}^{\infty} \d v\; \nabla_{\!\eps_2} f_z(-M) \,\cdot\, p_{\delta_2 - z}(v + M - \eps_2) \,\cdot\, g'_{\delta_2}(v).
\end{split}
\label{eq:Q_formula2}
\end{equation}
We denote these two terms by $\CI^{(\eta_2)}_1$ and $\CI^{(\eta_2)}_2$ respectively, and will compute their limits one by one.

\smallskip
\noindent {\bf The term $\CI^{(\eta_2)}_1$.} From Lemma~\ref{lem:f_g}, we conclude that $|\nabla_{\!\eps_2} f_z(-M)| \leq C_1$ uniformly in $\eps_2 \in (0,1]$, $z \in [0, \frac{A}{2}]$ and $u_1 \in \R$. Furthermore, for $\alpha_\star$ from the definition \eqref{eq:Dom}, Lemma~\ref{lem:g_converges} gives the bound $|g_{\delta_2}(-M + \eps_2)| \leq C_2 \eps_2$ uniformly in $\delta_2 \in [0, \delta_\star]$, $\eps_2 \in (0,\delta_2^{\alpha_\star}]$ and $u_2 \in \R$ (here, we use Lemma~\ref{lem:f_g}, which shows that $g_{\delta_2}$ has the required properties). Hence, for the constant $C = C_1 C_2 / \sqrt{\pi}$, we can bound
\begin{equation}
	|\CI^{(\eta_2)}_1| \leq C \frac{1}{\delta_2} \int_{0}^{\delta_2} \d z\; \frac{\eps_2}{\sqrt{\delta_2 - z}} = 2 C \frac{\eps_2}{\sqrt{\delta_2}}.
\end{equation}
Since for $\eta_2 \in \Dom$ we have $\eps_2 \leq \delta_2^{\alpha_\star}$ for $\alpha_\star > \frac{1}{2}$ (see \eqref{eq:Dom}), we conclude that
\begin{equation}\label{eq:I1_limit}
	\mylim{\eta_2 \searrow 0} |\CI^{(\eta_2)}_1| \leq 2 C \lim_{\delta_2 \searrow 0} \delta_2^{\alpha_\star - \frac{1}{2}} = 0,
\end{equation}
uniformly in $u_1, u_2 \in \R$.

\smallskip
\noindent {\bf The term $\CI^{(\eta_2)}_2$.} Changing the integration variable, we get
\begin{equation}
	\CI^{(\eta_2)}_2 = \frac{2}{\delta_2} \int_{0}^{\delta_2} \d z \int_{0}^{\infty} \d v\; \nabla_{\!\eps_2} f_z(-M) \,\cdot\, p_{\delta_2 - z}(v) \,\cdot\, g'_{\delta_2}(v - M + \eps_2).
\end{equation}
Let us define $\CI_2 = f'_0(-M) \cdot g'_{0}(-M)$ and
\begin{equation}\label{eq:I2_tilde}
	\wt{\CI}^{(\delta_2)}_2 = \frac{2}{\delta_2} \int_{0}^{\delta_2} \d z \int_{0}^{\infty} \d v\; f'_z(-M) \,\cdot\, p_{\delta_2 - z}(v) \,\cdot\, g'_{\delta_2}(v - M).
\end{equation}
We are going to prove that $\CI^{(\eta_2)}_2 - \wt{\CI}^{(\delta_2)}_2$ vanishes and $\wt{\CI}^{(\delta_2)}_2$ converges to $\CI_2$ as $\eta_2 \searrow 0$ in the domain $\Dom$.

We first prove that $\CI^{(\eta_2)}_2 - \wt{\CI}^{(\delta_2)}_2$ vanishes as $\eta_2 \searrow 0$ along $\Dom$. We have
\begin{align}
\MoveEqLeft
	\CI^{(\eta_2)}_2 - \wt{\CI}^{(\delta_2)}_2 \\
	&=\, \frac{2}{\delta_2} \int_{0}^{\delta_2} \d z \int_{0}^{\infty} \d v\; \bigl(\nabla_{\!\eps_2} f_z(-M) - f'_z(-M)\bigr) \,\cdot\, p_{\delta_2 - z}(v) \,\cdot\, g'_{\delta_2}(v - M + \eps_2)\\
	&\quad+ \frac{2}{\delta_2} \int_{0}^{\delta_2} \d z \int_{0}^{\infty} \d v\; f'_z(-M) \,\cdot\, p_{\delta_2 - z}(v) \,\cdot\, \bigl(g'_{\delta_2}(v - M + \eps_2) - g'_{\delta_2}(v - M)\bigr).
\end{align}
Denote these two terms by $\CI^{(\eta_2)}_{2, 1}$ and $\CI^{(\eta_2)}_{2, 2}$ respectively; we will show that they both vanish in the limit.

\smallskip
\noindent{\bf The term $\CI^{(\eta_2)}_{2,1}$.} We can bound
\begin{equation}
\begin{split}
|\CI^{(\eta_2)}_{2, 1}|
 &\leq \Big(\sup_{0 \leq z \leq \delta_2} \bigl|\nabla_{\!\eps_2} f_z(-M) - f'_z(-M)\bigr|\Big)\\
 &\qquad\times \frac{2}{\delta_2} \int_{0}^{\delta_2}\!\!\! \d z \left|\int_{0}^{\infty} \d v\;  p_{\delta_2 - z}(v) \,\cdot\, g'_{\delta_2}(v - M + \eps_2)\right|.
 \end{split}\label{eq:I21}
\end{equation}
From Lemma~\ref{lem:f_g} we get
\begin{equation}\label{eq:f_vanishes}
	\lim_{\eps_2 \searrow 0} \sup_{0 \leq z \leq \delta_2} \bigl|\nabla_{\!\eps_2} f_z(-M) - f'_z(-M)\bigr| = 0.
\end{equation}
Moreover, the monotonicity property of the Riemann integral gives
\begin{align}
\MoveEqLeft[6]
	\lim_{\delta_2 \searrow 0} \sup_{0 \leq \eps_2 \leq \delta_2^{\alpha_\star}} \frac{1}{\delta_2} \int_{0}^{\delta_2} \d z \left|\int_{0}^{\infty} \d v\;  p_{\delta_2 - z}(v) \,\cdot\, g'_{\delta_2}(v - M + \eps_2)\right| \\
	& \leq \lim_{\delta_2 \searrow 0} \frac{1}{\delta_2} \int_{0}^{\delta_2} \d z\; \sup_{0 \leq \eps_2 \leq \delta_2^{\alpha_\star}} \left|\int_{0}^{\infty} \d v\;  p_{\delta_2 - z}(v) \,\cdot\, g'_{\delta_2}(v - M + \eps_2)\right|.
\end{align}
From boundedness and continuity of the function $g'_{\delta_2}$ (Lemma~\ref{lem:f_g}) we conclude that the function
\begin{equation}
	z  \mapsto  \sup_{0 \leq \eps_2 \leq \delta_2^{\alpha_\star}} \left|\int_{0}^{\infty} \d v\;  p_{\delta_2 - z}(v) \,\cdot\, g'_{\delta_2}(v - M + \eps_2)\right|
\end{equation}
is continuous in $z \in (0, \delta_2]$. Then the fundamental theorem of calculus and Lemma~\ref{lem:heat2} yield
\begin{equation}
	\begin{split}
	\MoveEqLeft[7]
	\lim_{\delta_2 \searrow 0} \frac{1}{\delta_2} \int_{0}^{\delta_2} \d z\; \sup_{0 \leq \eps_2 \leq \delta_2^{\alpha_\star}} \left|\int_{0}^{\infty} \d v\;  p_{\delta_2 - z}(v) \,\cdot\, g'_{\delta_2}(v - M + \eps_2)\right| \\
		&=\lim_{\delta_2 \searrow 0} \sup_{0 \leq \eps_2 \leq \delta_2^{\alpha_\star}} \left|\int_{0}^{\infty} \d v\;  p_{\delta_2}(v) \,\cdot\, g'_{\delta_2}(v - M + \eps_2)\right| \\
		&\leq \lim_{\delta_2 \searrow 0} \sup_{0 \leq \eps_2 \leq \delta_2^{\alpha_\star}} \int_{0}^{\infty} \d v\;  p_{\delta_2}(v) \,\cdot\, |g'_{\delta_2}(v - M + \eps_2)| = \frac{1}{2} |g'_{0}(- M)|,
	\end{split}\label{eq:I21_intermediate}
\end{equation}
for $\alpha_\star$ from \eqref{eq:Dom}. From Lemma~\ref{lem:f_g} we conclude that all these limits hold uniformly in $u_1, u_2 \in \R$. Using \eqref{eq:I21_intermediate} together with \eqref{eq:f_vanishes} and \eqref{eq:I21} yields
\begin{equation}\label{eq:I21_limit}
	\mylim{\eta_2 \searrow 0} |\CI^{(\eta_2)}_{2,1}| = 0,
\end{equation}
uniformly in $u_1, u_2 \in \R$.

\smallskip
\noindent{\bf The term $\CI^{(\eta_2)}_{2,2}$.} We bound
\begin{equation}\label{eq:I22}
\begin{split}
	|\CI^{(\eta_2)}_{2, 2}| &\leq \Big( \sup_{0 \leq z \leq \delta_2} |f'_z(-M)|\Big)\\
	&\qquad\times \frac{2}{\delta_2} \int_{0}^{\delta_2}\!\!\! \d z \left|\int_{0}^{\infty}\!\!\! \d v\; p_{\delta_2 - z}(v) \,\cdot\, \bigl(g'_{\delta_2}(v - M + \eps_2) - g'_{\delta_2}(v - M)\bigr)\right|.
\end{split}
\end{equation}
Lemma~\ref{lem:f_g} implies that the supremum of $|f'_z(-M)|$ in \eqref{eq:I22} is bounded, and the function
\begin{equation}
	z  \mapsto  \int_{0}^{\infty} \d v\;  p_{\delta_2 - z}(v) \,\cdot\, \bigl(g'_{\delta_2}(v - M + \eps_2) - g'_{\delta_2}(v - M)\bigr)
\end{equation}
is continuous in $z \in (0, \delta_2]$. Then, in the same way as in \eqref{eq:I21_intermediate}, the fundamental theorem of calculus and Lemma~\ref{lem:heat2} give
\begin{align}
\MoveEqLeft[6]
	\lim_{\delta_2 \searrow 0} \sup_{0 \leq \eps_2 \leq \delta_2^{\alpha_\star}} \frac{1}{\delta_2} \int_{0}^{\delta_2} \d z \left|\int_{0}^{\infty} \d v\; p_{\delta_2 - z}(v) \,\cdot\, \bigl(g'_{\delta_2}(v - M + \eps_2) - g'_{\delta_2}(v - M)\bigr)\right|\\
	&= \lim_{\delta_2 \searrow 0} \sup_{0 \leq \eps_2 \leq \delta_2^{\alpha_\star}} \left|\int_{0}^{\infty} \d v\; p_{\delta_2}(v) \,\cdot\, \bigl(g'_{\delta_2}(v - M + \eps_2) - g'_{\delta_2}(v - M)\bigr)\right| = 0.
\end{align}
From Lemma~\ref{lem:f_g} we conclude that this limit and the bounds hold uniformly in $u_1, u_2 \in \R$. Using this limit in \eqref{eq:I22}, we see that
\begin{equation}\label{eq:I22_limit}
	\mylim{\eta_2 \searrow 0} |\CI^{(\eta_2)}_{2,2}| = 0.
\end{equation}
From \eqref{eq:I21_limit} and \eqref{eq:I22_limit}, we conclude that $\mylim{\eta_2 \searrow 0} \bigl|\CI^{(\eta_2)}_2 - \wt{\CI}^{(\delta_2)}_2 \bigr| = 0$ uniformly in $u_1, u_2 \in \R$.

We now prove that $\wt{\CI}^{(\delta_2)}_2$ converges to $\CI_2$ as $\eta_2 \searrow 0$. From Lemma~\ref{lem:f_g}, we conclude that the function
\begin{equation}
	z  \mapsto   \int_{0}^{\infty} \d v\; f'_z(-M) \,\cdot\, p_{\delta_2 - z}(v) \,\cdot\, g'_{\delta_2}(v - M)
\end{equation}
is continuous in $z \in (0, \delta_2]$. Then applying the fundamental theorem of calculus and Lemma~\ref{lem:heat2} to \eqref{eq:I2_tilde}, we get
\begin{align}\label{eq:I2_limit}
	\lim_{\delta_2 \searrow 0} \wt{\CI}^{(\delta_2)}_2 &= \lim_{\delta_2 \searrow 0} \frac{2}{\delta_2} \int_{0}^{\delta_2} \d z \int_{0}^{\infty} \d v\; f'_z(-M) \,\cdot\, p_{\delta_2 - z}(v) \,\cdot\, g'_{\delta_2}(v - M)\\
	&= \lim_{\delta_2 \searrow 0} 2 \int_{0}^{\infty} \d v\; f'_0(-M) \,\cdot\, p_{\delta_2}(v) \,\cdot\, g'_{\delta_2}(v - M) = \CI_2.
\end{align}
As before we conclude that the limit holds uniformly in $u_1, u_2 \in \R$.

From \eqref{eq:I1_limit} and \eqref{eq:I2_limit}, we obtain
\begin{equation}
	\begin{split}
	\MoveEqLeft[8]
	\mylim{\eta_2 \searrow 0} \frac{1}{|\eta_2|} \Bigl(\fQ_L^{(\eta_1; \eta_2)} - \fQ_L^{(\eta_1; 0)}\Bigr)(u_1, u_2)\\
	&= - f'_0(-M) \cdot g'_{0}(-M)\\
	&= -\partial_{2} \fP_{[-L,\fx_2]}^{\nohit\,(-\fg_{\eta_1, 0})} (u_1, -\fM)\,\cdot\,\partial_{1} \Theta^{(-\fM)}_{[\fx_2, L]} (-\fM, u_2),
	\end{split}\label{eq:Q1_limit_proof}
\end{equation}
uniformly in $u_1, u_2 \in \R$, which is the required formula \eqref{eq:Q_eps1_0_formula}.

 Dependence of the prelimiting kernel on $\eta_1$ is only through the function $\fP_{[-L,\fx_2]}^{\nohit\,(-\fg_{\eta_1, 0})}$ and  its $\partial_{2}$-derivative, which appear in \eqref{eq:Q_eps1_0_formula} and other formulas. From \eqref{eq:f_explicitly}, we can see that these functions are bounded uniformly in $\eta_1 \in (0, \frac{A}{2}]^2$; indeed, the kernel $\fP_{[-L,\fx_1 + \delta_1]}^{\nohit\,(-\fg_{\eta_1, 0})}$ is bounded because it is a density function, and boundedness of $\partial_{2} \Theta^{(-\fM)}_{[\fx_1 + \delta_1, x_2]}$ follows from \eqref{e:Theta_formula}. Hence, the limit \eqref{eq:Q_eps1_0} holds uniformly over $\eta_1 \in (0, \frac{A}{2}]^2$.

One can readily see that the limit \eqref{eq:Q_eps1_0} holds also uniformly over $(x_1, x_2) \in S$ and $M \in \bar S$. This is because all the involved kernels are given by hitting and transition distributions of Brownian motion, which are bounded uniformly in these variables. This finishes the proof of \eqref{eq:Q_eps1_0} and \eqref{eq:Q_eps1_0_formula}.
\smallskip

A proof of existence of the limit \eqref{eq:Q_0_eps2} and the formula \eqref{eq:Q_0_eps2_formula} is carried out in a similar manner. Indeed, we may write
\begin{align}
\frac{\Bigl(\fQ_L^{(\eta_1; 0)} - \fQ_L^{(0; 0)}\Bigr)(u_1, u_2)}{p_{2 L}(u_2 - u_1)} =& \P_{\fB(-L)=u_1, \fB(L) = u_2}\Bigl(\fB(y)>-\fg_{\eta_1, 0}(y) \text{ for } y \in [-L,L]\Bigr) \\
&- \P_{\fB(-L)=u_1, \fB(L) = u_2}\Bigl(\fB(y)>-\fM \text{ for } y \in [-L,L]\Bigr)\\
=& - \P_{\fB(-L)=u_1, \fB(L) = u_2}\Bigl(\fB(y)>-\fM \text{ for } y \in [-L,L],\, \\
& \hspace{0.5cm} \fB(z) \in (-\fM, -\fM + \eps_1] \text{ for some } z \in [\fx_1, \fx_1 + \delta_1)\Bigr).
\end{align}
Defining the stopping time $\bar{\tau}^{(a)} := \sup \{z < x_1 + \delta_1: \fB(z) = a\}$, as in \eqref{eq:Q_formula1} we write
\begin{align}
\frac{\Bigl(\fQ_L^{(\eta_1; 0)} - \fQ_L^{(0; 0)}\Bigr)(u_1, u_2)}{p_{2 L}(u_2 - u_1)}
&= - \P_{\fB(-L)=u_1, \fB(L)=u_2}\Bigl(\fB(y)>-M \text{ for } y \in [-L,L],\, \\
& \hspace{4.5cm} \bar{\tau}^{(-\fM + \eps_1)} \in [\fx_1, \fx_1 + \delta_1)\Bigr).\label{eq:Q2_stopping_time}
\end{align}
Then, as in \eqref{eq:Q_formula2}, we may show that
\begin{align}
	\frac{1}{|\eta_1|} \Bigl(\fQ_L^{(\eta_1; 0)} &- \fQ_L^{(0; 0)}\Bigr)(u_1, u_2) = - \frac{2}{\delta_1} \int_{0}^{\delta_1} \d z\; \nabla_{\!\eps_1} \bar{f}_z(-M) \,\cdot\, p_{\delta_1 - z}(0) \,\cdot\, \bar{g}_{\delta_1}(-M + \eps_1)\\
	&- \frac{2}{\delta_1} \int_{0}^{\delta_1} \d z \int_{-M + \eps_1}^{\infty} \d v\; \nabla_{\!\eps_1} \bar{f}_z(-M) \,\cdot\, p_{\delta_1 - z}(v + M - \eps_1) \,\cdot\, \bar{g}'_{\delta_1}(v),
\end{align}
where
\begin{equation}
	\bar{f}_\delta(v) = \Theta^{(-\fM)}_{[-L, x_1 + \delta]} (u_1, v), \qquad \bar{g}_{\delta}(v) = \Theta^{(-\fM)}_{[x_1 + \delta, L]} (v, u_2).
\end{equation}
Since these function are particular cases of the functions \eqref{eq:functions_f_g}, we can repeat the argument used in our derivation of \eqref{eq:Q1_limit_proof}. This shows that
\begin{align}
	\mylim{\eta_1 \searrow 0} \frac{1}{|\eta_1|} \Bigl(\fQ_L^{(\eta_1; 0)} - \fQ_L^{(0; 0)}\Bigr)(u_1, u_2) &= - \bar{f}'_0(-M) \cdot \bar{g}'_{0}(-M)\\
	&= -\partial_{2} \Theta^{(-\fM)}_{[-L, \fx_1]} (u_1, -\fM) \,\cdot\, \partial_{1} \Theta^{(-\fM)}_{[\fx_1, L]} (-\fM, u_2).\label{eq:Q1_limit_formula}
\end{align}
Uniformity of this limit over $(x_1, x_2) \in S$, $u_1, u_2 \in \R$ and $M \in \bar S$ can be seen as in \eqref{eq:Q1_limit_proof}. This finishes the proof of \eqref{eq:Q_0_eps2} and \eqref{eq:Q_0_eps2_formula}. To prove \eqref{eq:Q_0_0}, we also need the limit
\begin{equation}
\begin{split}
\MoveEqLeft[6]
	\mylim{\eta_1 \searrow 0} \frac{1}{|\eta_1|} \partial_2 \Bigl(\fQ_L^{(\eta_1; 0)} - \fQ_L^{(0; 0)}\Bigr)(u_1, u_2)\\
	&= -\partial_{2} \Theta^{(-\fM)}_{[-L, \fx_1]} (u_1, -\fM) \,\cdot\, \partial_{1} \partial_{2} \Theta^{(-\fM)}_{[\fx_1, L]} (-\fM, u_2),
\end{split}
\label{eq:Q1_limit_formula_derivative}
\end{equation}
which can be proved by repeating the proof of \eqref{eq:Q1_limit_formula}.
\smallskip

Now, we turn to the proof of existence of the last limit \eqref{eq:Q_0_0} and the formula \eqref{eq:Q_0_0_formula}. Using \eqref{eq:Q1_limit_proof}, we obtain
\begin{equation}\label{eq:Q_star_formula}
\begin{split}
\MoveEqLeft[3]
\Bigl(\fQ_L^{(\eta_1; \star)} - \fQ_L^{(0; \star)}\Bigr) (u_1, u_2)\\
&= \partial_{2} \Bigl(\fP_{[-L,\fx_2]}^{\nohit\,(-\fg_{\eta_1, 0})} - \fP_{[-L,\fx_2]}^{\nohit\,(-\fM)} \Bigr) (u_1, -\fM) \,\cdot\, \partial_{1} \Theta^{(-\fM)}_{[\fx_2, L]} (-\fM, u_2).
\end{split}
\end{equation}
We may assume that $\fx_1 + \delta_1 < \fx_2$. Using \eqref{eq:no_hit}, we may write the kernel in parentheses explicitly:
\begin{align}
\MoveEqLeft[11]
\frac{\Bigl(\fP_{[-L,\fx_2]}^{\nohit\,(-\fg_{\eta_1, 0})} - \fP_{[-L,\fx_2]}^{\nohit\,(-\fM)} \Bigr) (u_1, v_1)}{p_{\fx_2 + L}(v_1 - u_1)}\\
=&\, \P_{\fB(-L)=u_1, \fB(\fx_2) = v_1}\Bigl(\fB(y)>-\fg_{\eta_1, 0}(y) \text{ for } y \in [-L,\fx_2]\Bigr) \\
& - \P_{\fB(-L)=u_1, \fB(\fx_2) = v_1}\Bigl(\fB(y)>-\fM \text{ for } y \in [-L,\fx_2]\Bigr)\\
=& - \P_{\fB(-L)=u_1, \fB(\fx_2) = v_1}\Bigl(\fB(y)>-\fM \text{ for } y \in [-L,\fx_2],\, \\
& \hspace{0.5cm} \fB(z) \in (-\fM, -\fM + \eps_1] \text{ for some } z \in [\fx_1, \fx_1 + \delta_1)\Bigr).
\end{align}
As in \eqref{eq:Q2_stopping_time}, we can write
\begin{align}
\MoveEqLeft[2]
\frac{\Bigl(\fP_{[-L,\fx_2]}^{\nohit\,(-\fg_{\eta_1, 0})} - \fP_{[-L,\fx_2]}^{\nohit\,(-\fM)} \Bigr) (u_1, v_1)}{p_{\fx_2 + L}(v_1 - u_1)} \\
=&-\P_{\fB(-L)=u_1, \fB(x_2)=u_2}\Bigl(\fB(y)>-M \text{ for } y \in [-L,x_2],\, \bar{\tau}^{(-\fM + \eps_1)} \in [\fx_1, \fx_1 + \delta_1)\Bigr).
\end{align}
This formula is the same as \eqref{eq:Q2_stopping_time}, where the interval is $[-L, x_2]$ instead of $[-L, L]$ . In the same way as in \eqref{eq:Q1_limit_formula_derivative}, we arrive at
\begin{align}
\MoveEqLeft[11]
\mylim{\eta_1 \searrow 0} \frac{1}{|\eta_1|} \partial_{2} \Bigl(\fP_{[-L,\fx_2]}^{\nohit\,(-\fg_{\eta_1, 0})} - \fP_{[-L,\fx_2]}^{\nohit\,(-\fM)}\Bigr)(u_1, v_1)\\
&= -\partial_{2} \Theta^{(-\fM)}_{[-L, \fx_1]} (u_1, -\fM) \,\cdot\, \partial_{1} \partial_{2} \Theta^{(-\fM)}_{[\fx_1, \fx_2]} (-\fM, v_1).
\end{align}
Using this limit in \eqref{eq:Q_star_formula}, we obtain
\begin{align}
\mylim{\eta_1 \searrow 0} \frac{1}{|\eta_1|} &\Bigl(\fQ_L^{(\eta_1; \star)} - \fQ_L^{(0; \star)}\Bigr) (u_1, u_2) \\
&= \mylim{\eta_1 \searrow 0} \frac{1}{|\eta_1|} \partial_{2} \Bigl(\fP_{[-L,\fx_2]}^{\nohit\,(-\fg_{\eta_1, 0})} - \fP_{[-L,\fx_2]}^{\nohit\,(-\fM)} \Bigr) (u_1, -\fM) \,\cdot\, \partial_{1} \Theta^{(-\fM)}_{[\fx_2, L]} (-\fM, u_2)\\
&= -\partial_{2} \Theta^{-\fM}_{[-L, \fx_1]} (u_1, -\fM) \,\cdot\, \partial_{1} \partial_{2} \Theta^{(-\fM)}_{[\fx_1, \fx_2]} (-\fM, -\fM) \,\cdot\, \partial_{1} \Theta^{(-\fM)}_{[\fx_2, L]} (-\fM, u_2).
\end{align}
Using \eqref{e:Theta_formula}, we can compute explicitly
\begin{equation}
\partial_{1} \partial_{2} \Theta^{(-\fM)}_{[\fx_1, \fx_2]} (-\fM, -\fM) = -2 p''_{\fx_2 - \fx_1} (0) = \frac{1}{\sqrt{4\pi}} (\fx_2 - \fx_1)^{-\frac{3}{2}},
\end{equation}
which gives
\begin{align}
\MoveEqLeft[7]
\mylim{\eta_1 \searrow 0} \frac{1}{|\eta_1|} \Bigl(\fQ_L^{(\eta_1; \star)} - \fQ_L^{(0; \star)}\Bigr) (u_1, u_2)\\
&= \frac{1}{\sqrt{4\pi}} (\fx_2 - \fx_1)^{-\frac{3}{2}} \cdot \partial_{2} \Theta^{(-\fM)}_{[-L, \fx_1]} (u_1, -\fM) \cdot \partial_{1} \Theta^{(-\fM)}_{[\fx_2, L]} (-\fM, u_2).
\end{align}
As above, uniformity in the variables $u_1, u_2 \in \R$, $(x_1, x_2) \in S$ and $M \in \bar S$ follows readily from the properties of the kernels. This operator is rank one, which finishes the proof of \eqref{eq:Q_0_0} and \eqref{eq:Q_0_0_formula}.
\end{proof}

We will use the kernels \eqref{eqs:Q_limits} in Lemmas~\ref{lem:A_formulas} and \ref{lem:deriv1} to compute derivatives of the kernel \eqref{eq:fR}. For this, we will also use the following two functions
\begin{subequations}\label{eqs:D_def}
\begin{align}
\fD^{(\eta_1; \eta_2)}_{[\fx, L], +}(w) &:= \partial_{1} \Bigl(\fP_{[\fx, L]}^{\nohit\,(-\fg_{\eta_1, \eta_2})} \fT_{1/2,-L} \Gamma \varrho\Bigr) (-\fM, w),\label{eq:D1_def}\\
\fD^{(\eta_1; \eta_2)}_{[-L, \fx], -}(w) &:= \partial_{2} \Bigl(\varrho \Gamma \bigl(\fT_{1/2,-L}\bigr)^*\, \fP_{[-L, \fx]}^{\nohit\,(-\fg_{\eta_1, \eta_2})}\Bigr) (w, -\fM),\label{eq:D2_def}
\end{align}
\end{subequations}
which are defined for $|\fx| \leq L$. Here, the multiplicative operator $\Gamma := \Gamma_{1/2}$ is defined in \eqref{eq:Gamma_L}; $\varrho$ is the reflection operator $\varrho f(u) := f(-u)$; and we used operators defined in \eqref{eq:groupS} and \eqref{eq:no_hit}. To bound in Lemma~\ref{lem:A_formulas} the kernels in trace norm, we need the following bounds on the $L^2$ norms of the functions in  \eqref{eqs:D_def}.

\begin{lem}\label{lem:D_bounds}
Let the sets $S$ and $\bar S$ be defined in \eqref{eq:setsS}. The following bounds hold:
\begin{equation}\label{eqs:D_bounds}
\Bigl\| \fD^{(\eta_1; \eta_2)}_{[\fx, L], +} \Bigr\|_{L^2} \leq C, \qquad\qquad \Bigl\| \fD^{(\eta_1; \eta_2)}_{[-L, \fx], -} \Bigr\|_{L^2} \leq C,
\end{equation}
uniformly over $(\fx_1, \fx_2) \in S$, $|\fx| \leq \beta L$, $\fM \in \bar S$, and $\eta_1, \eta_2 \in [0, \frac{A}{2}]^2$. Here, $C \geq 0$ is a constant that depends only on $L$ and $\bar L$, the dependence of the functions in \eqref{eqs:D_def} on $x_1$, $x_2$ and $M$ comes through the function $\fg_{\eta_1, \eta_2}$ defined in  \eqref{eq:g_def}, and the constant $\beta > 0$ is from the definition of $\bJ_{L}$ in \eqref{eq:setJ}.
\end{lem}

\begin{proof}
We start our analysis with the function $\fD^{(\eta_1; \eta_2)}_{[\fx, L], +}$, defined in \eqref{eq:D1_def}. First of all, we rewrite the formula to simplify the apparent dependence on $L$. Let $\fB$ be Brownian motion of variance two and define $\tau_g$ to be the first hitting time of a function $g : \R \to \R$ from above. That is to say, if the starting time for $\fB$ is $x$, then
\begin{equation}\label{eq:tau-def}
\tau_g := \inf\bigl\{\ell \geq x : \fB(\ell) \leq g(\ell)\bigr\}.
\end{equation}
Then the transition distribution \eqref{eq:no_hit} may be written as
\begin{equation}
\begin{split}
\MoveEqLeft[6]
\fP_{[\fx, L]}^{\nohit\,(-\fg_{\eta_1, \eta_2})}(u_1, u_2)\\
&= p_{L - \fx}(u_2 - u_1) - \P_{\fB(\fx)=u_1, \fB(L) = u_2}\bigl(\tau_{-\fg_{\eta_1, \eta_2}} \leq L\bigr) \,\cdot\, p_{L - \fx}(u_2 - u_1),
\end{split}
\end{equation}
where $p_\ell$ is the transition kernel of $\fB$, defined in \eqref{eq:heat_kernel}. Using the strong Markov property, we may rewrite the last expression as
\begin{align}
\MoveEqLeft[6]
\fP_{[\fx, L]}^{\nohit\,(-\fg_{\eta_1, \eta_2})}(u_1, u_2)\\
&= p_{L - \fx}(u_2 - u_1) - \int_{x}^L \P_{\fB(x) = u_1} \bigl(\tau_{-\fg_{\eta_1, \eta_2}} \in \d s\bigr) \cdot p_{L - s} \bigl(u_2 -\fg_{\eta_1, \eta_2}(s)\bigr).
\label{eq:earlierformula}
\end{align}
Here $\P_{\fB(x) = u_1} (\tau_{-\fg_{\eta_1, \eta_2}} \in \bigcdot\,)$ is the probability measure on $\R$, generated by the random variable $\tau_{-\fg_{\eta_1, \eta_2}}$ under the initial data that $\fB(x) = u_1$. Just looking directly at this integral, it may not be immediately clear that it is finite since if $u_2 \leq -M$, the heat kernel $p_{L - s}$ tends to a Dirac delta function when $s \to L$. Let us briefly explain why the integral is well-defined.  Another way to write the integral is as
\begin{align}
	\E_{\fB(\fx)=u_1} &\Bigl[p_{L - \tau_{-\fg_{\eta_1, \eta_2}}} \bigl(u_2 -\fg_{\eta_1, \eta_2}(\tau_{-\fg_{\eta_1, \eta_2}})\bigr) \,\cdot\, \1{\tau_{-\fg_{\eta_1, \eta_2}} \leq L}\Bigr]\\
	= &\,\E_{\fB(\fx)=u_1} \Bigl[p_{L - \tau_{-\fg_{\eta_1, \eta_2}}} \bigl(u_2 -\fg_{\eta_1, \eta_2}(\tau_{-\fg_{\eta_1, \eta_2}})\bigr) \,\cdot\, \1{\tau_{-\fg_{\eta_1, \eta_2}} \leq x_2 +
	\delta_2}\Bigr] \label{eq:earlierformula1}\\
	&+ \E_{\fB(\fx)=u_1} \Bigl[p_{L - \tau_{-\fg_{\eta_1, \eta_2}}} (u_2 + M) \,\cdot\, \1{x_2 +
	\delta_2 < \tau_{-\fg_{\eta_1, \eta_2}} \leq L}\Bigr],
\end{align}
where in the last term we used the identity $\fg_{\eta_1, \eta_2}(\tau_{-\fg_{\eta_1, \eta_2}}) = -M$ almost surely when $x_2 + \delta_2 < \tau_{-\fg_{\eta_1, \eta_2}} \leq L$. The first expectation on the right-hand side of \eqref{eq:earlierformula1} is finite, because the heat kernel is bounded. Since $\tau_{-\fg_{\eta_1, \eta_2}}$ is a hitting time of a piecewise constant function, from \cite[Section~2.6]{karatzas1998brownian} we conclude that its probability distribution has a density
$\frac{\d}{\d s} \P_{\fB(x) = u_1} \bigl(\tau_{-\fg_{\eta_1, \eta_2}} \leq s\bigr)$,
and, moreover, that on the time interval $s \in (x_2 + \delta_2, L]$ this density is bounded above by a constant which depends on $x_2 + \delta_2$ and $L$ (cf. \eqref{eq:density_stopping_time} and \eqref{eq:density_stopping_time2} for  similar densities). Thus, the last expectation in \eqref{eq:earlierformula1} is bounded by a constant multiple of
$\int_{x_2 + \delta_2}^L \d s\; p_{L - s} (u_2 + M)$, which is finite as desired.

Returning to \eqref{eq:earlierformula}, we claim that this implies that
\begin{equation}\label{eq:Theta_new_formula}
\fP_{[\fx, L]}^{\nohit\,(-\fg_{\eta_1, \eta_2})}(u_1, u_2) = \Bigl(\fT_{0, L - \fx} - \SS_{0, L, [\fx, L]}^{(\eta_1; \eta_2)}\Bigr) (u_1, u_2),
\end{equation}
where $\fT_{t,x}(u,v)$ is defined in  \eqref{eq:fTdef} and
\begin{equation}\label{eq:S_bold}
\SS_{t, \fy, [\fx, L]}^{(\eta_1; \eta_2)}(u_1, u_2) := \E_{\fB(\fx)=u_1}\Bigl[ \fT_{t, \fy - \tau_{-\fg_{\eta_1, \eta_2}}} \bigl(\fB(\tau_{-\fg_{\eta_1, \eta_2}}), u_2 \bigr) \,\cdot\, \1{\tau_{-\fg_{\eta_1, \eta_2}} \leq L}\Bigr].
\end{equation}
To justify the claim, we observe that $p_\ell(u) = \fT_{0,\ell}(u)$ and $\fT_{t,x}(u,v) = \fT_{t,x}(u - v)$. Using these facts, \eqref{eq:Theta_new_formula} follows readily from \eqref{eq:earlierformula}. The first of these facts, that $p_\ell(u) = \fT_{0,\ell}(u)$, follows from \eqref{eq:groupS}.

Using \eqref{eq:Theta_new_formula}, we write  the kernel \eqref{eq:D1_def} in the form
\begin{equation}
\begin{split}
	\fD^{(\eta_1; \eta_2)}_{[\fx, L], +}(w) =& \int_{-\infty}^\infty \d v\; \partial_{1} \fT_{0, L - \fx} (-M, v) \cdot \bigl(\fT_{1/2,-L} \Gamma \varrho\bigr) (v, w)\\
	&- \int_{-\infty}^\infty \d v\; \partial_{1} \SS_{0, L, [\fx, L]}^{(\eta_1; \eta_2)} (-M, v) \cdot \bigl(\fT_{1/2,-L} \Gamma \varrho\bigr) (v, w).
\end{split}
\label{eq:D1_intermediate}
\end{equation}
Applying the composition identity \eqref{eq:S_properties}, the first integral on the right-hand side of \eqref{eq:D1_intermediate} equals
\begin{equation}\label{eq:D1_intermediate0}
	\partial_{1} \bigl(\fT_{1/2,-\fx} \Gamma \varrho\bigr) (-\fM, w).
\end{equation}
Furthermore, fast decay of the functions \eqref{eq:fTdef} at infinity allows us to apply Fubini's theorem and write the second integral in \eqref{eq:D1_intermediate} as
\begin{equation}\label{eq:D1_intermediate1}
\begin{split}
\MoveEqLeft[11]
	\partial_{u} \E_{\fB(\fx)=u}\Bigl[ \int_{-\infty}^\infty \d v\; \fT_{0, L - \tau_{-\fg_{\eta_1, \eta_2}}} \bigl(\fB(\tau_{-\fg_{\eta_1, \eta_2}}), v \bigr)\\
	&\times \bigl(\fT_{1/2,-L} \Gamma \varrho\bigr) (v, w) \,\cdot\, \1{\tau_{-\fg_{\eta_1, \eta_2}} \leq L}\Bigr] \bigg|_{u = -M},
\end{split}
\end{equation}
where we used the definition \eqref{eq:S_bold}. Applying the identity \eqref{eq:S_properties}, the integral inside the expectation can be written as
\begin{equation}
	\bigl(\fT_{1/2, - \tau_{-\fg_{\eta_1, \eta_2}}} \Gamma \varrho\bigr) \bigl(\fB(\tau_{-\fg_{\eta_1, \eta_2}}), w\bigr).
\end{equation}
Combining this identity with \eqref{eq:D1_intermediate0} and \eqref{eq:D1_intermediate1}, we get from \eqref{eq:D1_intermediate} that
\begin{equation}\label{eq:D1_new_formula}
\fD^{(\eta_1; \eta_2)}_{[\fx, L], +}(w) = \partial_{1} \bigl(\fT_{1/2,-\fx} \Gamma \varrho\bigr) (-\fM, w) - \partial_{1} \bigl( \SS_{1/2, 0, [\fx, L]}^{(\eta_1; \eta_2)} \Gamma \varrho\bigr) (-\fM, w).
\end{equation}
Dependence of this expression on $L$ is now more apparent, allowing us to take the limit $L \to \infty$.

Now, we will bound the $L^2$ norm of \eqref{eq:D1_new_formula}. For this, we need to bound the $L^2$ norms of the two terms in \eqref{eq:D1_new_formula}. The reflection operator $\varrho$ in \eqref{eq:D1_new_formula} does not influence the $L^2$ norm and can be omitted. Then Lemma~\ref{lem:S_integral_estimate} yields the  bound on the first term in \eqref{eq:D1_new_formula},
\begin{equation}
\Bigl\| \partial_{1} \bigl(\fT_{1/2, - \fx} \Gamma \varrho\bigr) (-\fM, \bigcdot) \Bigr\|_{L^2} \leq C_1,
\end{equation}
for all $|\fx| \leq \beta L$ and $\fM \in \bar S$, and for some constant $C_1 \geq 0$, depending on $L$ and $\bar L$.

Next, we turn to the second term in \eqref{eq:D1_new_formula}. Using \eqref{eq:S_bold} and integration by parts, we can write
\begin{equation}
\begin{split}
\SS_{1/2, 0, [\fx, L]}^{(\eta_1; \eta_2)} (v, w)
&= \int_\fx^{L} \P_{\fB(\fx)=v} \bigl( \tau_{-\fg_{\eta_1, \eta_2}} \in \d s\bigr) \,\cdot\, \bigl( \fT_{1/2, - s} \Gamma \varrho \bigr) \bigl(-\fg_{\eta_1, \eta_2}(s), w\bigr)\\
&= \int_\fx^{L} \P_{\fB(\fx)=v} \bigl( \tau_{-\fg_{\eta_1, \eta_2}} < s\bigr) \,\cdot\, \partial_s \bigl( \fT_{1/2, - s} \Gamma \varrho \bigr) \bigl(-\fg_{\eta_1, \eta_2}(s), w\bigr) \d s\\
&\quad\qquad + \P_{\fB(\fx)=v} \bigl( \tau_{-\fg_{\eta_1, \eta_2}} \leq L\bigr) \,\cdot\, \bigl( \fT_{1/2, - L} \Gamma \varrho \bigr) \bigl(-M, w\bigr)\\
&\quad\qquad - \P_{\fB(\fx)=v} \bigl( \tau_{-\fg_{\eta_1, \eta_2}} \leq \fx\bigr) \,\cdot\, \bigl( \fT_{1/2, - \fx} \Gamma \varrho \bigr) \bigl(-\fg_{\eta_1, \eta_2}(x), w\bigr),
\end{split}
\label{eq:last_term_formula}
\end{equation}
where we used $\fg_{\eta_1, \eta_2}(L) = M$. As in \eqref{eq:earlierformula}, we use in this formula the probability distribution $\P_{\fB(\fx)=v} ( \tau_{-\fg_{\eta_1, \eta_2}} \in \d s\,)$ generated by the random variable $\tau_{-\fg_{\eta_1, \eta_2}}$. The integration-by-parts formula holds for any probability distribution (not necessarily absolutely continuous) and can be found in \cite[Theorem~21.67]{Hewitt}. From \eqref{e:Theta_formula}, one can readily conclude that
\begin{align}
\Bigl|\partial_{v}\P_{\fB(\fx)=v} \bigl( \tau_{-\fg_{\eta_1, \eta_2}} < s\bigr) \big|_{v = -\fM} \Bigr| &\leq C (s - \fx)^{-\frac{1}{2}}, \quad \text{for}~s > \fx,\\
\Bigl|\partial_{v}\P_{\fB(\fx)=v} \bigl( \tau_{-\fg_{\eta_1, \eta_2}} < \fx\bigr) \big|_{v = -\fM} \Bigr| &= 0.
\end{align}
Formula \eqref{eq:fTdef} implies that $\partial_x \fT_{t,x}(z) = \fT_{t,x}''(z)$, where the derivatives on the right-hand side are taken with respect to $z$. Applying these bounds and identities to \eqref{eq:last_term_formula}, we learn that
\begin{align}
\Bigl|\partial_{1} \bigl(\SS_{1/2, 0, [\fx, L]}^{(\eta_1; \eta_2)} \Gamma \varrho\bigr) (-\fM, w)\Bigr| &\leq C \int_\fx^{L} (s - \fx)^{-\frac{1}{2}} \bigl|\bigl( \fT_{1/2, - s}'' \Gamma \varrho \bigr) (-\fM - w) \bigr| \d s \\
&\hspace{2cm} + C (L - \fx)^{-\frac{1}{2}} \bigl|\bigl( \fT_{1/2, -L} \Gamma \varrho \bigr) (-\fM - w)\bigr|.
\end{align}
Using the triangle inequality for the $L^2$ norm, we obtain
\begin{align}
\Bigl\| \partial_{1} \bigl(\SS_{1/2, 0, [\fx, L]}^{(\eta_1; \eta_2)} \Gamma \varrho\bigr)  (-\fM, \bigcdot) \Bigr\|_{L^2} &\leq C \int_\fx^{L} (s - \fx)^{-\frac{1}{2}} \bigl\| \bigl( \fT_{1/2, - s}'' \Gamma \varrho \bigr) (-\fM - \bigcdot) \bigr\|_{L^2} \d s \\
&\hspace{1.5cm} + C (L - \fx)^{-\frac{1}{2}} \bigl\|\bigl( \fT_{1/2, -L} \Gamma \varrho \bigr) (-\fM - \bigcdot)\bigr\|_{L^2}.\label{eq:term2}
\end{align}
In view of Lemma~\ref{lem:S_integral_estimate}, the last term in this expression is bounded by $(L - \fx)^{-\frac{1}{2}} C_2$.
Since $-L \leq - s \leq L$, the assumptions of Lemma~\ref{lem:S_integral_estimate} are satisfied, and the first term in \eqref{eq:term2} is bounded by
\begin{equation}
C_3 \int_\fx^{L} (s - \fx)^{-\frac{1}{2}} \d s =   \frac{1}{2} C_3 (L-\fx)^{\frac{1}{2}}.
\end{equation}
Combining the derived bounds on \eqref{eq:D1_new_formula}, we obtain the first bound in \eqref{eqs:D_bounds}.
\smallskip

Now, we will prove the second bound in \eqref{eqs:D_bounds}. From definition \eqref{eq:D2_def}, we get
\begin{equation}\label{eq:D-new}
\fD^{(\eta_1; \eta_2)}_{[-L, \fx], -}(w) = \partial_{1} \Bigl(\fP_{[-\fx, L]}^{\nohit\,(-\bar{\fg}_{\eta_2, \eta_1})} \fT_{1/2,-L} \Gamma \varrho \Bigr) (-\fM, w),
\end{equation}
for a new function $\bar \fg_{\eta_2, \eta_1}(y) := \fM - \eps_2 \cdot \1{y \in (-\fx_2 - \delta_2, -\fx_2]} - \eps_1 \cdot \1{y \in (-\fx_1 - \delta_1, -\fx_1]}$. The formula \eqref{eq:D-new} is the same as \eqref{eq:D1_def}, but for this new function $\bar \fg_{\eta_2, \eta_1}$. Then the second bound in \eqref{eqs:D_bounds} follows from the first one.
\end{proof}

Using the functions \eqref{eqs:D_def}, we define the rank-one kernels
\begin{subequations}\label{eqs:A_defs}
\begin{align}
\fA_L^{(\eta_1; \star)}(u_1, u_2) &:= \fD^{(\eta_1; 0)}_{[-L, \fx_2], -}(u_1) \,\cdot\, \fD^{(\eta_1; 0)}_{[\fx_2, L], +}(u_2), \label{eq:A1}\\
\fA_L^{(\star; 0)}(u_1, u_2) &:= \fD^{(0; 0)}_{[-L, \fx_1], -}(u_1) \,\cdot\, \fD^{(0; 0)}_{[\fx_1, L], +}(u_2),\label{eq:A2}\\
\fA_L^{(\star; \star)}(u_1, u_2) &:= \frac{1}{\sqrt{4\pi}} (\fx_2 - \fx_1)^{-\frac{3}{2}} \, \cdot\, \fD^{(0; 0)}_{[-L, \fx_1], -}(u_1) \,\cdot\, \fD^{(0; 0)}_{[\fx_2, L], +}(u_2).\label{eq:A3}
\end{align}
\end{subequations}
We prove in Lemma~\ref{lem:deriv1} that these kernels equal to certain derivatives of the kernel \eqref{eq:fR}. For this we need to show that \eqref{eqs:A_defs} are trace class.

\begin{lem}\label{lem:A_formulas}
The kernels \eqref{eqs:A_defs} can be written in terms of \eqref{eqs:Q_limits} as
\begin{subequations}\label{eqs:A_formulas}
\begin{align}
\fA_L^{(\eta_1; \star)} &= \varrho\Gamma(\fT_{1/2,-L})^* \fQ_L^{(\eta_1; \star)} \fT_{1/2,-L} \Gamma\varrho, \label{eq:A1_formula}\\
\fA_L^{(\star; 0)} &= \varrho\Gamma(\fT_{1/2,-L})^* \fQ_L^{(\star; \eta_2)} \fT_{1/2,-L} \Gamma\varrho,\label{eq:A2_formula}\\
\fA_L^{(\star; \star)} &= \varrho\Gamma (\fT_{1/2,-L})^* \fQ_L^{(\star; \star)} \fT_{1/2,-L} \Gamma \varrho,\label{eq:A3_formula}
\end{align}
\end{subequations}
where as before $\Gamma = \Gamma_{1/2}$.

Let the sets $S$ and $\bar S$ be defined in \eqref{eq:setsS}. Then these kernels are trace class, with trace norms bounded as follows:
\begin{equation}\label{eq:A_bounds}
\bigl\| \fA_L^{(\eta_1; \star)} \bigr\|_1 \leq C, \qquad \bigl\| \fA_L^{(\star; 0)} \bigr\|_1 \leq C, \qquad \bigl\|\fA_L^{(\star; \star)} \bigr\|_1 \leq C |\fx_2 - \fx_1|^{-\frac{3}{2}},
\end{equation}
uniformly in $(x_1, x_2) \in S$ and $M \in \bar S$.
\end{lem}
\begin{proof}
Using formulas \eqref{eq:Q_eps1_0_formula} and \eqref{eqs:D_def}, we may write the right-hand side of \eqref{eq:A1_formula} in the form
\begin{align}
\MoveEqLeft
\int_{-\infty}^\infty \d u_1 \int_{-\infty}^\infty \d u_2\; \varrho \Gamma (\fT_{1/2,-L})^*(w_1, u_1) \,\cdot\, \fQ_L^{(\eta_1; \star)} (u_1, u_2) \,\cdot\, (\fT_{1/2,-L} \Gamma \varrho) (u_2, \bar v_2)\\
&= \partial_{2} \Bigl( \varrho \Gamma (\fT_{1/2,-L})^* \fP_{[-L,\fx_2]}^{\nohit\,(-\fg_{\eta_1, 0})}\Bigr) (w_1, -\fM) \,\cdot\, \partial_{1} \Bigl(\Theta^{(-\fM)}_{[\fx_2, L]} \fT_{1/2,-L} \Gamma \varrho\Bigr) (-\fM, \bar v_2)\\
&= \fD^{(\eta_1; 0)}_{[-L, \fx_2], -}(w_1) \,\cdot\, \fD^{(0; 0)}_{[\fx_2, L], +}(\bar v_2),
\end{align}
which is exactly $\fA_L^{(\eta_1; \star)}$. Since the kernel \eqref{eq:A1} is rank-one, from \eqref{eq:norms_products} we conclude that
\begin{equation}
\bigl\| \fA_L^{(\eta_1; \star)} \bigr\|_1 \leq \bigl\| \fD^{(\eta_1; 0)}_{[-L, \fx_2], -} \bigr\|_{L^2} \, \cdot\, \bigl\| \fD^{(\eta_1; 0)}_{[\fx_2, L], +} \bigr\|_{L^2}.
\end{equation}
Then Lemma~\ref{lem:D_bounds} yields the first bound in \eqref{eq:A_bounds}.

In the same way, we prove formulas \eqref{eq:A2_formula} and \eqref{eq:A3_formula}; and bounds on the trace norms of \eqref{eq:A2} and \eqref{eq:A3}.
\end{proof}

Next we differentiate the kernel $\fR_L^{(\eta_1; \eta_2)}$, defined in \eqref{eq:fR}, with respect to $\eta_1$ and $\eta_2$.

\begin{lem}\label{lem:deriv1}
Let the sets $S$ and $\bar S$ be defined in~\eqref{eq:setsS}. The following limits hold in trace norm:
\begin{subequations}\label{eqs:derivs}
\begin{align}\label{eq:deriv1}
\mylim{\eta_2 \searrow 0} \frac{1}{|\eta_2|} \Bigl( \fR_L^{(\eta_1; \eta_2)} -  \fR_L^{(\eta_1; 0)}\Bigr) &= \fA_L^{(\eta_1; \star)},\\
\mylim{\eta_1 \searrow 0} \frac{1}{|\eta_1|} \Bigl( \fR_L^{(\eta_1; 0)} -  \fR_L^{(0; 0)}\Bigr)  &= \fA_L^{(\star; 0)}, \label{eq:deriv2}\\
\mylim{\eta_1 \searrow 0} \frac{1}{|\eta_1|} \Bigl( \fA_L^{(\eta_1; \star)} -  \fA_L^{(0; \star)}\Bigr) &= \fA_L^{(\star; \star)},\label{eq:deriv3}
\end{align}
\end{subequations}
 uniformly over $(x_1, x_2) \in S$ and $M \in \bar S$; and the limit \eqref{eq:deriv1} holds uniformly over $\eta_1 \in (0, \frac{A}{2}]^2$.
\end{lem}

\begin{proof}
To compute the limits of the kernels, we use the definitions of the $\epi$-kernel and $\hypo$-kernels, provided in \eqref{eq:Kd_epi} and \eqref{eq:Kd} respectively. We start by proving \eqref{eq:deriv1}. We have 
\begin{align}
\fR_L^{(\eta_1; \eta_2)} - \fR_L^{(\eta_1; 0)} &= \varrho \Bigl(\fK^{\hypo(-\fg_{\eta_1, \eta_2})}_{1/2, [-L,L]} - \fK^{\hypo(-\fg_{\eta_1, 0})}_{1/2, [-L,L]}\Bigr) \varrho\\
& = \varrho \Gamma(\fT_{1/2,-L})^* \Bigl(\fP_{[-L,L]}^{\hit(-\fg_{\eta_1, \eta_2})} - \fP_{[-L,L]}^{\hit(-\fg_{\eta_1, 0})}\Bigr)\fT_{1/2,-L} \Gamma \varrho\\
& = \varrho\Gamma(\fT_{1/2,-L})^* \Bigl(\fQ_L^{(\eta_1; 0)} - \fQ_L^{(\eta_1; \eta_2)}\Bigr) \fT_{1/2,-L} \Gamma \varrho,\label{e:kernel_d2}
\end{align}
where in the last identity we made use of \eqref{eq:hit_def} and the operator \eqref{e:Q_def}. Applying \eqref{eq:norms_products}, we get 
\begin{align}
&\bigl\| |\eta_2|^{-1} \bigl(\fR_L^{(\eta_1; \eta_2)} - \fR_L^{(\eta_1; 0)}\bigr) + \varrho\Gamma(\fT_{1/2,-L})^* \fQ_L^{(\eta_1; \star)} \fT_{1/2,-L} \Gamma\varrho \bigr\|_1 \\
&\qquad \leq \bigl\| \varrho\Gamma(\fT_{1/2,-L})^* \bigr\|_2 \bigl\| \bigl(|\eta_2|^{-1} \bigl(\fQ_L^{(\eta_1; 0)} - \fQ_L^{(\eta_1; \eta_2)}\bigr) + \fQ_L^{(\eta_1; \star)}\bigr) \fT_{1/2,-L} \Gamma \varrho \bigr\|_2.
\end{align}
Lemma~\ref{lem:S_integral_estimate} guarantees that the kernels $\Gamma(\fT_{1/2,-L})^*$ and $\fT_{1/2,-L} \Gamma$ have fast decays at infinity. Moreover, Lemma~\ref{lem:Q_limits} implies that the kernel in the princesses vanishes uniformly as $\eta_2 \searrow 0$. Hence, from the dominated convergence theorem we conclude that the preceding expression vanishes as $\eta_2 \searrow 0$, and we have the limit in trace class
\begin{align}
\mylim{\eta_2 \searrow 0} \frac{1}{|\eta_2|} \Bigl( \fR_L^{(\eta_1; \eta_2)} - \fR_L^{(\eta_1; 0)}\Bigr) &= - \varrho\Gamma (\fT_{1/2,-L})^* \mylim{\eta_2 \searrow 0} \frac{1}{|\eta_2|} \Bigl(\fQ_L^{(\eta_1; \eta_2)} - \fQ_L^{(\eta_1; 0)}\Bigr) \fT_{1/2,-L} \Gamma \varrho \\
&= - \varrho\Gamma(\fT_{1/2,-L})^* \fQ_L^{(\eta_1; \star)} \fT_{1/2,-L} \Gamma\varrho.\label{eq:limit1}
\end{align}
By \eqref{eq:A1_formula}, this is equal to the operator $-\fA_L^{(\eta_1; \star)}$, which is trace class and rank-one.
This finishes the proof of \eqref{eq:deriv1}.

In the same way, we can prove existence of the limit \eqref{eq:deriv2}. We have the limit in trace class
\begin{equation}
\mylim{\eta_1 \searrow 0} \frac{1}{|\eta_1|} \Bigl(  \fR_L^{(\eta_1; 0)} -   \fR_L^{(0; 0)}\Bigr) = - \varrho \Gamma (\fT_{1/2,-L})^* \fQ_L^{(\star; 0)} \fT_{1/2,-L} \Gamma \varrho,
\end{equation}
which, according to \eqref{eq:A2_formula}, equals $-\fA_L^{(\star; 0)}$. Moreover, formulas \eqref{eq:A1_formula} and \eqref{eq:A3_formula}, the dominated convergence theorem and Lemma~\ref{lem:Q_limits} yield the limit in trace class
\begin{align}
\mylim{\eta_1 \searrow 0} \frac{1}{|\eta_1|} \Bigl( \fA_L^{(\eta_1; \star)} -  \fA_L^{(0; \star)}\Bigr) &= - \varrho \Gamma (\fT_{1/2,-L})^* \mylim{\eta_1 \searrow 0} \frac{1}{|\eta_1|} \Bigl( \fQ_L^{(\eta_1; \star)} - \fQ_L^{(0; \star)} \Bigr) \fT_{1/2,-L} \Gamma \varrho\\
&= - \varrho \Gamma (\fT_{1/2,-L})^* \fQ_L^{(\star; \star)} \fT_{1/2,-L} \Gamma \varrho,
\end{align}
which, according to \eqref{eq:deriv3}, equals $- \fA_L^{(\star; \star)}$.
\end{proof}

\subsubsection{Proof of the limit \eqref{eqs:densities}}
\label{ss:first_deriv}

To prove existence of the limit \eqref{eqs:densities}, we compute the two limits in \eqref{eqs:limitsF} in turn.

We start by computing the limit \eqref{eq:limitF1}. By \cite[Theorem~4.1]{fixedpt}, the kernel $\fK^{\hypo(\fh_0)}_{1/2}$ is trace class; and \eqref{eq:norms} and \eqref{eq:norms_products} imply that the limits in Lemma~\ref{lem:deriv1} hold in trace class also for the kernels $\fK^{\hypo(\fh_0)}_{1/2} \fR_L^{(\eta_1; \eta_2)}$. Then using definition \eqref{eq:deriv_eta},  identity \eqref{eq:Fredholm_deriv} and Lemma~\ref{lem:deriv1}, we obtain
\begin{equation}
\begin{split}
\MoveEqLeft
 \dens^{(\eta_1; \star)}_{1, M, L}(x_1, x_2) = \mylim{\eta_2 \searrow 0} \frac{1}{|\eta_2|} \dens^{(\eta_1; \eta_2)}_{1, M, L}(x_1, x_2) \\
 &= D_{\eta_2} \det\Bigl(I- \fK^{\hypo(\fh_0)}_{1/2} \fR_L^{(\eta_1; \eta_2)} \Bigr) \Bigr|_{\eta_2 = 0} - D_{\eta_2} \det \Bigl(I-\fK^{\hypo(\fh_0)}_{1/2} \fR_L^{(0;  \eta_2)}\Bigr) \Bigr|_{\eta_2 = 0} \\
&= \det \Bigl(I- \fK^{\hypo(\fh_0)}_{1/2} \fR_L^{(\eta_1; 0)} \Bigr) \tr\bigg[\Bigl(I - \fK^{\hypo(\fh_0)}_{1/2} \fR_L^{(\eta_1; 0)}\Bigr)^{-1} \fK^{\hypo(\fh_0)}_{1/2} \fA_L^{(\eta_1; \star)}\bigg]\\
&\quad - \det \Bigl(I- \fK^{\hypo(\fh_0)}_{1/2} \fR_L^{(0; 0)} \Bigr) \tr\bigg[\Bigl(I - \fK^{\hypo(\fh_0)}_{1/2} \fR_L^{(0; 0)}\Bigr)^{-1} \fK^{\hypo(\fh_0)}_{1/2} \fA_L^{(0; \star)}\bigg].
\end{split}
\label{e:det_d2}
\end{equation}
The uniformity of the limits over $(x_1, x_2) \in S$, $M \in \bar S$ and $\eta_1 \in (0,\frac{A}{2}]^2$ follows from the uniformity of the kernels.
Thus we have shown that the limit in \eqref{eq:limitF1} (and hence also \eqref{e:densities1}) exists in the desired uniform manner. Moreover, we have provided a formula for this limit.  The formula can be simplified further. In shorthand, denote the kernels
\begin{equation}\label{eq:kernels12}
\fG_L^{(\eta_1)} := \fR_L^{(\eta_1; 0)}, \qquad\qquad \fH_L^{(\eta_1)} := \fR_L^{(\eta_1; 0)} - \fA_L^{(\eta_1; \star)}.
\end{equation}
Since the operator $\fA_L^{(\eta_1; \star)}$, given in  \eqref{eq:deriv1}, is rank one, we can use identity \eqref{eq:Fredholm_rank1} to write the first term in \eqref{e:det_d2} as
\begin{align}
\MoveEqLeft[9]
\det \Bigl(I- \fK^{\hypo(\fh_0)}_{1/2} \fR_L^{(\eta_1; 0)} \Bigr) \tr\bigg[\Bigl(I - \fK^{\hypo(\fh_0)}_{1/2} \fR_L^{(\eta_1; 0)}\Bigr)^{-1} \fK^{\hypo(\fh_0)}_{1/2} \fA_L^{(\eta_1; \star)}\bigg]\\
&= \det\Bigl(I- \fK^{\hypo(\fh_0)}_{1/2} \fH_L^{(\eta_1)} \Bigr) - \det\Bigl(I- \fK^{\hypo(\fh_0)}_{1/2} \fG_L^{(\eta_1)}\Bigr),
\end{align}
while the second term is obtained by setting $\eta_1 = 0$. This yields
\begin{align}
 \dens^{(\eta_1; \star)}_{1, M, L}(x_1, x_2) &= \det\Bigl(I- \fK^{\hypo(\fh_0)}_{1/2} \fH_L^{(\eta_1)} \Bigr) - \det\Bigl(I- \fK^{\hypo(\fh_0)}_{1/2} \fG_L^{(\eta_1)}\Bigr) \\
&\qquad \qquad  - \det\Bigl(I- \fK^{\hypo(\fh_0)}_{1/2} \fH_L^{(0)} \Bigr) + \det\Bigl(I- \fK^{\hypo(\fh_0)}_{1/2} \fG_L^{(0)}\Bigr).\label{e:det_d2_final}
 \end{align}

We now compute the limit \eqref{eq:limitF2} (and hence also \eqref{e:densities}). Formula \eqref{e:det_d2_final}, definition \eqref{eq:deriv_eta} and the differentiation formula \eqref{eq:Fredholm_deriv} yield
\begin{equation}
\begin{split}
&\dens_{1, M, L}(x_1, x_2) = \mylim{\eta_1 \searrow 0} \frac{1}{|\eta_1|} \dens^{(\eta_1; \star)}_{1, M, L}(x_1, x_2) \\
&= \det\Bigl(I- \fK^{\hypo(\fh_0)}_{1/2} \fH_L^{(0)}\Bigr) \tr\bigg[\Bigl(I- \fK^{\hypo(\fh_0)}_{1/2} \fH_L^{(0)}\Bigr)^{-1}  \fK^{\hypo(\fh_0)}_{1/2}  \mylim{\eta_1 \searrow 0} \frac{1}{|\eta_1|} \bigl( \fH_L^{(\eta_1)} - \fH_L^{(0)}\bigr)\bigg]\\
&\quad - \det\Bigl(I-\fK^{\hypo(\fh_0)}_{1/2} \fG_L^{(0)}\Bigr) \tr\bigg[\Bigl(I- \fK^{\hypo(\fh_0)}_{1/2} \fG_L^{(0)}\Bigr)^{-1} \fK^{\hypo(\fh_0)}_{1/2} \mylim{\eta_1 \searrow 0} \frac{1}{|\eta_1|} \bigl( \fG_L^{(\eta_1)} - \fG_L^{(0)}\bigr)\bigg].
\end{split}
\label{e:det_d2_final2}
\end{equation}
We will here rewrite the limits of the kernels, as we did  in \eqref{e:det_d2_final}. Using Lemma~\ref{lem:deriv1} and formula \eqref{eq:Fredholm_rank1}, the expression in the last line of \eqref{e:det_d2_final2} equals
\begin{align}
 \MoveEqLeft[6]
 \det \Bigl(I-\fK^{\hypo(\fh_0)}_{1/2} \fG_L^{(0)}\Bigr) \tr\bigg[\Bigl(I-\fK^{\hypo(\fh_0)}_{1/2} \fG_L^{(0)}\Bigr)^{-1} \fK^{\hypo(\fh_0)}_{1/2} \fA_L^{(\star; 0)}\bigg]\\
&= \det \Bigl(I- \fK^{\hypo(\fh_0)}_{1/2} \Bigl(\fG_L^{(0)} - \fA_L^{(\star; 0)}\Bigr) \Bigr) - \det \Bigl(I-\fK^{\hypo(\fh_0)}_{1/2} \fG_L^{(0)}\Bigr),\label{e:det2_d1_final}
\end{align}
where we used that $\fA_L^{(\star; 0)}$ is rank-one.

Now we turn to the kernel $\fH_L^{(\eta_1)}$ in \eqref{e:det_d2_final2}, for which we have
\begin{align}
\mylim{\eta_1 \searrow 0} \frac{1}{|\eta_1|} \bigl( \fH_L^{(\eta_1)} - \fH_L^{(0)}\bigr) 
&= \mylim{\eta_1 \searrow 0} \frac{1}{|\eta_1|} \bigl( \fG_L^{(\eta_1)} - \fG_L^{(0)}\bigr) - \mylim{\eta_1 \searrow 0} \frac{1}{|\eta_1|} \Bigl( \fA_L^{(\eta_1; \star)} - \fA_L^{(0; \star)}\Bigr)\\
&= \fA_L^{(\star; 0)} - \fA_L^{(\star; \star)},
\label{eq:G1_deriv}
\end{align}
where we used definitions \eqref{eq:kernels12} and Lemma~\ref{lem:deriv1}. The operators $\fA_L^{(\star; 0)}$ and $\fA_L^{(\star; \star)}$ are rank one. Hence, plugging \eqref{eq:G1_deriv} into \eqref{e:det_d2_final2} and applying \eqref{eq:Fredholm_rank1}, the expression in the second line of \eqref{e:det_d2_final2} can be written as
\begin{align}
\det&\Big(I-\fK^{\hypo(\fh_0)}_{1/2} \fH_L^{(0)}\Bigr) \tr\bigg[\Bigl(I-\fK^{\hypo(\fh_0)}_{1/2} \fH_L^{(0)}\Bigr)^{-1} \fK^{\hypo(\fh_0)}_{1/2} \Bigl(\fA_L^{(\star; 0)} - \fA_L^{(\star; \star)}\Bigr)\bigg]\\
&= \det\Bigl(I- \fK^{\hypo(\fh_0)}_{1/2} \Bigl(\fH_L^{(0)} - \fA_L^{(\star; 0)} + \fA_L^{(\star; \star)}\Bigr) \Bigr) - \det\Bigl(I- \fK^{\hypo(\fh_0)}_{1/2} \fH_L^{(0)} \Bigr). \label{e:det1_d1_final}
\end{align}

Combining formulas \eqref{e:det_d2_final2}, \eqref{e:det2_d1_final} and \eqref{e:det1_d1_final}, we conclude that
\begin{align}
\dens_{1, M, L}(x_1, x_2) &= \det\Bigl(I- \fK^{\hypo(\fh_0)}_{1/2} \Bigl(\fH_L^{(0)} - \fA_L^{(\star; 0)} + \fA_L^{(\star; \star)}\Bigr) \Bigr) \\
&\quad - \det\Bigl(I- \fK^{\hypo(\fh_0)}_{1/2} \fH_L^{(0)} \Bigr) \\
&\quad- \det \Bigl(I- \fK^{\hypo(\fh_0)}_{1/2} \Bigl(\fG_L^{(0)} - \fA_L^{(\star; 0)}\Bigr) \Bigr) + \det \Bigl(I-\fK^{\hypo(\fh_0)}_{1/2}\fG_L^{(0)}\Bigr).
\end{align}
Recalling the definition of the kernels \eqref{eq:kernels12}, we can write the last expression as 
\begin{equation}
\begin{split}
\dens_{1, M, L}(x_1, x_2) &= \det\Bigl(I- \fK^{\hypo(\fh_0)}_{1/2} \Bigl(\fR_L^{(0; 0)} - \fA_L^{(0; \star)} - \fA_L^{(\star; 0)} + \fA_L^{(\star; \star)}\Bigr) \Bigr) \\
&\quad - \det\Bigl(I- \fK^{\hypo(\fh_0)}_{1/2} \Bigl(\fR_L^{(0; 0)} - \fA_L^{(0; \star)}\Bigr) \Bigr)\\
&\quad - \det \Bigl(I- \fK^{\hypo(\fh_0)}_{1/2} \Bigl(\fR_L^{(0; 0)} - \fA_L^{(\star; 0)}\Bigr) \Bigr) + \det \Bigl(I-\fK^{\hypo(\fh_0)}_{1/2}\fR_L^{(0; 0)}\Bigr).
\end{split}
\label{e:final}
\end{equation}
This exact formula for $\dens_{1, M, L}(x_1, x_2)$ will be useful next, in Section \ref{sec:properties_of_P}.

\subsubsection{Taking the limit $L \to \infty$}
\label{sec:L-to-infinity}

Now we will prove that the limit \eqref{eq:P_limit} exists. For this we need to show that all kernels in \eqref{e:final} converge in trace norm as $L \to \infty$. Convergence of the kernel \eqref{eq:fR} follows from \eqref{eq:Kd}, i.e., 
\begin{equation}\label{eq:R-limit}
\fR := \lim_{L \to \infty} \fR_L^{(0; 0)} = \fK^{\epi(-\fM)}_{1/2},
\end{equation}
for the constant function $\fg_{0, 0} \equiv \fM$ defined in \eqref{eq:g_def}. 

Due to the identities \eqref{eqs:A_defs}, in order to prove convergence of the kernels $\fA_L$ in trace norm, we need to show convergence of the functions \eqref{eqs:D_def} in $L^2$. In \eqref{eq:D1_new_formula} we derived the formula 
\begin{equation}
\fD^{(0; 0)}_{[\fx, L], +}(w) = \partial_{1} \bigl(\fT_{1/2,-\fx} \Gamma \varrho\bigr) (-\fM, w) - \partial_{1} \bigl( \SS_{1/2, 0, [\fx, L]}^{(0; 0)} \Gamma \varrho\bigr) (-\fM, w),
\end{equation}
and we need to show that the last term converges in $L^2$ as $L \to \infty$. Using \eqref{eq:last_term_formula}, we have 
\begin{equation}
\SS_{1/2, 0, [\fx, L]}^{(0; 0)} (v, w) = \int_\fx^{L} \P_{\fB(\fx)=v} \bigl( \tau_{-\fM} \in \d s\bigr) \,\cdot\, \bigl( \fT_{1/2, - s} \Gamma \varrho \bigr) \bigl(-\fM, w\bigr),
\end{equation}
where $\tau_{-\fM}$ is the first hitting time of the set $(-\infty, -\fM]$, as defined in \eqref{eq:tau-def}. From \cite[Proposition~3.7]{flat} we get 
\begin{equation}
\lim_{L \to \infty} \partial_{1} \SS_{1/2, 0, [\fx, L]}^{(0; 0)} (v, w) = - \partial_{1} \bigl( \fT_{1/2, - \fx} \Gamma \varrho \bigr) \bigl(-2\fM - v, w\bigr),
\end{equation}
and a fast decay of the functions allows to get this convergence in $L^2$. Then
\begin{equation}\label{eq:D-limit}
\fD_{x}(w) := \lim_{L \to \infty} \fD^{(0; 0)}_{[\fx, L], +}(w) = 2 \partial_{1} \bigl(\fT_{1/2,-\fx} \Gamma \varrho\bigr) (-\fM, w).
\end{equation}
Furthermore, from \eqref{eq:D-new} we get $\lim_{L \to \infty} \fD^{(0; 0)}_{[-L, \fx], -}(w) = \fD_{-x}(w)$ in $L^2$. Then \eqref{eqs:A_defs} yields the following limits in trace norm
\begin{equation}\label{eq:A-limits}
\begin{aligned}
\fA^{(0; \star)}(u_1, u_2) := \lim_{L \to \infty} \fA_L^{(0; \star)}(u_1, u_2) &= \fD_{-\fx_2}(u_1) \,\cdot\, \fD_{\fx_2}(u_2), \\
\fA^{(\star; 0)}(u_1, u_2) := \lim_{L \to \infty} \fA_L^{(\star; 0)}(u_1, u_2) &= \fD_{-\fx_1}(u_1) \,\cdot\, \fD_{\fx_1}(u_2), \\
\fA^{(\star; \star)}(u_1, u_2) := \lim_{L \to \infty} \fA_L^{(\star; \star)}(u_1, u_2) &= \frac{1}{\sqrt{4\pi}} (\fx_2 - \fx_1)^{-\frac{3}{2}} \, \cdot\, \fD_{-\fx_1}(u_1) \,\cdot\, \fD_{\fx_2}(u_2). 
\end{aligned}
\end{equation}
As before, these kernels depend on $x_i$ and $M$, but we prefer to suppress it from the notation. Using continuity of the Fredholm determinant, we get from \eqref{e:final} the limit \eqref{eq:P_limit} with the limiting function
\begin{equation}
\begin{split}
\dens_{1, M}(x_1, x_2) &= \det\Bigl(I- \fK^{\hypo(\fh_0)}_{1/2} \Bigl(\fR - \fA^{(0; \star)} - \fA^{(\star; 0)} + \fA^{(\star; \star)}\Bigr) \Bigr) \\
&\quad - \det\Bigl(I- \fK^{\hypo(\fh_0)}_{1/2} \Bigl(\fR - \fA^{(0; \star)}\Bigr) \Bigr)\\
&\quad - \det \Bigl(I- \fK^{\hypo(\fh_0)}_{1/2} \Bigl(\fR - \fA^{(\star; 0)}\Bigr) \Bigr) + \det \Bigl(I-\fK^{\hypo(\fh_0)}_{1/2}\fR\Bigr).
\end{split}
\label{e:final-infinity}
\end{equation}

\subsection{Properties of the function $\dens_{t, \fM, L}$}
\label{sec:properties_of_P}

We now prove the properties of the function $\dens_{t, \fM, L}$, which are stated in Proposition~\ref{prop:densities}.

\subsubsection{Continuity}

Dependence of the function $\dens_{t, \fM, L}$ on $x_1$, $x_2$ and $M$ comes through the kernels $\fA_L^{(0; \star)}$, $\fA_L^{(\star; 0)}$ and $\fA_L^{(\star; \star)}$ in formula \eqref{e:final}. These kernels are rank-one, and are defined in \eqref{eqs:A_defs} via the two functions \eqref{eqs:D_def}. Furthermore, these two functions depend on these variables only through the kernels $\fP_{[a, b]}^{\nohit\,(-\fg_{\eta_1, \eta_2})}$ and their first-order derivatives. This non-hitting density of Brownian motion depends continuously on $x_1$, $x_2$ and $M$, from which we can conclude that the kernels in \eqref{e:final} are continuous as well. Then Lemma~\ref{lem:det}(\ref{it:det_continuous}) implies continuity of the Fredholm determinants in \eqref{e:final} and hence of the function $\dens_{t, M, L}(x_1, x_2)$.

\subsubsection{Scaling property}
\label{ss:scaling_of_P}

From identity \eqref{eq:density2}, we conclude
\begin{equation}
\dens_{t, \fM, L}^{(\eta_1; \eta_2)}(x_1, x_2) = \dens_{1, t^{-1/3} \fM, t^{-2/3} L}^{(\eta^{(t)}_1, \eta^{(t)}_2)} \bigl(t^{-\frac{2}{3}} x_1, t^{-\frac{2}{3}} x_2\bigr),
\end{equation}
where $\eta^{(t)}_i = (t^{-1/3} \eps_i, t^{-2/3} \delta_i)$. Hence, from \eqref{e:densities}, we have
\begin{align}
\dens_{t, \fM, L}(x_1, x_2) &= t^{-2} \mylim{\eta_1 \searrow 0}\, \mylim{\eta_2 \searrow 0} \frac{1}{|\eta^{(t)}_1| |\eta^{(t)}_2|} \dens_{1, t^{-1/3} \fM, t^{-2/3} L}^{(\eta^{(t)}_1, \eta^{(t)}_2)} \bigl(t^{-\frac{2}{3}} x_1, t^{-\frac{2}{3}} x_2\bigr) \\
&= t^{-2} \dens_{1, t^{-1/3} \fM, t^{-2/3} L} \bigl(t^{-\frac{2}{3}} x_1, t^{-\frac{2}{3}} x_2\bigr),
\end{align}
which is exactly \eqref{e:densities_scale}.

\subsubsection{Boundedness}
\label{ss:bound_on_P}


Now, we will prove the bound \eqref{eq:P_bound}. For this, we write \eqref{e:final-infinity} as
\begin{equation}
\dens_{1, \fM, L} = \dens^{(1)}_{1, \fM, L} + \dens^{(2)}_{1, \fM, L} + \dens^{(3)}_{1, \fM, L},
\end{equation}
where
\begin{subequations}\label{eq:P_definitions}
\begin{align}
\dens^{(1)}_{1, \fM, L}(\fx_1, \fx_2) :=& \det\Bigl(I- \fK^{\hypo(\fh_0)}_{1/2} \Bigl(\fR^{(0; 0)}_L - \fA^{(0; \star)}_L - \fA^{(\star; 0)}_L + \fA^{(\star; \star)}_L\Bigr) \Bigr) \\
&- \det\Bigl(I- \fK^{\hypo(\fh_0)}_{1/2} \Bigl(\fR^{(0; 0)}_L - \fA^{(0; \star)}_L - \fA^{(\star; 0)}_L \Bigr) \Bigr),\label{eq:P1_def}\\
\dens^{(2)}_{1, \fM, L}(\fx_1, \fx_2) :=& \det\Bigl(I- \fK^{\hypo(\fh_0)}_{1/2} \Bigl(\fR^{(0; 0)}_L - \fA^{(0; \star)}_L - \fA^{(\star; 0)}_L\Bigr) \Bigr) \\
& - \det\Bigl(I- \fK^{\hypo(\fh_0)}_{1/2} \Bigl(\fR^{(0; 0)}_L - \fA^{(0; \star)}_L \Bigr) \Bigr)\label{eq:P2_def}\\
\dens^{(3)}_{1, \fM, L}(\fx_1, \fx_2) :=& \det \Bigl(I-\fK^{\hypo(\fh_0)}_{1/2} \fR^{(0; 0)}_L\Bigr) \\
&- \det \Bigl(I- \fK^{\hypo(\fh_0)}_{1/2} \Bigl(\fR^{(0; 0)}_L - \fA^{(\star; 0)}_L\Bigr) \Bigr). \label{eq:P3_def}
\end{align}
\end{subequations}
Then \eqref{eq:P_bound} will follow immediately, if we prove that
\begin{subequations}\label{eq:P_bounds_new}
\begin{align}
\dens^{(1)}_{1, \fM, L}(\fx_1, \fx_2) &\leq C |\fx_2 - \fx_1|^{-\frac{3}{2}},\label{eq:P1_bound}\\
\dens^{(2)}_{1, \fM, L}(\fx_1, \fx_2) \leq C,& \qquad \dens^{(3)}_{1, \fM, L}(\fx_1, \fx_2) \leq C,\label{eq:P2_bound}
\end{align}
\end{subequations}
uniformly in $\fx_1, \fx_2 \in \bJ_L$, $\fx_1 < \fx_2$ and $|\fM| \leq \bar L$, where $C\geq 0$ depends only on $L$, $\bar L$ and $\alpha$.
\smallskip

Thus, it remains to prove \eqref{eq:P_bounds_new}. Let us recall the definitions of the kernels involved in \eqref{eq:P_definitions} and show that their trace norms are bounded. The kernel $\fK^{\hypo(\fh_0)}_{1/2}$ is defined in \eqref{eq:Kd}, from which we conclude that $\bigl\| \fK^{\hypo(\fh_0)}_{1/2} \bigr\|_1 \leq C_1$, for a finite constant $C_1 \geq 0$ which depends on the initial data $\fh_0$ merely in terms of its maximal value (see the comments after \eqref{eq:Kd}). The kernel $\fR^{(0; 0)}_L$ is defined in \eqref{eq:fR}. Since the prelimiting kernel in \eqref{eq:Kd} is trace class, the definitions \eqref{eq:fR} and \eqref{eq:Kd_epi} imply that $\fR^{(0; 0)}_L$ is trace class. The kernels $\fA^{(\star; 0)}_L$,  $\fA^{(0; \star)}_L$ and $\fA^{(\star; \star)}_L$ are defined in \eqref{eqs:A_defs}, and Lemma~\ref{lem:A_formulas} proves that they are trace class.

Applying the bounds \eqref{eq:det_bound_trace}, \eqref{eq:norms} and  \eqref{eq:norms_products} to the determinants in \eqref{eq:P1_def}, and using boundedness of the trace norms of the involved operators, we obtain
\begin{align}\label{eq:P1_bound_intermediate}
\dens^{(1)}_{1, \fM, L}(\fx_1, \fx_2) \leq C_2 \bigl\| \fA_L^{(\star; \star)} \bigr\|_1,
\end{align}
for some constant $C_2 \geq 0$. Using the bound \eqref{eq:A_bounds} in \eqref{eq:P1_bound_intermediate}, we obtain \eqref{eq:P1_bound}, as required. 

Consider now the functions $\dens^{(2)}_{1, \fM, L}$ and $\dens^{(3)}_{1, \fM, L}$.  We apply \eqref{eq:det_bound_trace}, \eqref{eq:norms} and \eqref{eq:norms_products} to the two pairs of determinants in \eqref{eq:P2_def} and \eqref{eq:P3_def}. For this, we use boundedness of the trace norms of the involved kernels, which we stated above. Then we find that
\begin{equation}\label{eq:P2_bound_intermediate1}
\dens^{(2)}_{1, \fM, L}(\fx_1, \fx_2) \leq C_4 \bigl\| \fA^{(\star; 0)}_L \bigr\|_1, \qquad \dens^{(3)}_{1, \fM, L}(\fx_1, \fx_2) \leq C_4 \bigl\| \fA^{(0; \star)}_L \bigr\|_1.
\end{equation}
Applying \eqref{eq:A_bounds} to \eqref{eq:P2_bound_intermediate1}, we obtain the required bounds \eqref{eq:P2_bound}.

\subsection{Upper bound on the twin peaks probability}\label{sec:uppbd}

Using Proposition~\ref{prop:densities}, we can prove the bound in Theorem~\ref{thm:densities2-intro} on the twin peaks probability. 

%
%

\begin{proof}[Proof of Theorem~\ref{thm:densities2-intro}]
We may write the probability on the left-hand side of \eqref{eq:q_def} for the cut-off as
\begin{align}\label{eq:q_equality}
\P_{\fh_0} \Bigl(\fh^L_t \in \TP{A, \bar L}^{\eps}\Bigr) = \sum_{M \in \eps \zz} \P_{\fh_0} \Bigl(\fh^L_t \in \TP{A, \bar L}^{\eps},\, \Max(\fh^L_t) \in (M-\eps, M]\Bigr).
\end{align}
Let us define the $\eps$-fattening of $\bJ_{\bar L}$ by $\bJ_{\bar L}^{\eps} := \bigcup_{x \in \bJ_{\bar L}} [x - \eps, x + \eps]$. From the definition of the set $\TP{A, \bar L}^{\eps}$, we conclude that, if $\bJ_{\bar L^{1/2}}^{\eps} \cap (M-\eps, M] = \emptyset$, then the term indexed by $M$ in the sum \eqref{eq:q_equality} vanishes. When this term does not vanish, we can estimate it by
\begin{align}
&\P_{\fh_0} \Bigl(\fh^L_t \in \TP{A, \bar L}^{\eps},\, \Max(\fh^L_t) \in (M-\eps, M]\Bigr) \\
&\quad \leq \P_{\fh_0} \Bigl(\exists\, x_1, x_2 \in \bJ_{L} : x_2 - x_1 \geq A,\; \fh_{t}(x_i) > M - 2 \eps,\; \fh_{t}(y) \leq M ~\text{for}~ y \in [- L,  L]\Bigr).
\end{align}
For $0 < \delta \leq A$, denote $\overline \eta = (2 \eps, \delta)$, and let $A_\delta$ be the largest value in $\delta \zz$ such that $0 < A_\delta \leq A$. Then, using the function $\dens_{t, M, L}^{(\eta_1; \eta_2)}(x_1, x_2)$, defined in \eqref{eq:density}, we have the upper bound
\begin{equation}\label{eq:bound_by_sums}
\P_{\fh_0} \Bigl(\fh^L_t \in \TP{A, \bar L}^{\eps}\Bigr) \leq \sum_{\substack{x_1, x_2 \in (\delta \zz) \cap \bJ_{ L}^{\delta} \\ x_2 - x_1 \geq A_\delta}} \sum_{M \in (\eps \zz) \cap \bJ_{\bar L^{1/2}}^{\eps}} \dens^{(\overline\eta; \overline\eta)}_{t, M, L}(x_1, x_2).
\end{equation}
For $0 < \delta_1 \leq \delta_2$ and $0 < \eps_1 \leq \eps_2$, let us set $\overline \eta_1 = (2 \eps_1, \delta_1)$ and $\overline \eta_2 = (2 \eps_2, \delta_2)$. Moreover, we define
\begin{equation}
	\CI^{(\overline\eta_1; \overline\eta_2)} := \sum_{\substack{x_1, x_2 \in (\delta_1 \zz) \cap \bJ_{ L}^{\delta_1} \\ x_2 - x_1 \geq A_{\delta_1}}} \sum_{M \in (\eps_1 \zz) \cap \bJ_{\bar L^{1/2}}^{\eps_1}} \dens^{(\overline\eta_1; \overline\eta_2)}_{t, M, L}(x_1, x_2),
\end{equation}
where we prefer to suppress dependence on the variables $t$, $A$, etc. Then \eqref{eq:bound_by_sums} reads as
\begin{equation}
\P_{\fh_0} \Bigl(\fh^L_t \in \TP{A, \bar L}^{\eps}\Bigr) \leq \CI^{(\overline\eta; \overline\eta)}.
\end{equation}

From \eqref{e:densities1}, we conclude that for every fixed $\bar \eta_1 > 0$ the following limit holds:
\begin{equation}\label{eq:limit_to_integrals1}
\mylim{\bar \eta_2 \searrow 0} \frac{1}{|\bar \eta_2|} \CI^{(\bar \eta_1; \bar \eta_2)} = \sum_{\substack{x_1, x_2 \in (\delta_1 \zz) \cap \bJ_{ L}^{\delta_1} \\ x_2 - x_1 \geq A_{\delta_1}}} \sum_{M \in (\eps_1 \zz) \cap \bJ_{\bar L^{1/2}}^{\eps_1}} \dens^{(\overline\eta_1; \star)}_{t, M, L}(x_1, x_2)=:\CI^{(\bar \eta_1; \star)}.
\end{equation}
Since the limit \eqref{e:densities1} holds uniformly over $\eta_1$ and since the function $\dens^{(\overline\eta; \overline\eta)}_{t, M, L}(x_1, x_2)$ is positive, we conclude that there exist constants $\chi_1 > 0$ and $C_1 > 0$, independent of $\bar \eta_1$ and $\bar \eta_2$, such that
\begin{equation}\label{eq:I1_bound}
\CI^{(\bar \eta_1; \bar \eta_2)} \leq C_1 |\bar \eta_2|\, \CI^{(\bar \eta_1; \star)},
\end{equation}
for all $\bar \eta_1, \bar \eta_2 \in (0, \chi_1]^2$ and $\bar \eta_2 \in \Dom$ (this domain is defined in \eqref{eq:Dom}). Now we claim that
\begin{equation}\label{eq:limit_to_integrals2}
\mylim{\bar \eta_1 \searrow 0} \eps_1^{-1} |\bar \eta_1| \CI^{(\bar \eta_1; \star)} = C_2 \int_{-\beta L}^{\beta L - A} \d x_1 \int_{x_1 + A}^{\beta L} \d x_2 \int_{-\beta \sqrt{\bar L}}^{\beta \sqrt{\bar L}} \d M\; \dens_{t, M, L}(x_1, x_2)=:\CI,
\end{equation}
for a constant $C_2 > 0$ which we define later, and where the constant $\beta$ arises from the definition of $\bJ_{ L}$. We note that continuity of the function $\dens_{t, M, L}$, proved in Proposition~\ref{prop:densities}, guarantees that the latter triple integral is well-defined. We postpone our proof of this limit for the moment, and show how it implies the bound \eqref{eq:q_def}.

From \eqref{eq:limit_to_integrals2}, we conclude that there exist constants $\chi_3 > 0$ and $C_3 > 0$ such that
\begin{equation}\label{eq:I2_bound}
\CI^{(\bar \eta_1; \star)} \leq C_3 \frac{\eps_1}{|\bar \eta_1|} \CI,
\end{equation}
for all $\bar \eta_1 \in (0,\chi_3]^2 \cap \Dom$. Hence, from  \eqref{eq:bound_by_sums}, and from the two bounds \eqref{eq:I1_bound} and \eqref{eq:I2_bound}, we get
\begin{equation}\label{eq:C4eCI}
	\P_{\fh_0} \Bigl(\fh^L_t \in \TP{A, \bar L}^{\eps}\Bigr) \leq C_4 \eps \CI
\end{equation}
where $C_4 = C_1 C_2 C_3$. Using the scaling property \eqref{e:densities_scale}, changing the variables of integration, denoting $\bar L_t := t^{-2/3} \bar L$, $L_t := t^{-2/3} L$ and $A_t := t^{-2/3} A$, and using the bound in \eqref{eq:P_bound} on $\dens_{1, M, L_t} ( x_1, x_2)$, we may write the preceding expression explicitly as
\begin{equation}
C_4 \eps \CI = C_5 \eps t^{-\frac{1}{3}} \int_{-\beta  L_t}^{\beta L_t - A_t} \d x_1 \int_{x_1 + A_t}^{\beta  L_t} \d x_2 \int_{-\beta \sqrt{\bar L_t}}^{\beta \sqrt{\bar L_t}} \d M\; |x_2 - x_1|^{-\frac{3}{2}}
\end{equation}
where the constants $C_5$ and $c$ depend on $L_{T_0}$, $\bar L_{T_0}$ and $\alpha$ (the constant $\alpha$ arises from the upper bound on the initial data $\fh_0$). In Theorem \ref{thm:densities2-intro}, we have assumed that $t\in [T_0,T]$. For this range of~$t$ and given the condition that $x_2-x_1\geq A_t\leq A_{T_0}$ from the integral above, we can bound above $\exp\big\{c |x_2 - x_1|^{-\frac{3}{2}}\big\} \leq \exp \big\{ c A_{T_0}^{-3/2} \big\}$. Inserting this bound and evaluating the integrals yields another constant depending just on $\beta, T_0, L, \bar{L}$ and $A$. Thus we have shown that
$C_4 \eps \CI \leq C_6 \eps t^{-1 / 3}$. In light of \eqref{eq:C4eCI}, this yields~\eqref{eq:q_def} and completes the proof.

\smallskip

All that remains is to prove the limit \eqref{eq:limit_to_integrals2}. For this, let us define the finite sets
\begin{equation}
S_\delta := \left\{(x_1, x_2) : x_i \in (\delta \zz) \cap \bJ_{ L}^{\delta},\; x_2 - x_1 \geq A_\delta\right\}, \qquad \bar S_\eps := (\eps \zz) \cap \bJ_{\bar L^{1/2}}^{\eps},
\end{equation}
 over which the summation in \eqref{eq:bound_by_sums} is performed. Recalling that $\overline \eta_1 = (2 \eps_1, \delta_1)$, we may write
\begin{align}
\eps_1^{-1} |\bar \eta_1| \CI^{(\bar \eta_1; \star)} - \CI &= \eps_1^{-1} |\bar \eta_1|^2 \sum_{\substack{(x_1, x_2) \in S_{\delta_1} \\ M \in \bar S_{\eps_1}}} \left( \frac{1}{|\overline \eta_1|} \dens^{(\overline\eta_1; \star)}_{t, M, L}(x_1, x_2) - \dens_{t, M, L}(x_1, x_2) \right)\\
&\hspace{3cm} + \biggl( \eps_1^{-1} |\bar \eta_1|^2 \sum_{\substack{(x_1, x_2) \in S_{\delta_1} \\ M \in \bar S_{\eps_1}}} \dens_{t, M, L}(x_1, x_2) - \CI \biggr).\label{eq:limit_to_integrals_proof}
\end{align}
Call the first term on the right-hand side $\CI_1$ and the second one $\CI_2$.
The cardinalities $|S_{\delta_1}|$ and $|\bar S_{\eps_1}|$ are respectively proportional to $\delta_1^{-2}$ and $\eps_1^{-1}$, and we can define the finite constants $\lim_{\delta_1 \searrow 0} \delta_1^2 |S_{\delta_1}| = C_7$ and $\lim_{\eps_1 \searrow 0} \eps_1 |\bar{S}_{\eps_1}| = C_8$, where $C_7$ depends on $L$, $A$ and $\beta$, and $C_8$ depends on $\bar L$ and $\beta$. Moreover, we define the sets
\begin{equation}
S := \bigcap_{\delta_1 > 0} S_{\delta_1} = \bigl\{(x_1, x_2) : x_i \in \bJ_{L},\; x_2 - x_1 \geq A\bigr\}, \qquad \bar S := \bigcap_{\eps_1 > 0} \bar{S}_{\eps_1} = \bJ_{\bar L^{1/2}}.
\end{equation}
Proposition~\ref{prop:densities} implies continuity of the function $\dens_{t, M, L}(x_1, x_2)$ for $(x_1, x_2) \in S$ and $M \in \bar S$. Hence, if we take the constant $C_2 = \lim_{\eps_1, \delta_1 \searrow 0} \eps_1^{-1} |\bar \eta_1|^2 |S_{\delta_1}| |\bar{S}_{\eps_1}| = 4 C_7 C_8$ in \eqref{eq:limit_to_integrals2}, then $|\CI_2|\to 0$ in the limit $\eps_1, \delta_1 \searrow 0$, in view of the definition of a multiple integral as a limit of Riemann sums.

To deal with the other term, observe that
\begin{equation}\label{eq:limit_to_integrals_proof2}
|\CI_2|\leq \eps_1^{-1} |\bar \eta_1|^2 |S_{\delta_1}| |\bar S_{\eps_1}| \sup_{\substack{(x_1, x_2) \in S_{\delta_1} \\ M \in \bar S_{\eps_1}}} \left| \frac{1}{|\overline \eta_1|} \dens^{(\overline\eta_1; \star)}_{t, M, L}(x_1, x_2) - \dens_{t, M, L}(x_1, x_2) \right|.
\end{equation}
The multiplier $\eps_1^{-1} |\bar \eta_1|^2 |S_{\delta_1}| |\bar S_{\eps_1}|$ converges to $C_2$ as $\eps_1, \delta_1 \searrow 0$, and Proposition~\ref{prop:densities} implies that the expression \eqref{eq:limit_to_integrals_proof2} vanishes as $\overline \eta_1 \searrow 0$ along the domain $\Dom$. This finishes the proof of \eqref{eq:limit_to_integrals2}.
\end{proof}

With this, we may close out the proof of Theorem~\ref{thm:johansson}.

\begin{proof}[Proof of Theorem~\ref{thm:johansson}]
Lemma~\ref{lem:KPZ_is_bounded} implies that, for any $t > 0$, the KPZ fixed point $\fh_t$ almost surely has a maximizer. Then uniqueness of the maximizer follows immediately from Theorem~\ref{thm:densities2-intro} by first taking the limit $\varepsilon\to 0$, then $\bar L \to \infty$ and then $A \to 0$.
\end{proof}



\renewcommand{\red}[1]{{\color{red} #1}}

\section{The lower bound on the twin peaks probability}
\label{sec:TP_probability}




In this section, we prove Theorem~\ref{thm:TP_probability}: at any positive time, the KPZ fixed point lies in the set $\TP{A,L}^{\eps}$ defined in \eqref{eq:setTP} with probability at least a constant multiple of $\varepsilon$.


Theorem~\ref{thm:TP_probability} offers a lower bound on the probability of the twin peaks' event. The general~$t$ version of this result follows from the $t=1$ case in light of the 1:2:3 invariance of the KPZ fixed point recalled in Proposition~\ref{prop:sym}(\ref{123a}). Thus in this section we set $t=1$ and, for notational convenience, we will denote $\hlim_1$ by $\hlim$; the claim in Theorem~\ref{thm:TP_probability} that $\smallconst$ and $L_0$ can be taken to depend on $t$ continuously will be handled separately. Also, we prefer to suppress the initial state $\fh_0$ from notation for the probability measure.

Recall that Assumption~\ref{a:initial_state_pd} says that there exist $\alpha$, $\gamma$, and $\lambda$ such that  $\fh_0(y)$ is $-\infty$ for all $y< -\lambda$ and satisfies $\fh_0(y) \leq \alpha -\gamma |y|^{1/2}$ for all $y\in\R$. In this section, we assume without loss of generality that $\alpha = \lambda = 0$. We may do so because the occurrence of the twin peaks' event $\TP{A,L}^{\eps}$ is unchanged, provided that $L$ is altered suitably, under arbitrary shifts of $\fh_0$ in the horizontal and vertical directions; i.e., under transformations of the form $\fh_0(\bigcdot) \mapsto \fh_0(\bigcdot+z_1) + z_2$ for fixed $z_1,z_2\in\R$.
So we consider the function $\tilde\fh_0(y) := \fh_0(y-\lambda) - (\alpha+\gamma |\lambda|^{1/2})$, which satisfies $\tilde\fh_0(y) = -\infty$ for all $y<0$. For this function, we have that
$$\tilde\fh_0(y) \leq \alpha -(\alpha + \gamma |\lambda|^{1/2}) - \gamma |y-\lambda|^{1/2} \leq -\gamma|y|^{1/2}, \textrm{ for all }y\in \R,$$
by noting that, for $a,b>0$, $a^{1/2}+b^{1/2} \geq (a+b)^{1/2}$ and that $|y-\lambda|+|\lambda| \geq |y|$ by the triangle inequality. 
 Thus, for $\fh_0$  satisfying Assumption~\ref{a:initial_state_pd}, we may indeed assume that $\alpha = \lambda = 0$, i.e.,
 \begin{equation}\label{eq:h0bound}
 \fh_0(y) \leq -\gamma |y|^{1/2} \textrm{ for all }y\in \R \text{ and } \fh_0(y) = -\infty \text{ for } y < 0;
 \end{equation}
of course, as mentioned previously, making this simplification may need a modification in the value of $L$ for which we consider the event $\TP{A,L}^\eps$. More precisely, we have that $\fh_t \in \TP{A,L}^\eps \iff \tilde\fh_t \in \TP{A,\smash{\tilde L}}^\eps$ where $\smash{\tilde\fh_t}$ is obtained from $\fh_t$ by applying the same horizontal and vertical shift that was applied above to $\fh_0$ to give $\tilde \fh_0$ (so $\tilde\fh_t$ has the distribution of the KPZ fixed point started from initial condition $\tilde\fh_0$ at time $t$) and $\tilde L$ is a function of $L$, $\alpha$, $\gamma$ and $\lambda$. This simplification will aid us in some later technical arguments.

Many of the estimates made in this section will depend on $\fh_0$, and it will be helpful to be precise about which aspects of $\fh_0$ are relevant. Thus, in this section the parameter $\theta$ will be such that
$$\sup_{|y|\leq \theta}\fh_0(y) \geq -\theta.$$
Estimates will depend on $\fh_0$ only through $\gamma$ and $\theta$; this is because we assume as above that $\alpha = \lambda = 0$, but note that the transformation described to make this simplification modifies the value of~$\theta$.

In contrast to the paper at large, we work here mainly with the prelimiting model of Brownian last passage percolation. By doing so, we gain access to certain important tools: these include the Brownian description given by the distributional identification of the ``melon'' function of Brownian LPP (to be introduced shortly) with Dyson Brownian motion; and a remarkable deterministic identity (Proposition \ref{p.melon lpp identity}) relating last passage values in the Brownian environment with those in the melon-transformed environment. (A few words on this topic have been offered in explaining Item \ref{heur:3} in Section \ref{sec:heuristic}.) We work with the prelimiting model since the key identity (given as\eqref{e.prelimiting gen profile formula} ahead) is easier to work with there than at the level of the KPZ fixed point. (More precisely, while one can define $\fh_1$ via an LPP problem through the infinite parabolic Airy line ensemble directly as was done in \cite{sarkar2020brownian}, this approach would be more technically involved for our argument---for example in making sense of the weight of an infinite path, and since the Markov property holds for Dyson Brownian motion in the prelimit in a technically simpler form compared to the full infinite line ensemble. For such reasons we have adopted the approach through the prelimit.) Once we establish a prelimiting version of Theorem \ref{thm:TP_probability} (Proposition \ref{p.twin peaks prelimit} below), we use the convergence of Brownian LPP to the KPZ fixed point (Lemma \ref{l.convergence to fixed point}) to deduce the theorem.
Below, in Figure~\ref{fig:outline_lower}, is a diagrammatic representation of the structure of this section.

{
  \footnotesize
  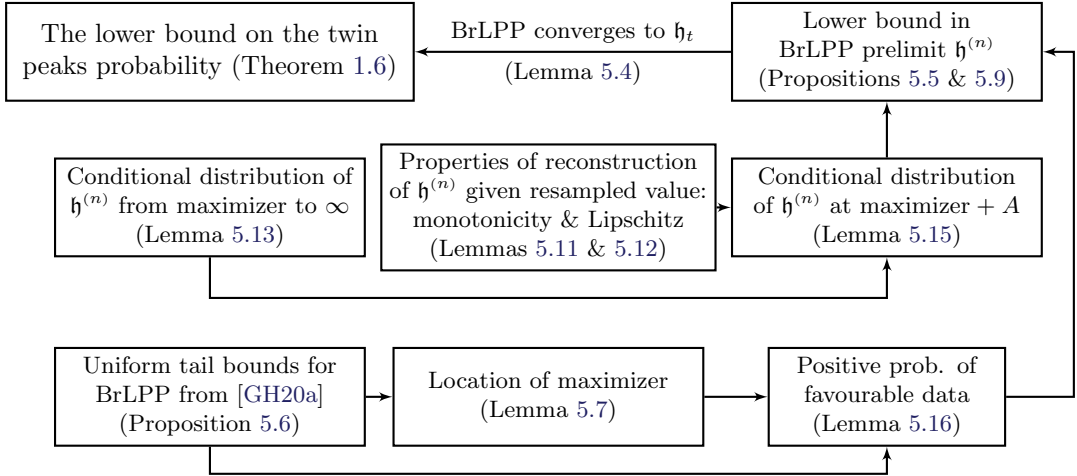
\begin{figure}[h]
  \begin{tikzpicture}[scale=0.9, auto,
  	block_main/.style = {rectangle, draw=black, thick, fill=white, text width=16em, text centered, minimum height=4em, font=\small},
	block_med/.style = {rectangle, draw=black, fill=white, thick, text width=12em, text centered, minimum height=4em},
	block_large/.style = {rectangle, draw=black, fill=white, thick, text width=13em, text centered, minimum height=4em},
	block_small/.style = {rectangle, draw=black, fill=white, thick, text width=9em, text centered, minimum height=4em},
        line/.style ={draw, thick, -latex', shorten >=0pt}]
      	
     \node [block_main] (lower_bound) at (-4,0) {The lower bound on the twin peaks probability (Theorem~\ref{thm:TP_probability})};
	
	\node [block_med] (prelimit) at (6,0) {Lower bound in BrLPP prelimit $\h$\\(Propositions~\ref{p.twin peaks prelimit} \& \ref{p.resampling twin peaks})};
	
	\node [block_med] (tail_bounds) at (-4,-5.1) {Uniform tail bounds for BrLPP from \cite{LPPtools}\\(Proposition~\ref{l.BrLPP uniform tail bound})};

	\node [block_med] (max_loc) at (1,-5.1) {Location of maximizer\\(Lemma~\ref{l.tightness of prelimiting maximizer})};
	
	\node [block_small] (fav_data) at (6,-5.1) {Positive prob. of favourable data\\(Lemma~\ref{l.fav probability})};

	\node [block_med] (Z_dist) at (6,-2.3) {Conditional distribution of $\h$ at $\text{maximizer} + A$\\(Lemma~\ref{l.conditional distribution})};

	\node [block_large] (reconstruction_prop) at (1,-2.3) {Properties of reconstruction of $\h$ given resampled value: monotonicity \& Lipschitz\\(Lemmas~\ref{l.monotonicity of L^f} \& \ref{l.lipschitz})};

	\node [block_med] (cond_ray) at (-4,-2.3) {Conditional distribution of $\h$ from maximizer to $\infty$\\(Lemma~\ref{l.first conditional statement})};

    \begin{scope}[every path/.style=line]
		\path (prelimit) -- node[above] {BrLPP converges to $\fh_t$} node[below] {(Lemma~\ref{l.convergence to fixed point})} (lower_bound);
		\path (fav_data.east) -- ++(1,0) |- (prelimit);
		\path (Z_dist.north) -- (Z_dist.north |- prelimit.south);
		\path (tail_bounds) -- (max_loc);
		\path (max_loc) -- (fav_data);
		\path (tail_bounds.south) -- ++(0,-0.4) -| (fav_data);
		\path (reconstruction_prop) -- (Z_dist);
		\path (cond_ray.south) -- ++(0,-0.6) -| (Z_dist);
    \end{scope}
  \end{tikzpicture}
  \captionof{figure}{Structure of this section.}
  \label{fig:outline_lower}
  \end{figure}
}

\subsection{Preliminaries}
\subsubsection{The model}
We denote the integer interval $\{1, \ldots, n\}$ by $\intint{n}$. Consider a sequence of continuous functions $f=(f_1, \ldots, f_n) : \intint{n}\times [0,\infty)\to \R$. We will depict these functions as in Figure~\ref{f.BrlPP}. The functions $f_1$ through $f_n$ are each indexed by a spatial variable which lies respectively along one of $n$ horizontal \emph{lines}, with the top line indexing $f_1$ and the bottom line indexing $f_n$. The function values along these lines represent an \emph{environment}.

Let $0\leq y\leq x$. An upright path $\gamma$ from $(y,n)$ (i.e., position $y$ on line $n$) to $(x,1)$ is a path which starts at $(y,n)$ and moves rightwards, jumping up from one line  to the next at certain times until it reaches $(x,1)$: see Figure~\ref{f.BrlPP}. An upright path is parametrized by its \emph{jump times} $\{t_i\}_{i=1}^{n-1}$ at which it jumps from the $(i+1)$\textsuperscript{st} line to the $i$\textsuperscript{th} line. Define $\Pi^n_{y,x}$ to be the set of upright paths from $(y,n)$ to $(x,1)$. The \emph{weight} of $\gamma \in \Pi^n_{y,x}$ in $f$ is denoted $f[\gamma]$ and defined by
\begin{equation}\label{e.path weight definition}
f[\gamma] = \sum_{i=1}^{n-1} \Big(f_i(t_{i-1}) - f_i(t_{i})\Big),
\end{equation}
where $\{t_i\}_{i=1}^{n-1}$ are the jump times of $\gamma$, with $t_n=y$ and $t_0=x$. This expression is thus the sum of increments of $f$ along the portions of $\gamma$ on each line.
We define the last passage value in $f$ from $(y,n)$ to $(x,1)$ by
\begin{equation}\label{e.lpp definition}
f[(y,n)\to (x,1)] = \sup_{\gamma\in \Pi^n_{y,x}} f[\gamma].
\end{equation}
If the set $\Pi^n_{y,x}$ is empty, i.e., if $y>x$, we define the passage value to be $-\infty$. The model of Brownian LPP is specified by taking $f$ to be a collection of $n$ independent standard Brownian motions defined on $[0, \infty)$.

\begin{figure}
\includegraphics{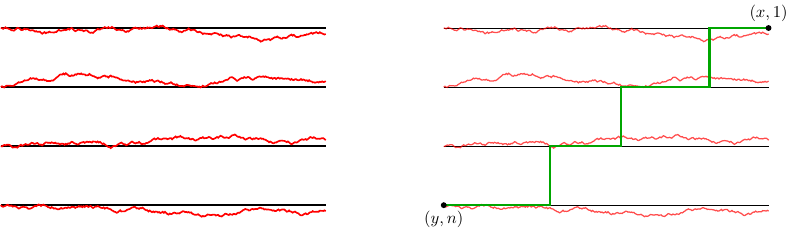}
\caption{Left: A depiction of the environment given by $f$. The functions $f_i$ corresponding to each line are graphed in red on the corresponding black line for visual clarity; the function values themselves need not be ordered. Right: An upright path $\gamma$ from $(y,n)$ to $(x,1)$ is depicted in green; note that in a formal sense the depicted vertical portions are not part of the path. The path's weight is the sum of the increments of $f_i$ along the portion of the $i$\textsuperscript{th} line $\gamma$ spends on it.}
\label{f.BrlPP}
\end{figure}

\subsubsection{The melon function}
We define the weight of a collection of disjoint (except possibly at shared endpoints) upright paths as the sum of the weights of the constituent paths. Then for $j\in\intint{n}$, we define $f[(y,n)^j\to (x,1)^j]$ to be the maximum weight over all collections of $j$ disjoint paths from $(y,n)$ to $(x,1)$.

With this, we may define the \emph{melon} function $\Wf = ((\Wf)_1, \ldots, (\Wf)_n):\intint{n}\times[0,\infty) \to \R$ by
%
\begin{equation}\label{e.Wf definition}
(\Wf)_j(x) = f[(0,n)^j\to (x,1)^j] - f[(0,n)^{j-1}\to (x,1)^{j-1}],
\end{equation}
for $j\geq 2$ and $(\Wf)_1(x) = f[(0,n)\to (x,1)]$. Two important deterministic properties are that the curves of $\Wf$ are \emph{ordered}, meaning that $(\Wf)_i(\bigcdot) \geq (\Wf)_{i+1}(\bigcdot)$ for each $i\in\intint{n-1}$ (see, for instance, the discussion at the start \cite[Section 4]{Landscape} and references given there); and an inference concerning last passage values in the melon environment:
\begin{proposition}[Special case of Proposition 4.1 of \cite{Landscape}]\label{p.melon lpp identity}
Let $f=(f_1,\ldots, f_n):\intint{n}\times[0,\infty)\to \R$ be continuous, and let $y\leq x$. Then
$$f[(y,n)\to (x,1)] = (\Wf)[(y,n)\to (x,1)].$$
\end{proposition}

In particular, this proposition applies when $f = B$, a collection of $n$ Brownian motions, as considered below in Section \ref{sec:blpp}.
An important technical condition imposed by the definition of $\Wf$, as well as by Proposition~\ref{p.melon lpp identity}, is that the domain of each line is $[0,\infty)$, rather than $\R$. It is because of this condition that we consider only initial conditions which are $-\infty$ for all small enough arguments; which is to say, this is why we require Assumption~\ref{a:initial_state_pd}(b).

\subsubsection{Scaling limit of Brownian LPP}\label{sec:blpp}
We will need the convergence of Brownian LPP values to the parabolic Airy sheet, a convergence that holds uniformly on compact sets. Let $B:\intint{n}\times \R\to \R$ be a collection of $n$ independent two-sided Brownian motions. Define $\S_n:\R^2\to\R\cup\{-\infty\}$ by
\begin{equation}\label{e.S_n}
\S_n(y,x) := n^{-1/3}\left(B[(2y n^{2/3},n) \to (n+2x n^{2/3}, 1)] - 2n - 2(x-y)n^{2/3}\right);
\end{equation}
the co-domain includes $-\infty$ simply to handle the case that $2yn^{2/3}> n+2xn^{2/3}$.
(Although here we allow $y$ to be negative, the definitions in \eqref{e.path weight definition} and \eqref{e.lpp definition} easily generalize.)

Here is a statement of convergence of Brownian LPP to the parabolic Airy sheet. (It is simply to have a cleaner version of this statement that we allowed $y$ to be negative in \eqref{e.S_n}.)

\begin{prop}[Theorem~1.3 of \cite{Landscape}]\label{t.BrLPP to Airy sheet}
In the topology of uniform convergence on compact sets, we have the convergence in law
\begin{equation}\label{e.para Airy BrLPP connection}
\S(y, x) = \lim_{n\to\infty} \S_n(y, x).
\end{equation}
\end{prop}
%
%

We will generally work with a centred and scaled version of $\WB$. Indeed, let $\L_n=(\L_{n,1},\ldots, \L_{n,n}): \intint{n}\times[-\frac{1}{2}n^{1/3}, \infty) \to \R$ be given by
\begin{align}\label{e.P_n definition}
\L_{n,j}(x) = n^{-1/3}\left((\WB)_j(n+2xn^{2/3}) - (2n+2xn^{2/3})\right).
\end{align}
%
Here, $\L$ indicates ``parabolic'', as these objects converge to the parabolic Airy line ensemble (though we will not use this fact, as we only require the convergence of Brownian LPP values to the parabolic Airy \emph{sheet} as in Proposition 5.3). 
Since $(\WB)_j(\bigcdot)$ is ordered, and $(\WB)_j(0) = 0$ for $j\in \intint{n}$, we see that, for $x\geq -\frac{1}{2}n^{1/3}$,
\begin{equation}\label{e.P S identity}
\begin{split}
\L_{n,1}(x) 
&= n^{-1/3}\left(B[(0,n)\to(n+2xn^{2/3}, 1)] - (2n+2xn^{2/3})\right)\\
&= \S_n(0,x).
\end{split}
\end{equation}
We used the definition of $\WB_{n,1}$ \eqref{e.Wf definition} for the first equality and \eqref{e.S_n} for the second.

Note also that
\begin{equation}\label{e.P LPP value}
\begin{split}
\MoveEqLeft[4]
\L_n[(y,n) \to (x,1)]\\
&= n^{-1/3}\left((\WB)[(n+2yn^{2/3}, n) \to (n+2xn^{2/3}, 1)] - 2(x-y)n^{2/3}\right),
\end{split}
\end{equation}
for all $-\frac{1}{2}n^{1/3}\leq y<x$. We find then that, for $y>0$ with $x > y-\frac{1}{2}n^{1/3}$,
\begin{equation}\label{e.melon LPP and S_n}
\S_n(y,x) = \L_n[(-\tfrac{1}{2}n^{1/3}+y, n) \to (x, 1)] - n^{2/3}
\end{equation}
by comparing \eqref{e.P LPP value} to the definition  \eqref{e.S_n} of $\S_n$, and using Proposition~\ref{p.melon lpp identity}.

 We may now define the prelimiting version of $\hlim$, denoted $\h:[-n^{1/60}, n^{1/60}] \to \R\cup\{-\infty\}$, by
\begin{equation}\label{e.prelimiting gen profile formula}
\begin{split}
\h(x) &= \sup_{0 \leq y \leq n^{1/60}} \Big(\hlim_0(y) + \S_n(y, x)\Big) \\
%
%
&=\sup_{0 \leq y \leq n^{1/60}} \left(\hlim_0(y) + \L_n[(-\tfrac{1}{2}n^{1/3}+y, n) \to (x, 1)] - n^{2/3}\right),
\end{split}
\end{equation}
since by assumption $\fh_0(y) = -\infty$ for $y< 0$.
The final equality follows from \eqref{e.melon LPP and S_n}. We adopt the upper limit of $n^{1/60}$ on $y$ and $|x|$ in order to meet a technical hypothesis in the application of an upcoming estimate Proposition~\ref{l.BrLPP uniform tail bound}  from \cite{LPPtools}; note that $n^{1/60}\to\infty$ as $n\to\infty$ and so in the limit $y$ and $x$ can be thought of as respectively taking any non-negative value and any real value.
Finally, for given $x$, a path in the environment defined by $\L_n$ which achieves the supremum implicit in $\L_n[(-\tfrac{1}{2}n^{1/3}+y, n) \to (x, 1)]$ in the last equality of \eqref{e.prelimiting gen profile formula} is called a \emph{geodesic}.

%
%
The next lemma translates the well-known fact that $\WB$ can be described as non-intersecting Brownian motions to a similar statement about $\L_n$'s distribution; see also Figure~\ref{f.P_n}.

\begin{figure}
\includegraphics[scale=1.2]{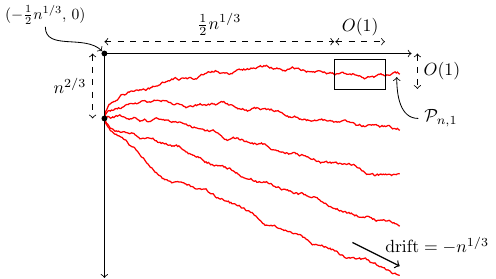}
\caption{A depiction of $\L_n$ from \eqref{e.P_n definition}. The vertical shift by $n^{-2/3}$, drift of $-n^{-1/3}$, and scaling by $n^{-1/3}$ are picked so that the height and fluctuations of $\L_{n,1}$ in a unit order interval are also of unit order, as emphasized by the black box of unit order height and width.} 
\label{f.P_n}
\end{figure}

\begin{lemma}\label{l.dyson brownian motion}
The distribution of $\L_n:\intint{n}\times [-\frac{1}{2}n^{1/3}, \infty)\to \R$ is that of $n$ independent Brownian motions of rate two and of drift $-n^{1/3}$, with common initial value $-n^{2/3}$, and conditioned on mutual non-intersection.
\end{lemma}

\begin{proof}
We may identify $\WB$ with $n$-level Dyson Brownian motion \cite{o2002representation}, which may be defined as $n$ independent Brownian motions of rate one and zero drift, with common initial value zero, conditioned on mutual non-intersection (the singular conditioning made precise via a suitable limiting procedure or a Doob $h$-transform); see, for example, \cite{dyson,mehta,grabiner1999brownian}. The expressions for the rate, drift, and initial value in the sought statement
follow from the definition~\eqref{e.P_n definition} of $\L_n$.
\end{proof}

%

Since we ultimately need to make inferences about $\hlim$, we require that $\h\to \hlim$ on compact sets. This is recorded in the next statement, which we will prove in Section~\ref{s.technical estimates} using the convergence of $\S_n$ to $\S$ from Proposition~\ref{t.BrLPP to Airy sheet}.

\begin{lemma}\label{l.convergence to fixed point} Let $\hlim_0:\R\to\R\cup\{-\infty\}$ satisfy Assumption~\ref{a:initial_state_pd}. Then we have that $\h \to \hlim$ in distribution in the topology of uniform convergence on compact sets.
\end{lemma}

We will prove Theorem~\ref{thm:TP_probability} by deriving the following analogous statement for $\h$. 

\begin{proposition}\label{p.twin peaks prelimit}
Let $\hlim_0:\R\to\R\cup\{-\infty\}$ satisfy Assumption~\ref{a:initial_state_pd} and suppose that $A>0$. There exist $L_0$ and $\smallconst>0$ (both depending on $\gamma$, $\theta$, and $A$) such that, for all $L> L_0$, there exists $n_0$ (depending on $\gamma$, $\theta$, and $L$) so that for all $n>n_0$ and $\varepsilon\in(0,1)$,
$$\P\left(\h\in\TP{A,L}^{\eps}\right) \geq \smallconst\varepsilon.$$
Further, $L_0$ and $\smallconst$ may be taken to depend continuously on $\gamma$, $\theta$, and $A$.
\end{proposition}
We show how Proposition~\ref{p.twin peaks prelimit} implies Theorem~\ref{thm:TP_probability}, and then turn to the proof of the former.

\begin{proof}[Proof of Theorem~\ref{thm:TP_probability}]
As we have noted, we may take $t=1$ and use the notation $\hlim$ in place of~$\fh_1$. 
%
%
%
%
%
%
Let $L>L_0$ with $L_0$ as given in Proposition~\ref{p.twin peaks prelimit}. By combining the fact that $\h$ converges to $\hlim$ uniformly on compact sets (Lemma \ref{l.convergence to fixed point}) with the Portmanteau theorem and Proposition~\ref{p.twin peaks prelimit}, we see that
$$\P\big(\hlim\in\TP{A,L}^{\eps}\big) \geq \lim_{n\to\infty}\P\left(\h\in\TP{A,L}^{\eps}\right) \geq \smallconst\varepsilon.$$
For general $t>0$, we must show that $\smallconst$ and $L_0$ can be taken to depend on $t$ continuously. This follows since the KPZ scaling, to move from the $\varepsilon$-twin peaks' event for $t>0$  to $t=1$, modifies $A$, $\gamma$, $\theta$, and $\varepsilon$ in a manner that depends continuously on $t$; the dependence of $\smallconst$ and $L_0$ on these quantities in Proposition~\ref{p.twin peaks prelimit} is also continuous.
This completes the proof of Theorem~\ref{thm:TP_probability}.
\end{proof}


We now turn to the proof of Proposition~\ref{p.twin peaks prelimit}. We start with a section on some technical estimates regarding the location of maximizers of $\h$ as well as of the maximum in its
 definition \eqref{e.prelimiting gen profile formula}.

\subsection{Locations of maximizers}\label{s.technical estimates}

 We will need a uniform tail bound for $\S_n(y, x)$ as $x$ and $y$ vary over compact intervals. Such a bound is proved in \cite{LPPtools}, and we state it here now.

\begin{proposition}[Proposition~3.15 of \cite{LPPtools}]\label{l.BrLPP uniform tail bound}
There exist finite positive constants $n_0$, $K_0$, $C$, and~$c$ such that, for $n\geq n_0$, $K_0\leq K \leq n^{1/30}$, and $0<R < n^{1/46}$,
$$\P\left(\sup_{x,y\in[-R,R]} \left|\S_n(y, x) + (y-x)^2\right| > K\right) \leq CR^2\exp(-cK^{3/2}).$$
\end{proposition}

The proof of Proposition~\ref{l.BrLPP uniform tail bound} as given in \cite{LPPtools} is not difficult and follows a strategy used earlier in \cite{basu2014last} to prove a similar statement in another model of LPP; essentially, one considers a fine discretization of the set of endpoints in $[-R,R]$ and uses known one-point tail bounds and a union bound to get the uniform-over-endpoints statement.

It is to handle the bounds of $n^{1/30}$ and $n^{1/46}$ on $K$ and $R$ that we have imposed that $y\leq n^{1/60}$ in the definition of $\h$ and restricted $\h$'s domain to $[-n^{1/60}, n^{1/60}]$.

Next, we show uniform-in-$n$ control over $x_0^{n}$, the maximizer of $\h$.



\begin{lemma}\label{l.tightness of prelimiting maximizer}Let $\hlim_0:\R\to\R\cup\{-\infty\}$ satisfy Assumption~\ref{a:initial_state_pd} .
Let $x_0^{n}$ be the maximizer of $\h$ of largest absolute value. Given $\delta>0$, there exist $n_0$ and $M<\infty$ (both depending on $\gamma$, $\theta$, and $\delta$) such that, for all $n\geq n_0$,
$$\P\left(|x_0^{n}| > M\right) \leq \delta.$$
Further, $M$ can be taken to depend on $\gamma$, $\theta$, and $\delta$ continuously.
\end{lemma}

\begin{proof}
Since by definition $x_0^{n} \in [-n^{1/60}, n^{1/60}]$, it is enough to prove that 
$$\P\left(\sup_{M < |x| < n^{1/60}} \h(x) > \h(0)\right) < \delta$$
for large enough $M$ depending only on $\delta$, $\gamma$ and $\theta$. For any $K>0$, we may bound above the left-hand side by
\begin{equation}\label{e.prelimit maximizer location breakup}
\begin{split}
\MoveEqLeft[10]
\P\left(\sup_{M<|x|<n^{1/60}} \h(x) > \h(0)\right)\\
&\leq \P\left(\sup_{M<|x|<n^{1/60}} \h(x) > -K\right) + \P\left(\h(0) < -K\right).
\end{split}
\end{equation}
We will find a $K$ such that both terms are less than $\delta/2$. The second term is easier to bound, and we address it first. Let $y_0 \in [-\theta, \theta]$ be such that $\hlim_0(y_0) \geq -2\theta$ and set $K \geq 4\theta$. From the formula \eqref{e.prelimiting gen profile formula} for $\h$, we see that
\begin{align*}
\P\left(\h(0) < -K\right) \leq \P\Big(\hlim_0(y_0) + \S_n(y_0,0) < -K\Big)
&\leq \P\Big(\S_n(y_0, 0) < -K/2\Big).
\end{align*}
We can bound this probability using the one-point tail information for $\S_n(y, x)$ from Proposition~\ref{l.BrLPP uniform tail bound}. Doing so shows that, for large enough $K$, the probability is less than $\delta/2$.

Returning to the first term on the right-hand side of \eqref{e.prelimit maximizer location breakup}, recall that, by \eqref{eq:h0bound}, there exists $\gamma>0$ such that $\hlim_0(y) \leq -\gamma |y|^{1/2}$ for all $y\in \R$. Then we have from \eqref{e.prelimiting gen profile formula} that
\begin{align}
\P\left(\sup_{M<|x|<n^{1/60}} \h(x) > -K\right) &\leq \P\Bigg(\sup_{\substack{M<|x|<n^{1/60}\\ 0\leq y\leq n^{1/60}}}\left(\S_n(y, x) - \gamma |y|^{1/2}\right) > -K\Bigg)\nonumber\\
&\leq \sum_{i=M}^{\lceil n^{1/60}\rceil}\,\, \sum_{j=0}^{\lceil n^{1/60}\rceil}\P\Bigg(\sup_{\substack{|x|\in [i,i+1]\\ y \in [j,j+1]}}\S_n(y, x) > -K + \gamma |j|^{1/2}\Bigg). \label{e.maximizer bounding sum breakup}
\end{align}
Now we want to apply Proposition~\ref{l.BrLPP uniform tail bound} to bound each summand. We see that
\begin{align}
\MoveEqLeft[4]
\P\Bigg(\sup_{\substack{|x|\in [i,i+1]\\ y\in [j,j+1]}}\S_n(y, x) > -K + \gamma |j|^{1/2}\Bigg)\nonumber\\
& \leq \P\Bigg(\sup_{\substack{|x|\in [i,i+1]\\ y\in [j,j+1]}}\left(\S_n(y, x) + (x-y)^2\right) > -K + \gamma|j|^{1/2} + (|i-j|-1)^2\Bigg).\label{e.prob breakup double sum}
\end{align}
We need a lower bound on the right-hand side of the preceding line's probability, which we record next.

\begin{lemma}
For $i\geq 4$ and $j\geq 1$,
\begin{equation}\label{e.to show square}
\gamma j^{1/2} + (|i-j|-1)^2 \geq \frac{\gamma}{2}j^{1/2} + c_\gamma i^{1/2},
\end{equation}
where $c_\gamma = \min\left(\frac{1}{16}, \frac{\gamma}{2\sqrt{2}}\right)$.
\end{lemma}

\begin{proof}
Note that
\begin{align*}
\gamma j^{1/2} + (|i-j|-1)^2 = \gamma j^{1/2} + i^2(|1-j/i|-1/i)^2.
\end{align*}
If $j/i\geq 1/2$, we ignore the second term and write $\gamma j^{1/2}$ as $\gamma j^{1/2}/2 + \gamma j^{1/2}/2 \geq \gamma j^{1/2}/2 + \gamma i^{1/2}/2\sqrt2$.
If $j/i<1/2$ and $i\geq 4$, then
\begin{align*}
(|1-j/i|-1/i)^2 \geq (\tfrac{1}{2}-1/i)^2 \geq \frac{1}{16}.
\end{align*}
Noting that $i^2 \geq i^{1/2}$ completes the proof.
\end{proof}

%
%

Using \eqref{e.to show square} in \eqref{e.prob breakup double sum} gives
\begin{align*}
\MoveEqLeft[6]
\P\Bigg(\sup_{\substack{|x|\in [i,i+1]\\ y\in \gamma^{-1/2}[j,j+1]}}\left(\S_n(y, x) + (x-y)^2\right) > -K + \gamma j^{1/2} + (|i-j|-1)^2\Bigg)\\
&\leq \P\Bigg(\sup_{\substack{|x|\in [i,i+1]\\ y\in [j,j+1]}}\left(\S_n(y, x) + (x-y)^2\right) > \gamma j^{1/2} + \frac{1}{2}c_\gamma i^{1/2}\Bigg).
\end{align*}
Here we assumed that $K \leq \smash{\frac{c_\gamma}{2}i^{1/2}}$ for $i\geq M$, which can be assured by supposing that $M$ is large enough.
We apply Proposition~\ref{l.BrLPP uniform tail bound} after noting that we are permitted to do so since for $i \leq n^{1/60}$ and $j\leq n^{1/60}$, we have $i^{1/2}, j^{1/2} < n^{1/120}$ and $[j,j+1],[i, i+1] \subseteq [-n^{1/46}, n^{1/46}]$. Thus we see, for positive constants $c$ and $C$ depending on $\gamma$,
$$\P\Bigg(\sup_{\substack{|x|\in [i,i+1]\\ y\in [j,j+1]}}\S_n(y, x) > -K+ \gamma j^{1/2}\Bigg) \leq C\max \{ i^2, j^2\} \exp \big\{ -c(i^{3/4}+ j^{3/4}) \big\}.$$
Returning to the sum in \eqref{e.maximizer bounding sum breakup} and substituting this bound yields, for $M$ large enough,
\begin{align*}
\P\left(\sup_{M<|x|<n^{1/60}} \h(x) > -K\right)
&\leq \sum_{i=M}^{\lceil n^{1/60}\rceil}\,\,\sum_{j=0}^{\lceil n^{1/60}\rceil} C\max \{ i^2, j^2 \} \exp \big\{ -c(i^{3/4}+ j^{3/4}) \big\}\\
&\leq \sum_{i=M}^{\infty} C'i^2\exp \{ -ci^{3/4} \} ,
\end{align*}
which may be made smaller than $\delta/2$ by choosing $M$ suitably high (which overall depends on $\gamma$ and $\theta$)
 and by further assuming that $n> M^{60}$, if required. It may be easily checked that the dependence of $M$ on these quantities is continuous. This completes the proof of Lemma~\ref{l.tightness of prelimiting maximizer}.
\end{proof}


The proof of Lemma~\ref{l.convergence to fixed point} on the convergence of $\h$ to $\hlim$ follows similar lines, relying on a bound on the location of the maximizer in the definition of $\h$ \eqref{e.prelimiting gen profile formula}. We give it now.

\begin{proof}[Proof of Lemma~\ref{l.convergence to fixed point}]
Set $y_n(x)$ equal to $\argmax_y \left(\hlim_0(y) + \S_n(y,x)\right)$ (taking the choice of largest absolute value when it is not unique). Fix $M>0$. We claim that $(y_n(x))_{n\in\N}$ is uniformly tight for $x\in[-M,M]$. To show this, let $\varepsilon>0$ be given and let $R\geq M\vee \theta$, so that we may choose $y_0 \in[0,R]$ such that $\hlim_0(y_0) \geq -2\theta$. Then we observe that, for every $K>0$,
\begin{equation}
\begin{split}
\MoveEqLeft
\P\big(y_n(x) > R\big) \\
&\leq \P\left(\sup_{R\leq y \leq n^{1/60}} \big(\hlim_0(y) + \S_n(y,x)\big) > \hlim_0(y_0) + \S_n(y_0,x)\right)\\
&\leq \P\left(\sup_{R\leq y \leq n^{1/60}} \big(\hlim_0(y) + \S_n(y,x)\big) > -K\right) + \P\Big(\hlim_0(y_0) + \S_n(y_0,x) < -K\Big).\end{split}
\label{e.maximizer tightness split up}
\end{equation}
We set $K \geq 4\theta$ large enough that $\hlim_0(y_0) \geq -2\theta \geq -K/2$ and $K \geq (y_0-x)^2$ (for example by setting $K\geq 4R^2$). Then we bound the second term of \eqref{e.maximizer tightness split up} by $CR^2\exp(-cK^{3/2})$ (uniformly for all $x\in[-M,M]$) by Proposition~\ref{l.BrLPP uniform tail bound}. Thus, for all $K$ large enough, the second term is at most $\varepsilon/2$.

Next we bound the first term of \eqref{e.maximizer tightness split up}. By a union bound,
\begin{align*}
\P\left(\sup_{R\leq y \leq n^{1/60}} \big(\hlim_0(y) + \S_n(y,x)\big) > -K\right)
&\leq \sum_{j=R}^{\lceil n^{1/60}\rceil } \P\left(\sup_{y \in [j,j+1]} \big(\S_n(y,x) - \gamma |y|^{1/2}\big) > -K\right)\\
&\leq \sum_{j=R}^{\lceil n^{1/60}\rceil } \P\left(\sup_{y \in [j,j+1]} \S_n(y,x) > -K + \gamma j^{1/2}\right).
\end{align*}
Setting $R$ such that $\gamma j^{1/2} > 2K$ for all $j\geq R$ shows that each summand in the last display is bounded above by $Cj^2 \exp(-c j^{3/4})$, uniformly over $x\in[-M,M]$, again by Proposition~\ref{l.BrLPP uniform tail bound}. This expression is summable in $j$. So taking $R$ sufficiently large implies that the sum is bounded above by $\varepsilon/2$. Thus, for such $R$, we find that, for $x\in[-M,M]$,
$$\P\big(y_n(x) > R) \leq \varepsilon,$$
so that the claimed uniform tightness is obtained, because $y_n(x) \geq 0$ almost surely by our assumption on~$\fh_0$.

That the maximizer sequence has a convergent subsequence, combined with the uniform convergence on compact sets of $\S_n$ to $\S$, implies that $\h\to \hlim$ uniformly on compact sets as well. To see this, fix $M>0$ and let $\mc K$ be a random compact set such that $y_n(x) \in \mc K$ for all $n$ and $x\in[-M,M]$. Then simple manipulations show that
\begin{equation*}
\sup_{x\in[-M,M]} \left|\h(x) - \hlim(x)\right| \leq \sup_{\substack{x\in[-M,M]\\ y\in \mc K}} \Big|\S_n(y,x) - \S(y,x)\Big| \to 0.\qedhere
\end{equation*}
\end{proof}

\subsection{The resampling framework}
To prove Proposition~\ref{p.twin peaks prelimit}, we will prove the following stronger proposition, from which the former immediately follows. Let $\beta$ be as in the definition of $\bJ_L$ in \eqref{eq:setJ}.

\begin{proposition}\label{p.resampling twin peaks}
Let $\hlim_0:\R\to\R\cup\{-\infty\}$ satisfy Assumption~\ref{a:initial_state_pd} and suppose  $A>0$. There exist $\smallconst >0$ and $L_0>0$ (both depending on $\gamma$, $\theta$, and $A$) such that, for all $L>L_0$, there exists $n_0$ (depending on $\gamma$, $\theta$, and $L$) such that, for all $n>n_0$ and $\varepsilon\in(0,1)$,
%
$$\P\bigg(\sup_{x\in[x_0^{n}+A,x_0^{n}+A+2]}\!\!\!\! \h(x) > \h(x_0^{n}) - \varepsilon;\,\,  |\h(x_0^{n})| \leq \beta L^{1/2};\,\,  |x_0^{n}|\leq \beta L-A-2\bigg) \geq \smallconst\varepsilon,$$
where $x_0^{n} = \argmax_{|x|\leq n^{1/60}} \h(x)$ and is taken to be the largest (i.e., not necessarily greatest in absolute value) one on the event that it is not unique; we will use shorthand $\smash{x_0=x_0^{n}}$ below at times. Further, $L_0$ and $\smallconst$ can be taken to depend on $\gamma,$ $\theta,$ and $A$ continuously.
\end{proposition}

The strengthening of Proposition~\ref{p.resampling twin peaks} relative to Proposition~\ref{p.twin peaks prelimit} is that we now assert that it is possible to achieve the twin peaks' event of separation $A$ by moving at most distance two to the right of the maximizer $x_0$ beyond the imposed distance $A$.

The proof of Proposition~\ref{p.resampling twin peaks} follows a Gibbsian resampling argument. (We will recall the Brownian Gibbs property precisely in Section~\ref{s.brownian gibbs}.) This argument is considerably easier in the case where $\fh_0$ is a narrow wedge; in Section \ref{sec:nwoutline}, we explain how this case works and then give the more general proof of this proposition. To set up the argument, we must first recall that $\h$ can be expressed in terms of $\L$ via the variational problem in \eqref{e.prelimiting gen profile formula}; and that $\L$ can be expressed in terms of non-intersecting Brownian motions via Lemma \ref{l.dyson brownian motion}. Roughly put, then, $\h$ may be expressed in terms of non-intersecting Brownian motions. We will make use of the Gibbs resampling property for these motions, filtered through the variational problem. To do this, we need to define a $\sigma$-algebra $\F$ that contains the data which will \emph{not} be resampled. We will study the $\F$-conditional distribution of $\h$ on $[x_0+A,x_0+A+2]$ and show that, with probability at least $\smallconst \varepsilon$, an event occurs  which implies that $\h \in \TP{A,L}^\eps$.

To describe $\F$, we need some notation: for a function $f:I\to \R$ and an interval $[a,b]\subseteq I$,  the \emph{bridge} of $f$ on $[a,b]$, denoted $f^{[a,b]}:[a,b]\to \R$, is given by
\begin{equation}\label{e.bridge notation}
f^{[a,b]}(x) = f(x) - \frac{x-a}{b-a}\cdot f(b) - \frac{b-x}{b-a}\cdot f(a);
\end{equation}
this is the function obtained by affinely transforming $f$ so that its values at $a$ and $b$ vanish. This notation clashes with a similar one in Section~\ref{sec:KPZ_application}, but \eqref{e.bridge notation} will define $f^{[a,b]}$ within this section.

The $\sigma$-algebra $\F$ is defined to be the one generated by the following collection of random variables:
\begin{enumerate}
  \item The maximizer of $\h$: $x_0 = x_0^{n} = \argmax_{x\in[-n^{1/60}, n^{1/60}]} \h(x)$.
  \item The curve data of $\L_n$: $\big\{\L_{n,j}(x): j \in\llbracket2,n\rrbracket,\, x\geq -\frac{1}{2}n^{1/3} \text{ or } j=1,\, x\not\in [x_0+A, x_0+A+2]\big\}$.
  \item The side bridge data of the top curve in $[x_0+A, x_0+A+2]$: $\L_{n,1}^{[x_0+A,x_0+A+1]}$ and $\L_{n,1}^{[x_0+A+1,x_0+A+2]}$.
  (Here $\L_{n,1}^{[x_0+A,x_0+A+1]}$ is the function on $[x_0+A,x_0+A+1]$ defined via \eqref{e.bridge notation} with $f= \L_{n,1}$, and similarly for $\L_{n,1}^{[x_0+A+1,x_0+A+2]}$.)
\end{enumerate}

Conditionings on similar collections of data have been used in earlier works such as \cite{hammond2017brownian,calvert2019brownian}. There, however, the interval of focus---our $[x_0+A,x_0+A+2]$---is either deterministic or a stopping domain (an analogue of a stopping time suited to the spatial nature of the Brownian Gibbs property used there). This means that the conditional distribution is more easily analysed using standard Markovian properties. Here, $x_0$ is a random variable which depends on the entirety of $\L_n$ and so is rather non-Markovian. This complicates the analysis considerably; a careful treatment will be provided in Section~\ref{s.F-conditionl distribution}, for which we set up some notation and record some observations in the rest of this section.

Conditional on $\F$, the only randomness left in determining $\h$ is the value of the random variable $Z :=\L_{n,1}(x_0+A+1)$. Given a value of $Z$ labelled $z$, and the data of $\F$, we can reconstruct $\L_{n,1}(\bigcdot)$; when $z$ is distributed according to the correct $\F$-conditional distribution of $Z$, this reconstruction may be thought of as the $\F$-conditional distribution of $\L_n$. We will denote the reconstruction by $\L_{n,1}^{z}(\bigcdot):[-\tfrac{1}{2}n^{1/3}, \infty)\to \R$.
It is given by the formula
\begin{equation}\label{e.L^z formula}
\L_{n,1}^{z}(x) = \begin{cases}
\L_{n,1}(x) &
\begin{aligned}[l]
\textrm{for }&x\in [-\tfrac{1}{2}n^{1/3}, \infty)\\&\setminus [x_0+A, x_0+A+2],
\end{aligned}
\\[18pt]
\begin{aligned}[l]
&(x_0+A+1-x)\L_{n,1}(x_0+A)\\&\ \ +(x-(x_0+A))z + \L_{n,1}^{[x_0+A+0,x_0+A+1]}(x)
\end{aligned} &  
\begin{aligned}[l]
 \textrm{for } x\in(x_0&+A,\\
  &x_0+A+1);
 \end{aligned}\\[26pt]
\begin{aligned}
&(x-(x_0+A+1))\L_{n,1}(x_0+A+2)\\&\ \ + (x_0+A+2-x)z+ \L_{n,1}^{[x_0+A+1,x_0+A+2]}(x)
\end{aligned}
 &  
 \begin{aligned}[l]
 \textrm{for } x\in[x_0&+A+1,\\
  &x_0+A+2);
 \end{aligned}
\end{cases}
\end{equation}
while for $j\geq 2$ and $x\in \R$, $\L_{n,j}^{z}(x) = \L_{n,j}(x)$. Note that $\L_n^z$ is $\F$-measurable.


\subsection{The Brownian Gibbs property}\label{s.brownian gibbs}

\begin{figure}
\includegraphics[width=\linewidth]{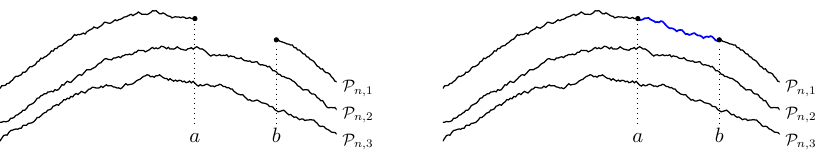}
\caption{A depiction of the Brownian Gibbs resampling procedure. On the left in black is all the data contained in the $\sigma$-algebra $\F_{\mathrm{ext}}$. The $\F_{\mathrm{ext}}$-conditional distribution of $\L_{n,1}$ on $[a,b]$ is that of a Brownian bridge (in blue) of rate two from $(a,\L_{n,1}(a))$ to $(b,\L_{n,1}(b))$ conditioned on not intersecting $\L_{n,2}$ on $[a,b]$.}
\label{f.bg resample}
\end{figure}

Here, we explain this property of $\L_n$, which was introduced and significantly leveraged in  \cite{BrownianGibbs}. It will be used in the proof of an important upcoming statement, Lemma~\ref{l.conditional distribution}, on the $\F$-conditional distribution of $Z$.

For a fixed interval $[a,b]\subseteq (-\frac{1}{2}n^{1/3}, \infty)$, define $\F_{\mathrm{ext}}$ to be the $\sigma$-algebra generated by $\{\L_{n,1}(x) : x\in[-\frac{1}{2}n^{1/3}, \infty) \setminus [a,b]\}$ and $\{\L_{n,j}(x) : j\in \llbracket2,n\rrbracket, x\geq [-\frac{1}{2}n^{1/3},\infty)\}$, i.e., the data of everything \emph{external} to $[a,b]$ on the top line. The \emph{Brownian Gibbs} property asserts that the $\F_{\mathrm{ext}}$-conditional distribution of $\L_{n,1}(\bigcdot)$ on $[a,b]$ is that of a Brownian bridge of rate two from $(a, \L_{n,1}(a))$ to $(b, \L_{n,1}(b))$ which is conditioned  not to intersect the second curve $\L_{n,2}(\bigcdot)$. This can be interpreted as saying that $\L_{n,1}$ can be \emph{resampled} on $[a,b]$ without changing its law by sampling a Brownian bridge with prescribed endpoints and conditioning it to avoid the second curve: see Figure~\ref{f.bg resample}.


\begin{lemma}\label{l.P_n Brownian gibbs}
The ensemble $\L_{n}$ has the Brownian Gibbs property.
\end{lemma}

This statement is equivalent to the one that $n$-level Dyson Brownian motion has the Brownian Gibbs property. Though well-known and used in previous works \cite{BrownianGibbs,dauvergne2018basic}, we were unable to locate a precise proof of this fact in the literature. However, it is fairly straightforward given the fact that an ensemble of Brownian bridges with strictly ordered endpoints conditioned on the (positive probability) event of non-intersection has the Brownian Gibbs property, and we will sketch the proof of Dyson Brownian motion having the Brownian Gibbs property given this fact. That non-intersecting Brownian bridges have the Brownian Gibbs property is very intuitive, but was formally proved only recently, in \cite{dimitrov2020characterization}.

\begin{proof}[Proof of Lemma~\ref{l.P_n Brownian gibbs}]
As mentioned, this follows from $\L_n$ being an affine transformation of $n$-level Dyson Brownian motion (Lemma~\ref{l.dyson brownian motion}) and the latter ensemble having the Brownian Gibbs property. We briefly outline how to show that $n$-level Dyson Brownian motion $\mathrm{DBM}_n : \intint{n}\times[0,\infty)\to\R$ has this property.

Let $[a,b] \subseteq (0,\infty)$. We first condition on the $\sigma$-algebra generated by $\{\mathrm{DBM}_{n,j}(x) : j\in\llbracket 1,n\rrbracket, x \in [0, \infty)\setminus [a,b]\}$. The Markov property of $\mathrm{DBM}_n$ then implies that this conditional distribution depends on only the boundary values $\{\mathrm{DBM}_{n,j}(x) : j\in\intint{n}, x\in\{a,b\}\}$. Then the conditional distribution is that of $n$ non-intersecting Brownian bridges with the given endpoints, as can be verified by comparing the transition probabilities of this ensemble (which makes use of the Karlin-McGregor formula \cite{karlinMcGregor} for non-intersecting strong Markov processes) with that of the conditioned Dyson Brownian motion (see, for example, \cite[Section~3]{warren2007dyson} for the transition probability formulas of Dyson Brownian motion). The ensemble of non-intersecting Brownian bridges, quite naturally, has the Brownian Gibbs property \cite[Lemma 2.13]{dimitrov2020characterization}.
 \end{proof}

\subsection{An outline of the argument in the narrow-wedge case}\label{sec:nwoutline}
Before proving Proposition~\ref{p.resampling twin peaks}, we give an outline of the argument in the simpler narrow-wedge case, under which $\fh_0$ is zero at the origin and $-\infty$ elsewhere.

First observe from \eqref{e.prelimiting gen profile formula} that, for this initial condition, $\h(x) = \S_n(0,x) = \L_{n,1}(x)$ for $|x|\leq n^{1/60}$. In particular, $\h$ is a function of only the top line of $\L_n$, and the same is true for $x_0=x_0^{n}$ defined earlier. The collection of curves $\L$ can be expressed in terms of non-intersecting Brownian motions via Lemma \ref{l.dyson brownian motion}. We will show that, for some $\smallconst$ (depending on $A$), it is with probability at least $\smallconst \eps$ that the curve $\L_{n,1}$ comes within $\eps$ of its maximum $\h(x_0)$ in the window $[x_0+A,x_0+A+2]$. Our first step, below, will be to identify the $\F$-conditional distribution of $Z=\L_{n,1}(x_0+A+1)$. (The event $\h\in\TP{A,L}^\eps$ also imposes conditions on the location of the maximizer and the value of the maximum, but these are more easily handled and not discussed here.)

%

\smallskip

\emph{Step 1: The $\F$-conditional distribution of $Z$}. As mentioned, $x_0$ is defined in terms of the whole curve $\L_{n,1}(\bigcdot)$ and so is non-Markovian; in particular, it is not a stopping time. But it is intuitively plausible, based on the definition of $x_0$ and the Brownian Gibbs property of $\L_n$, that the distribution of $\L_{n,1}(\bigcdot)$ on $[x_0, n^{1/60}]$ conditional on $\L_{n,1}(\bigcdot)$ outside of the interval $(x_0, n^{1/60})$ and the lower curves $\L_{n,2},\L_{n,3},\ldots$ is of a Brownian bridge (of rate two) between the appropriate endpoints conditioned on (i) not intersecting the lower curve $\L_{n,2}$ and (ii) not exceeding $\L_{n,1}(x_0)$. This intuition is correct and is carefully stated in Lemma~\ref{l.first conditional statement}.

This description of $\L_{n,1}$ on $[x_0, n^{1/60}]$ makes it easy to derive the distribution of $Z$ conditional on~$\F$. Indeed, when we also condition on the data of $\L_{n,1}$ on $[x_0,x_0+A]$ and $[x_0+A+2, n^{1/60}]$, we see that $\L_{n,1}$ on $[x_0+A,x_0+A+2]$ has the law of Brownian bridge of rate two with endpoints $\L_{n,1}(x_0+A)$ and $\L_{n,1}(x_0+A+2)$ which is conditioned to again (i) not intersect the lower curve and (ii) not exceed $\L_{n,1}(x_0)$. To get from this collection of conditioning data to~$\F$, we only have to include the side bridge data $\L_{n,1}^{[x_0+A,x_0+A+1]}$ and $\L_{n,1}^{[x_0+A+1,x_0+A+2]}$; classical decompositions of Brownian bridge then say that the $\F$-conditional distribution of $Z$ is that of a normal random variable of appropriate $\F$-measurable mean and unit variance, conditioned on the reconstruction $\L_{n,1}^Z(\bigcdot)$ again satisfying (i) and (ii).

\begin{figure}
\includegraphics[width=0.8\textwidth]{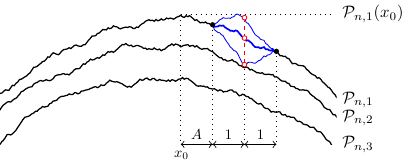}
\caption{A depiction of the resampling in the proof outline for the narrow-wedge case. The sigma-algebra $\F$ contains information about all of the thick curves including the thick blue curve, except that it forgets the location of the red circle whose value is denoted by $Z=\L_{n,1}(x_0+A+1)$. The reconstruction $\L^z_{n,1}$ is given by linearly shifting the left and right blue bridges to meet the value $z$ of the red circle as it varies. The thin blue lines here represent two possible reconstructions. The random variable $Z$ is restricted by the fact that the reconstruction must not exceed the global maximum $\L_{n,1}(x_0)$ (here denoted by a horizontal dotted black line) and must not intersect the second curve $\L_{n,2}$. The vertical dashed red line represents the possible range $[\Corner^{\downarrow}, \Corner^{\uparrow}]$ for $Z$ (with the upper and lower red circles corresponding to these bounds).}
\label{f.Z resampling}
\end{figure}

We can simplify this description of $Z$. Essentially, condition (i) places a lower bound on how large $Z$ can be, while (ii) places an upper bound: see Figure~\ref{f.Z resampling}. To make this rigorous, we observe that the reconstruction $\L_{n,1}^z(x)$ is monotone in $z$ for every $x$ from its formula \eqref{e.L^z formula}. These lower and upper bounds are $\F$-measurable random variables; they are \emph{corners} of the acceptable range of values $Z$ can take, and we label them respectively $\Corner^{\downarrow}$ and $\Corner^{\uparrow}$. Thus, $Z$ is a normal random variable with explicit $\F$-measurable mean and unit variance, conditioned on lying in the interval $[\Corner^{\downarrow}, \Corner^{\uparrow}]$.

\smallskip

\emph{Step 2: Finding the sweet spot for $Z$}. It is easy to see that $\Corner^{\uparrow}$ is such that, when $z=\Corner^{\uparrow}$, $\sup_{x\in [x_0+A,x_0+A+2]}\L_{n,1}^z(x) = \L_{n,1}(x_0)$. With this in hand, the linear---and so certainly Lipschitz---relationship of $\L_{n,1}^z(x)$ with $z$ for every fixed $x$ tells us that reducing $z$ by $\varepsilon$ from $\Corner^{\uparrow}$ reduces the value of $\sup_{x\in [x_0+A,x_0+A+2]}\L_{n,1}^z(x)$ by an amount of order $\eps$. Thus, to cause $\TP{A,L}^\eps$ to occur, it is enough to have $Z$ fall within a sweet spot interval $I_{\eps}$ of length of order $\eps$ with upper endpoint $\Corner^{\uparrow}$.

\smallskip

\emph{Step 3: The probability of hitting the sweet spot.} It remains to bound below the probability that $Z$ falls in $I_\eps$. To do so, we need control over two things: the $\F$-measurable mean of $Z$, and the value of $\Corner^{\uparrow}$. (We can ignore $\Corner^{\downarrow}$, i.e., take it to be $-\infty$, as we are only aided in proving a probability lower bound if its value is larger.) For this purpose, we consider a selection of favourable $\F$-measurable data $\fav$ which is defined by demanding bounds on these quantities, as well as on the location of the maximizer and the value of the maximum: see Section~\ref{s.favourable data} ahead. We then show that $\fav$ occurs with a probability that is uniformly positive in  $n$. Given control over the mean and $\Corner^{\uparrow}$ on a positive probability event, the form of the normal density guarantees that the probability of $Z$ falling in the order $\varepsilon$ length interval is at least some constant multiple of $\varepsilon$. This completes the proof outline in the narrow-wedge case.

\smallskip

\emph{Complications with general initial data.} The narrow-wedge case provided a number of simplifications, the primary being the equality of $\h$ and the top line of $\L_n$. This had two effects, both in Step 1: we could define $x_0$ in terms of only $\L_{n,1}(\bigcdot)$, i.e., without the lower curves (making it simpler to consider that process' distribution on $[x_0, n^{1/60}]$); and we could infer the existence of a valid interval $[\Corner^{\downarrow}, \Corner^{\uparrow}]$ for the $\F$-conditional distribution of $Z$ via monotonicity properties of only $\L_n$. Both these aspects will need modification in the general case.

Because we can perform Brownian resamplings only with $\L_n$, we need the representation of $\h$ in terms of $\L_n$ recorded in the last equality of \eqref{e.prelimiting gen profile formula}, which relies on the identification of LPP values in the original and melon environments cited in Proposition~\ref{p.melon lpp identity}. Note that $\h$, and so also $x_0$, is now defined in terms of all the curves of $\L_n$, not just the first. More specifically, while in the narrow wedge case we could work with the function values of $\L_{n,1}(\bigcdot)$, in the general case we have to analyse \emph{last passage values through the environment given by $\L_n$}. This is the underlying complication that causes all the others in the general case.

To achieve a description of $Z$ in terms of $[\Corner^{\downarrow}, \Corner^{\uparrow}]$ in the general case, we first need a formula for $\hz{z}$, the reconstruction of $\h$ when $Z=z$, in terms of $\L_n^z$. This will be recorded shortly in \eqref{e.L^f representation}. Then we need a monotonicity statement about $\hz{z}(x)$ for fixed $x$ that will allow us to express the condition that $z$ is such that $\sup_{x\in[x_0+A,x_0+A+2]} \hz{z}(x) \leq \h(x_0)$ as an upper bound on $z$, just as we did with the monotonicity statement for $\L_{n,1}^z$ above in simplifying the condition (ii). Such a monotonicity statement is actually not true for $\hz{z}$, and we circumvent this difficulty by deriving one instead for the weights of individual paths (as opposed to their supremum $\hz{z}$) in the reconstructed environment. This is Lemma~\ref{l.monotonicity of L^f} recorded ahead. Finally, we need to know that $\hz{z}(x)$ is Lipschitz in $z$ for each fixed $x$: see  Lemma~\ref{l.lipschitz} whose argument also proceeds in a pathwise manner.

With these modifications, the proof in the general case proceeds largely along the lines of the narrow-wedge case outlined here. We move to giving the details next, starting with the facts needed to handle the  general case's main complications, namely the monotonicity and Lipschitz properties of $\hz{z}$.

\subsection{The reconstructed $\h$ and its properties}
Using \eqref{e.prelimiting gen profile formula}, we provide a formula in terms of~$\L_n^z$ for the reconstructed $\h$, denoted $\hz{z}:[-n^{1/60}, n^{1/60}]\to\R\cup\{-\infty\}$. It is:
\begin{equation}\label{e.L^f representation}
\hz{z}(x) = \sup_{0\leq y\leq n^{1/60}} \Big(\hlim_0(y) + \L_n^{z}[(-\tfrac12n^{1/3}+ y, n) \to (x, 1)] - n^{2/3}\Big).
\end{equation}
%
%
Since $\L_n^{z}$ is $\F$-measurable for all $z\in \R$, so is $\hz{z}$. As we noted in the preceding section, it is also immediate from the formula of $\L_{n,1}^{z}(x)$ that, for any $x\in \R$, $\L_{n,1}^z(x)$ is non-decreasing as a function of $z$. The function $\hz{z}(x)$ further enjoys a monotonicity property in $z$ that is slightly more complicated, which we record in the next lemma.

Recall that, for an upright path $\gamma$, $\L_n^z[\gamma]$ is the weight 
of $\gamma$ in the environment $\L_n^z$. Imitating \eqref{e.prelimiting gen profile formula}, let
$$\hz{z}(\gamma) = \hlim_0(y) + \L^z[\gamma] - n^{2/3},$$
where $y$ is the starting coordinate of $\gamma$ on line $n$.

\begin{lemma}[Monotonicity of $\hz{z}$ in $z$]\label{l.monotonicity of L^f}
For each upright path $\gamma$ starting on line $n$ and ending on  line $1$, the process $z \mapsto \hz{z}(\gamma)$ is non-increasing almost surely; or constant almost surely; or non-decreasing almost surely. Moreover, it is an almost sure event that:
\begin{enumerate}
  \item if $-n^{1/60}\leq x\leq x_0+A$, then $\hz{z}(x) = \h(x)$ for all $z\in \R$; and
  \item if $x\geq x_0+A+2$, then $\hz{z}(x)$ is non-increasing in $z$.
\end{enumerate}
\end{lemma}

\begin{proof}
Let $u$ be the coordinate at which $\gamma$ jumps to the top line (i.e., line $1$); and let $x$ be its ending coordinate. Let $\gamma^{u-}$ be $\gamma$ restricted to its path before coordinate $u$, i.e., $\gamma$'s restriction to the lower $n-1$ lines,  indexed by $\llbracket 2,n \rrbracket$. Then
\begin{align}
\hz{z}(\gamma) &= \hz{z}(\gamma^{u-}) + \L_{n,1}^{z}(x) - \L_{n,1}^z(u)\nonumber\\
&= \h(\gamma^{u-}) + \L_{n,1}^{z}(x) - \L_{n,1}^z(u);\label{e.path weight in resampled env}
\end{align}
the latter equality because the environment of the lower $n-1$ lines of $\L^z_n$ does not depend on $z$. The claim that $\hz{z}(\gamma)$ is monotone and the nature of its monotonicity now follow readily by examining the increment $\L_{n,1}^z(x) - \L_{n,1}^z(u)$ from the definitions in \eqref{e.L^z formula}.

Now we move to proving the two numbered claims.
For (1), consider the set of paths which end at $x$. We {\em claim} that, for such paths $\gamma$, $\hz{z}(\gamma)$ is constant in $z$;  which implies (1). The claim follows by noting that, if $x\leq x_0+A$, then $u\leq x_0+A$; and so $\L_{n,1}^{z}(x) - \L_{n,1}^{z}(u)$ does not depend on $z$ from \eqref{e.L^z formula}. This completes the claim by the decomposition \eqref{e.path weight in resampled env}.

A similar argument holds for (2). We {\em claim} that, for any path $\gamma$ which ends at $x$, $\hz{z}(\gamma)$ is non-increasing in $z$. To see this, we use the same decomposition as \eqref{e.path weight in resampled env}, and observe that it is enough to prove that $\L_{n,1}^z(x) - \L_{n,1}^z(u) = \L_{n,1}(x) - \L_{n,1}^z(u)$ is non-increasing in $z$. There are two cases: $u\in [x_0+A,x_0+A+ 2]$ and $u\not\in [x_0+A,x_0+A+2]$. In the first, $\L_{n,1}^z(u)$ is non-decreasing in $z$; while, in the second, it is constant, as we see from \eqref{e.L^z formula}.  This completes the proof of the claim and thus also of Lemma~\ref{l.monotonicity of L^f}.
%
%
%
%
%
\end{proof}

While a similar monotonicity statement in $z$ as Lemma~\ref{l.monotonicity of L^f}(1) and (2) holds for $\hz{z}(x)$ when $x\in [x_0+A, x_0+A+1]$, there is no counterpart when $x\in [x_0+A+1, x_0+A+2]$. In the notation of the preceding proof, this is because, in the latter case, the type of monotonicity for the weight of a given path ending at $x$ depends on the value of $u$, the location at which the path jumps to the top line from the second line: for a certain range of $u$ depending on $x$, the weight of any path with jump point $u$ will be increasing in $z$; while, for larger $u$, it will be decreasing. Since $\hz{z}(x)$ maximizes over all such paths, no monotonicity holds for this quantity.
In contrast, notice from the proof of Lemma~\ref{l.monotonicity of L^f} that, for $x\in[-n^{1/60}, x_0+A]$ or $x\geq x_0+A+2$, the weight of any path ending at $x$ has a single form of monotonicity for all possible $u$. It is in order to handle this absence of $z$-monotonicity for $\hz{z}(x)$ when $x\in [x_0+A+1, x_0+A+2]$ that we proved the first statement of Lemma~\ref{l.monotonicity of L^f}, concerning the monotonicity of the weight of \emph{single} paths.

\begin{lemma}[$\sup \hz{z}$ is Lipschitz in $z$]\label{l.lipschitz}
It holds almost surely that, for any $z_1, z_2\in \R$,
$$\left|\sup_{x\in [x_0+A,x_0+A+ 2]}\hz{z_1}(x) - \sup_{x\in[x_0+A,x_0+A+ 2]}\hz{z_2}(x)\right| \leq 2 |z_1-z_2|.$$
\end{lemma}

\begin{proof}
For convenience of notation, let us define
$$h(z) = \sup_{x\in[x_0+A,x_0+A+ 2]}\hz{z}(x).$$
The arguments that we will present hold on the probability one event $\Omega$ that, for each $z\in \R$, there exist $x\in [x_0+A,x_0+A+ 2]$, $y\in [0, n^{1/60}]$ and an upright path $\Gamma^{z,x}$ (for which we use the capital Greek letter to emphasise the path's randomness) ending at $x$, such that
$$h(z) = \hlim_0(y) + \L_n^z[\Gamma^{z,x}] - n^{2/3};$$
that the supremum in the definition of $h$ is achieved uses the compactness of $[x_0+A,x_0+A+ 2]$.

By symmetry, it is enough to prove that $h(z_1) - h(z_2) \leq 2 |z_1-z_2|.$
On the event $\Omega$, we see that, for some $y\in [0, n^{1/60}]$, $x\in [x_0+A,x_0+A+ 2]$, and upright path $\Gamma^{z_1,x}$,
\begin{align*}
h(z_1) &= \hlim_0(y) + \L^{z_1}_n[\Gamma^{z_1,x}] - n^{2/3},\\
h(z_2) &\geq \hlim_0(y) + \L^{z_2}_n[\Gamma^{z_1,x}] - n^{2/3},
\end{align*}
and so
$$h(z_1) - h(z_2) \leq \L_n^{z_1}[\Gamma^{z_1,x}] - \L_n^{z_2}[\Gamma^{z_1,x}].$$
Let $u$ be the coordinate where $\Gamma^{z_1,x}$ jumps to the top line. Since the environments defined by $\L_n^{z_1}$ and $\L_n^{z_2}$ differ only in the top line, we see that
\begin{align*}
\L_n^{z_1}[\Gamma^{z_1,x}] - \L_n^{z_2}[\Gamma^{z_1,x}]
&=\big[\L_{n,1}^{z_1}(x) - \L_{n,1}^{z_1}(u)\big] - \big[\L_{n,1}^{z_2}(x) - \L_{n,1}^{z_2}(u)\big] \\
&= \big[\L_{n,1}^{z_1}(x) - \L_{n,1}^{z_2}(x)\big] - \big[\L_{n,1}^{z_1}(u) - \L_{n,1}^{z_2}(u)\big]\\
&\leq 2|z_1-z_2|.
\end{align*}
The inequality follows from the definition of $\L^z_{n,1}$ in \eqref{e.L^z formula}.
The equalities and the bound hold deterministically  on $\Omega$. This completes the proof of Lemma~\ref{l.lipschitz}.
\end{proof}

\subsection{The $\F$-conditional distribution of $Z$}\label{s.F-conditionl distribution}
We next move towards a description of the $\F$-conditional distribution of $Z$.
First we define the canonical filtration for the top curve, $\F^{\mathrm{past}}_t = \sigma\left(\L_{n,1}(s) : s\in[ -\frac{1}{2}n^{1/3}, t]\right)$. We also define a filtration that captures the future of the process by $\F^{\mathrm{future}}_t = \sigma\left(\L_{n,1}(s) :s \geq t\right)$.
Certain additional $\sigma$-algebras are needed. Let $\F_{\mathrm{lower}}$ be the $\sigma$-algebra generated by the lower $n-1$ curves, i.e., $\F_{\mathrm{lower}} = \sigma(\L_{n,j}(x) : x\geq -\frac{1}{2}n^{1/3}, j\in\llbracket 2,n\rrbracket)$. Let $\F^{\mathrm{past}}_{x_0}$ be the $\sigma$-algebra generated by all sets of the form
\begin{align}\label{e.sigma-algebra defn}
F_s \cap \{x_0 > s\},
\end{align}
where $F_s$ ranges over all elements of $\F^{\mathrm{past}}_s$ and $s$ ranges over $[-\frac{1}{2}n^{1/3},\infty)$. This $\sigma$-algebra encodes the information known by time $x_0$. If $x_0$ were a stopping time, $\F^{\mathrm{past}}_{x_0}$ would coincide with the usual $\sigma$-algebra associated with such times. Let $\F^{\mathrm{past}}_{x_0+A}$ be defined similarly to $\F^{\mathrm{past}}_{x_0}$ in \eqref{e.sigma-algebra defn} with $x_0+A$ replacing $x_0$. Let $\F^{\mathrm{future}}_{x_0+A+2}$ be defined as the $\sigma$-algebra generated by all sets of the form $F_s \cap \{x_0+A+2 < s\}$, where $F_s$ ranges over all elements of $\F^{\mathrm{future}}_s$ and $s$ ranges over $[-\frac{1}{2}n^{1/3}, \infty)$.

Finally, let $\F'$ be generated by $\F_{\mathrm{lower}} \cup \F^{\mathrm{past}}_{x_0}\cup \F^{\mathrm{future}}_{n^{1/60}}$, and let $\F''$ be generated by $\F'\cup\F^{\mathrm{past}}_{x_0+A}\cup\F^{\mathrm{future}}_{x_0+A+2}$ (which should be thought of as equalling $\F_{\mathrm{lower}}\cup\F^{\mathrm{past}}_{x_0+A}\cup\F^{\mathrm{future}}_{x_0+A+2}$, since typically $x_0+A+2$ will be less than $n^{1/60}$). Observe that $\F$ is the $\sigma$-algebra generated by $\F''$ along with the side bridge data on $[x_0+A,x_0+A+1]$ and $[x_0+A+1,x_0+A+ 2]$; i.e., $\L_{n,1}^{[x_0+A,x_0+A+1]}$ and $\L_{n,1}^{[x_0+A+1,x_0+A+ 2]}$. See \eqref{e.bridge notation} to recall the notation.

We start by describing the $\F'$-conditional distribution of $\L_{n,1}(\bigcdot)$, which will then be used to give the $\F$-conditional distribution of $Z=\L_{n,1}(x_0+A+1)$ in Lemma~\ref{l.conditional distribution}.
To give the first statement,  and with a slight abuse of notation that, we hope, will not  cause confusion with the earlier defined $\L_n^z$, define $\L_{n,1}^B(\bigcdot):[-\frac{1}{2}n^{1/3}, \infty)$  by
$$\L_{n,1}^B(x) = \begin{cases}
\L_{n,1}(x)  & \textrm{for }-\frac{1}{2}n^{1/3}\leq x\leq x_0 \text{ or } x\geq n^{1/60}\\
B(x) & \textrm{for } x_0\leq x\leq n^{1/60},
\end{cases}$$
where $B:[x_0, n^{1/60}]\to \R$ is a given function with $B(x_0) = \L_{n,1}(x_0)$ and $B(n^{1/60}) = \L_{n,1}(n^{1/60})$; also, let $\L_{n,j}^B(x) = \L_{n,j}(x)$ for $j\in \llbracket2,n\rrbracket$ and $x$ in the domain.
Define $\hz{B}$ (with a similar notational abuse) to be the reconstructed $\h$, given as in \eqref{e.L^f representation} by
$$
\hz{B}(x) = \sup_{0\leq y\leq n^{1/60}} \Big(\hlim_0(y) + \L_n^{B}[(-\tfrac12n^{1/3}+ y, n) \to (x, 1)]-n^{2/3}\Big).
$$

\begin{lemma}\label{l.first conditional statement}
Conditionally on $\F'$,  $\left\{\L_{n,1}(x) : x\geq -\frac{1}{2}n^{1/3}\right\}$ has the law of $\L_{n,1}^B$, where $B:[x_0, n^{1/60}]\to \R$ is Brownian bridge of rate two from $(x_0, \L_{n,1}(x_0))$ to $(n^{1/60}, \L_{n,1}(n^{1/60}))$ conditioned on non-intersection of the second curve $\L_{n,2}(\bigcdot)$ and on $\sup_{x_0\leq x\leq n^{1/60}}\hz{B}(x) \leq \h(x_0)$.
\end{lemma}

The proof will mainly rely on \cite{millar1978path}. This paper identifies the distribution of a homogeneous strong Markov process $X : [0,\infty) \to E$ on the unbounded interval whose left endpoint is the maximizer $t_0$ of $\Phi(X(t))$, for a given continuous function $\Phi: E\to [-\infty, \infty]$, as the same Markov process started at $X(t_0)$ but under the (typically singular) conditioning that the earlier maximum value is respected, i.e., that $\Phi(X_t)$ does not attain a higher maximum after~$t_0$. Here, $E$ is the state space of the Markov process, a set that is supposed compact in \cite{millar1978path}.

\begin{proof}[Proof of Lemma~\ref{l.first conditional statement}]
First we recall that $x_0\geq -n^{1/60}$ by definition and so $\F^{\mathrm{past}}_{-n^{1/60}}\subseteq \F^{\mathrm{past}}_{x_0}$. Conditional on $\F_{\mathrm{lower}}$, $\F^{\mathrm{past}}_{-n^{1/60}}$ and $\F^{\mathrm{future}}_{n^{1/60}}$, the distribution of $\{\L_{n,1}(x) : -n^{1/60} \leq x \leq n^{1/60}\}$ is that of a Brownian bridge of rate two on $[-n^{1/60},n^{1/60}]$ with starting value $\L_{n,1}(-n^{1/60})$ and ending value $\L_{n,1}(n^{1/60})$ conditioned on non-intersection with the second curve $\L_{n,2}(\bigcdot)$. This is the statement of the Brownian Gibbs property of $\L_{n}$, Lemma~\ref{l.P_n Brownian gibbs}. In particular, the conditioned process is Markov (and non-homogeneous), and, since Brownian bridge is a strong Markov process and the conditioning event is almost surely of positive probability, the same is true of the conditioned space-time process. (Here we consider the space-time process so as to have a time-homogeneous Markov process: see \cite[Chapter III, exercise 1.10]{revuz2013continuous}.)

Consider the process $X : [-n^{1/60}, \infty)\to [-\infty,\infty]^2\times[-n^{1/60},\infty]$ defined by $X(t) := (\L_{n,1}(t\wedge n^{1/60}), \h(t\wedge n^{1/60}), t)$; here $[-\infty,\infty]$ and $[-n^{1/60},\infty]$ are compactifications of $\R$ and $[-n^{1/60},\infty)$, and are employed so that the state space of $X$ is compact. We consider $X$ to start at time $-n^{1/60}$, and to be killed at time $n^{1/60}$, so that the maximizer of the second component of $X$ is $x_0^{n} = \argmax_{|x|\leq n^{1/60}} \h(x)$. (To be precise, as earlier we will be working with the final maximizer, i.e., the largest one, on the event that there are several. To see that there is such a final maximizer for the process $\h(t\wedge n^{1/60})$, note that there must be a final one on the interval $[-n^{1/60}, n^{1/60}]$ by continuity of $\h$, and that, by Lemma~\ref{l.maximizer has no atoms} ahead, the final one is almost surely not $n^{1/60}$.)

We claim that, conditionally on $\F_{\mathrm{lower}}$, $\F^{\mathrm{past}}_{-n^{1/60}}$ and $\F^{\mathrm{future}}_{n^{1/60}}$, $X$ is a homogeneous strong Markov process. To see this, first define the process $X'$ by $X'(t) = \big( \L_{n,1}(t\wedge n^{1/60}), \h(t \wedge n^{1/60}) \big)$. It is enough to prove that $X'$ is a \emph{non-homogeneous} Markov process, as then, by the same trick as used a few lines above, the space-time process $X(t) = (X'(t), t)$ is necessarily a homogeneous Markov process.

To show that $X'$ is Markov under this conditioning, we state a formula that expresses $\h(t+s)$ in terms of $\h(t)$ and data contained in $\F_{\mathrm{lower}}$, $\F^{\mathrm{past}}_{-n^{1/60}}$ and $\F^{\mathrm{future}}_{n^{1/60}}$:
\begin{equation}\label{e.h^n markov}
\begin{split}
\h(t+s) &= \max\bigg\{\h(t) + \L_{n,1}(t+s) - \L_{n,1}(t),\\
&\hspace{-0.3cm} \sup_{\substack{y\\ u> t}}\Big(\hlim_0(y) + \L_n[(-\tfrac{1}{2}n^{1/3}+y, n)\to (u,2)] + \L_{n,1}(t+s) - \L_{n,1}(u) - n^{2/3}\Big)\bigg\};
\end{split}\end{equation}
the supremum over $y$ and $u$ is taken over choices such that $0\leq y\leq n^{1/60}$ and $y\leq -\frac{1}{2}n^{1/3}+~u$.

This formula follows by considering the location $u$ that the geodesic for endpoint $t+s$ jumps to the top line. It is the first term that attains the maximum when $u\leq t$; and it is the second that does so when $u>t$. In the first case, $\h(t+s)$ is equal to $\h(t)$ plus the remaining increment on the top line as the $(t+s)$-geodesic must pass through $(t,1)$; this is because $\h(t)$ is the value attained by a similar maximization problem. In the second case, we have rewritten the formula \eqref{e.prelimiting gen profile formula} for $\h$ by decomposing the last passage value at $(u,2)$, which lies on the geodesic by definition.

Observe that, since we have conditioned on the lower $n-1$ curves, $\L_n[(-\frac{1}{2}n^{1/3}+y,n)\to (u,2)]$ is a deterministic function of $y$ and $u$. Thus, conditional on $(\L_{n,1}(t), \h(t))$, $\F_{\mathrm{lower}}$, $\F^{\mathrm{past}}_{-n^{1/60}}$ and $\F^{\mathrm{future}}_{n^{1/60}}$, it holds that $\h\left((t+s)\wedge n^{1/60}\right)$ is measurable with respect to $\{\L_{n,1}(x): t\leq x \leq n^{1/60}\}$, which is conditionally independent of $\F^{\mathrm{past}}_t$ given $\L_{n,1}(t)$ by the Markov property of $\L_{n,1}(\bigcdot)$. This proves that $X'$ is Markov, and so $X$ is a homogeneous Markov process. (We also used that the canonical filtration of $\h$ is contained in the filtration generated by $\F^{\mathrm{past}}_t$ and $\F_{\mathrm{lower}}$.) This argument reduced the Markov property of $X'$ to that of $\L_{n,1}$; the reduction also works to show that $X'$ is strong Markov since $\L_{n,1}$ is strong Markov.

Since the state space of $X$ is compact, we may apply the results of \cite{millar1978path}.
Consider the projection map $\Phi: [-\infty,\infty]^2\times[-n^{1/60},\infty] \to [-\infty, \infty]$ given by $(x,y,z)\mapsto y$. Then $x_0=x_0^{n} = \argmax_{t\geq -n^{1/60}} \Phi(X(t))$, the largest one if the argmax is not unique. The main theorem of \cite{millar1978path} implies that, conditionally on $\F_{\mathrm{lower}}$, $\F^{\mathrm{past}}_{-n^{1/60}}$ and $\F^{\mathrm{future}}_{n^{1/60}}$, the process $\{X(x_0+t) : t> 0\}$ is conditionally independent of $\F^{\mathrm{past}}_{x_0}$ given the data $(\L_{n,1}(x_0), \h(x_0), x_0)$. Further, this process is Markov and has the law of $X$ conditioned on the event that $\Phi(X(t)) \leq \Phi(X(x_0))$ for all $t$.

By projecting to the first coordinate of $X$, we see that this statement is equivalent to $\{\L_{n,1}(x) : x\geq -\frac{1}{2}n^{1/3}\}$ having the distribution of $\L_n^B$, where $B:[x_0, n^{1/60}] \to \R$ is a Brownian bridge of rate two from $(x_0, \L_{n,1}(x_0))$ to $(n^{1/60},\L_{n,1}(n^{1/60}))$ conditioned on (i) $B(\bigcdot) > \L_{n,2}(\bigcdot)$ on $[x_0, n^{1/60}]$ and (ii) $\sup_{x_0 \leq x \leq n^{1/60}}\hz{B}(x) \leq \h(x_0)$. We get an equivalent condition on projecting because the second component of $X$ is determined by the first along with the lower curve data.
This completes the proof of Lemma~\ref{l.first conditional statement}.
\end{proof}

The following lemma was used in the proof of Lemma~\ref{l.first conditional statement} and establishes that the distribution of the maximizer of $\h(\bigcdot)$ has no atoms.

\begin{lemma}\label{l.maximizer has no atoms}
Fix $z$ with $-n^{1/60}<z\leq n^{1/60}$. Almost surely, $\h(z) \neq \sup_{|x|\leq n^{1/60}}\h(x)$.
\end{lemma}

\begin{proof}
We use the formula \eqref{e.prelimiting gen profile formula} for $\h(z)$ in terms of a last passage problem through $\L_n$. It is a consequence of Lemma~\ref{l.P_n Brownian gibbs} that $\L_{n,j}(\bigcdot) - \L_{n,j}(z-1)$ is absolutely continuous with respect to Brownian motion of rate two on $[z-1,z+1]$ for each $j\in\intint{n}$ (see for example \cite[Proposition~4.1]{BrownianGibbs}). Thus we have with probability one that $z$ is not a local maximizer of $\L_{n,j}$ for any $j\in\intint{n}$, which is the event we work on now. Let $\Gamma$ be the geodesic associated to $\h(z)$ implicit in \eqref{e.prelimiting gen profile formula} and let $J\in\intint{n}$ be the index of the line that $\Gamma$ visits at time $z^-$ (if $J>1$, this means that the geodesic jumps to line 1 at location $z$, and so the top line's values do not contribute to the last passage value). Let $\tilde z$ belong to the interval of time that $\Gamma$ spends on line $J$ and be such that $\L_{n,J}(\tilde z) > \L_{n,J}(z)$. Now we can consider a modification of $\Gamma$ that has endpoint $(\tilde z,1)$, which it jumps to from $(\tilde z, J)$; this implies that $z$ is not a maximizer of $\h(\bigcdot)$ since $\h(\tilde z) > \h(z)$ and $|\tilde z| \leq n^{1/60}$.
\end{proof}

With these preliminaries, we now state what the $\F$-conditional distribution of $Z$ is; recall that $Z = \L_{n,1}(x_0+A+1)$.

\begin{lemma}\label{l.conditional distribution}
There exist $\F$-measurable random variables $\Corner^{\downarrow}$ and $\Corner^{\uparrow}$ such that the following holds. Conditionally on $\F$, and on the $\F$-measurable event that $x_0+A+2 \leq n^{1/60}$, the distribution of $Z$ is a normal random variable with mean $\frac{1}{2}(\L_{n,1}(x_0+A) + \L_{n,1}(x_0+A+2))$ and variance one, conditional on lying inside $[\Corner^{\downarrow}, \Corner^{\uparrow}]$. Further, when $z=\Corner^{\uparrow}$,
\begin{equation}\label{e.when z equals corner}
\sup_{x\in [x_0+A,x_0+A+2]} \hz{z}(x) = \h(x_0).
\end{equation}
\end{lemma}

In the discussion of the narrow-wedge case, the $\Corner^\uparrow$ and $\Corner^\downarrow$ random variables had a clear interpretation respectively as the largest value of $Z$ such that the reconstruction $\L_{n,1}^{Z}$ at no point exceeds $\L_{n,1}(x_0)$ and the smallest value of $Z$ such that the reconstruction at no point intersects $\L_{n,2}$; each variable handled one condition. For general initial conditions, these random variables play analogous but slightly different roles. In particular, they are respectively the largest and smallest values of $Z$ such that $\hz{Z}$ at no point exceeds $\h(x_0)$ and $\L_{n,1}^Z$ does not intersect $\L_{n,2}$.
However, it may not be the case that each variable separately handles one of the conditions: because of a larger class of possible geodesic paths, $\Corner^{\downarrow}$ may also play a role in preventing $\hz{Z}$ from exceeding $\h(x_0)$, unlike in the narrow-wedge case. This will be seen in $\Corner^{\downarrow}$'s definition in the proof, to which we turn  now.

\begin{proof}[Proof of Lemma~\ref{l.conditional distribution}]
From Lemma~\ref{l.first conditional statement}, we know that the process $\{\L_{n,1}(x) : x_0\leq x\leq n^{1/60}\}$, conditionally on $\F'$, has the law of Brownian bridge of rate two with appropriate endpoints conditioned on not intersecting $\L_{n,2}$ and on $\sup_{x_0\leq x\leq n^{1/60}}\h(x) \leq \h(x_0)$. Let $\mathsf{Val}$ (short for ``valid'') be this conditioning event, i.e.,
\begin{equation}\label{e.conditioning event definition}
\mathsf{Val} := \Big\{\L_{n,1}(x) > \L_{n,2}(x) \ \ \  \forall x\in[x_0, n^{1/60}]\Big\} \cap \left\{\sup_{x_0\leq x\leq n^{1/60}}\h(x) \leq \h(x_0)\right\}.
\end{equation}
Recall that $\F''$ is generated by $\F'$ and the additional data of $\L_{n,1}(\bigcdot)$ on $[x_0,x_0+A] \cup [x_0+A+2, n^{1/60}]$. Note that the $\sigma$-algebra $\F$ is generated by $\F''$ along with the additional side bridge data, i.e., $\L_{n,1}^{[x_0+A,x_0+A+1]}$ and $\L_{n,1}^{[x_0+A+1,x_0+A+2]}$.

Recall the following decomposition of a Brownian bridge $B$ of arbitrary endpoints and rate $\sigma^2$ on an interval $[a,b]$, with $c\in[a,b]$: conditionally on $B(a)$, $B(b)$, and the side bridges $B^{[a,c]}$ and $B^{[c,b]}$, the distribution of $B(c)$ is that of a normal random variable with mean $\frac{b-c}{b-a}B(a)+\frac{c-a}{b-a}B(b)$ and variance $\sigma^2\frac{(c-a)(b-c)}{b-a}$. This is because such a Brownian bridge can be decomposed into three independent parts: the side bridges $B^{[a,c]}$ and $B^{[c,b]}$ (both of which have the Brownian bridge law) and the value of $B(c)$ (which is normally distributed as specified).

This decomposition implies the following. By conditioning $\L_{n,1}(\bigcdot)$ on the side bridge data in addition to $\F''$, and on the $\F$-measurable event that $x_0+A+2 \leq n^{1/60}$, the $\F$-conditional distribution of $Z =\L_{n,1}(x_0+A+1)$ is that of a normal random variable with mean $\frac{1}{2}(\L_{n,1}(x_0+A) + \L_{n,1}(x_0+A+2))$ and variance one (as the Brownian bridge is of rate two), conditioned on $\mathsf{Val}$ occurring.

We claim that there exist $\F$-measurable random variables $\Corner^{\downarrow}$ and $\Corner^{\uparrow}$ such that the occurrence of $\mathsf{Val}$ is equivalent to $Z$ lying in $[\Corner^{\downarrow}, \Corner^{\uparrow}]$.

We start by focusing on the second event on the right-hand side of \eqref{e.conditioning event definition}. Consider a fixed upright path $\gamma$, and recall the definition of $\hz{z}(\gamma)$ from \eqref{e.L^f representation}. The second event is equivalent to the event that $Z$ is such that $\hz{Z}(\gamma) \leq \h(x_0)$ for each upright path $\gamma$ with endpoint lying in $[x_0, n^{1/60}]$. Now recall that, by Lemma~\ref{l.monotonicity of L^f}, with probability one, $\hz{z}(\gamma)$ is monotone (i.e., non-increasing, non-decreasing, or constant) in $z$. Thus the condition that $\hz{z}(\gamma) \leq \h(x_0)$ yields, for each such upright path $\gamma$, a condition that $z$ lies in an interval $I_\gamma$ of the form $(-\infty, \infty)$, $(-\infty, r_\gamma^{\uparrow})$, or $(r_\gamma^{\downarrow}, \infty)$ for some real number $r_\gamma^{\uparrow}$ or $r_{\gamma}^{\downarrow}$; which form of interval applies depends on the nature of the monotonicity of $\hz{z}(\gamma)$. Note that $r_\gamma^{\uparrow}$ and $r_{\gamma}^{\downarrow}$ are $\F$-measurable.

To satisfy the second event in the intersection defining $\mathsf{Val}$ in \eqref{e.conditioning event definition}, $Z$ must lie in the intersection of all of the $I_\gamma$ as $\gamma$ varies over the set of upright paths with endpoint in $[x_0, n^{1/60}]$. To satisfy the first event in \eqref{e.conditioning event definition}, we must also ensure that the value of $Z$ gives non-intersection of $\L_{n,1}^z(\bigcdot)$ with $\L_{n,2}(\bigcdot)$. Recall from the definition  \eqref{e.L^z formula} of $\L_{n,1}^z$ that $\L_{n,1}^z(x)$ is non-decreasing in $z$ for all $x$. Since the $\L_n$ curves are also ordered, it follows that satisfying the first event in the definition of $\mathsf{Val}$ is equivalent to $Z$ lying in an infinite ray $I_{\mathrm{lower}} = (r_{\mathrm{lower}},\infty)$. Further, $r_{\mathrm{lower}}$ is an $\F$-measurable random variable.

The idea now is to consider the intersection of $I_{\mathrm{lower}}$ and the intervals $I_{\gamma}$ corresponding to all paths $\gamma$ with endpoint in $[x_0, n^{1/60}]$. So, let
$$I = I_{\mathrm{lower}} \cap \bigcap_{\gamma} I_\gamma.$$
Since we need $I$ to be $\F$-measurable, we take the big intersection over only a countable collection of upright paths $\gamma$. More precisely, the intersection is taken over the set of upright paths $\gamma$ which have start point and endpoint lying in $\Q$ (with the endpoint in $[x_0, n^{1/60}]$) and all of whose jump times from one line to the next occur at values in~$\Q$. The continuity of the curves of $\L_n$ implies that this countable dense intersection is sufficient to ensure that $Z\in I$ implies the satisfaction of $\mathsf{Val}$.

In principle, $I$ may be the empty set. But it is not: by the definition of $x_0$ and $I$, and since the curves of $\L_n$ obey the non-intersection condition, the value $\L_{n,1}(x_0+A+1)$ almost surely lies in~$I$.

We define $\Corner^{\downarrow} = \inf I$ and $\Corner^{\uparrow} = \sup I$, which are clearly $\F$-measurable. We note the characterization of $\Corner^{\uparrow}$ as the largest value of $Z$ which satisfies the second event of $\mathsf{Val}$, i.e., that $\sup_{x_0\leq x \leq n^{1/60}}\hz{z}(x) \leq \h(x_0)$.

We are left with proving the last assertion \eqref{e.when z equals corner}. Note that it is immediate from the definition of $I$ that, when $Z=\Corner^{\uparrow}$, there exists some $x'\neq x_0$ such that $\hz{z}(x) = \h(x_0)$. We claim that at least one such $x'$ must lie in $[x_0+A,x_0+A+2]$. Suppose to the contrary that $\sup_{x\in [x_0+A,x_0+A+2]} \hz{z}(x) < \h(x_0)$. Consider $\hz{Z+\varepsilon}$ for small $\varepsilon>0$. We see from Lemma~\ref{l.monotonicity of L^f} that $\hz{Z+\varepsilon}(x) \leq \hz{Z}(x) \leq \h(x_0)$ for all $x\not\in [x_0+A,x_0+A+2]$. But for all small enough $\varepsilon$, we would still have $\sup_{x\in [x_0+A,x_0+A+2]} \hz{Z+\varepsilon}(x) < \h(x_0)$ by Lemma~\ref{l.lipschitz}, contradicting the definition of $I$ via the characterization of $\Corner^{\uparrow}$ noted in the previous paragraph. This completes the proof of Lemma~\ref{l.conditional distribution}.
\end{proof}

\subsection{Positive probability favourable data} \label{s.favourable data}

We next define a $\F$-measurable favourable event $\fav_{K, L}$, which we will show holds with positive probability. Recall that we are proving Proposition~\ref{p.resampling twin peaks}, which asserts a lower bound on the probability of the twin peaks' event. The argument for this proposition will rely on resampling some randomness, namely $Z = \L_{n,1}(x_0+A+1)$, conditionally on the data in $\F$. The role of the favourable event is to specify a class of good $\F$-measurable data under which the resampling can be analysed more easily.

We again adopt the shorthand of $x_0$ for $x_0^{n}$, and let $\mu := \frac{1}{2}\bigl(\L_{n,1}(x_0+ A) + \L_{n,1}(x_0+A+2)\bigr)$ be the $\F$-measurable mean of the normal random variable in the description of $Z$'s $\F$-conditional distribution from Lemma~\ref{l.conditional distribution}. Also let $\beta$ be as in the definition of $\bJ_L$ in \eqref{eq:setJ}. We set
$$\fav_{K,L} = \msf F_1 \cap \msf F_2 \cap \msf F_3' \cap \msf F_4',$$
where
\begin{align*}
\msf F_1 &= \Big\{\Corner^{\uparrow} \leq 4K\Big\}, \qquad &\msf F_2 &= \Big\{\mu \in [-K, K] \Big\},\\
\msf F_3' &= \Bigl\{|\h(x_0)| \leq \beta L^{1/2}\Bigr\}, \qquad &\msf F_4' &= \Big\{|x_0|\leq \beta L-A-2\Big\},
\end{align*}
(we will shortly define $\msf F_3$ and $\msf F_4$ to be modified versions of $\msf F_3'$ and $\msf F_4'$ which will be more convenient to work with).

Let us now say a few words on the form of $\fav_{K,L}$ and the proof idea of Proposition~\ref{p.resampling twin peaks}. From Lemma~\ref{l.conditional distribution}, we see that, when $Z = \Corner^{\uparrow}$, we have that $\sup_{x\in [x_0+A,x_0+A+2]} \h(x) = \h(x_0)$. Also, we see from Lemma~\ref{l.lipschitz} that reducing $Z$'s value from this level affects $\sup_{x\in [x_0+A,x_0+A+2]} \h(x)$ in a Lipschitz manner. Thus the event in Proposition~\ref{p.resampling twin peaks} occurs if $Z$ is within order $\varepsilon$ of $\Corner^{\uparrow}$.

We know from Lemma~\ref{l.conditional distribution} that, conditionally on $\F$, $Z$ is distributed as a normal random variable with mean $\mu$ and variance one conditioned on lying inside $[\Corner^{\downarrow}, \Corner^{\uparrow}]$. Thus to get a good lower bound on the event that $Z$ is close to $\Corner^{\uparrow}$, it is enough to know that, with positive probability, the mean of $Z$ is not too extreme and that the upper limit $\Corner^{\uparrow}$ is not too high. These are the first two events in the intersection defining $\fav_{K,L}$, and the mentioned positive probability lower bound is the content of the next lemma. The third event  handles the first extra condition in the definition of $\TP{A,L}^\eps$ \eqref{eq:setTP} on the value of $\h$'s maximum, while the final event in $\fav_{K,L}$'s intersection is imposed merely to ensure the second extra condition in \eqref{eq:setTP}, that twin peaks occurs in the interval $[-\beta L, \beta L]$.

\begin{lemma}\label{l.fav probability}
Let $\hlim_0:\R\to\R\cup\{-\infty\}$ satisfy Assumption~\ref{a:initial_state_pd} and consider $A>0$. There exist $K$ and $L_0$ (both depending on $\gamma$, $\theta$, and $A$) such that, for all $L>L_0$, there exists $n_0$ (depending on $\gamma$, $\theta$, and $L$) so that, for all $n>n_0$,
$$\P(\fav_{K,L}) \geq \frac{1}{2}.$$
Further, $K$ and $L_0$ may be made to depend on $\gamma$, $\theta$ and $A$ in a continuous manner.
\end{lemma}

\begin{proof}
We start by specifying some further good events. Let $y_0\in[-\theta,\theta]$ be some fixed real number such that  $\hlim_0(y_0) \geq -2\theta$. For $M>0$ to be specified later, define
\begin{align*}
\msf F_3 &= \bigl\{|\h(x_0)| \leq K\bigr\}, \qquad &\msf F_4 &= \big\{|x_0| \leq M\big\},\\ \qquad 
\msf F_5 &= \Big\{\S_n(y_0, -M) \geq -K/2\Big\}, \qquad &\msf F_6 &= \Big\{\L_{n,1}(-M) \leq K/2\Big\}.
\end{align*}
Note that $\fav_{K, L}\supseteq \msf F_1\cap\msf F_2\cap\msf F_3\cap \msf F_4$ when $\beta L\geq M+A+2$ and $\beta L^{1/2}\geq K$. We set $L_0$ high enough that both conditions on $L$ are met whenever $L\geq L_0$ (when $K$ and $M$ are set, which will be done in a way not depending on $L$).

We will show that $\P\left(\fav_{K, L}^c\right) \leq \tfrac{1}{2}$ for large enough $K$ and $L$ by showing the stronger statement that, for $\delta=1/10$ and appropriate choices of $K$, $L$, and $M$, we have
$$\P\left(\bigcup_{i=1}^6 \msf F_i^c\right) \leq 5\delta = \frac{1}{2}.$$

\noindent\emph{Bounding $\P(\msf F_4^c)$}: This is a simple application of Lemma~\ref{l.tightness of prelimiting maximizer}, which yields $M = M(\gamma,\theta)$ such that $\P(\msf F_4^c) \leq \delta = \frac{1}{10}$ for $n\geq n_0 = n_0(\gamma,\theta)$. We fix the value of $M$ obtained here for the rest of the proof.
\medskip

\noindent\emph{Bounding $\P(\msf F_2^c\cap \msf F_4)$}: On the event $\msf F_4$ we have $|x_0|\leq M$. Also, for large enough $K$ depending only on $\delta=1/10$, $A$, and $M$,
$$\P\left(\inf_{|x|\leq M+A+2} \L_{n,1}(x) < -K\right) \leq \delta/2 \quad\text{and}\quad \P\left(\sup_{|x| \leq M+A+2} \L_{n,1}(x) > K\right) \leq \delta/2;$$
this is implied by Proposition~\ref{l.BrLPP uniform tail bound} after recalling from \eqref{e.P S identity} that $\L_{n,1}(x) = \S_n(0, x)$.
Thus $\P(\msf F_2^c\cap \msf F_4)$ is at most $\delta$, since $\mu = \frac{1}{2}(\L_{n,1}(x_0+ A) + \L_{n,1}(x_0+A+2))$ is bounded above and below on $\msf F_4$ by $\sup_{x\in[-M-A-2, M+A+2]} \L_{n,1}(x)$ and $\inf_{x\in[-M-A-2, M+A+2]} \L_{n,1}(x)$.
\medskip

\noindent\emph{Bounding $\P(\msf F_5^c)$ and $\P(\msf F_6^c)$}: These correspond respectively to lower and upper tails on one-point last passage values, i.e., on $\S_n(y,x)$ for fixed $y$ and $x$, because $\L_{n,1}(-M) = \S_n(0, -M)$ in view of \eqref{e.P S identity}. Thus, we obtain $\P(\msf F_5^c) \leq \delta$ and $\P(\msf F_6^c) \leq \delta$ by applying Proposition~\ref{l.BrLPP uniform tail bound} in a closed interval of unit length around the starting and ending points and by setting $K$ high enough, depending on $\theta$.
\medskip

\noindent\emph{Bounding $\P(\msf F_3^c\cap \msf F_4 \cap \msf F_5)$}: We first bound the probability that $\h(x_0) \geq K$ on $\msf F_4$. Recall that, by assumption, $ \hlim_0(y) \leq - \gamma |y|^{1/2}$ for all $y\in \R$. 
Note that, on $\msf F_4$,
$\h(x_0) = \sup_{|x|\leq M} \h(x).$
Thus,
\begin{align*}
\sup_{|x|\leq M}\h(x)
= \sup_{\substack{|x|\leq M\\ 0\leq y \leq n^{1/60}}}\Big(\hlim_0(y) + \S_n(y, x)\Big)
&\leq \sup_{\substack{|x|\leq M\\ 0\leq y \leq n^{1/60}}} \left(\S_n(y,x) - \gamma |y|^{1/2}\right).
\end{align*}
By a union bound, we can bound $\P\left(\sup_{|x|\leq M} \h(x) > K\right)$ above by
\begin{align*}
\sum_{j=0}^{\lceil n^{1/60}\rceil } \P\Bigg(\sup_{\substack{|x|\leq M\\ y \in[j,j+1]}} \left(\S_n(y,x) - \gamma |y|^{1/2}\right) > K\Bigg) 
&\leq \sum_{j=0}^{\lceil n^{1/60}\rceil } \P\Bigg(\sup_{\substack{|x|\leq M\\ y \in[j,j+1]}} \S_n(y,x) > K + \gamma j^{1/2}\Bigg)\\
&\leq \sum_{j=0}^{\lceil n^{1/60}\rceil } C\max(M^2, j^2) \exp\left(-c(K^{3/2} + j^{3/4})\right),
\end{align*}
the final inequality by an application of Proposition~\ref{l.BrLPP uniform tail bound}. The final expression can be made less than $\delta=\frac{1}{10}$ by raising $K$ appropriately (depending on $M$ and $\gamma$) if needed. Doing so, we learn that $\P(\{\h(x_0)\geq K\}\cap \msf F_4) \leq \delta$.

Next we bound the probability that $\h(x_0) \leq -K$ on $\msf F_5$.  Recall that $y_0\in [-\theta,\theta]$ is such that $\fh_0(y_0) \geq -2\theta$, and increase $K$ if needed so that $\fh_0(y_0) \geq -K/3$. Since $x_0$ is the maximizer of $\h$, and  \eqref{e.prelimiting gen profile formula} holds, we see that
$$\h(x_0) \geq \fh_0(y_0) + \S_n(y_0, -M) \geq -5K/6 > -K,$$
the last inequality holding on $F_5$. Thus $\P(\{\h(x_0) \leq -K\}\cap F_5) = 0$. Overall we have shown that $\P(\msf F_3^c \cap \msf F_4\cap \msf F_5) \leq \delta$.

\medskip

\noindent\emph{Bounding $\P(\msf F_1^c\cap \msf F_3\cap \msf F_4\cap \msf F_5\cap \msf F_6)$}: Recall that we have set $y_0$ and $K$ such that $\hlim_0(y_0) \geq -K/3 > -K$. Observe that $\smash{\Corner^{\uparrow}} > 4K$ implies that there is a value of $z$ in $[4K, \infty)$ such that $\h(x_0) = \sup_{x\in [x_0+A,x_0+A+2]} \hz{z}(x)$ by Lemma~\ref{l.conditional distribution}. But we will now show that if $z\geq 4K$, then, on the event $\bigcap_{i=3}^6 \msf F_i$, we have that $\hz{z}(x_0+A+1) > \h(x_0)$; this is a contradiction and so the probability we are bounding must be zero.
We use the formula for $\hz{z}$ from \eqref{e.L^f representation}, and the formula \eqref{e.melon LPP and S_n} relating $\S_n$ and LPP values through $\L_n$. Indeed, if $z\geq 4K$ and the event $\bigcap_{i=3}^6 \msf F_i$ holds, then
\begin{align*}
\hz{z}(x_0+A+1) &= \sup_{0\leq y \leq n^{1/60}} \Big(\hlim_0(y) + \L_n^{z}[(-\tfrac12n^{1/3}+ y, n) \to (x_0+A+1, 1)]-n^{2/3}\Big)\\
&\geq -K + \left(\L_n^z[(-\tfrac12n^{1/3}+ y_0, n) \to (-M, 1)] - n^{2/3}\right)\\
&\quad + \L_n^z[(-M, 1) \to (x_0+A+1,1)]\\
&= -K + \S_n(y_0, -M) + z - \L_{n,1}(-M)\\
&\geq 3K + \S_n(y_0, -M)  - \L_{n,1}(-M) \geq 2K.
\end{align*}
The first inequality bounded the supremum by the choice of $y=y_0$ and used our assumption that $\hlim_0(y_0) > -K$; the penultimate inequality used the assumption that $z\geq4K$;  and the final inequality used the bounds that hold on $\msf F_5\cap\msf F_6$.
%
The conclusion $\hz{z}(x_0+A+1)  \geq 2K$ contradicts $\h(x_0) \leq K$, which holds on $\msf F_4$, since, on this event, $x_0+A+1 \in [-M,M]$. Thus the probability we are bounding is zero.

Overall we have shown that $\P\big(\bigcup_{i=1}^6 \msf F_i^c\big) \leq 5\delta = 1/2$. It may be easily checked that the setting of $K$ and $L_0$ can be made to depend on $\gamma$, $\theta$, and $A$ continuously, completing the proof of Lemma~\ref{l.fav probability}.
\end{proof}

\subsection{Performing the resampling: the proof of Proposition~\ref{p.resampling twin peaks}}

In this proof,
we will need a monotonicity property of conditional probabilities of the normal distribution. The proof is a straightforward calculation that we omit here, but details are available in \cite[Lemma~5.15]{calvert2019brownian}.

\begin{lemma}\label{l.normal monotonicity}
Fix $r>0$, $m\in\R$, and $\sigma^2>0$, and let $X$ be distributed as $N(m, \sigma^2)$. Then the quantity $\P(X \geq s-r \mid X\leq s)$ is a strictly decreasing function of $s\in\R$.
\end{lemma}

\begin{proof}[Proof of Proposition~\ref{p.resampling twin peaks}]
We fix $K$ and $L_0$ as given by Lemma~\ref{l.fav probability}. For any given $L>L_0$, we have that $\P(\fav_{K,L}) \geq 1/2$.

Recall $x_0  = x_0^{n}= \argmax_{|x|\leq n^{1/60}} \h(x)$. We have that
\begin{align*}
\MoveEqLeft
\P\left(\sup_{x\in [x_0+A,x_0+A+2]} \h(x) > \h(x_0) - \varepsilon;\, |\h(x_0)| \leq \beta L^{1/2};\, |x_0|\leq \beta L-A-2\right)\\
&\hspace{-0.3cm}= \E\left[\P\left(\sup_{x\in [x_0+A,x_0+A+2]} \h(x) > \h(x_0) - \varepsilon \ \bigg|\ \F\right)\one_{|\h(x_0)| \leq \beta L^{1/2}\, |x_0|\leq \beta L-A-2}\right].
\end{align*}
We have to bound below the inner conditional probability. Set $\fav = \fav_{K,L}$ for notational convenience. Recall that the occurrence of $\fav$ implies that $|x_0|\leq \beta L-A-2$ and $|\h(x_0)| \leq \beta L^{1/2}$. We claim that, on $\fav$, the event of the inner conditional probability is implied by
\begin{equation}\label{e.I_epsilon}
Z := \L_{n,1}(x_0+A+1) \in \left[\Corner^{\downarrow}\vee (\Corner^{\uparrow}-\varepsilon/2),\  \Corner^{\uparrow}\right]=:I_{\varepsilon}.
\end{equation}
The validity of the claim follows from two facts. The first is that $\sup_{x\in [x_0+A,x_0+A+2]} \hz{z}(x) = \h(x_0)$ when $z=\Corner^{\uparrow}$ (from Lemma~\ref{l.conditional distribution}); and the second is that,  almost surely for all $z\in\R$, $\left|\sup_{x\in[x_0+A,x_0+A+2]} \hz{z_1}(x) - \sup_{x\in[x_0+A,x_0+A+2]} \hz{z_2}(x)\right|$ is at most $2|z_1-z_2|$ (from Lemma~\ref{l.lipschitz}). We apply this to $z_1 = \Corner^{\uparrow}$ and $z_2\in I_\varepsilon$.

With this preparation, we see that
\begin{align}
\MoveEqLeft[12]
\P\left(\sup_{x\in [x_0+A,x_0+A+2]} \h(x) > \h(x_0) - \varepsilon\ \bigg|\ \F\right)\one_{|\h(x_0)|\leq \beta L^{1/2},\, |x_0|\leq \beta L-A-2}\nonumber\\
&\geq \P\left(\sup_{x\in [x_0+A,x_0+A+2]} \h(x) > \h(x_0) - \varepsilon\ \bigg|\ \F\right)\cdot \one_{\fav}\nonumber\\
&\geq \P\big(Z \in I_{\varepsilon}\ \big|\ \F\big)\cdot \one_{\fav} . \label{e.Z in interval prob}
\end{align}
Recall now from Lemma~\ref{l.conditional distribution} that $Z$ is distributed as a normal random variable with mean $\mu = \frac{1}{2}(\L_{n,1}(x_0+A) + \L_{n,1}(x_0+A+2))$ and variance one, conditioned on lying inside $[\Corner^{\downarrow}, \Corner^{\uparrow}]$. Observe from \eqref{e.I_epsilon} that $I_\varepsilon$ is one of two intervals: $[\Corner^{\downarrow}, \Corner^{\uparrow}]$ or $[\Corner^{\uparrow}-\varepsilon/2, \Corner^{\uparrow}]$. In the first case, the conditional probability in \eqref{e.Z in interval prob} equals one. We show now that, in the second case, the conditional probability is bounded below by $\smallconst\varepsilon$ for some constant $\smallconst>0$.

We let $N$ be a standard normal random variable with mean zero and variance one. Then, on the event $\fav\cap\{\Corner^{\downarrow} < \Corner^{\uparrow}-\varepsilon/2\}$, we have that $|\mu| \leq K$ and $\Corner^{\uparrow} \leq 4K$, which implies that, on the same event,
\begin{align*}
\P\big(Z \in I_\varepsilon\ \big|\ \F\big)
&= \P\left(N +\mu \in [\Corner^{\uparrow} -\varepsilon/2, \Corner^{\uparrow}]\  \Big|\  N+\mu \in [\Corner^{\downarrow}, \Corner^{\uparrow}], \F\right)\\
&= \frac{\P\left(N +\mu \in [\Corner^{\uparrow} -\varepsilon/2, \Corner^{\uparrow}]\  \Big|\ \F\right)}{\P\left(N+\mu \in [\Corner^{\downarrow}, \Corner^{\uparrow}]\  \Big|\   \F\right)}\\
&\geq \P\left(N+\mu \geq \Corner^{\uparrow} -\varepsilon/2 \  \Big|\  N+\mu \leq \Corner^{\uparrow}, \F\right)\\
&\geq \P\left(N +\mu  \geq 4K -\varepsilon/2\  \Big|\ N+\mu \leq 4K, \F\right)\\
&\geq  \P\big(N +\mu  \in [4K -\varepsilon/2, 4K]\ \big|\ \F\big).
\end{align*}
For the first equality, we interpret $\P(Z\in \cdot \mid \F)$ as a regular conditional distribution, which exists as $Z$ takes values in $\R$ (see \cite[Theorem~6.3]{Kallenberg}); conditioning on an event $E$  is then understood by the usual equality of conditional probability with a ratio of probabilities, i.e., $\P(\cdot \mid E, \F) = \P(\cdot\cap E\mid \F)/\P(E\mid \F)$. Then the first equality follows from the characterization of $Z$'s law recalled above after \eqref{e.Z in interval prob}. The second equality can be seen by noting that we are working on the event that $\Corner^\downarrow < \Corner^\uparrow-\varepsilon/2$, and some simple manipulations of the probabilities in the ratio gives the third line, i.e., the first inequality.
The penultimate inequality used the monotonicity property of normal random variables recorded in Lemma~\ref{l.normal monotonicity} and that $\Corner^{\uparrow}\leq 4K$ on $\fav$. Now the form of the normal density gives that the final expression is bounded below by $\smallconst\varepsilon\cdot\one_\fav$ for some $\smallconst>0$ depending only on $K$, since $|\mu|\leq K$; further, this dependence is clearly continuous in $K$.

Substituting  into \eqref{e.Z in interval prob} this bound, as well as the earlier bound of one in the case that $\Corner^{\downarrow}\geq \Corner^{\uparrow}-\varepsilon/2$, gives that
\begin{align*}
\MoveEqLeft[24]
\P\left(\sup_{x\in [x_0+A,x_0+A+2]} \h(x) > \h(x_0) - \varepsilon;\, |\h(x_0)| \leq \beta L; \, |x_0|\leq \beta L-A-2\right)\\
&\geq \smallconst\varepsilon\cdot \P\big(\fav_{K,L}\big) \geq \frac{1}{2}\smallconst \varepsilon,
\end{align*}
in view of $K$ and $L$ being such that $\P(\fav_{K,L}) \geq 1/2$. Relabelling $\smallconst$ completes the proof of Proposition~\ref{p.resampling twin peaks}.
\end{proof}


\section{Proof of Theorem~\ref{thm:main}}
\label{sec:main_proof}



We start by making some definitions. Let the set $\TTP_{\!A, L}^{\eps}$ be defined in \eqref{eq:setTP} and let $\TTP_{\!A, L} := \TTP_{\!A, L}^{0}$. Then for $L \geq A > 0$, $T > T_0 > 0$ and $\eps \geq 0$, we define the set
\begin{equation}\label{eq:times_bounded}
\CT_{[T_0, T], A, L}^{\eps}(\fh) := \bigl\{t \in [T_0,T] : \fh_t \in \TTP_{\!A, L}^{\eps}\bigr\}.
\end{equation}
We denote $\CT_{[T_0, T], A, L}(\fh) := \CT^{0}_{[T_0, T], A, L}(\fh)$. 


%
For lighter notation, we suppress the initial data $\fh_0$ from the probability measure $\P_{\fh_0}$.

We will first prove \eqref{eq:dimTP} for the set $\CT_{[T_0, T], A, L}(\fh)$ defined in \eqref{eq:times_bounded}, where $A > 0$ and $L \geq A$. All constants which appear in this proof can depend on $T_0$, $T$, $L$ and $A$, and we will not indicate these dependences. If $\CT_{[T_0, T], A, L}(\fh) = \emptyset$, then $\dim(\CT_{[T_0, T], A, L}(\fh)) = 0$ and the required bound is trivial. Hence, we will consider the case $\CT_{[T_0, T], A, L}(\fh) \neq \emptyset$.
\medskip

\emph{Step 1: From twin peaks at exceptional time to $\eps$-twin-peaks at deterministic time.} We start by introducing a dyadic partition of the time interval $[T_0, T]$. More precisely, for any integer  $n \geq 1$, we define the set
$$\CA_n := \{i \in \nn : 1 \leq i \leq \lceil 2^n (T - T_0) \rceil\}$$
and the times $\ft_{n, i} := T_0 + 2^{-n} i$ indexed by $i \in \CA_n$. Furthermore, for these values of $n$ and $i$, we define the set
\begin{equation}
\CJ_{n, i} := \bigl\{t \in [T_0,T] :  |\ft_{n, i} - t| \leq (2+\tfrac{1}{n})^{-n - 1}\bigr\},
\end{equation}
which contains a small left-neighbourhood of $\ft_{n, i}$; the $1/n$ term provides a little padding that will be useful shortly. We will show that, if there is an exceptional time present in $\CJ_{n, i}$, i.e., $\CJ_{n, i}\cap \CT_{[T_0,T],A,L}(\fh) \neq \emptyset$, then $\fh^L_{\ft_{n,i}} \in \TP{A, L + 1}^{\eps}$ for a suitably chosen $\eps$, where $\fh^L$ is the cutoff defined in \eqref{eq:cut-off}.

If $\CJ_{n, i} \cap \CT_{[T_0, T], A, L}(\fh) \neq \emptyset$, then let $\tau$ be the smallest random time point in $\CJ_{n, i} \cap \CT_{[T_0, T], A, L}(\fh)$, and let $\chi_1, \chi_2 \in \bJ_L$ (this set is defined in \eqref{eq:setJ}) be two random spatial points such that $|\chi_1 - \chi_2 | \geq A$ and $\Max(\fh_\tau) = \fh_\tau(\chi_1) = \fh_\tau(\chi_2)$. Since the pair of points $\chi_1$, $\chi_2$ may not be  unique, we choose it so that $\chi_1 + \chi_2$ takes the minimum value.

Recall that, for any $\nu > 0$, the KPZ fixed point is almost surely locally $(\tfrac{1}{3}-\nu)$-H\"{o}lder continuous in the time variable uniformly in the space variable (see Lemma~\ref{lem:Holder-time}). Owing to this, and also to $\ft_{n,i}, \tau \in [T_0, T+1]$, and $\chi_\ell \in \bJ_L$, it follows that, for $\ell = 1$ and~$2$,
\begin{equation}\label{eq:Holderr}
|\fh^L_{\ft_{n,i}}(\chi_\ell) - \fh^L_{\tau}(\chi_\ell)| = |\fh_{\ft_{n,i}}(\chi_\ell) - \fh_{\tau}(\chi_\ell)| \leq \CC_1 2^{-(\frac{1}{3}-\nu) n}
\end{equation}
almost surely. Here,  $\CC_1 \geq 0$ is an almost surely finite random constant that is independent of $n$ and~$i$ but that depends on $L$, $T$ and $\nu$. Moreover, $\ft_{n,i}, \tau \in [T_0, T+1]$ and Lemma~\ref{lem:max_Holder} yield
\begin{equation}\label{eq:C_2}
	\bigl| \Max(\fh^L_{\ft_{n,i}}) - \Max(\fh^L_{\tau}) \bigr| \leq \CC_2 2^{-(\frac{1}{3}-\nu) n},
\end{equation}
for $\CC_2 \geq 0$, which  is an almost surely finite random constant that is independent of $n$ and $i$ but that depends on $L$, $T$ and $\nu$.

Combining \eqref{eq:Holderr} and \eqref{eq:C_2}, we find that, for $\ell = 1$ and $2$,
\begin{equation}
\Max(\fh^L_{\ft_{n,i}}) - \fh^L_{\ft_{n,i}}(\chi_\ell) \leq \bigl(\fh^L_{\tau}(\chi_\ell) - \fh^L_{\ft_{n,i}}(\chi_\ell)\bigr) + \bigl(\Max(\fh^L_{\ft_{n,i}}) - \Max(\fh^L_{\tau})\bigr) \leq \CC_3 2^{-(\frac{1}{3}-\nu) n},
\end{equation}
where $\CC_3 = \CC_1 + \CC_2$ is an almost surely finite random constant that is independent of $n$ and $i$ but that depends on $L$, $T$ and $\nu$. Let $K$ be a deterministic constant (which we will eventually send to infinity). Then, for all $n$ sufficiently large, we have that $\fh^L_{\ft_{n,i}} \in \TP{A, L + 1}^{\eps}$ almost surely on the event that $\left\{\CJ_{n, i} \cap \CT_{[T_0, T], A, L}(\fh) \neq \emptyset\right\}\cap\left\{\CC_3 \leq K\right\}$, where $\eps = K 2^{-( 1 / 3-\nu)n}$. The increase of $L$ by $1$ comes from \eqref{eq:C_2}: if $\Max(\fh_{\tau}) \in \bJ_{L^{1/2}}$, then $\Max(\fh^L_{\ft_{n,i}}) \in \bJ_{(L + 1)^{1/2}}$ as soon as $\eps$ is sufficiently small.
\medskip

\emph{Step 2: The open cover and its Hausdorff pre-measure.} We estimate the Hausdorff pre-measure \eqref{eq:pre-measure} of $\CT_{[T_0, T], A, L}(\fh)$ by choosing a suitable open covering. For this, we define the set
\begin{equation}
\CJ^*_{n, i} := \bigl\{t \in [T_0,T] : |t - \ft_{n, i}| < 2^{-n - 1}\bigr\},
\end{equation}
which is open and has diameter $2^{-n}$; note also the inclusion $\CJ_{n, i} \subset \CJ^*_{n, i}$ due to the presence of the padding term $1/n$ in the former's definition.

We further define the set $\CI_n := \{i \in \CA_n : \CJ_{n, i} \cap \CT_{[T_0, T], A, L}(\fh) \neq \emptyset\}$.
Then one can see that the following open covering holds: $\CT_{[T_0, T], A, L}(\fh)\subset \bigcup_{i \in \CI_n} \CJ^*_{n, i}$. Moreover, for any $\alpha > 0$, we have that
\begin{equation}\label{eq:m-bound}
m_{\alpha, n} \bigl(\CT_{[T_0, T], A, L}(\fh) \bigr) := \sum_{i \in \CI_n} \diam(\CJ^*_{n, i})^\alpha = 2^{-n\alpha} |\CI_n|,
\end{equation}
where $\diam$ is the diameter of a set. The definition of the Hausdorff pre-measure \eqref{eq:pre-measure} yields
\begin{equation}\label{eq:UpperBound2}
\CH^\alpha_{\delta_n}\bigl(\CT_{[T_0, T], A, L}(\fh) \bigr) \leq m_{\alpha, n} \bigl(\CT_{[T_0, T], A, L}(\fh) \bigr),
\end{equation}
where $\delta_n = 2^{-n - 1}$, which vanishes as $n \to \infty$.

Next we estimate the quantity $m_{\alpha, n}(\CT_{[T_0, T], A, L}(\fh))$. Let us define the set
$$\CI^*_n := \bigl\{i \in \CA_n : \fh^L_{\ft_{n,i}} \in \TP{A, L + 1}^{\eps}\bigr\},$$
where $\eps = K 2^{-(1 / 3-\nu)n}$ as above. Then, as we proved in Step 1, $\CI_n \subset \CI^*_n$ almost surely on the event $\{\CC_3 \leq K\}$. Hence \eqref{eq:m-bound} yields, on the same event,
\begin{equation}
m_{\alpha, n}\bigl(\CT_{[T_0, T], A, L}(\fh)\bigr) \leq 2^{-n\alpha}|\CI^*_n|
\end{equation}
almost surely. Taking expectations, we obtain
\begin{equation}\label{eq:m_bound}
\E\bigl[m_{\alpha, n}\bigl(\CT_{[T_0, T], A, L}(\fh)\bigr)\one_{\CC_3 \leq K}\bigr] \leq 2^{-n\alpha} \E\bigl[|\CI^*_n|\one_{\CC_3\leq K}\bigr],
\end{equation}
where
\begin{equation}
\E\bigl[|\CI^*_n|\one_{\CC_3\leq K}\bigr] = \sum_{i \in \CA_n} \P \bigl( \fh^L_{\ft_{n,i}} \in \TP{A, L + 1}^{\eps}, \CC_3\leq K\bigr) \leq \sum_{i \in \CA_n} \P \bigl( \fh^L_{\ft_{n,i}} \in \TP{A, L + 1}^{\eps}\bigr).
\end{equation}
 Using Theorem~\ref{thm:densities2-intro}, we see that each summand is at most $C \epsilon \ft_{n,i}^{-1/3}$ for a constant $C$. Now, by the definitions of $\ft_{n,i} =T_0+ 2^{-n} i$
 and $\eps= K 2^{-(1 / 3-\nu)n}$, we get, for a constant $C_3$,
 $$\P\bigl(\fh^L_{\ft_{n,i}} \in \TP{A, L + 1}^{\eps}\bigr) \leq C  K 2^{-(\frac{1}{3}-\nu)n} (T_0+2^{-n} i)^{-1/3} \leq C_3K 2^{\nu n} i^{-1/3}.$$
 In the second inequality, we used  $(T_0+2^{-n}i)^{-1/3} \leq 2^{n/3} i^{-1/3}$.  Hence,
\begin{equation}
\E\bigl[|\CI^*_n|\one_{\CC_3\leq K}\bigr]  \leq C_3 K2^{\nu n} \sum_{i \in \CA_n} i^{-1/3} \leq C_4 K2^{\nu n} (2^n T)^{\frac{2}{3}} = C_4 KT^{\frac{2}{3}} 2^{n(\frac{2}{3} + \nu)},
\end{equation}
for a new constant $C_4 \geq 0$. Combining this bound with \eqref{eq:m_bound}, we obtain
\begin{align}
\E\bigl[m_{\alpha, n}\bigl(\CT_{[T_0, T], A, L}(\fh)\bigr)\one_{\CC_3 \leq K}\bigr] &\leq 2^{-n\alpha}\cdot C_4 KT^{\frac{2}{3}}  2^{n(\frac{2}{3} + \nu)}
= C_4 KT^{\frac{2}{3}} 2^{n (\frac{2}{3}-\alpha + \nu)}.
\end{align}
For any $\alpha > \frac{2}{3}$, we may choose $0 < \nu < \alpha-\frac{2}{3}$, so that the last expression vanishes as $n \to \infty$. Combining this with \eqref{eq:UpperBound2} and the definition \eqref{eq:measure}, we conclude
$$\E\bigl[\CH^\alpha_{\delta_n}(\CT_{[T_0, T], A, L}(\fh))\one_{\CC_3\leq K}\bigr] \leq  C_4 KT^{\frac{2}{3}} 2^{n (\frac{2}{3}-\alpha + \nu)}$$
for any $\alpha > \frac{2}{3}$. Then the monotone convergence theorem and the almost sure finiteness of $\CC_3$ yield, for any $\alpha > \frac{2}{3}$, that
\begin{equation}\label{eq:upper-bound-expectation}
\E\bigl[\CH^\alpha \bigl(\CT_{[T_0, T], A, L}(\fh)\bigr)\bigr] = \lim_{K\to\infty}\lim_{n \to \infty} \E\bigl[\CH^\alpha_{\delta_n}(\CT_{[T_0, T], A, L}(\fh))\one_{\CC_3\leq K}\bigr] = 0.
\end{equation}
By \eqref{eq:Hausdorff_dim}, this shows, for every $T_0>0$, $A>0$, and $L>0$, that the Hausdorff dimension of $\CT_{[T_0, T], A, L}(\fh)$ is almost surely at most $2/3$.
\medskip

\emph{Step 3: Inferring the Hausdorff dimension upper bound for $\CT_{T,A}(\fh)$.} Recall the countable stability property of Hausdorff dimension \eqref{eq:monotone}, which says that the Hausdorff dimension of a countable union of spaces is the supremum of the Hausdorff dimensions of the individual spaces. This yields that the Hausdorff dimension of $\CT_{T, A}(\fh)$ is at most $2/3$ after noting that this set can be written as a countable union of sets of the form $\CT_{[T_0, T], A, L}(\fh)$:
\begin{equation}\label{eq:countable stability decomp}
\CT_{T,A}(\fh) = \bigcup_{L=1}^{\infty} \CT_{[L^{-1}, T], A, L}(\fh).
\end{equation}
This equality is implied by the straightforward set monotonicity properties of $\CT_{[T_0, T], A, L}(\fh)$ in each of the variables $T_0$, $A$, and $L$. Thus the Hausdorff dimension of $\CT_{T,A}(\fh)$ is almost surely at most $2/3$ when $A>0$. If $A=0$, we can replace the $A$ in the right-hand side \eqref{eq:countable stability decomp} with $L^{-1}$ and the same reasoning applies. This completes the proof of Theorem~\ref{thm:main}.

\appendix

\section{Hausdorff measure and dimension}
\label{sec:Hausdorff}

We recall the definitions of Hausdorff measure and Hausdorff dimension for subsets of~$\R$---a treatment of Hausdorff dimension in $\R^d$ may be found in \cite[Section~4]{Mattila}.

\begin{defn}\thmtitle{Hausdorff measure and dimension} 
 Let $S \subset \R$ and $0 \leq \alpha < \infty$.
\begin{enumerate}
\item For $0 < \delta \leq \infty$, we define the \emph{$\delta$-approximate packing pre-measure}
\begin{equation}\label{eq:pre-measure}
\CH^\alpha_\delta(S) := \inf \left\{ \sum_{i = 1}^\infty \diam (J_i)^\alpha : S \subset \bigcup_{i = 1}^\infty J_i,\; \diam (J_i) < \delta \right\},
\end{equation}
where the infimum is taken over all countable covers of $S$ by sets $J_i \subset \R$ with $\diam(J_i) < \delta$, and where $\diam(J_i)$ is the supremum of all distances between points in $J_i$.
\item We define the \emph{$\alpha$-dimensional Hausdorff measure}
\begin{equation}\label{eq:measure}
\CH^\alpha(S) := \lim_{\delta \to 0} \CH^\alpha_\delta(S) = \sup_{\delta > 0} \CH^\alpha_\delta(S).
\end{equation}
We note that $\CH^\alpha(S)$ is well-defined (although it can be infinite), because $\CH^\alpha_\delta(S)$ monotonically increases as $\delta \searrow 0$.
\item The \emph{Hausdorff dimension} of $S$ is defined by the following equivalent formulas:
\begin{align}\label{eq:Hausdorff_dim}
\dim(S) &:= \sup \{\alpha \geq 0 : \CH^\alpha(S) > 0\} = \sup \{\alpha \geq 0 : \CH^\alpha(S) = \infty\} \\
&= \inf \{\alpha \geq 0 : \CH^\alpha(S) < \infty\} = \inf \{\alpha \geq 0 : \CH^\alpha(S) = 0\}.
\end{align}
Note that $0 \leq \dim(S) \leq 1$ for any $S \subset \R$.
\end{enumerate}
\end{defn}

\begin{rem}\label{rem:dimension}
We will need the following facts about the Hausdorff dimension, which are implied by \eqref{eq:Hausdorff_dim}:
\begin{enumerate}
\item if $\CH^\alpha(S) < \infty$, then $\dim(S) \leq \alpha$;
\item\label{it:2} if $\CH^\alpha(S) > 0$, then $\dim(S) \geq \alpha$.
\end{enumerate}
\end{rem}

Clearly, the Hausdorff measure is monotone in the sense that $\CH^\alpha(S) \leq \CH^\alpha(S')$ for $S \subset S'$. Moreover, the Hausdorff dimension enjoys monotonicity and countable stability properties, which can be found below Definition $4.8$ in \cite{Mattila}:
\begin{equation}\label{eq:monotone}
\dim(S) \leq \dim(S'), \qquad\qquad \dim \left( \bigcup_{i \in \nn} S_i \right) = \sup_{i \in \nn} \dim (S_i),
\end{equation}
for any $S \subset S'$ and for any sets $S_i \subset \R$.



\section{Trace class operators and Fredholm determinants}
\label{sec:Fredholm}

We list several properties of Fredholm determinants of which we make use. For more information, see \cite{Simon}.
We first provide basic definitions. For the separable Hilbert space $\CH = L^2(\R)$, the trace norm of a bounded linear operator $A$ is $\| A \|_1 := \sum_{n = 1}^\infty \langle e_n, |A| e_n \rangle$, where $\{e_n\}_{n \geq 1}$ is any orthonormal basis of $\CH$, and $|A| = \sqrt{A^* A}$ is the unique positive square root of the operator $A^* A$. The space of trace class operators $\CB_1(\CH)$ contains such $A$ for which $\| A \|_1 < \infty$. For $A \in \CB_1(\CH)$, the trace is defined by $\tr(A) := \sum_{n = 1}^\infty \langle e_n, A e_n \rangle$ and the Hilbert-Schmidt norm is $\| A \|_2 := \sqrt{\tr(|A|)^2}$. The space of Hilbert-Schmidt operators $\CB_2(\CH)$ contains such $A$ for which $\| A \|_2 < \infty$.

Let $\| A \|_\op$ be the operator norm of $A$. Then we have the following relations between the norms:
\begin{equation}\label{eq:norms}
\| A \|_\op \leq \| A \|_2 \leq \| A \|_1.
\end{equation}
If $B$ is another bounded linear operator on $\CH$, then the following bounds hold, assuming that the involved norms are finite:
\begin{equation}\label{eq:norms_products}
\| A B \|_1 \leq \| A \|_2 \| B \|_2, \qquad \| A B \|_2 \leq \| A \|_2 \| B \|_\op, \qquad \| A B \|_2 \leq \| A \|_\op \| B \|_2.
\end{equation}

The \emph{Fredholm determinant} of an integral trace-class operator $A$ on $\CH$ is defined by
\begin{equation}\label{eq:det}
\det(I + A)_{\CH} := 1 + \sum_{n = 1}^\infty \frac{1}{n!} \int_{\R} \cdots \int_{\R} \det[A(x_i, x_j)]_{i,j = 1}^n \d x_1 \cdots \d x_n,
\end{equation}
where $A : \R^2 \to \R$ is the integral kernel of the operator. One important property of the Fredholm determinant is that it is invariant under conjugation $A \mapsto \Gamma^{-1} K \Gamma$. For further properties of Fredholm determinants, we refer to \cite{Simon}.

We often work with operators $A(v)$, parametrized by a variable $v$. For these, we will use the following properties.

\begin{lem}\label{lem:det}
Let a family of operators $A(v)$, parametrized by some vectors $v \in \R^n$, converges in trace norm to $A(v_0) \in \CB_1(\CH)$ as $v \to v_0$. Then
\begin{enumerate}[label=\normalfont{(\arabic*)},ref=\arabic*]
\item $\lim_{v \to v_0}\tr(A(v)) = \tr(A(v_0))$.
\item\label{it:det_continuous} $\lim_{v \to v_0}\det(I + A(v))_{\CH} = \det(I + A(v_0))_{\CH}$.
\item If $I + A(v_0)$ and $I + A(v)$ are invertible for all $v$ close to $v_0$, then
\begin{equation}
(I + A(v))^{-1} \xrightarrow{v \to v_0} (I + A(v_0))^{-1} \quad \text{in}~\CB_1(\CH).
\end{equation}
\item\label{det:deriv} Let $g (v) = \prod_{i = 1}^n |v_i|$ and let there be an operator $\partial A(v_0) \in \CB_1(\CH)$ such that
\begin{equation}
\tfrac{1}{g (v - v_0)} \bigl(A(v) - A(v_0)\bigr) \xrightarrow{v \to v_0} \partial A(v_0) \quad \text{in}~\CB_1(\CH),
\end{equation}
then
\begin{align}\label{eq:Fredholm_deriv}
\tfrac{1}{g (v - v_0)} &\bigl( \det (I + A(v))_{\CH} - \det (I + A(v_0))_{\CH}\bigr) \\
&\qquad \xrightarrow{v \to v_0} \tr \bigl[(I + A(v_0))^{-1} \partial A(v_0) \bigr] \det (I + A(v_0))_{\CH}.
\end{align}
\end{enumerate}
\end{lem}

\noindent The first property follows from \cite[Theorem~3.1]{Simon}. The latter two follow from \cite[Corollary~5.2]{Simon} and \cite[Eq.~5.1]{Simon} respectively.


We make use of a formula concerning a rank-one perturbation of an operator. 
For $A, B \in \CB_1(\CH)$, where $B$ is rank-one, then we have
\begin{equation}\label{eq:Fredholm_rank1}
\det (I + A + B)_{\CH} = \bigl(I + \tr \bigl[(I+A)^{-1} B\bigr]\bigr) \det(I + A)_{\CH}.
\end{equation}
This identity is obtained using approximations of the kernels by finite-dimensional matrices, for which this rank-one perturbation identity cane be found in \cite[p.~475]{Meyer}. From this identity we get
\begin{equation}
\bigl| \det (I + A + B)_{\CH} - \det (I + A)_{\CH} \bigr| \leq \| (I+A)^{-1} \|_1 \| B\|_1 \bigl|\det(I + A)_{\CH}\bigr|, \label{eq:det_bound_trace}
\end{equation}
where we used \cite[Eq.~3.1]{Simon}, the fact that $\CB_1(\CH)$ is an ideal, and \cite[Eq.~3.6]{Simon}.

For any trace class operators $A$ and $B$, the following bounds hold \cite[Theorem~3.4]{Simon}:
\begin{equation}\label{eq:det_bound}
\bigl| \det (I + A)_{\CH} - \det (I + B)_{\CH} \bigr| \leq \| A - B\|_1 e^{\| A\|_1 + \|B\|_1 + 1} \leq \| A - B\|_1 e^{\| A - B\|_1 + 2 \|B\|_1 + 1}.
\end{equation}

\section{Bounds on functions}
\label{sec:kernels_bounds}

We provide several estimates on Airy function and the functions $\fT_{t,x}$ defined in \eqref{eq:fTdef}. These are used in the proof of Proposition~\ref{prop:densities}.

\begin{lem}\label{lem:S_bound}
For any integer $n \geq 0$, there exists a constant $C > 0$, such that
\begin{equation}\label{eq:S_bound}
\bigl|\fT^{(n)}_{1,x}(z)\bigr| \leq C e^{F(x, z)},
\end{equation}
where $F(x, z) := -\tfrac{1}{3} x^3 + x y - \tfrac{2}{3} (y \vee 0)^{3 / 2}$ with $y = x^2 - z$, and where $\fT^{(n)}_{1,x}(z)$ denotes the $n^{\text{th}}$ derivative with respect to $z$.
\end{lem}

\begin{proof}
We first derive a bound on the Airy function and its derivatives. The Airy function can be written in terms of the Bessel functions as
\begin{equation}\label{eq:Airy_function}
\Ai(z) = \frac{1}{\pi} \sqrt{\frac{z}{3}} K_{1/3} (\zeta), \qquad \Ai(-z) = \frac{\sqrt{z}}{3} \bigl(J_{-1/3} (\zeta) + J_{1/3} (\zeta) \bigr),
\end{equation}
where $z \geq 0$; $\zeta = \frac{2}{3} z^{\frac{3}{2}}$; $J$ is a Bessel function of the first kind; and $K$ is a modified Bessel function of the second kind \cite[Section~2.2.4]{AiryBook}. Then, from \cite[Eq.~10.29.5]{NIST:DLMF}, we have
\begin{equation}\label{eq:K_recursion}
K^{(n)}_{\nu}(z) = 2^{-n} \sum_{k = 0}^n (-1)^k \binom{n}{k} K_{\nu - n + 2 k}(z).
\end{equation}
Moreover, the function $|K_\nu(z)|$ is bounded by a constant multiple of $e^{- z}$ for $z \geq 0$, where the constant depends on $\nu$ (see \cite[Eq.~10.25.3]{NIST:DLMF}). To be more precise, we have the bound $|K_\nu(z)| \leq C z^{-1/2} e^{- z}$ for large $z \geq 0$. However, the slowly decaying factor $z^{-1/2}$ does not play any role in our analysis. Then, from \eqref{eq:Airy_function} and \eqref{eq:K_recursion}, we obtain
\begin{equation}\label{eq:Airy_bound1}
\bigl|\Ai^{(n)}(z)\bigr| \leq C(n) e^{- \frac{2}{3} z^{3/2}},
\end{equation}
for $n \geq 0$ and $z \geq 0$.

Now, we will bound the Airy function $\Ai(-z)$ for $z > 0$. Equation 10.6.7 in \cite{NIST:DLMF} implies that the functions $J$ satisfy relation \eqref{eq:K_recursion}. Moreover, from \cite[Eq.~10.7.8]{NIST:DLMF}, we readily conclude that $|J_\nu(z)|$ is bounded by a constant depending on $\nu$. A slightly stronger bound holds for large values of $z$: $|J_\nu(z)| \leq C z^{-1/2}$. As for the function $K_\nu$ we will ignore this slow decay. Combining \eqref{eq:Airy_function} with these properties of the functions $J$, we obtain
\begin{equation}\label{eq:Airy_bound2}
\bigl|\Ai^{(n)}(-z)\bigr| \leq C(n),
\end{equation}
for $n \geq 0$ and $z \geq 0$.

From \eqref{eq:Airy_bound1} and \eqref{eq:Airy_bound2}, we conclude that $\bigl|\Ai^{(n)}(z)\bigr| \leq C(n) e^{- \frac{2}{3} (z \vee 0)^{3/2}}$. Applying this bound to \eqref{eq:fTdef}, we arrive at \eqref{eq:S_bound}.
\end{proof}

Next we derive bounds on the $L^2$ norms of the derivatives of $\fT^{(n)}_{t,x}(z)$. The bound \eqref{eq:S_bound} implies that these functions are not integrable on $\R$, and we need to multiply them by a fast decaying function to gain integrability. This is exactly the reason for conjugation of the kernel in Proposition~\ref{prop:KPZfp}.

\begin{lem}\label{lem:S_integral_estimate}
For  $t > 0$, $L > 0$, $\bar L > 0$, $u \in \R$, $x \in \R$, and  integer $n \geq 0$,
\begin{equation}\label{eq:S_integral_estimate}
\bigl\| \bigl(\fT^{(n)}_{t,x} \Gamma_t\bigr) (u, \bigcdot) \bigr\|_{L^2} \leq C,
\end{equation}
for $|x| \leq L$ and $|u| \leq \bar L$, where the constant $C \geq 0$ depends on $n$, $t$, $L$ and $\bar L$. Here, $\fT^{(n)}_{t,x}(z)$ is the $n^{\text{th}}$ derivative with respect to $z$; the multiplicative operator $\Gamma_t$ is defined in \eqref{eq:Gamma_L}.
\end{lem}

\begin{proof}
The definition \eqref{eq:fTdef} yields $\fT_{t,x}(z) = t^{-1/3} \fT_{1, t^{-2/3} x}(t^{-1/3} z)$ for $t \neq 0$. This means that the bound \eqref{eq:S_integral_estimate} for any $t > 0$ follows from the same bound for $t = 1$. Recalling the definition of $\Gamma = \Gamma_1$ in \eqref{eq:Gamma_L}, we may write
\begin{equation}
\bigl\| \bigl(\fT^{(n)}_{1,x} \Gamma\bigr) (u, \bigcdot) \bigr\|^2_{L^2} = \int_{-\infty}^\infty \fT^{(n)}_{1,x} (u - v)^2 e^{2 G(v)} \d v.
\end{equation}
Applying then the bound \eqref{eq:S_bound}, we get
\begin{align}
\MoveEqLeft
\bigl\| \bigl(\fT^{(n)}_{1,x} \Gamma\bigr) (u, \bigcdot) \bigr\|^2_{L^2} \\
&\leq C_1 \int_{-\infty}^\infty \exp\left\{2 F(x,u - v) + 2 G(v)\right\} \d v \\
&= C_1 e^{\frac{4}{3} x^3 - 2 x u} \int_{-\infty}^\infty \exp\left\{- \frac{4}{3} ((x^2 - u + v) \vee 0)^{3/2} + 2 x v + 2 \kappa \sgn(v) |v|^{3/2}\right\} \d v,\label{eq:integral}
\end{align}
where the constant $C_1$ depends on $n$. In order to estimate this integral, we will split the interval of integration into two subintervals: $v \leq u - x^2$ and $v > u - x^2$.

If $v \leq u - x^2$ in the integral \eqref{eq:integral}, then, on this interval of integration, the expression \eqref{eq:integral} equals
\begin{equation}
C_1 e^{\frac{4}{3} x^3 - 2 x u} \int_{-\infty}^{(u - x^2) \wedge 0} \exp\left\{2 x v - 2 \kappa |v|^{3/2}\right\} \d v + C_1 e^{\frac{4}{3} x^3 - 2 x u} \int_{(u - x^2) \wedge 0}^{u - x^2} \exp\left\{2 x v + 2 \kappa |v|^{3/2}\right\} \d v.
\end{equation}
Naturally, for $|x| \leq L$ and $|u| \leq \bar L$, both these integrals are bounded by a constant depending on $L$ and $\bar L$.

In the case that $v > u - x^2$, the expression in \eqref{eq:integral} may be written as
\begin{equation}\label{eq:integral2}
\begin{split}
\MoveEqLeft
C_1 e^{\frac{4}{3} x^3 - 2 x u } \int_{u - x^2}^{(u - x^2) \vee 0} \exp\left\{- \frac{4}{3} (x^2 - u + v)^{3/2} + 2 x v - 2 \kappa |v|^{3/2}\right\} \d v \\
&+ C_1 e^{\frac{4}{3} x^3 - 2 x u} \int_{{(u - x^2) \vee 0}}^\infty \exp\left\{- \frac{4}{3} (x^2 - u + v)^{3/2} + 2 x v + 2 \kappa v^{3/2}\right\} \d v.
\end{split}
\end{equation}
For $|x| \leq L$ and $|u| \leq \bar L$, the first term in \eqref{eq:integral2} is bounded by a constant depending on $L$ and $\bar L$. Now, we will bound the second integral in \eqref{eq:integral2}. For this, we use the inequality
\begin{equation}
\tfrac{4}{3} (x^2 - u + v)^{\frac{3}{2}} \geq \alpha [(x^2 - u) \vee 0]^{\frac{3}{2}} + \bigl(\tfrac{4}{3} - \alpha\bigr) [v - (u - x^2) \vee 0]^{\frac{3}{2}},
\end{equation}
which holds for any $\alpha < \frac{4}{3}$. Then the second integral in \eqref{eq:integral2} is estimated by
\begin{align}
\MoveEqLeft
C_1 e^{\frac{4}{3} x^3 - 2 x u - \alpha [(x^2 - u) \vee 0]^{3/2}}\\
&\qquad\times \int_{{(u - x^2) \vee 0}}^\infty\exp\left\{- (\tfrac{4}{3} - \alpha) [v - (u - x^2) \vee 0]^{3/2} + 2 x v + 2 \kappa v^{3/2}\right\} \d v \\
&= C_1 e^{\frac{4}{3} x^3 - 2 x u - \alpha [(x^2 - u) \vee 0]^{3/2}}\\
&\qquad\times \int_{0}^\infty \exp\left\{- (\tfrac{4}{3} - \alpha) v^{3/2} + 2 x v + 2 \kappa (v + (u - x^2) \vee 0)^{3/2}\right\} \d v\\
&\leq C_1 e^{\frac{4}{3} x^3 - 2 x u - (\alpha - 2^{3/2} \kappa) [(x^2 - u) \vee 0]^{3/2}} \int_{0}^\infty \exp\left\{- (\tfrac{4}{3} - \alpha - 2^{3/2} \kappa) v^{3/2} + 2 x v\right\} \d v,
\end{align}
where in the last line we have used Jensen's inequality for the function $|\bigcdot|^{3/2}$. As soon as $\kappa > 0$ is small enough, the last integral is bounded by a constant depending on $x$, which implies that the last term in \eqref{eq:integral2} is bounded by a constant depending on $L$ and $\bar L$, if $|x| \leq L$ and $|u| \leq \bar L$.

Combining \eqref{eq:integral} with the derived bounds, we obtain the bound \eqref{eq:S_integral_estimate} for $t = 1$, as required.
\end{proof}

\printbibliography[heading=apa]

@book {Meyer,
    AUTHOR = {Meyer, Carl},
     TITLE = {Matrix analysis and applied linear algebra},
 PUBLISHER = {Society for Industrial and Applied Mathematics (SIAM),
              Philadelphia, PA},
      YEAR = {2000},
     PAGES = {xii+718},
      ISBN = {0-89871-454-0},
   MRCLASS = {15-01},
  MRNUMBER = {1777382},
       DOI = {10.1137/1.9780898719512},
       URL = {https://doi-org.ezproxy.cul.columbia.edu/10.1137/1.9780898719512},
}

@book {Hewitt,
    AUTHOR = {Hewitt, Edwin and Stromberg, Karl},
     TITLE = {Real and abstract analysis. {A} modern treatment of the theory
              of functions of a real variable},
    SERIES = {Second printing corrected},
 PUBLISHER = {Springer-Verlag, New York-Berlin},
      YEAR = {1969},
     PAGES = {viii+476},
   MRCLASS = {26.00 (28.00)},
  MRNUMBER = {0274666},
}

@book {Kallenberg,
    AUTHOR = {Kallenberg, Olav},
     TITLE = {Foundations of modern probability},
    SERIES = {Probability and its Applications (New York)},
   EDITION = {Second},
 PUBLISHER = {Springer-Verlag, New York},
      YEAR = {2002},
     PAGES = {xx+638},
      ISBN = {0-387-95313-2},
   MRCLASS = {60-01},
  MRNUMBER = {1876169},
MRREVIEWER = {Klaus D. Schmidt},
       DOI = {10.1007/978-1-4757-4015-8},
       URL = {https://doi-org.ezproxy.cul.columbia.edu/10.1007/978-1-4757-4015-8},
}

@article {SS2010,
    AUTHOR = {Schramm, Oded and Steif, Jeffrey E.},
     TITLE = {Quantitative noise sensitivity and exceptional times for
              percolation},
   JOURNAL = {Ann. of Math. (2)},
  FJOURNAL = {Annals of Mathematics. Second Series},
    VOLUME = {171},
      YEAR = {2010},
    NUMBER = {2},
     PAGES = {619--672},
      ISSN = {0003-486X},
   MRCLASS = {60K35 (82B43)},
  MRNUMBER = {2630053},
MRREVIEWER = {Anatoly Yambartsev},
       DOI = {10.4007/annals.2010.171.619},
       URL = {https://doi-org.ezproxy.cul.columbia.edu/10.4007/annals.2010.171.619},
}

@article {BKS1999,
    AUTHOR = {Benjamini, Itai and Kalai, Gil and Schramm, Oded},
     TITLE = {Noise sensitivity of {B}oolean functions and applications to
              percolation},
   JOURNAL = {Inst. Hautes \'{E}tudes Sci. Publ. Math.},
  FJOURNAL = {Institut des Hautes \'{E}tudes Scientifiques. Publications
              Math\'{e}matiques},
    NUMBER = {90},
      YEAR = {1999},
     PAGES = {5--43},
      ISSN = {0073-8301},
   MRCLASS = {60B15 (60K35 68Q15 82B43 94C10)},
  MRNUMBER = {1813223},
MRREVIEWER = {H. Kesten},
       URL = {http://www.numdam.org.ezproxy.cul.columbia.edu/item?id=PMIHES_1999__90__5_0},
}

@article {GPS2010,
    AUTHOR = {Garban, Christophe and Pete, G\'{a}bor and Schramm, Oded},
     TITLE = {The {F}ourier spectrum of critical percolation},
   JOURNAL = {Acta Math.},
  FJOURNAL = {Acta Mathematica},
    VOLUME = {205},
      YEAR = {2010},
    NUMBER = {1},
     PAGES = {19--104},
      ISSN = {0001-5962},
   MRCLASS = {60K35 (60F10)},
  MRNUMBER = {2736153},
MRREVIEWER = {Antal A. J\'{a}rai},
       DOI = {10.1007/s11511-010-0051-x},
       URL = {https://doi-org.ezproxy.cul.columbia.edu/10.1007/s11511-010-0051-x},
}

@article {grabiner1999brownian,
    AUTHOR = {Grabiner, David J.},
     TITLE = {Brownian motion in a {W}eyl chamber, non-colliding particles,
              and random matrices},
   JOURNAL = {Ann. Inst. H. Poincar\'{e} Probab. Statist.},
  FJOURNAL = {Annales de l'Institut Henri Poincar\'{e}. Probabilit\'{e}s et
              Statistiques},
    VOLUME = {35},
      YEAR = {1999},
    NUMBER = {2},
     PAGES = {177--204},
      ISSN = {0246-0203},
   MRCLASS = {60J65 (17B45 60K35)},
  MRNUMBER = {1678525},
MRREVIEWER = {Akihito Hora},
       DOI = {10.1016/S0246-0203(99)80010-7},
       URL = {https://doi-org.ezproxy.cul.columbia.edu/10.1016/S0246-0203(99)80010-7},
}

@article {hammond2017brownian,
    AUTHOR = {Hammond, Alan},
     TITLE = {Brownian regularity for the {A}iry line ensemble, and
              multi-polymer watermelons in {B}rownian last passage
              percolation},
   JOURNAL = {Mem. Amer. Math. Soc.},
  FJOURNAL = {Memoirs of the American Mathematical Society},
    VOLUME = {277},
      YEAR = {2022},
    NUMBER = {1363},
     PAGES = {v+133},
      ISSN = {0065-9266},
      ISBN = {978-1-4704-5229-2; 978-1-4704-7095-1},
   MRCLASS = {60K35 (82C22 82C43)},
  MRNUMBER = {4403929},
       DOI = {10.1090/memo/1363},
       URL = {https://doi.org/10.1090/memo/1363},
}

@article{dauvergne2018basic,
  title={Bulk properties of the {A}iry line ensemble},
  author={Dauvergne, Duncan and Vir{\'a}g, B{\'a}lint},
  Eprint={1812.00311},
  Eprinttype = {arxiv},
  year={2021+},
  journal={Ann. Probab., to appear},
}

@article {BatesGangulyHammond,
    AUTHOR = {Bates, Erik and Ganguly, Shirshendu and Hammond, Alan},
     TITLE = {Hausdorff dimensions for shared endpoints of disjoint
              geodesics in the directed landscape},
   JOURNAL = {Electron. J. Probab.},
  FJOURNAL = {Electronic Journal of Probability},
    VOLUME = {27},
      YEAR = {2022},
     PAGES = {Paper No. 1, 44},
   MRCLASS = {60K35 (28A78 60G15 60G57 60K37)},
  MRNUMBER = {4361743},
       DOI = {10.1214/21-ejp706},
       URL = {https://doi.org/10.1214/21-ejp706},
}

@article {corwinReview,
    AUTHOR = {Corwin, Ivan},
     TITLE = {The {K}ardar-{P}arisi-{Z}hang equation and universality class},
   JOURNAL = {Random Matrices Theory Appl.},
  FJOURNAL = {Random Matrices. Theory and Applications},
    VOLUME = {1},
      YEAR = {2012},
    NUMBER = {1},
     PAGES = {1130001, 76},
      ISSN = {2010-3263},
   MRCLASS = {82B31 (60B20 60K35 60K37)},
  MRNUMBER = {2930377},
       DOI = {10.1142/S2010326311300014},
       URL = {https://doi-org.ezproxy.cul.columbia.edu/10.1142/S2010326311300014},
}

@article {corwin2016,
    AUTHOR = {Corwin, Ivan and Liu, Zhipeng and Wang, Dong},
     TITLE = {Fluctuations of {TASEP} and {LPP} with general initial data},
   JOURNAL = {Ann. Appl. Probab.},
  FJOURNAL = {The Annals of Applied Probability},
    VOLUME = {26},
      YEAR = {2016},
    NUMBER = {4},
     PAGES = {2030--2082},
      ISSN = {1050-5164},
   MRCLASS = {60K35 (60B20 82B41 82C24 82C31)},
  MRNUMBER = {3543889},
MRREVIEWER = {Florent Benaych-Georges},
       DOI = {10.1214/15-AAP1139},
       URL = {https://doi-org.ezproxy.cul.columbia.edu/10.1214/15-AAP1139},
}

@article{hht,
  title={{A KPZ cocktail--shaken, not stirred: Toasting 30 years of kinetically roughened surfaces}},
  author={Halpin-Healy, Timothy and Takeuich, Kazumasa},
  journal={J. Stat. Phys.},
  volume={160},
  pages={794--814},
  year={2015},
}

@book {karatzas1998brownian,
    AUTHOR = {Karatzas, Ioannis and Shreve, Steven E.},
     TITLE = {Brownian motion and stochastic calculus},
    SERIES = {Graduate Texts in Mathematics},
    VOLUME = {113},
   EDITION = {Second},
 PUBLISHER = {Springer-Verlag, New York},
      YEAR = {1991},
     PAGES = {xxiv+470},
      ISBN = {0-387-97655-8},
   MRCLASS = {60J65 (35K99 35R60 60G44 60H10 60J60)},
  MRNUMBER = {1121940},
       DOI = {10.1007/978-1-4612-0949-2},
       URL = {https://doi-org.ezproxy.cul.columbia.edu/10.1007/978-1-4612-0949-2},
}

@article {o2002representation,
    AUTHOR = {O'Connell, Neil and Yor, Marc},
     TITLE = {A representation for non-colliding random walks},
   JOURNAL = {Electron. Comm. Probab.},
  FJOURNAL = {Electronic Communications in Probability},
    VOLUME = {7},
      YEAR = {2002},
     PAGES = {1--12},
      ISSN = {1083-589X},
   MRCLASS = {60J65 (15A52 60J27 60J45 60K25)},
  MRNUMBER = {1887169},
MRREVIEWER = {Joanna B. Mitro},
       DOI = {10.1214/ECP.v7-1042},
       URL = {https://doi-org.ezproxy.cul.columbia.edu/10.1214/ECP.v7-1042},
}

@book{revuz2013continuous,
  title={Continuous martingales and {B}rownian motion},
  author={Revuz, Daniel and Yor, Marc},
  volume={293},
  year={2013},
  publisher={Springer Science \& Business Media}
}

@book{morters2010brownian,
  title={Brownian motion},
  author={M{\"o}rters, Peter and Peres, Yuval},
  volume={30},
  year={2010},
  publisher={Cambridge University Press}
}

@article {basu2019fractal,
    AUTHOR = {Basu, Riddhipratim and Ganguly, Shirshendu and Hammond, Alan},
     TITLE = {Fractal geometry of {${\rm Airy}_2$} processes coupled via the
              {A}iry sheet},
   JOURNAL = {Ann. Probab.},
  FJOURNAL = {The Annals of Probability},
    VOLUME = {49},
      YEAR = {2021},
    NUMBER = {1},
     PAGES = {485--505},
      ISSN = {0091-1798},
   MRCLASS = {60K35 (28A80 60H15 82B43 82D60)},
  MRNUMBER = {4203343},
       DOI = {10.1214/20-AOP1444},
       URL = {https://doi.org/10.1214/20-AOP1444},
}

@misc{basu2014last,
  title={Last passage percolation with a defect line and the solution of the slow bond problem},
  author={Basu, Riddhipratim and Sidoravicius, Vladas and Sly, Allan},
  Eprint={1408.3464},
  Eprinttype = {arxiv},
  year={2014}
}

@article{LPPdynamics,
  title={Stability and chaos in dynamical last passage percolation},
  author={Ganguly, Shirshendu and Hammond, Alan M},
  journal={Forthcoming},
  year={2020},
}

@misc{LPPtools,
  title={The geometry of near ground states in {G}aussian polymer models},
  author={Ganguly, Shirshendu and Hammond, Alan},
  Eprint={2010.05836},
  Eprinttype = {arxiv},
  year={2020}
}

@misc{calvert2019brownian,
  title={Brownian structure in the {KPZ} fixed point},
  author={Calvert, Jacob and Hammond, Alan and Hegde, Milind},
  Eprint={1912.00992},
  Eprinttype = {arxiv},
  year={2019}
}

@misc{sarkar2020fixedpoint,
  title={Convergence of exclusion processes and {KPZ} equation to the {KPZ} fixed point},
  author={Quastel, Jeremy and Sarkar, Sourav},
  Eprint={2008.06584},
  Eprinttype = {arxiv},
  year={2020}
}

@misc{virag2020fixedpoint,
  title={The heat and the landscape {I}},
  author={Vir{\'a}g, B{\'a}lint},
  Eprint={2008.07241},
  Eprinttype = {arxiv},
  year={2020}
}

@article{sarkar2020brownian,
  title={Brownian absolute continuity of the {KPZ} fixed point with arbitrary initial condition},
  author={Sarkar, Sourav and Vir{\'a}g, B{\'a}lint},
  Eprint={2002.08496},
  Eprinttype = {arxiv},
  year={2020},
  journal={Ann. Probab., to appear}
}

@article{millar1978path,
  title={A path decomposition for {M}arkov processes},
  author={Millar, PW},
  journal={Ann. Probab.},
  volume={6},
  number={2},
  pages={345--348},
  year={1978},
  publisher={Institute of Mathematical Statistics}
}

@misc{NIST:DLMF,
         key = "{DLMF}",
       title = "{NIST Digital Library of Mathematical Functions}",
howpublished = "http://dlmf.nist.gov/, Release 1.0.26 of 2020-03-15",
         url = "http://dlmf.nist.gov/",
        note = "F.~W.~J. Olver, A.~B. {Olde Daalhuis}, D.~W. Lozier, B.~I. Schneider,
                R.~F. Boisvert, C.~W. Clark, B.~R. Miller, B.~V. Saunders,
                H.~S. Cohl, and M.~A. McClain, eds."}

@article{Schehr,
	Author = {Schehr, Gr{\'e}gory},
	Da = {2012/11/01},
	Doi = {10.1007/s10955-012-0593-8},
	Isbn = {1572-9613},
	Journal = {J. Stat. Phys.},
	Number = {3},
	Pages = {385--410},
	Title = {Extremes of {$N$} Vicious Walkers for Large {$N$}: Application to the Directed Polymer and {KPZ} Interfaces},
	Ty = {JOUR},
	Url = {https://doi.org/10.1007/s10955-012-0593-8},
	Volume = {149},
	Year = {2012}}

@article{BLS,
author = {Baik,Jinho  and Liechty,Karl  and Schehr,Grégory},
title = {On the joint distribution of the maximum and its position of the {A}iry$_2$ process minus a parabola},
journal = {J. Math. Phys.},
volume = {53},
number = {8},
year = {2012},
doi = {10.1063/1.4746694},

URL = {
        https://doi.org/10.1063/1.4746694

},
}

@article {MR3300961,
    AUTHOR = {Quastel, Jeremy and Remenik, Daniel},
     TITLE = {Tails of the endpoint distribution of directed polymers},
   JOURNAL = {Ann. Inst. Henri Poincar\'{e} Probab. Stat.},
  FJOURNAL = {Annales de l'Institut Henri Poincar\'{e} Probabilit\'{e}s et
              Statistiques},
    VOLUME = {51},
      YEAR = {2015},
    NUMBER = {1},
     PAGES = {1--17},
      ISSN = {0246-0203},
   MRCLASS = {60K35 (82C23)},
  MRNUMBER = {3300961},
MRREVIEWER = {Olga Leonidovna Izyumtseva},
       DOI = {10.1214/12-AIHP525},
       URL = {https://doi-org.ezproxy.cul.columbia.edu/10.1214/12-AIHP525},
}

@article {Dynkin,
    AUTHOR = {Dynkin, E. and Jushkevich, A.},
     TITLE = {Strong {M}arkov processes},
   JOURNAL = {Teor. Veroyatnost. i Primenen.},
  FJOURNAL = {Akademija Nauk SSSR. Teorija Verojatnoste\u{\i} i ee Primenenija},
    VOLUME = {1},
      YEAR = {1956},
     PAGES = {149--155},
      ISSN = {0040-361x},
   MRCLASS = {60.0X},
  MRNUMBER = {0088103},
MRREVIEWER = {K. L. Chung},
}

@book {BaikBook,
    AUTHOR = {Baik, Jinho and Deift, Percy and Suidan, Toufic},
     TITLE = {Combinatorics and random matrix theory},
    SERIES = {Graduate Studies in Mathematics},
    VOLUME = {172},
 PUBLISHER = {American Mathematical Society, Providence, RI},
      YEAR = {2016},
     PAGES = {xi+461},
      ISBN = {978-0-8218-4841-8},
   MRCLASS = {60B20 (30E25 33E17 41A60 47B35 82C23)},
  MRNUMBER = {3468920},
MRREVIEWER = {Terence Tao},
}

@article {NQR20,
    AUTHOR = {Nica, Mihai and Quastel, Jeremy and Remenik, Daniel},
     TITLE = {One-sided reflected {B}rownian motions and the {KPZ} fixed
              point},
   JOURNAL = {Forum Math. Sigma},
  FJOURNAL = {Forum of Mathematics. Sigma},
    VOLUME = {8},
      YEAR = {2020},
     PAGES = {Paper No. e63, 16},
   MRCLASS = {60K35 (82C43)},
  MRNUMBER = {4190063},
MRREVIEWER = {Shuta Nakajima},
       DOI = {10.1017/fms.2020.56},
       URL = {https://doi.org/10.1017/fms.2020.56},
}

@book {RM,
    AUTHOR = {Anderson, Greg W. and Guionnet, Alice and Zeitouni, Ofer},
     TITLE = {An introduction to random matrices},
    SERIES = {Cambridge Studies in Advanced Mathematics},
    VOLUME = {118},
 PUBLISHER = {Cambridge University Press, Cambridge},
      YEAR = {2010},
     PAGES = {xiv+492},
      ISBN = {978-0-521-19452-5},
   MRCLASS = {60B20 (46L53 46L54)},
  MRNUMBER = {2760897},
MRREVIEWER = {Terence Tao},
}

@book {AiryBook,
    AUTHOR = {Vall\'{e}e, Olivier and Soares, Manuel},
     TITLE = {Airy functions and applications to physics},
 PUBLISHER = {Imperial College Press, London},
      YEAR = {2010},
     PAGES = {x+202},
      ISBN = {978-1-84816-548-9; 1-84816-548-X},
   MRCLASS = {33C10 (33-02 81S30)},
  MRNUMBER = {2722693},
       DOI = {10.1142/p709},
       URL = {https://doi.org/10.1142/p709},
}

@incollection {1710.02635,
    AUTHOR = {Quastel, Jeremy and Matetski, Konstantin},
     TITLE = {From the totally asymmetric simple exclusion process to the
              {KPZ} fixed point},
 BOOKTITLE = {Random matrices},
    SERIES = {IAS/Park City Math. Ser.},
    VOLUME = {26},
     PAGES = {251--301},
 PUBLISHER = {Amer. Math. Soc., Providence, RI},
      YEAR = {2019},
   MRCLASS = {60K35 (82C27)},
  MRNUMBER = {3971157},
}

@misc{Landscape,
Author = {Duncan Dauvergne and Janosch Ortmann and Bálint Vir\'{a}g},
Title = {The directed landscape},
Year = {2018},
Eprint = {1812.00309},
Eprinttype = {arxiv},
}

@book {Mattila,
    AUTHOR = {Mattila, Pertti},
     TITLE = {Geometry of sets and measures in {E}uclidean spaces},
    SERIES = {Cambridge Studies in Advanced Mathematics},
    VOLUME = {44},
      NOTE = {Fractals and rectifiability},
 PUBLISHER = {Cambridge University Press, Cambridge},
      YEAR = {1995},
     PAGES = {xii+343},
      ISBN = {0-521-46576-1; 0-521-65595-1},
   MRCLASS = {28A75 (49Q20)},
  MRNUMBER = {1333890},
MRREVIEWER = {Harold Parks},
       DOI = {10.1017/CBO9780511623813},
       URL = {https://doi.org/10.1017/CBO9780511623813},
}

@article {JohanssonPoly,
    AUTHOR = {Johansson, Kurt},
     TITLE = {Discrete polynuclear growth and determinantal processes},
   JOURNAL = {Comm. Math. Phys.},
  FJOURNAL = {Communications in Mathematical Physics},
    VOLUME = {242},
      YEAR = {2003},
    NUMBER = {1-2},
     PAGES = {277--329},
      ISSN = {0010-3616},
   MRCLASS = {82C22 (60F17 60K35)},
  MRNUMBER = {2018275},
MRREVIEWER = {Timo Sepp\"{a}l\"{a}inen},
       DOI = {10.1007/s00220-003-0945-y},
       URL = {https://doi.org/10.1007/s00220-003-0945-y},
}

@article {BrownianGibbs,
    AUTHOR = {Corwin, Ivan and Hammond, Alan},
     TITLE = {Brownian {G}ibbs property for {A}iry line ensembles},
   JOURNAL = {Invent. Math.},
  FJOURNAL = {Inventiones Mathematicae},
    VOLUME = {195},
      YEAR = {2014},
    NUMBER = {2},
     PAGES = {441--508},
      ISSN = {0020-9910},
   MRCLASS = {60K35 (60B05 60F05 60J65)},
  MRNUMBER = {3152753},
       DOI = {10.1007/s00222-013-0462-3},
       URL = {https://doi.org/10.1007/s00222-013-0462-3},
}

@article {Takeuchi,
    AUTHOR = {Takeuchi, Kazumasa A.},
     TITLE = {An appetizer to modern developments on the
              {K}ardar-{P}arisi-{Z}hang universality class},
   JOURNAL = {Phys. A},
  FJOURNAL = {Physica A. Statistical Mechanics and its Applications},
    VOLUME = {504},
      YEAR = {2018},
     PAGES = {77--105},
      ISSN = {0378-4371},
   MRCLASS = {82B05},
  MRNUMBER = {3805500},
       DOI = {10.1016/j.physa.2018.03.009},
       URL = {https://doi.org/10.1016/j.physa.2018.03.009},
}

@article {MaxOfAiry2,
    AUTHOR = {Flores, Gregorio Moreno and Quastel, Jeremy and Remenik,
              Daniel},
     TITLE = {Endpoint distribution of directed polymers in {$1+1$}
              dimensions},
   JOURNAL = {Comm. Math. Phys.},
  FJOURNAL = {Communications in Mathematical Physics},
    VOLUME = {317},
      YEAR = {2013},
    NUMBER = {2},
     PAGES = {363--380},
      ISSN = {0010-3616},
   MRCLASS = {60K35 (60B20 82C22 82C44 82D60)},
  MRNUMBER = {3010188},
MRREVIEWER = {Dimitri Petritis},
       DOI = {10.1007/s00220-012-1583-z},
       URL = {https://doi.org/10.1007/s00220-012-1583-z},
}

@book {Simon,
    AUTHOR = {Simon, Barry},
     TITLE = {Trace ideals and their applications},
    SERIES = {Mathematical Surveys and Monographs},
    VOLUME = {120},
   EDITION = {Second},
 PUBLISHER = {American Mathematical Society, Providence, RI},
      YEAR = {2005},
     PAGES = {viii+150},
      ISBN = {0-8218-3581-5},
   MRCLASS = {47L20 (47A40 47A55 47B10 47B36 47E05 81Q15 81U99)},
  MRNUMBER = {2154153},
MRREVIEWER = {Pavel B. Kurasov},
}

@article {flat,
    AUTHOR = {Quastel, Jeremy and Remenik, Daniel},
     TITLE = {How flat is flat in random interface growth?},
   JOURNAL = {Trans. Amer. Math. Soc.},
  FJOURNAL = {Transactions of the American Mathematical Society},
    VOLUME = {371},
      YEAR = {2019},
    NUMBER = {9},
     PAGES = {6047--6085},
      ISSN = {0002-9947},
   MRCLASS = {60K35 (35K59 35R60 82C21)},
  MRNUMBER = {3937318},
       DOI = {10.1090/tran/7338},
       URL = {https://doi.org/10.1090/tran/7338},
}

@article{prahoferSpohn,
	Author = {Pr{\"a}hofer, Michael and Spohn, Herbert},
	Fjournal = {Journal of Statistical Physics},
	Journal = {J.~Stat.~Phys.},
	Number = {5-6},
	Pages = {1071--1106},
	Title = {Scale invariance of the {PNG} droplet and the {A}iry process},
	Volume = {108},
	Year = {2002}}

@article{tracyWidom,
	Author = {Tracy, Craig A. and Widom, Harold},
	Fjournal = {Communications in Mathematical Physics},
	Journal = {Comm. Math. Phys.},
	Number = {1},
	Pages = {151--174},
	Title = {Level-spacing distributions and the {A}iry kernel},
	Volume = {159},
	Year = {1994}}

@book{abrSteg,
	Author = {Abramowitz, Milton and Stegun, Irene A.},
	Pages = {xiv+1046},
	Publisher = {National Bureau of Standards Applied Mathematics Series},
	Title = {Handbook of mathematical functions with formulas, graphs, and mathematical tables},
	Volume = {55},
	Year = {1964}}

@article{tracyWidom2,
	Author = {Tracy, Craig A. and Widom, Harold},
	Fjournal = {Communications in Mathematical Physics},
	Journal = {Comm. Math. Phys.},
	Number = {3},
	Pages = {727--754},
	Title = {On orthogonal and symplectic matrix ensembles},
	Volume = {177},
	Year = {1996}}

@article{cqr,
	Author = {Corwin, Ivan and Quastel, Jeremy and Remenik, Daniel},
	Doi = {10.1007/s00220-012-1582-0},
	Fjournal = {Communications in Mathematical Physics},
	Issn = {0010-3616},
	Journal = {Comm. Math. Phys.},
	Number = {2},
	Pages = {347-362},
	Publisher = {Springer-Verlag},
	Title = {Continuum Statistics of the {A}iry$_2$ Process},
	Url = {http://dx.doi.org/10.1007/s00220-012-1582-0},
	Volume = {317},
	Year = {2013}}

@article{bfp,
	Author = {Borodin, Alexei and Ferrari, Patrik L. and Pr{\"a}hofer, Michael},
	Fjournal = {International Mathematics Research Papers. IMRP},
	Journal = {Int. Math. Res. Pap. IMRP},
	Pages = {Art. ID rpm002, 47},
	Title = {Fluctuations in the discrete {TASEP} with periodic initial configurations and the {${\rm Airy}_1$} process},
	Year = {2007}}

@incollection{quastelCDM,
	Author = {Quastel, Jeremy},
	Booktitle = {Current developments in mathematics, 2011},
	Publisher = {Int. Press, Somerville, MA},
	Title = {The {K}ardar-{P}arisi-{Z}hang equation},
	Year = {2011}}

@book{mehta,
	Address = {Boston, MA},
	Author = {Mehta, Madan Lal},
	Edition = {Second},
	Isbn = {0-12-488051-7},
	Mrclass = {82-02 (15A52 60B99 60K35 82B41)},
	Mrnumber = {1083764 (92f:82002)},
	Mrreviewer = {B. S. Nahapetian},
	Pages = {xviii+562},
	Publisher = {Academic Press Inc.},
	Title = {Random matrices},
	Year = {1991}}

@article{cqrFixedPt,
	Author = {Corwin, Ivan and Quastel, Jeremy and Remenik, Daniel},
	Doi = {10.1007/s10955-015-1243-8},
	Fjournal = {Journal of Statistical Physics},
	Issn = {0022-4715},
	Journal = {J. Stat. Phys.},
	Language = {English},
	Number = {4},
	Pages = {815-834},
	Publisher = {Springer US},
	Title = {Renormalization Fixed Point of the {KPZ} Universality Class},
	Url = {http://dx.doi.org/10.1007/s10955-015-1243-8},
	Volume = {160},
	Year = {2015}}

@incollection{quastelRem-review,
	Author = {Quastel, Jeremy and Remenik, Daniel},
	Booktitle = {Topics in Percolative and Disordered Systems},
	Pages = {121-171},
	Series = {Springer Proceedings in Mathematics \& Statistics},
	Title = {Airy processes and variational problems},
	Volume = 69,
	Year = 2014}

@article {fixedpt,
    AUTHOR = {Matetski, Konstantin and Quastel, Jeremy and Remenik, Daniel},
     TITLE = {The {KPZ} fixed point},
   JOURNAL = {Acta Math.},
  FJOURNAL = {Acta Mathematica},
    VOLUME = {227},
      YEAR = {2021},
    NUMBER = {1},
     PAGES = {115--203},
      ISSN = {0001-5962},
   MRCLASS = {60K35},
  MRNUMBER = {4346267},
       DOI = {10.4310/acta.2021.v227.n1.a3},
       URL = {https://doi.org/10.4310/acta.2021.v227.n1.a3},
}

@ article{TASEPandOthers,
	Author = {Matetski, Konstantin and Remenik, Daniel},
	Title = {{TASEP} and generalizations: {M}ethod for exact solution},
	Note={To appear},
	Shorthand = {MQR+},
	Year = {2022},
	JOURNAL = {Probab. Theory Related Fields},
  	FJOURNAL = {Probability Theory and Related Fields},
      ISSN = {1432-2064},
     CODEN = {NUPBBO},
       DOI = {10.1007/s00440-022-01129-w},
       URL = {https://doi.org/10.1007/s00440-022-01129-w},
}

@article {pimentel_2014,
    AUTHOR = {Pimentel, Leandro P. R.},
     TITLE = {On the location of the maximum of a continuous stochastic
              process},
   JOURNAL = {J. Appl. Probab.},
  FJOURNAL = {Journal of Applied Probability},
    VOLUME = {51},
      YEAR = {2014},
    NUMBER = {1},
     PAGES = {152--161},
      ISSN = {0021-9002},
   MRCLASS = {60G17 (60G10 60J65)},
  MRNUMBER = {3189448},
MRREVIEWER = {Martynas Manstavi\v{c}ius},
       DOI = {10.1239/jap/1395771420},
       URL = {https://doi-org.ezproxy.cul.columbia.edu/10.1239/jap/1395771420},
}

@Article{quastelSpohn,
author="Quastel, Jeremy and Spohn, Herbert",
title="The One-Dimensional {KPZ} Equation and Its Universality Class",
journal="J. Stat. Phys.",
fjournal="Journal of Statistical Physics",
year="2015",
month="Aug",
day="01",
volume="160",
number="4",
pages="965--984",

issn="1572-9613",
doi="10.1007/s10955-015-1250-9",
url="https://doi.org/10.1007/s10955-015-1250-9"
}

@article{borodinPetrov,
author = "Borodin, Alexei and Petrov, Leonid",
doi = "10.1214/13-PS225",
fjournal = "Probability Surveys",
journal = "Probab. Surveys",
pages = "1--58",
publisher = "The Institute of Mathematical Statistics and the Bernoulli Society",
title = "Integrable probability: {F}rom representation theory to {M}acdonald processes",
url = "https://doi.org/10.1214/13-PS225",
volume = "11",
year = "2014"
}

@article {karlinMcGregor,
    AUTHOR = {Karlin, Samuel and McGregor, James},
     TITLE = {Coincidence probabilities},
   JOURNAL = {Pacific J. Math.},
  FJOURNAL = {Pacific Journal of Mathematics},
    VOLUME = {9},
      YEAR = {1959},
     PAGES = {1141--1164},
      ISSN = {0030-8730},
   MRCLASS = {60.00},
  MRNUMBER = {114248},
MRREVIEWER = {F. L. Spitzer},
       URL = {http://projecteuclid.org/euclid.pjm/1103038889},
}

@incollection{noumi2002tropical,
    AUTHOR = {Noumi, Masatoshi and Yamada, Yasuhiko},
     TITLE = {Tropical {R}obinson-{S}chensted-{K}nuth correspondence and
              birational {W}eyl group actions},
 BOOKTITLE = {Representation theory of algebraic groups and quantum groups},
    SERIES = {Adv. Stud. Pure Math.},
    VOLUME = {40},
     PAGES = {371--442},
 PUBLISHER = {Math. Soc. Japan, Tokyo},
      YEAR = {2004},
   MRCLASS = {05E10 (20F55)},
  MRNUMBER = {2074600},
       DOI = {10.2969/aspm/04010371},
       URL = {https://doi.org/10.2969/aspm/04010371},
}

@incollection {corwin2020invariance,
    AUTHOR = {Corwin, Ivan},
     TITLE = {Invariance of polymer partition functions under the geometric
              {RSK} correspondence},
 BOOKTITLE = {Stochastic analysis, random fields and integrable
              probability---{F}ukuoka 2019},
    SERIES = {Adv. Stud. Pure Math.},
    VOLUME = {87},
     PAGES = {89--136},
 PUBLISHER = {Math. Soc. Japan, Tokyo},
      YEAR = {2021},
   MRCLASS = {05E05},
  MRNUMBER = {4397419},
}

@article {dauvergne2020hidden,
    AUTHOR = {Dauvergne, Duncan},
     TITLE = {Hidden invariance of last passage percolation and directed
              polymers},
   JOURNAL = {Ann. Probab.},
  FJOURNAL = {The Annals of Probability},
    VOLUME = {50},
      YEAR = {2022},
    NUMBER = {1},
     PAGES = {18--60},
      ISSN = {0091-1798},
   MRCLASS = {60K35 (05A15)},
  MRNUMBER = {4385122},
       DOI = {10.1214/21-aop1527},
       URL = {https://doi.org/10.1214/21-aop1527},
}

@article {caffarelli1982partial,
    AUTHOR = {Caffarelli, L. and Kohn, R. and Nirenberg, L.},
     TITLE = {Partial regularity of suitable weak solutions of the
              {N}avier-{S}tokes equations},
   JOURNAL = {Comm. Pure Appl. Math.},
  FJOURNAL = {Communications on Pure and Applied Mathematics},
    VOLUME = {35},
      YEAR = {1982},
    NUMBER = {6},
     PAGES = {771--831},
      ISSN = {0010-3640},
   MRCLASS = {35Q10 (76D05)},
  MRNUMBER = {673830},
MRREVIEWER = {Tai Ping Liu},
       DOI = {10.1002/cpa.3160350604},
       URL = {https://doi.org/10.1002/cpa.3160350604},
}

@article {dyson,
    AUTHOR = {Dyson, Freeman J.},
     TITLE = {A {B}rownian-motion model for the eigenvalues of a random
              matrix},
   JOURNAL = {J. Math. Phys.},
  FJOURNAL = {Journal of Mathematical Physics},
    VOLUME = {3},
      YEAR = {1962},
     PAGES = {1191--1198},
      ISSN = {0022-2488},
   MRCLASS = {81.60},
  MRNUMBER = {148397},
MRREVIEWER = {G. K\"{a}ll\'{e}n},
       DOI = {10.1063/1.1703862},
       URL = {https://doi.org/10.1063/1.1703862},
}

@article {dimitrov2020characterization,
    AUTHOR = {Dimitrov, Evgeni and Matetski, Konstantin},
     TITLE = {Characterization of {B}rownian {G}ibbsian line ensembles},
   JOURNAL = {Ann. Probab.},
  FJOURNAL = {The Annals of Probability},
    VOLUME = {49},
      YEAR = {2021},
    NUMBER = {5},
     PAGES = {2477--2529},
      ISSN = {0091-1798},
   MRCLASS = {82C22 (60J65)},
  MRNUMBER = {4317710},
MRREVIEWER = {Utkir Rozikov},
       DOI = {10.1214/21-aop1513},
       URL = {https://doi.org/10.1214/21-aop1513},
}

@article{dauvergne2022non,
  title={Non-uniqueness times for the maximizer of the {KPZ} fixed point},
  author={Dauvergne, Duncan},
  journal={arXiv preprint arXiv:2202.01700},
  year={2022}
}

@article{dauvergne2021scaling,
  title={The scaling limit of the longest increasing subsequence},
  author={Dauvergne, Duncan and Vir{\'a}g, B{\'a}lint},
  journal={arXiv preprint arXiv:2104.08210},
  year={2021}
}

@article{warren2007dyson,
  title={Dyson's {B}rownian motions, intertwining and interlacing},
  author={Warren, Jon},
  journal={Electronic Journal of Probability},
  volume={12},
  pages={573--590},
  year={2007},
  publisher={Institute of Mathematical Statistics and Bernoulli Society}
}

\end{document}